\documentclass{article} % For LaTeX2e
\usepackage{iclr2023_conference,times}

\usepackage{hyperref}
\usepackage{url}
\usepackage[toc,page,header]{appendix}
\usepackage{minitoc}

\title{Phase transition for detecting  a small  \\  community in a large network}

\author{\qquad \quad \quad  \,  \textbf{Jiashun Jin} \\
\, \quad \, Carnegie Mellon University \\
\quad \,  \texttt{jiashun@stat.cmu.edu} 
\and 
\hspace{-7 cm} \textbf{Zheng Tracy Ke} \\
\hspace{-7 cm} Harvard University \\ 
\hspace{-7 cm} \texttt{zke@fas.harvard.edu} \\
\AND \\
\qquad \qquad \textbf{Paxton Turner} \\ 
\qquad \quad \,  Harvard University \\
\texttt{paxtonturner@fas.harvard.edu} \\
\and \\
\textbf{Anru R. Zhang} \\
Duke University  \\
\texttt{anru.zhang@duke.edu}
}

\usepackage{url}            % simple URL typesetting
\usepackage{booktabs}       % professional-quality tables
\usepackage{amsfonts}       % blackboard math symbols
\usepackage{nicefrac}       % compact symbols for 1/2, etc.
\usepackage{microtype}      % microtypography
\usepackage{xcolor}         % colors
\usepackage{wrapfig}
\usepackage{amsmath}
\usepackage{graphicx}
\usepackage{enumerate, enumitem}
\usepackage{amsthm,amsmath,amsfonts,amssymb}
\usepackage{caption}
\usepackage{subcaption}

\usepackage[normalem]{ulem}

\numberwithin{equation}{section}
\theoremstyle{plain}
\newtheorem{thm}{Theorem}[section] %(If you want theorem numbered
\newtheorem{lemma}{Lemma}[section] %%    with section number.
\newtheorem{cor}{Corollary}[section]
\newtheorem{prop}{Proposition}[section]

\newtheorem{defi}{Definition}[section]
\newtheorem{conjecture}{Conjecture}[section]
\newcommand{\bed}{\begin{defi}}
\newcommand{\eed}{\end{defi}}
\newtheorem*{lemma*}{Lemma}
\newtheorem*{thm*}{Theorem}
\newtheorem*{cor*}{Corollary}
\newtheorem*{lemma**}{\sout{Lemma}}
\newtheorem*{thm**}{\sout{Theorem}}
\newtheorem*{cor**}{\sout{Corollary}}

\newcommand{\eps}{\varepsilon}

\newcommand{\bitem}{\begin{itemize}}
\newcommand{\eitem}{\end{itemize}}

\newcommand{\goto}{\rightarrow}

\newcommand{\beqn}{\begin{equation}}
\newcommand{\eeqn}{\end{equation}}
\newcommand{\balign}{\begin{align}}
\newcommand{\ealign}{\end{align}}

\newcommand{\s}{\sigma}

\newcommand{\tr}{\mathrm{tr}}
\usepackage{amssymb}
\newcommand{\beq}{\begin{equation}}
\newcommand{\eeq}{\end{equation}}

\newcommand{\diag}{\mathrm{diag}}

\allowdisplaybreaks
\usepackage{xr}
\newcommand{\vp}{\varphi} %varphi
\newcommand{\E}{\mathbb{E}} %expectation
\newcommand{\mb}[1]{\mathbb{#1}} %blackboard bold
\newcommand{\mf}[1]{\mathbf{#1}} %mathbold 
\newcommand{\mc}[1]{\mathcal{#1}} %mathcal 
\newcommand{\rp}[1]{^{(#1)}} %raise to parentheses
\newcommand{\ti}{\tilde } %tilde 
\newcommand{\p}{\mb{P}}
\newcommand{\R}{\mb{R}}
\newcommand{\var}{\mathrm{Var}}
\newcommand{\cov}{\mathrm{Cov}}

\newcommand{\T}{\mathsf{T}}

\newcommand{\teta}{\tilde{\eta}} 
\newcommand\num{\addtocounter{equation}{1}\tag{\theequation}}
\newcommand{\les}{\lesssim} 
\newcommand{\oO}{\overline{\Omega}} 
\newcommand{\gam}{\beta \circ \theta}

\iclrfinalcopy 
\begin{document}

\maketitle

	\vspace{-0.5cm}

\begin{abstract}
	How to detect a small  community  in a large network is an interesting problem, including clique detection as a special case, where a naive degree-based $\chi^2$-test was shown to be powerful   in the presence of an Erd\H{o}s-Renyi background. Using Sinkhorn's theorem, we show that the signal captured by the $\chi^2$-test may be a modeling artifact, and it may disappear once we replace the Erd\H{o}s-Renyi model by a broader network model.  We show that the recent SgnQ test is more appropriate for such a setting. The test is optimal in detecting communities with sizes comparable to the whole network, but has never been studied for our setting, which is substantially different and more challenging.   Using a degree-corrected block model (DCBM), we establish phase transitions of this testing problem concerning the size of the small community and the edge densities in small and large communities.   When the size of the small community is larger than $\sqrt{n}$, the SgnQ test is optimal  for it attains the computational lower bound (CLB), the information lower bound for methods allowing polynomial computation time. When the size of the small community is smaller than $\sqrt{n}$, we establish the parameter regime where the SgnQ test has full power and make some conjectures of the CLB. We also study the classical information lower bound (LB) and show that there is always a gap between the CLB and LB in our range of interest. 	
\end{abstract}

\doparttoc % Tell to minitoc to generate a toc for the parts
\faketableofcontents % Run a fake tableofcontents command for the partocs

%\part{} % Start the document part
%\parttoc % Insert the document TOC

\vspace{-0.4cm}

\section{Introduction} \label{sec:intro} 

Consider an undirected network with $n$ nodes and $K$ communities. 
We assume $n$ is large and the network is connected for convenience. 
We are interested in testing whether $K = 1$ or $K > 1$ and the sizes 
of some of the communities are much smaller than $n$  (communities are scientifically meaningful but mathematically hard to define;  intuitively, they are clusters of nodes that have more edges ``within" than ``across" \citep{Jin2015, zhao2012consistency}). 
The problem is a special case of network global testing, a topic that has 
received a lot of attention (e.g., \cite{JKL2018, JKL2019}).  However, 
existing works focused on the so-called {\it balanced case}, 
where the sizes of communities are at the same order. Our case is 
{\it severely unbalanced}, where the sizes of some communities are much 
smaller than $n$ (e.g., $n^{\eps}$).  %This case is substantially different from that 
%in \cite{JKL2018, JKL2019},  so the results there do not directly apply.  

The problem also includes 
clique detection (a problem of primary interest in graph learning  \citep{alon1998finding, ron2010finding}) as a special case.  
Along this line, \cite{Ery1, Ery0} have made 
remarkable progress. In detail,  they considered the problem of testing whether a 
graph is generated from a  one-parameter  Erd\H{o}s-Renyi  model or 
a two-parameter model: %(sometimes referred to as planted dense subgraph \cite{hajek2015computational})
for any nodes $1 \leq i, j \leq n$, the probability that they have an edge equals $b$ if $i, j$ both are in a small planted subset and equals $a$ otherwise. 
%	The idea was further extended by \cite{Ery2}   to 
%	a more general case. 
A remarkable conclusion of  these papers %\cite{Ery1, Ery0} 
is: a naive degree-based  $\chi^2$-test is optimal, provided that the clique size is in a certain range.  
Therefore, at first glance, it seems that the problem has been elegantly solved, at least to some extent. 

Unfortunately, recent progress in network testing tells a very different story: 
the signal captured by the $\chi^2$-test may be a modeling artifact. It 
may disappear once we replace the models in \cite{Ery1, Ery0} 
by a properly broader model. When this happens,  the $\chi^2$-test will be asymptotically powerless in the whole range of parameter space.   

We explain the idea with the popular {\it Degree-Corrected Block Model (DCBM)} \citep{karrer2011stochastic}, though it is valid in broader settings.  Let $A\in\mathbb{R}^{n,n}$ be the network adjacency matrix, where $A(i,j)\in\{0,1\}$ indicates whether there is an edge between nodes $i$ and $j$, $1\leq i, j\leq n$. By convention, we do not allow for self-edges, so the diagonals of $A$ are always 0.  
Suppose there are $K$ communities, ${\cal C}_1, \ldots, {\cal C}_K$.  For each node $i$, $1 \leq i \leq n$, we use a parameter $\theta_i$ to model the degree heterogeneity and $\pi_i$ to model the membership:  
when $i \in {\cal C}_k$,  $\pi_i(\ell) = 1$ if $\ell = k$ and $\pi_i(\ell) = 0$ otherwise.  
For a $K \times K$ symmetric and irreducible non-negative matrix $P$ that models the community structure, DCBM  assumes that  the upper triangle of $A$ contains independent Bernoulli random variables satisfying\footnote{In this work we use $M'$ to denote the transpose of a matrix or vector $M$.}
\begin{equation} \label{Model-1} 
	\mathbb{P}(A(i, j) = 1)  = \theta_i \theta_j \pi_i' P \pi_j, \qquad 1 \leq i, j \leq n.  
\end{equation} 
In practice, we interpret $P(k, \ell)$  
as the baseline connecting probability between communities $k$ and $\ell$.  
Write $\theta = (\theta_1, \theta_2, \ldots, \theta_n)'$, $\Pi=[\pi_1,\pi_2, \ldots,\pi_n]'$,  and $\Theta = \diag(\theta) \equiv  \diag(\theta_1,\theta_2,\ldots,\theta_n)$.  Introduce $n \times n$  matrices $\Omega$ and $W$ by $\Omega = \Theta \Pi P \Pi' \Theta$ and  $W=A-\mathbb{E}[A]$. We can re-write (\ref{Model-1})  as 
\begin{equation}  \label{Model-2}
	A = \Omega - \mathrm{diag}(\Omega) + W. 
\end{equation} 
We call $\Omega$ the {\it Bernoulli probability matrix}  and $W$ the noise matrix. When $\theta_i$ in the same community are equal,  DCBM reduces to the Stochastic Block Model (SBM) \citep{SBM}. When $K = 1$,  the SBM reduces to  the Erd\H{o}s-Renyi model, where $\Omega(i, j)$ take the same value for all 
$1 \leq i, j \leq n$.

We first describe why the signal captured by the $\chi^2$-test in  \cite{Ery1, Ery0} is a modeling artifact. Using Sinkhorn's matrix scaling theorem \citep{Sinkhorn}, it is possible to build a null DCBM with $K = 1$ that has no community structure and an alternative DCBM with $K \geq 2$ and clear community structure such that the two models have the \textit{same} expected degrees. Thus, we do not expect that degree-based test such as $\chi^2$ can tell them apart. We make this Sinkhorn argument precise in Section \ref{subsec:identifiability} and show the failure of $\chi^2$ in  Theorem \ref{thm:chi2}. 

In the Erd\H{o}s-Renyi setting in \cite{Ery1}, the null has one parameter and the alternative has two parameters.  In such a setting, we cannot have degree-matching.   
In these cases, a naive degree-based $\chi^2$-test may have good power, but it is due to the  
very specific models they choose.  For clique detection in more realistic settings, 
we prefer to use a broader model such as the DCBM, where by the degree-matching argument above, the 
$\chi^2$-test is asymptotically powerless.

This motivates us to look for a different test.   One candidate is the scan statistic \cite{Ery2}. However, 
a scan statistic is only computationally feasible when each time we scan a very small subset of nodes. 
For example, if each time we only scan a finite number of nodes, then the computational cost is polynomial; we call the test the {\it Economic Scan Test (EST)}.  Another candidate may come from the Signed-Polygon test family \citep{JKL2019},  including the Signed-Quadrilateral (SgnQ) as a special case. Let $\hat{\eta} = ({\bf 1}_n A {\bf 1}_n)^{-1/2}  A {\bf 1}_n$ and $\widehat{A} = A - \hat{\eta} \hat{\eta}$.  Define 
$Q_n = \sum_{i_1, i_2, i_3, i_4 (dist)} \widehat{A}_{i_1 i_2} \widehat{A}_{i_2 i_3} \widehat{A}_{i_3 i_4} \widehat{A}_{i_4 i_1}$ where the shorthand $(dist)$ indicates we sum over distinct indices.   The SgnQ test statistic is 
\begin{equation} \label{DefineSgnQ} 
	\psi_n = \bigl[Q_n - 2(\|\hat{\eta}\|^2-1)^2\bigr] / \sqrt{8(\|\hat{\eta}\|^2-1)^4}. 
\end{equation} 
SgnQ is computationally attractive because it can be evaluated in time $O(n^2 \bar{d})$, where $\bar{d}$ is the average degree of the network \citep{JKL2019}.  

Moreover, it was shown in \cite{JKL2019} that (a)  when $K = 1$ (the null case),  $\psi_n \goto N(0,1)$, and (b) 
when $K > 1$ and all communities are at the same order (i.e., a balanced alternative case), 
the SgnQ test achieves the classical information lower bound (LB) for global testing and so is optimal. 
Unfortunately, our case is much more delicate: the signal of interest is contained in a community with a size 
that is much smaller than $n$ (e.g., $n^{\eps}$), so the signal can be easily overshadowed by 
the noise term of $Q_n$. Even in the simple alternative case where we only have two communities (with sizes $N$ and $(n-N)$),  it is unclear (a) how the lower bounds vary as $N / n \goto 0$, and especially whether there is a gap between the computation lower bound (CLB) and classical information lower bound (LB), and (b) to what extent the SgnQ test attains the CLB and so is optimal.

\vspace{-2mm}
%%%%%%%%%%%%
%%%%%%%%%%%%
%%%%%%%%%%%%
\subsection{Results and contributions} 

We consider the problem of detecting a small community in the DCBM. In this work, we specifically focus on the case $K = 2$ as this problem already displays a rich set of phase transitions, and we believe it captures the essential behavior for constant $K > 1$.  Let $N \ll n$ denote the size of this small community under the alternative. Our first contribution analyzes the power of SgnQ for this problem, extending results of \cite{JKL2019} that focus on the balanced case. Let $\lambda_1 = \lambda_1(\Omega)$. In Section \ref{subsec:SgnQnull}, we  define a population counterpart $\tilde \Omega$ of $\hat A$ and let $\widetilde \lambda = \lambda_1(\tilde \Omega)$. We show that SgnQ has full power if $\widetilde \lambda_1 / \sqrt{ \lambda_2} \to \infty$, which reduces to $N(a-c)/\sqrt{nc} \to \infty$ in the SBM case.

For optimality, we obtain a computational lower bound (CLB), relying on the low-degree polynomial conjecture, which is a standard approach in studying CLB (e.g.,  \cite{kunisky2019notes}). Consider a case where $K = 2$ and we have a small community with size $N$. 
Suppose the edge probability within the community and outside the community are $a$ and $c$, where 
$a > c$. 
The quantity $(a-c)/\sqrt{c}$ acts as the {\it Node-wise Signal-to-Noise Ratio (SNR)} for the detection problem.\footnote{Note that the node-wise SNR captures the ratio of the mean difference and standard deviation of Bernoulli($a$) versus Bernoulli($c$), which motivates our terminology.}
When $N\gg \sqrt{n}$, we find that the CLB is completely determined by $N$ and node-wise SNR; moreover, SgnQ matches with the CLB and is optimal. 
When $N\ll \sqrt{n}$, the situation is more subtle: if the node-wise SNR $(a-c)/\sqrt{c} \goto 0$ (weak signal case), we show the problem is computationally hard and the LB depends on $N$ and the node-wise SNR. If $(a-c)/\sqrt{c}\gg n^{1/2}$ (strong signal case), then SgnQ solves the detection problem. In the range $1 \ll (a-c)/\sqrt{c} \ll n^{1/2}$ (moderate signal case), the CLB depends on not only $N$ and the node-wise SNR but also the background edge density $c$. In this regime, we make conjectures of the CLB, from the study of the aforementioned economic scan test (EST).  Our results are summarized in Figure \ref{fig:Phase} and explained in full detail in Section \ref{subsec:phase}. 

\begin{wrapfigure}[18]{r}{0.5\textwidth}
	%	\centering
	\vspace{-10pt}
	\includegraphics[trim={0cm 0cm 0cm 0cm},clip,width=.4\textwidth,angle=-90]{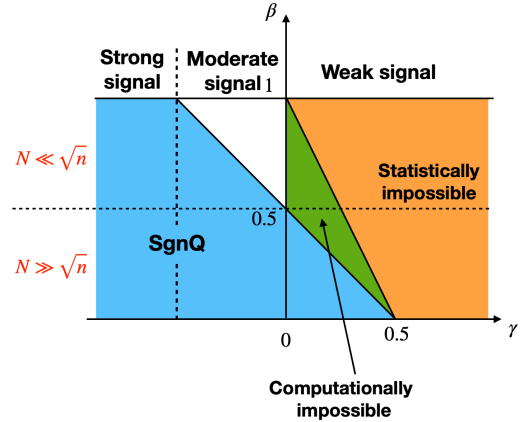}
	\caption{Phase diagram ($(a-c)/\sqrt{c}=n^{-\gamma}$ and $N=n^{1-\beta}$).}\label{fig:Phase}
	\vspace{-10pt}
\end{wrapfigure}

We also obtain the classical information lower bound (LB), 
and discover that as $N / n \goto 0$,  there is big gap between CLB and LB. Notably the LB is achieved by an (inefficient) signed scan test. In the balanced case in \cite{JKL2019}, the SgnQ test is optimal among all tests (even those that are allowed unbounded computation time), and such a gap does not exist.   

We also show that that the naive degree-based $\chi^2$-test is asymptotically powerless due to the aforementioned degree-matching phenomenon. 

%	We remark that o
Our statistical lower bound, computational lower bound, and the powerlessness of $\chi^2$ based on degree-matching are also valid for all $K>2$ since any model with $K \geq 2$ contains $K = 2$ as a special case. We also expect that our lower bounds are tight for these broader models and that our lower bound constructions for $K = 2$ represent the least favorable cases when community sizes are severely unbalanced. 

%{\sout{For simplicity, we only provide full proof for the case of $(K, m) = (2, 1)$.  The extension to 
%		general $(K, m)$ is straightforward, but with a much more tedious proof. } } 

Compared to \cite{Ery0, Ery1}, we consider network global testing in a more realistic setting, and show that optimal tests there (i.e., a naive degree-based $\chi^2$ test) may be 
asymptotically powerless here.  Compared with \cite{Ery2}, our setting is very different (they considered a setting where both the null and alternative are DCBM with $K = 1$). Compared to the study in the balanced case (e.g., \cite{JKL2018, JKL2019, GaL2017a}), our 
study is more challenging for two reasons. First, in the balanced case, there is no gap between the UB (the upper bound provided by the SgnQ test)  and LB, so there is no need to derive the CLB, which is usually technical demanding. Second, the size of the smaller community can get as small as $n^{\eps}$, where $\eps >  0$ is any constant. Due this imbalance in community sizes, the techniques of \cite{JKL2019} do not directly apply.   As a result, our proof involves the careful study of the $256$ terms that compose SgnQ, which requires using bounds tailored specifically for the severely unbalanced case.

%	Therefore, to prove our results, we have to analyze each term carefully, so the signal would not be overwhelmed by the noise.
%	In comparison, the standard deviation of the SgnQ test statistics is at a much larger order. 

Our study of the CLB is connected to that of \cite{hajek2015computational} in the Erd\"{o}s-Renyi setting of \cite{Ery1}. \cite{hajek2015computational} proved via computational reducibility that the naive $\chi^2$-test is the optimal polynomial-time test (conditionally on the planted clique hypothesis). We also note work of \cite{chen2016statistical} that studied a $K$-cluster generalization of the Erd\"{o}s-Renyi model of \cite{Ery1,Ery0} and provided conjectures of the CLB.  Compared to our setting, these models are very different because the expected degree profiles of the null and alternative  differ significantly.  In this work we consider the DCBM model, where due to the subtle phenomenon of degree matching between the null and alternative hypotheses, both CLB and LB are different from those obtained by \cite{hajek2015computational}.

%	Our work is also connected to property testing in theoretical computer science \citep{goldreich1998property,goldreich2011testing}. \cite{goldreich2011testing} studies testing if a given graph is $\epsilon$-close (two graphs are said to be $\epsilon$-close if one can be obtained from the other by removing or deleting an $\epsilon$-fraction of its edges) to one that has community structure, in the sense of its normalized adjacency matrix having a large second eigenvalue $\hat \lambda_2^*$. Although our notion of communities is also related to spectral properties, the test statistic $\hat \lambda_2^*$ can be very noisy due to the severe degree heterogeneity in DCBM, and moreover it does not have a tractable null. For these reasons, the algorithms of property testing are not directly applicable in our setting. 

{\bf Notations:}
We use ${\bf 1}_n$ to denote a $n$-dimensional vector of ones.  
For a vector $\theta = (\theta_1, \ldots, \theta_n)$, 
$\diag(\theta)$ is the diagonal matrix where the $i$-th diagonal entry is $\theta_i$. 
For a matrix $\Omega \in \mathbb{R}^{n\times n}$, $\diag(\Omega)$ is the diagonal matrix where the $i$-th diagonal entry is $\Omega(i, i)$.  For a vector $\theta \in \mathbb{R}^n$, $\theta_{max} = \max\{\theta_1, \ldots, \theta_n\}$ and $\theta_{min} = \min\{\theta_1, \ldots, \theta_n\}$. 
For two positive sequences $\{a_n\}$ and $\{b_n\}$, we write $a_n \asymp b_n$ if $c_1 \leq a_n/b_n \leq c_2$ for constants $c_2 > c_1 > 0$. We say $a_n \sim b_n$ if $(a_n/b_n) = 1 + o(1)$.

\section{Main results} 

In Section~\ref{subsec:identifiability}, following our discussion on Sinkhorn's theorem in Section \ref{sec:intro},  we introduce calibrations (including conditions on  identifiability and balance) that are appropriate for severely unbalanced DCBM and illustrate with some examples. 
In Sections~\ref{subsec:SgnQnull}-\ref{subsec:chi2}, we analyze the power of the SgnQ test and compare it with the $\chi^2$-test. In Sections~\ref{subsec:statLB}-\ref{subsec:compLB}, 
we discuss the information lower bounds (both the LB and CLB) and show that SgnQ test is optimal among polynomial time tests, when $N\gg\sqrt{n}$.
In Section~\ref{subsec:EST}, we study the EST and make some conjectures of the CLB when $N\ll \sqrt{n}$. In Section~\ref{subsec:phase}, we summarize our results and present the phase transitions.   

\vspace{-2mm}

\subsection{DCBM for severely unbalanced networks: identifiability, balance metrics, and global testing }  
\label{subsec:identifiability}

	In the DCBM (\ref{Model-1})-(\ref{Model-2}),  $\Omega = \Theta \Pi P \Pi' \Theta$. It is known that the matrices $(\Theta, \Pi,  P)$ are not identifiable. One issue is that $(\Pi, P)$ are only unique up to a permutation: for a $K \times K$ permutation matrix $Q$, $\Pi P \Pi = (\Pi Q) (Q' P Q) (\Pi Q)'$.  
This issue is easily fixable in applications so is usually neglected. 
A bigger issue is that,  $(\Theta, P)$ are not uniquely defined. 
For example, 
fixing a positive diagonal matrix $D\in\mathbb{R}^{K\times K}$,  let $P^* =DPD$ and $\Theta^* = \diag(\theta_1^*, \theta_2^*, \ldots, \theta_n^*)$  where 
$\theta_i^* =\theta_i/\sqrt{D(k,k)}$ if $i \in {\cal C}_k$, $1 \leq k \leq K$.   It is seen that $\Theta \Pi P \Pi' \Theta= \Theta^* \Pi  P^* \Pi' \Theta^*$,   so $(\Theta, P)$ are not uniquely defined.  

To motivate our identifiability condition, we formalize the degree-matching argument discussed in the introduction. Fix $(\theta, P)$ and let $h = (h_1, \ldots, h_K)'$ and $h_k > 0$ is the fraction of nodes in community $k$, $1 \leq k \leq K$. 
%	and suppose $\pi_i$ are iid from a distribution $F$. 
%	Let $h = \mathbb{E}_F[\pi_i]$. 
By the main result of \cite{Sinkhorn}, there is a unique positive diagonal matrix $D = \diag(d_1, \ldots, d_K)$ such that $D P D h = {\bf 1}_K$.  Consider a pair of two DCBM, a null with $K =1$ and an alternative with $K > 1$, 
with parameters  $\Omega= \Theta {\bf 1}_n {\bf 1}_n'  \Theta \equiv \theta \theta'$ and $\Omega^*(i,j) = \theta_i^* \theta_j^* \pi_i ' P \pi_j$ with $\theta_i^* = d_k \theta_i$ if $i \in {\cal C}_k$, $1 \leq k \leq K$, respectively. Direct calculation shows that node $i$ has the same expected degree under the null and alternative. 

%	It is seen that $\theta_i^*\pi_i=\theta_iD\pi_i$ for all $1\leq i\leq n$. Fix $i$. In the null, the expected degree of node $i$ is $\sum_{j=1}^n\theta_i\theta_j=\theta_i\|\theta\|_1$. In the alternative, conditioning on $\pi_i$, the expected degree of node $i$ is $\theta_i^* \pi_i' \mathbb{E}[\sum_{j = 1}^n \theta_j^*  P\pi_j]=\theta_i \|\theta\|_1  \pi_i' (D P D) h  = \theta_i \|\theta\|  \pi_i' {\bf 1}_K$, which is also equal to $\theta_i \|\theta\|_1$.     
%	Here, we have neglected the matrix $\diag(\Omega)$  in (\ref{Model-2}),  the effect of which is secondary.  
%	Also, note that $\Omega^* = \Theta \Pi (D P D) \Pi' \Theta$ and $(D P D) h = {\bf 1}_K$.

%	In summary, fix an alternative DCBM ($K > 1$) with a Bernoulli probability matrix $\Omega^*$, we can always use the Sinkhorn's theorem to redefine $(\Theta, P)$ (while $\Pi$ is unchanged)  so that for the new $(\Theta, P)$, 
%	$\Omega^* = \Theta \Pi P \Pi' \Theta$ and $P h = {\bf 1}_K$.  This motivates an interesting {\it identifiability condition for DCBM}; see Section \ref{subsec:identifiability}. Moreover,  if we pair the alternative with the null with $\Omega = \Theta {\bf 1}_n {\bf 1}_n' \Theta \equiv \theta \theta'$,  then for each node, the expected degrees in two models match with each other, and a naive degree-based $\chi^2$-test as in \cite{Ery1} is asymptotically powerless. %in separating the alternative from the null.   

There are many ways to resolve the issue. For example, in the balanced case (e.g., \cite{JKL2019, JKLW2022}), we can resolve it by requiring that $P$ has unit diagonals. However, for our case, 
this is inappropriate. Recall that, in practice, $P(k, \ell)$ represents   
as the baseline connecting probability between community $k$ and $\ell$.  If we forcefully 
rescale $P$ to have a unit diagonal here, both $(P, \Theta)$ lose their practical meanings.    

Motivated by the degree-matching argument, 
we propose an identifiability condition that is more appropriate for the severely unbalanced DCBM.    
By our discussion in Section \ref{sec:intro}, for any DCBM with a Bernoulli probability matrix $\Omega$, we can always use Sinkhorn's 
theorem to define $(\Theta, P)$ (while $\Pi$ is unchanged)  such that for the new $(\Theta, P)$, $\Theta = \Theta \Pi P \Pi' \Theta$ 
and $P h \propto {\bf 1}_K$, where $h = (h_1, \ldots, h_K)'$ and $h_k > 0$ is the fraction of nodes in community $k$, $1 \leq k \leq K$.    This motivates the following identifiability condition (which is more appropriate for our case): 
%%%%%%%%%%%
%%%%%%%%%%%
%%%%%%%%%%%
\beq \label{identifiable}
\|\theta\|_1=n, \qquad Ph \propto {\bf 1}_K, \quad\mbox{where $h_k$ is fraction of nodes in ${\cal C}_k$, $1 \leq k \leq K$}.  
\eeq
%The following lemma is proved in the supplement. 
\begin{lemma} \label{lem:identifiability}
	For any $\Omega$ that satisfies the DCBM (\ref{Model-2}) {\color{black}{and has positive diagonal elements}}, we can always find $(\Theta, \Pi, P)$ such that 
	$\Omega = \Theta \Pi P \Pi' \Theta$ and (\ref{identifiable}) holds.  Also,   any $(\Theta, P)$ that satisfy  
	$\Omega =  \Theta \Pi P \Pi' \Theta$ and (\ref{identifiable})  are unique.  
\end{lemma} 

Moreover, for network balance, the following two vectors in $\mathbb{R}^K$ are natural metrics: \begin{equation} \label{Definedg} 
	d = (\|\theta\|_1)^{-1}  \Pi' \Theta {\bf 1}_n, \qquad g = (\|\theta\|)^{-2} \Pi' \Theta^2 \Pi {\bf 1}_K,   
\end{equation} 
In the balanced case (e.g., \cite{JKL2019, JKLW2022}),  we usually assume the entries of $d$ and $g$ are at the same order.  For our setting, this is not the case. 
%{\sout{For illustration, we consider some examples.}}   

{ \color{black}{ Next we introduce the null and alternative hypotheses that we consider. Under each hypothesis, we impose the identifiability condition \eqref{identifiable}.  } }

{\bf \color{black}{General null model for the DCBM.  } } When $K = 1$ and $h = 1$,  $P$ is scalar (say, $P = \alpha$),  and $\Omega = \alpha \theta \theta'$ satisfies $\|\theta\|_1 = n$ by \eqref{identifiable}.  The expected total degree is $\alpha (\|\theta\|_1^2 - \|\theta\|^2)\sim \alpha\|\theta\|_1^2 = n^2\alpha $ under mild conditions, so we view $\alpha$ as the parameter for network sparsity.  In this model, 
$d = g = 1$.  

{\bf \color{black}{  Alternative model for the DCBM }  }.  We assume $K = 2$ and that the sizes of the two communities, ${\cal C}_0$ and ${\cal C}_1$,   are $(n-N)$  and $N$, respectively. %where we may assume $N \ll n$. 
For some positive numbers $a, b, c$, we have 
\begin{equation} \label{2group-DCBM}
	P = \left[
	\begin{array}{ll} 
		a & b \\
		b & c \\  
	\end{array}  
	\right],  
	\qquad
	\mbox{and} 
	\qquad 
	\Omega(i,j) = \left\{
	\begin{array}{ll} 
		\theta_i \theta_j  \cdot a, & \qquad  \mbox{if $i, j \in {\cal C}_1$}, \\ 
		\theta_i \theta_j \cdot c, &\qquad \mbox{if $i, j \in {\cal C}_0$}, \\
		\theta_i \theta_j \cdot b, &\qquad \mbox{otherwise}.  \\ 
	\end{array}
	\right. 
\end{equation} 
In the classical clique detection problem (e.g.,  \cite{Ery2}), $a$ and $c$ are the baseline probability where two nodes have an edge when both of them are {\it in} the clique and {\it outside} the clique, respectively. By \eqref{identifiable},  $a\epsilon+b(1-\epsilon)=b\epsilon+c(1-\epsilon)$ if we write $\epsilon=N/n$. Therefore, 
\beq \label{2group-DCBM-b}
b =  (c(n-N)-aN) /(n-2N). 
\eeq
Note that this is the {\it direct result}  of Sinkhorn's theorem and the parameter calibration we choose, not a condition we choose for technical convenience.   Write $d = (d_0, d_1)'$ and $g = (g_0, g_1)'$.  
It is seen that $d_0 = 1 - d_1$,  $g_0 = 1 - g_0$,  $d_1= \|\theta\|_1^{-1} \sum_{i \in {\cal C}_1} \theta_i$,  and $g_1 = \|\theta\|^{-2} \sum_{i \in {\cal C}_1} \theta_i^2$.  If all $\theta_i$ are at the same order, then 
$d_1 \asymp g_1 \asymp (N/n)$ and $d_0 \sim g_0 \sim 1$. We also observe that $b = c +O(a\epsilon)$ which makes the problem seem very close to \cite{Ery1,Ery2}, although in fact the problems are quite different.

{\bf \color{black}{Extension} }. An extension of our alternative is that,  for the $K$ communities,  the sizes of $m$ of them are at the order of $N$, for an $N \ll n$ and an integer $m$,   $1 \leq m < K$,  and 
the sizes of remaining $(K-m)$ are at the order of $n$.  In this case, $m$ entries of $d$ are $O(N/n)$ and other entries are $O(1)$; same for $g$.

\vspace{-2mm}

%%%%%%%%%%%
%%%%%%%%%%%
%%%%%%%%%%%
\subsection{The SgnQ test:  limiting null, p-value, and power}  \label{subsec:SgnQnull}
\label{subsec:null}  
In the null case, $K = 1$ and we assume $\Omega = \alpha \theta \theta'$, where $\|\theta\|_1=n$. As $n \goto \infty$, both 
$(\alpha, \theta)$ may vary with $n$. Write $\theta_{\max}=\|\theta\|_\infty$. We assume 
\begin{equation} \label{nullconditions} 
	n\alpha\to\infty, \qquad \mbox{and}\qquad \alpha\theta^2_{\max} \log(n^2\alpha)\to 0.  
\end{equation} 
The following theorem is adapted from \cite{JKL2019} and the proof is omitted. 
\begin{thm}[Limiting null of the SgnQ statistic] \label{thm:null-SgnQ}   
	Suppose the null hypothesis is true and the regularity conditions (\ref{identifiable}) and (\ref{nullconditions}) hold. As $n \goto \infty$, $\psi_n \goto N(0,1)$ in law. 
\end{thm}  
We have two comments. First, since the DCBM has many parameters (even in the null case), it is not an easy task to find a test statistic with a limiting null that is completely parameter free. For example, if we use the largest eigenvalue of $A$ as the test statistic, it is unclear how to normalize it so to have such a limiting null. 
Second, since the limiting null is completely explicit,   we can approximate the (one-sided) $p$-value of $\psi_n$ by  $\mathbb{P}(N(0, 1) \geq \psi_n)$.  
The p-values are useful in practice,  as we show in our numerical experiments.. For example, using a recent data set on the statisticians' publication \citep{JBES2022},  for each author, we can construct 
an ego network and apply the SgnQ test. We can then use the $p$-value to measure the co-authorship diversity of the author. Also,  in many hierarchical community detection algorithms (which 
are presumably  recursive, aiming to estimate the tree structure of communities),  we can use the p-values to determine whether we should further divide a sub-community in each stage of the algorithm (e.g. \cite{JBES2022}).  

%	\subsection{The UB and power analysis of the SgnQ test}   \label{subsec:SgnQalt}
The power of the SgnQ test hinges on the matrix $\widetilde{\Omega} = \Omega - ({\bf 1}_n' \Omega {\bf 1}_n)^{-1} \Omega {\bf 1}_n {\bf 1}_n' \Omega$.   
%	For $d$ and $g$  in (\ref{Definedg}), let $G \in \mathbb{R}^{K, K}$ be the diagonal matrix $G = \diag(g) = \diag(g_1, g_2, \ldots, g_K)$. 
By basic algebra,  
\beq \label{centeredOmega}
\widetilde{\Omega} = \Theta \Pi \widetilde{P} \Pi' \Theta, \qquad \mbox{where \;  $\widetilde{P} = P - (d' P d)^{-1} P d d' P$}.  
\eeq
Let $\tilde{\lambda}_1$ be the largest (in magnitude) eigenvalue of 
$\widetilde{\Omega}$. Lemma \ref{lemma:d} is proved in the supplement. 
%%%%%%%%%%%%%%%
%%%%%%%%%%%%%%%
%%%%%%%%%%%%%%%
{ \color{black}{	\begin{lemma}\label{lemma:d} 
			The rank and trace of the matrix $\widetilde{\Omega}$ are $(K-1)$  and $\|\theta\|^2 \diag(\ti P)' g$, respectively.  When $K = 2$, $\tilde{\lambda}_1 = \mathrm{trace}(\widetilde{\Omega}) = \| \theta \|^2(ac - b^2)(d_0^2 g_1 + d_1^2 g_0)/ (a d_1^2 + 2b d_0 d_1 + c d_0^2)$.    
\end{lemma} } }
\vspace{-0.3cm}
As a result of this lemma, we observe that in the SBM case, $d = h$ and thus $\widetilde \lambda_1 = \lambda_2 \asymp N(a-c)$. 
To see intuitively that the power of the SgnQ test hinges on 
$\tilde \lambda_1^4/\lambda_1^2$, if we heuristically replace the terms of SgnQ by population counterparts, we obtain
\begin{align*}
	Q_n &= \sum_{\substack{i_1, i_2, i_3, i_4 (distinct)}} \hat A_{i_1 i_2} \hat A_{i_2 i_3} \hat A_{i_3 i_4} \hat A_{i_4 i_1} 
	\approx \mathrm{trace}([\Omega - \eta \eta']^4) 
	= \mathrm{trace}(\widetilde \Omega^4)
	= \tilde \lambda_1^4.
\end{align*}	

%	To give an intuitive discussion of the power of the SgnQ test, first note that $\mathrm{Var}(Q_n) \approx (\|\widehat{\eta}\|^2 - 1)^4 \approx \lambda_1^4$, as suggested by \eqref{DefineSgnQ}. We show that the power of 
%	\begin{align*}
%		Q_n &= \sum_{\substack{i_1, i_2, i_3, i_4 (distinct)}}  (A_{i_1 i_2} - \hat \eta_{i_1} \hat \eta_{i_2}  ) (A_{i_2 i_3} - \hat \eta_{i_2} \hat \eta_{i_3}  )
%		(A_{i_3 i_4} - \hat \eta_{i_3} \hat \eta_{i_4}  )
%		(A_{i_4 i_1} - \hat \eta_{i_4} \hat \eta_{i_1}  )
%		\\ &
%		\approx \mathrm{trace}([\Omega - \eta \eta']^4) 
%		= \mathrm{trace}(\widetilde \Omega^4)
%		= \tilde \lambda_1^4
%	\end{align*}

%{\begin{lemma**} 
%		\sout{The rank and trace of the matrix $\widetilde{\Omega}$ are $(K-1)$  and $\|\theta\|^2 
%			(d' P (I_K - G)d) / (d' P d)$, respectively.  When $K = 2$, $\tilde{\lambda}_1 = \mathrm{trace}(\widetilde{\Omega}) = \|\theta\|^2  (c d_0^2 g_1 + a d_1^2 g_0 + b d_0 d_1) / (c d_0^2  + 2 b d_0 d_1 + a d_1^2)$.    }
%\end{lemma**}  }
%HERE

%	
%	, Intuitively, SgnQ works well because its power hinges on
%	
%	Under the null, it is known that $\mathrm{Var}(Q_n) \approx (\|\widehat{\eta}\|^2 - 1)^4 \approx \lambda_1^4$ (see [3]) 
\vspace{-0.4cm}
We now formally discuss the power of the SgnQ test.  
%{\sout{For simplicity, we focus on the setting in Example 2, but our 
%		methods are readily extended to the setting of Example 3.} }  
{\color{black}We focus on the alternative hypothesis in Section \ref{subsec:identifiability}.}  %Let $\theta^{(1)}$ and $\theta^{(0)}$ be the sub-vectors of $\theta$ restricted to the indices in ${\cal C}_1$ and ${\cal C}_0$, respectively. 
%The vector $\theta\circ \eta$ consists of two sub-vectors, $\sqrt{a}\theta^{(1)}$ and $\sqrt{c}\theta^{(0)}$. 
%In the most interesting parameter range, the sub-vector $\sqrt{c}\theta^{(0)}$ has a dominating contribution in the $\ell^1$-norm and $\ell^2$-norm of  $\theta\circ\eta$, and $\sqrt{a}\theta^{(1)}$ determines the $\ell^\infty$-norm of  $\theta\circ\eta$. Then, a sufficient condition for \eqref{altconditions} is  
%We need the following regularity conditions. 
Let $d=(d_1, d_0)'$ and $g=(g_1,g_0)'$ be as in \eqref{Definedg}, and let $\theta_{\max,0}=\max_{i\in {\cal C}_0}\theta_i$ and $\theta_{\max,1}=\max_{i\in {\cal C}_1}\theta_i$. 
Suppose 
\beq \label{altconditions}
d_1\asymp g_1\asymp N/n, \qquad a \theta^2_{\max,1} =O(1), \qquad cn\to\infty, \qquad c\theta^2_{\max,0} \log(n^2c)\to 0. 
\eeq
%	\pax{Can we just remove the first condition $d_1\asymp g_1\asymp N/n$? I don't think we need it. The reason is that SgnQ doesn't care about the clique size -- it always works if the signal is big enough.}
These conditions are mild. For example, when $\theta_i$'s are at the same order, the first inequality in \eqref{altconditions} automatically holds, and the other inequalities in \eqref{altconditions} hold if $a\leq C$ for an absolute constant $C>0$, $cn\to\infty$, and $c\log(n)\to 0$.   

Fixing $0 < \kappa < 1$, let $z_{\kappa} > 0$ be the value such that $\mathbb{P}(N(0, 1) \geq z_{\kappa}) = \kappa$. The level-$\kappa$ SgnQ test rejects the null if and only if  $\psi_n \geq z_{\kappa}$, where $\psi_n$ is as in \eqref{DefineSgnQ}. Theorem \ref{thm:alt-SgnQ} and Corollary~\ref{cor:SgnQtest} are proved in the supplement. Recall that our alternative hypothesis is defined in Section \ref{subsec:identifiability}. By \textit{power} we mean the probability that the alternative hypothesis is rejected, minimized over all possible alternative DCBMs satisfying our regularity conditions.

{\color{black}{
		\begin{thm}[Power of the SgnQ test] 
			\label{thm:alt-SgnQ} 
			Suppose that \eqref{altconditions} holds, and let $\kappa\in (0,1)$. Under the alternative hypothesis, if $| \ti \lambda_1 | / \sqrt{\lambda_1} \to \infty$, the power of the level-$\kappa$ SgnQ test tends to $1$. 
		\end{thm}
}}	

%{\color{black}{\begin{thm**}[\sout{Power of the SgnQ test}] 
%			\sout{Consider a DCBM with $2$ communities as in \eqref{2group-DCBM},   where   \eqref{identifiable} and \eqref{altconditions} hold.  For $\kappa\in (0,1)$ and let $n \goto \infty$.  
%				If $| \ti \lambda_1 | / \sqrt{\lambda_1} \to \infty$, 
%				then the power of the level-$\kappa$ SgnQ test tends to $1$. }
%\end{thm**}}}

{\color{black}{
		\begin{cor}
			\label{cor:SgnQtest}
			Suppose the same conditions of Theorem \ref{thm:alt-SgnQ} hold, and additionally $\theta_{\max} \leq C \theta_{\min}$ so all $\theta_i$ are at the same order.  In this case, $\lambda_1 \asymp cn$ and $|\tilde{\lambda}_1| \asymp N(a-c)$, and the power of the level-$\kappa$ SgnQ test tends to $1$ if $N(a-c)/\sqrt{cn}\to\infty$. 
		\end{cor}
	}
}

%{\color{black}{\begin{cor**} 
%			\sout{Suppose the same conditions of Theorem \ref{thm:alt-SgnQ} holds, and $\theta_{\max} \leq C \theta_{\min}$ so all $\theta_i$ are at the same order.  In this case, $\lambda_1 \asymp cn$ and $|\tilde{\lambda}_1| \asymp N(a-c)$, and the Type II error of the SgnQ test tends to $0$ if $N(a-c)/\sqrt{cn}\to\infty$. }
%\end{cor**}}}
In Theorem \ref{thm:alt-SgnQ} and Corollary \ref{cor:SgnQtest}, if $\kappa = \kappa_n$ and $\kappa_n \goto 0$ slowly enough, then  the results continues to hold, and the sum of Type I and Type II errors of the SgnQ test at level-$\kappa_n$ $\goto 0$.

%We have a few comments. First, Theorem~\ref{thm:alt-SgnQ} is proved under more general 

The power of the SgnQ test was only studied in the balanced case \citep{JKL2019}, but our setting is a severely unbalanced case, where the community sizes are at different orders as well as the entries of $d$ and $g$.
In the balanced case, the signal-to-noise ratio of SgnQ is governed by $|\lambda_2|/\sqrt{\lambda}_1$, but in our setting, the signal-to-noise ratio is governed by  $|\tilde{\lambda}_1|/\sqrt{\lambda_1}$. 
The proof is also subtly different. Since the entries of $P$ are at different orders, many terms deemed negligible in the power analysis of the balanced case may become non-negligible in the unbalanced case and require careful analysis. 

%{\color{red} {\bf Remark}. The SgnQ test is an order-$4$ Signed Polygon 
%	test. In [18], it was shown that the SNR of an order-$m$ Signed Polygon 
%	test hinges on $(\sum_{k = 2}^K \lambda_k^m) / \sqrt{\sum_{k = 1}^K \lambda_k^m}$.
%	Therefore, when $m$ is odd, signal cancellations may occur, meaning that 
%	$|\sum_{k = 2}^K \lambda_k^m| \ll \sum_{k = 2}^K |\lambda_k|^m$.  
%	Note that when $K = 2$ or $m$ is even, signal cancellation is not possible. 
%	For this reason, we prefer to use SgnQ than the SgnT test, which is an order-$3$ 
%	Signed-Polygon test. See [18] for more discussion. }

%their proofs cannot be easily adapted to the unbalanced case here. Their proofs rely on $\|P\|_{\max}=o(1)$, but in our setting, $\|P\|_{\max}=a$, and $a=O(1)$ is the most interesting case (e.g., the smaller community may have a constant edge density, such as in the problem of planted clique detection). Moreover, the signal to noise ratio  is also different from the balanced case. 
%In the balanced case, the signal-to-noise ratio is governed by $|\lambda_2|/\sqrt{\lambda_1}$, but in the unbalanced case, it is captured by $|\widetilde{\lambda}_1|/|\lambda_1|$ as shown in Theorem~\ref{thm:alt-SgnQ}. 

\vspace{-2mm}

\subsection{Comparison with  the naive degree-based $\chi^2$-test} \label{subsec:chi2}
Consider a setting where $\Omega = \alpha \Theta {\bf 1}_n {\bf 1}_n'  \Theta \equiv  \alpha \theta \theta'$ under the null and $\Omega = \Theta \Pi P 
\Pi' \Theta$ under the alternative, and (\ref{identifiable}) holds.  
When $\theta$ is unknown, it is unclear how to apply the $\chi^2$-test: the null case has $n$ unknown parameters $\theta_1, \ldots, \theta_n$, and we need to use the degrees to estimate $\theta_i$ first. As a result, the resultant $\chi^2$-statistic may be trivially $0$. Therefore,  we consider 
a simpler SBM case where $\theta = {\bf 1}_n$. In this case, $\Omega = \alpha  {\bf 1}_n {\bf 1}_n$, and 
$\Omega =  \Pi P \Pi'$ and the null case only has one unknown parameter $\alpha$.  Let $y_i$ be the degree of node $i$, and let $\hat{\alpha}=[n(n-1)]^{-1}{\bf 1}_n'A{\bf 1}_n$.  The  $\chi^2$-statistic is  
\beq \label{DefineChi2}
X_n = \sum_{i=1}^n  (y_i-n\hat{\alpha})^2/[(n-1)\hat{\alpha}(1-\hat{\alpha})]. 
\eeq
%%%%%%%%%%%
%%%%%%%%%%%
%%%%%%%%%%%
%{\bf Remark}. This includes the  degree-based $\chi^2$-test in the SBM case by  \cite{Ery1,cammarata2022power} as a special case, where $\theta={\bf 1}_n$, $y=A{\bf 1}_n$, $\hat{\alpha}=n^{-2}({\bf 1}_n'A{\bf 1}_n)$, and $X_n=(n\hat{\alpha})^{-1}\sum_{i=1}^n(y_i-n\hat{\alpha})^2$. 
%The test  is called `degree-based' because in this special case, $y_i$ is the degree of node $i$, and $n\hat{\alpha}$ is approximately the average node degree. Here, \eqref{DefineChi2} extends the $\chi^2$-statistic  in \cite{Ery1,cammarata2022power} from $\theta={\bf 1}_n$ to a general $\theta$.  
%%%%%%%%%%%%%
%%%%%%%%%%%%%
%%%%%%%%%%%%%
It is seen that as $n\alpha\to\infty$ and $\alpha\to 0$,   $(X_n-n)/\sqrt{2n}\goto N(0,1)$ in law. For a fixed level $\kappa\in (0,1)$, 
consider the $\chi^2$-test that rejects the null if and only if $(X_n-n)/\sqrt{2n} > z_{\kappa}$. 
Let $\alpha_0 =n^{-2}({\bf 1}_n' \Omega {\bf 1}_n)$. The  
power of the $\chi^2$-test hinges on 
%\beq \label{PowerChi2}
the quantity $(n\alpha_0)^{-1}\|(\Omega {\bf 1}_n - n\alpha_0)\|^2=(n\alpha_0)^{-1} \|\Pi Ph - (h'Ph)^{-1}{\bf 1}_n\|^2=0$, if $Ph\propto {\bf 1}_K$. The next theorem is proved in the supplement. 

{\color{black}{\begin{thm} \label{thm:chi2} 
			Suppose $\theta = \mf{1}_n$ and \eqref{altconditions} holds. If $|\tilde{\lambda}_1|/\sqrt{\lambda_1}\to\infty$ under the alternative hypothesis, the power of the level-$\kappa$ SgnQ test goes to $1$, while the power of the level-$\kappa$  $\chi^2$-test goes to $\kappa$. 
\end{thm} }}

\vspace{-2mm}

\subsection{The statistical lower bound and the optimality of  the scan test}  \label{subsec:statLB}
%\sout{Same as before, we focus on the setting in Example 2, but all results are extendable to Example 3. }
For lower bounds,  it is standard to consider a random-membership DCBM \citep{JKL2019}, where $\|\theta\|_1=n$, $ P$ is as in \eqref{2group-DCBM}-\eqref{2group-DCBM-b} and for a number $N \ll n$,  $\Pi = [\pi_1, \pi_2, \ldots, \pi_n]'$ 
satisfies 
\beq \label{random-Pi}
\pi_i = (X_i, 1-X_i), \qquad \mbox{where $X_i$ are iid Bernoulli($\eps)$ with $\eps = N/n$}.  
\eeq
%The next theorem is proved in the supplementary material: 

{\color{black}{\begin{thm}[Statistical lower bound] \label{thm:statLB1}
			Consider the null and alternative hypotheses of Section \ref{subsec:identifiability}, and assume that \eqref{random-Pi} is satisfied, $\theta_{\max}\leq C\theta_{\min}$ and $Nc/\log n\to\infty$. If $\sqrt{N}(a-c)/\sqrt{c}\to 0$, then for any test, the sum of the type-I and type-II errors tends to $1$. 
		\end{thm}
}}

%\begin{thm**}[\sout{Statistical lower bound}] 
%	\sout{Consider a pair of hypotheses, where the alternative hypothesis satisfies \eqref{2group-DCBM}-\eqref{2group-DCBM-b}, \eqref{random-Pi} and that $\|\theta\|_1=n$, and the null hypothesis is such that $\Omega=\alpha \theta\theta'$, with $\alpha=a(N/n)+b(1-N/n)$. Suppose $\theta_{\max}\leq C\theta_{\min}$ and $Nc/\log n\to\infty$. If $\sqrt{N}(a-c)/\sqrt{c}\to 0$, then for any test, the sum of the type-I and type-II errors tends to $1$. }
%\end{thm**}

%\begin{thm}[Lower bound for the case of $Nc\to\infty$] \label{thm:statLB1}
%	Let $a, c, \epsilon$ satisfy that $0<c<a<1$, $0<\epsilon <1/2$ and $c(1-\epsilon)> a\epsilon$. Let $\theta=(\theta_1,\theta_2,\ldots,\theta_n)'$ be a positive vector satisfying that $\|\theta\|^2=n$ and $c\theta_{\max}^2=o(1)$. Consider a pair of hypotheses, where the alternative hypothesis is as in \eqref{alt-model1}-\eqref{alt-model2} and the null hypothesis is as in \eqref{null-model}. Suppose $\epsilon\asymp N/n$, where $N\to\infty$, $N/n\to 0$ and $Nc\to\infty$. 
%	If 
%	\[
%	\theta_{\max}\sqrt{N}(a-c)/\sqrt{c}\to 0,
%	\]
%	then any test is asymptotically powerless. 
%\end{thm}

To show the tightness of this lower bound, we introduce the signed scan test, 
by adapting the idea in \cite{Ery1} from the SBM case to the DCBM case. Unlike the SgnQ test and the $\chi^2$-test,  
signed scan test is not a polynomial time test, but it provides sharper upper bounds. 
Let $\hat{\eta}$ be the same as in \eqref{DefineSgnQ}. For any subset $S\subset\{1,2,\ldots,n\}$, let ${\bf 1}_{S}\in\mathbb{R}^n$ be the vector whose $i$th coordinate is $1\{i\in S\}$. 
Define the signed scan statistic
\beq \label{DefineScan}
\phi_{sc} = \max_{S \subset \{1,2,\ldots,n\}: |S| = N} 
{\bf 1}'_S \big( A - \hat \eta \hat \eta' \big) {\bf 1}_S. 	
\eeq 
%%%%%%%%%%%%
%%%%%%%%%%%%
%%%%%%%%%%%%

{\color{black}{	\begin{thm}[Tightness of the statistical lower bound] \label{thm:scan}
			Consider the signed scan test \eqref{DefineScan} that rejects the null hypothesis if $\phi_{sc}>t_n$. Under the assumptions of Theorem~\ref{thm:statLB1}, if $\sqrt{N}(a-c)/\sqrt{c\log(n)}\to\infty$, then there exists a sequence $t_n$ such that the sum of type I and type II errors of the signed scan test tends to $0$. 
		\end{thm}
}}	

%{\color{black}{	\begin{thm*}[\sout{Tightness of the statistical lower bound}]
%			\sout{Consider the same pair of hypotheses as in Theorem~\ref{thm:statLB1}. 
%				Suppose $\theta_{\max}\leq C\theta_{\min}$ and $Nc/\log(n)\to\infty$. Consider the signed scan test \eqref{DefineScan}, which rejects the null hypothesis if $\phi_{sc}>t_n$. 
%				If $\sqrt{N}(a-c)/\sqrt{c\log(n)}\to\infty$, then there exists a sequence $t_n$ such that the sum of type I and type II errors of the signed scan test tends to $0$.}
%		\end{thm*}
%}}
By Theorems~\ref{thm:statLB1}-\ref{thm:scan} and 
%	{\color{red}{Theorem~\ref{thm:alt-SgnQ}}} 
Corollary \ref{cor:SgnQtest}, the two hypotheses are asymptotically indistinguishable if $\sqrt{N}(a-c)/\sqrt{c}\to 0$, and are asymptotically distinguishable by the SgnQ test if $N(a-c)/\sqrt{cn}\to \infty$.  Therefore, the lower bound is sharp, up to log-factors, and the signed scan test is nearly optimal. 
Unfortunately, the signed scan test is not polynomial-time computable. Does there exist a polynomial-time computable 
test that is optimal? We address this in the next section.

\vspace{-2mm}

\subsection{The computational lower bound} \label{subsec:compLB}
%We study the computational barrier of this testing problem in this section. 
Consider the same hypothesis pair as in Section~\ref{subsec:statLB}, where  $K=2$, $ P$ is as in \eqref{2group-DCBM}-\eqref{2group-DCBM-b}, and $\Pi$ is as in \eqref{random-Pi}. 
For simplicity, we only consider SBM, i.e., $\theta_i\equiv 1$. 
The low-degree polynomials argument emerges recently as a major tool to predicting the average-case computational barriers in a wide range of high-dimensional problems \citep{hopkins2017bayesian, hopkins2017power}. Many powerful methods, such as spectral algorithms and approximate message passing, can be formulated as functions of the input data, where the functions are polynomials with degree at most logarithm of the problem dimension. In comparison to many other schemes of developing computational lower barriers, the low-degree polynomial method yields the same threshold for various average-case hardness problems, such as community detection in the SBM \citep{hopkins2017bayesian} and (hyper)-planted clique detection \citep{hopkins2018statistical,luo2022tensor}. %spiked tensor model \citep{hopkins2017power, hopkins2018statistical, kunisky2019notes}, %spiked Wishart model \citep{bandeira2020computational}, 
%sparse PCA \citep{ding2019subexponential}. %spiked Wigner model \citep{kunisky2019notes}. 
The foundation of the low-degree polynomial argument is the following {\it low-degree polynomial conjecture} \citep{hopkins2017power} :
\begin{conjecture}[Adapted from \cite{kunisky2019notes}] \label{conjecture}
	Let $\mathbb{P}_n$ and $\mathbb{Q}_n$ denote a sequence of probability measures with sample space $\mathbb{R}^{n^k}$ where $k = O(1)$. Suppose that every polynomial $f$ of degree $O(\log n)$ with $\mathbb{E}_{\mathbb{Q}_n} f^2 = 1$ is bounded under $\mathbb{P}_n$ with high probability as $n\to \infty$ and that some further regularity conditions hold. Then there is no polynomial-time test distinguishing $\mathbb{P}_n$ from $\mathbb{Q}_n$ with type I and type II error tending to $0$ as $n \to \infty$.
\end{conjecture}
We refer to \cite{hopkins2018statistical} for a precise statement of this conjecture's required regularity conditions. The low-degree polynomial computational lower bound for our testing problem is as follows. 

%%%%%%%%%%%%%%%%%%%%%%%%%%%%%%%%%%%%%%%%%%%%%

\begin{thm}[Computational lower bound]\label{thm:compLB}
	{\color{black}{Consider the null and alternative hypotheses in Section \ref{subsec:identifiability}, and assume $\theta_i\equiv 1$ and \eqref{random-Pi} holds.}} As $n\to\infty$,
	assume $c < a$, $c<1-\delta$ for constant $\delta>0$, $N < n/3$,  $D = O(\log n)$, and
	$\limsup_{n\to\infty} \left\{ \left(\log_n\frac{N}{\sqrt{n}} +  \log_n\frac{a-c}{\sqrt{c}}\right) \vee \left(\sqrt{D/2-1}\log_n\frac{a-c}{\sqrt{c}}\right)\right\} < 0.$
	For any series of degree-$D$ polynomials $\phi_n: A \to \mathbb{R}$, whenever $\mathbb{E}_{H_0}\phi_n(A)=0, \textrm{Var}_{H_0}(\phi_n(A))=1$, we must have $\mathbb{E}_{H_1} \phi_n(A) = o(1)$. This implies if Conjecture~\ref{conjecture} is true, there is no consistent polynomial-time test for this problem. %based on polynomial of degree at most $D = O(\log n)$, 
	%has the sum of type I and type II errors tending to $1$. 
\end{thm}

%\begin{thm}[Computational lower bound]\label{thm:compLB}
%	Consider the same pair of hypotheses as in Theorem~\ref{thm:statLB1}, and additionally, assume $\theta_i\equiv 1$. As $n\to\infty$,
%	assume $c < a$, $c<1-\delta$ for constant $\delta>0$, $N < n/3$,  $D = O(\log n)$, and
%	$\limsup_{n\to\infty} \left\{ \left(\log_n\frac{N}{\sqrt{n}} +  \log_n\frac{a-c}{\sqrt{c}}\right) \vee \left(\sqrt{D/2-1}\log_n\frac{a-c}{\sqrt{c}}\right)\right\} < 0.$
%	For any series of degree-$D$ polynomials $\phi_n: A \to \mathbb{R}$, whenever $\mathbb{E}_{H_0}\phi_n(A)=0, \textrm{Var}_{H_0}(\phi_n(A))=1$, we must have $\mathbb{E}_{H_1} \phi_n(A) = o(1)$. This implies if Conjecture~\ref{conjecture} is true, there is no consistent polynomial-time test for this problem. %based on polynomial of degree at most $D = O(\log n)$, 
%	%has the sum of type I and type II errors tending to $1$. 
%\end{thm}

By Theorem~\ref{thm:compLB}, if both $(a-c)/\sqrt{c}\lesssim 1$ and $N(a-c)/\sqrt{cn}\to 0$, the testing problem is computationally infeasible. The region where the testing problem is statistically possible but the SgnQ test loses power corresponds to $N(a-c)/\sqrt{cn}\to0$. 
If $N\gtrsim \sqrt{n}$, Theorem~\ref{thm:compLB} already implies that this is the computationally infeasible region; in other words, SgnQ achieves the CLB and is optimal.  If $N=o(\sqrt{n})$, SgnQ solves the detection problem only when $(a-c)/\sqrt{c}\gg n^{1/2}$, i.e. when the node-wise SNR is strong. We discuss the case of moderate node-wise SNR in the next subsection.

% `moderate signal' case $1 \ll (a-c)/\sqrt{c}\ll  n^{1/2}$ in the next subsection. 
%
%We view $(a-c)/\sqrt{c}$ as the Signal-to-Noise Ratio (SNR) and call $(a-c)/\sqrt{c}\to 0$ and $(a-c)/\sqrt{c}\to\infty$ the `weak signal' case and `strong signal' case, respectively. In other words, for $N=o(\sqrt{n})$, SgnQ is optimal in the `weak signal' case. 

%when $(a-c)/\sqrt{c}\lesssim 1$ or $N\gtrsim n^{1/2}$, in the Region of Possibility where the SgnQ test loses power, there exits no other polynomial test (of degree at most $O(\log(n))$) that can succeed. Therefore, the SgnQ test is already optimal in the cases where $N\gtrsim n^{1/2}$ or $(a-c)/\sqrt{c}\lesssim 1$, if we restrict attention to polynomial-time tests.  

\vspace{-2mm}

\subsection{The power of EST, and discussions of the tightness of CLB}  \label{subsec:EST}
When $N=o(\sqrt{n})$ and $(a-c)/\sqrt{c}\to\infty$ both hold, the upper bound by SgnQ does not match with the CLB. It is unclear whether the CLB is tight. 
To investigate the CLB in this regime, we consider other possible polynomial-time tests. The economic scan test (EST) is one candidate. 
%\sout{Given a fixed constant $\rho \in (0,1)$, the EST statistic is defined to be $\phi_{EST}^{(\rho)}  \equiv \sup_{ |S| \leq 1/\epsilon } \sum_{i,j \in S} A_{ij}$. It can be computed in time $n^{1/\rho}$, which is polynomial time. } {\color{black}{\sout{Consider the EST that rejects the null if and only if $\phi_{EST} = (1/2)\flo{1/\rho}(\flo{1/\rho} - 1)$ (i.e., there exists a clique of size $\flo{1/\rho}$). 
%			We study the power of EST in the next theorem. For simplicity, we consider the SBM, i.e. where $\theta = {\bf 1}_n$.}}} 
{\color{black}{Given fixed positive integers $v$ and $e$, the EST statistic is defined to be $
		\phi_{EST}^{(v)}  \equiv \sup_{ |S| \leq v } \sum_{i,j \in S} A_{ij}$, and the EST is defined to reject if and only if $\phi_{EST}^{(v)} \geq e$.
		EST can be computed in time $O(n^{v})$, which is polynomial time.}} 
%Intuitively, if we replace each $A(i,j)$ by a Gaussian variable with the same mean and variance, we expect the EST to have a full power if $(a-c)/\sqrt{c}\to\infty$. 
%However, the network model has a quite different behavior from the 
{\color{black}{For simplicity, we consider the SBM, i.e. where $\theta = {\bf 1}_n$, and a specific setting of parameters for the null and alternative hypotheses.
}}

{\color{black}{
		\begin{thm}[Power of EST]
			\label{thm:EST} 
			Suppose $\beta\in [1/2, 1)$ and  $0<\omega < \delta<1$ are fixed constants. Under the alternative, suppose $\theta = \mf{1}_n$, \eqref{random-Pi} holds, $N=n^{1-\beta}$, $a = n^{-\omega}$, and $c =n^{-\delta}$. Under the null, suppose $\theta = \mf{1}_n$ and $\alpha = a(N/n)+b(1-N/n)$. 
			If $\omega/(1-\beta) < \delta$, the sum of type I and type II errors of the  EST with $v$ and $e$ satisfying $\omega/(1 - \beta) < v/e < \delta$ tends to $0$.
		\end{thm}		
}}

%{\color{black}{\begin{thm*}[\sout{Power of EST}]
%			%Consider a DCBM with $2$ communities as in \eqref{2group-DCBM},   where   \eqref{identifiable} and \eqref{altconditions} hold. Suppose $\theta={\bf 1}_n$. 
%			\sout{Consider a pair of hypotheses, where the alternative hypothesis satisfies \eqref{2group-DCBM}-\eqref{2group-DCBM-b}, \eqref{random-Pi} and that $\theta={\bf 1}_n$, and the null hypothesis is such that $\Omega=\alpha \theta\theta'$, with $\alpha=a(N/n)+b(1-N/n)$. Suppose  $N=n^{1-\beta}$, $a = n^{-\omega}$, and $c =n^{-\delta}$, for constants $\beta\in [1/2, 1)$ and  $0<\omega < \delta<1$. Let $\rho = \frac{(1 - \beta) \alpha + \delta}{2}$ in the EST. 
%				If $\omega < (1-\beta)\delta$, the sum of type I and type II errors of the  EST tends to $0$. }
%		\end{thm*}
%}}

\begin{figure}[t]
	%						\begin{subfigure}[b]{0.33\textwidth}
	%				\centering
	%				\includegraphics[width=\textwidth]{PDF/asymptotic_normality.png}
	%			\end{subfigure}
	\centering
	\includegraphics[width=0.325\textwidth]{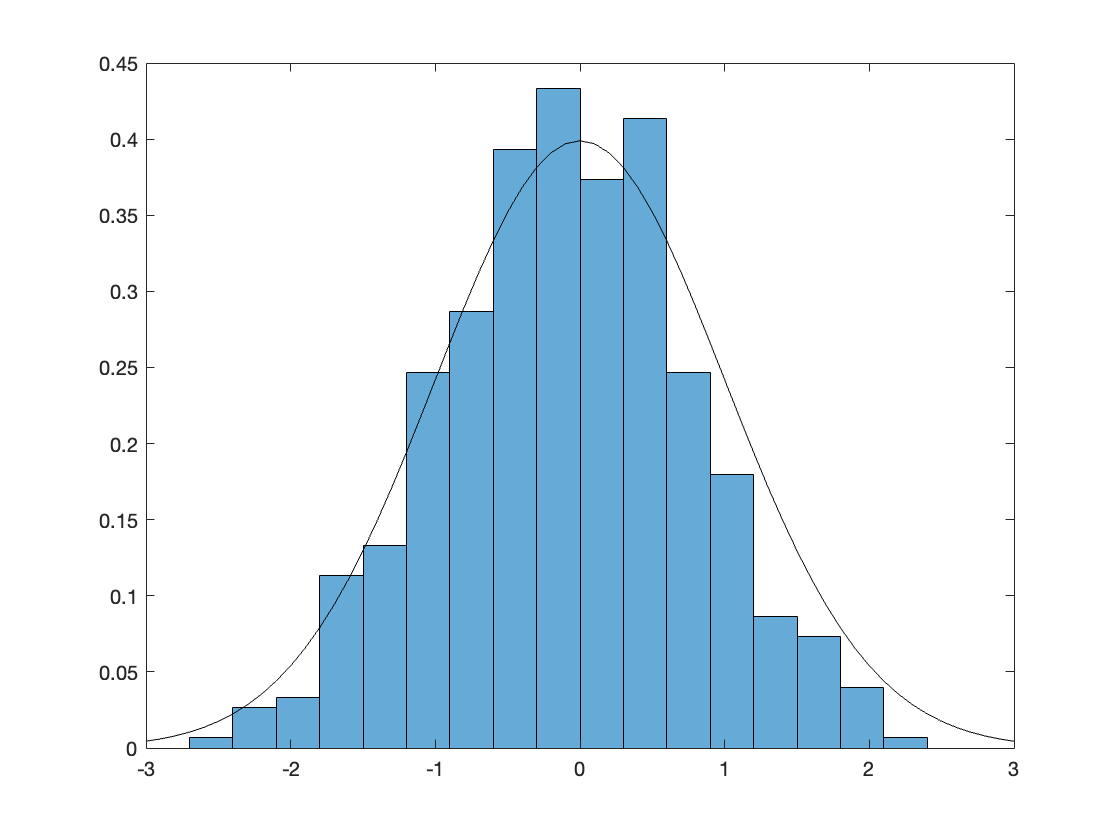}
	\includegraphics[width=.33\textwidth]{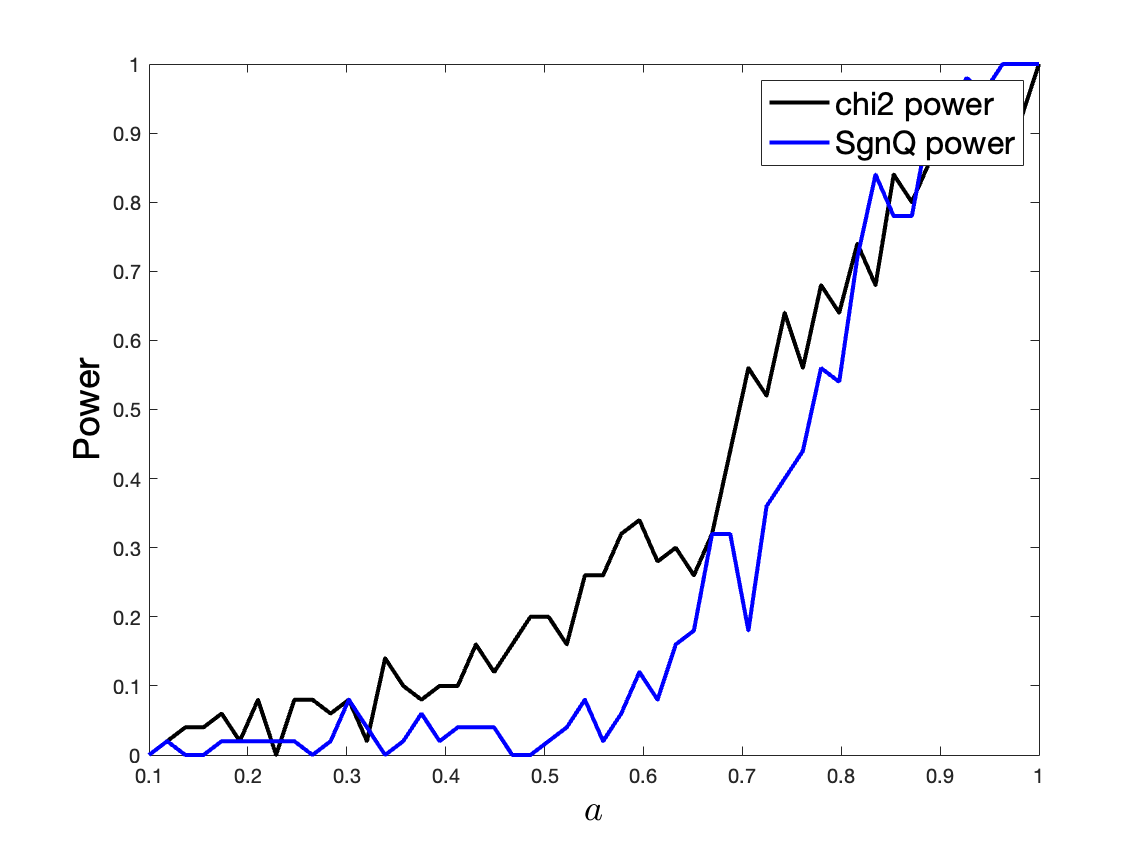}
	\includegraphics[width=.33\textwidth]{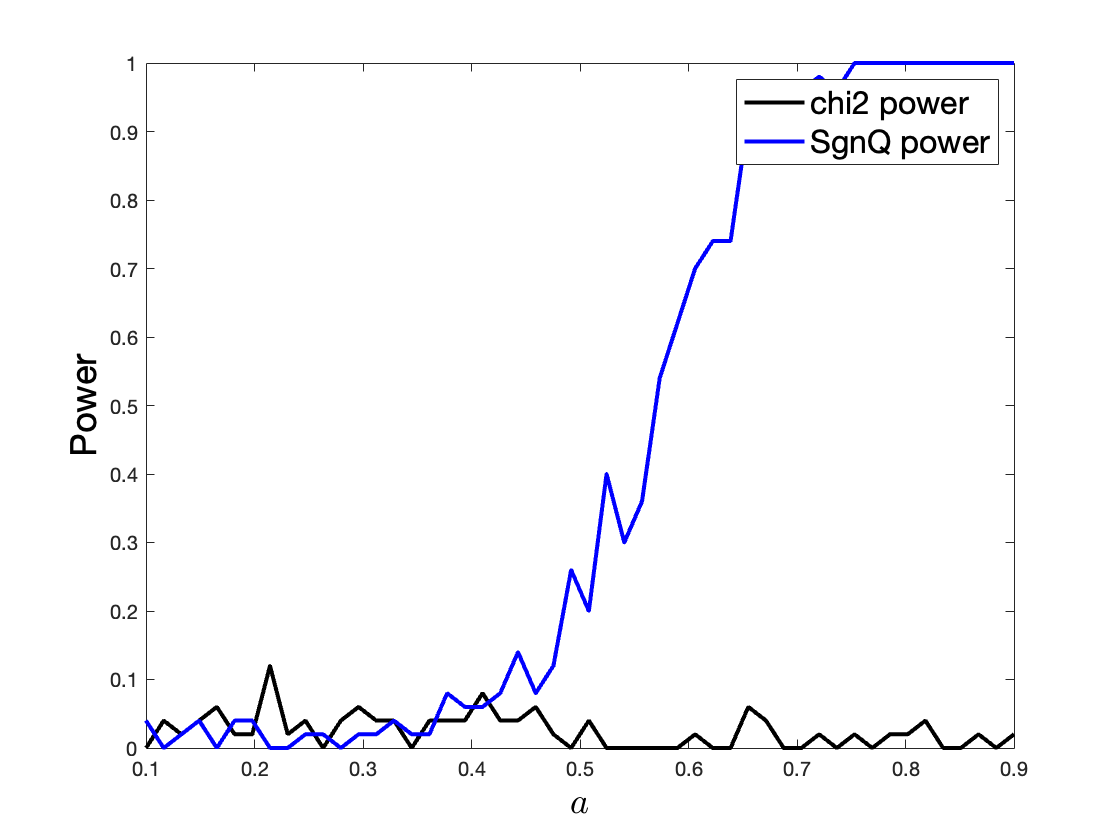}
	\caption{Left: Null distribution of SgnQ ($n = 500$). Middle and right: Power comparison of SgnQ and $\chi^2$ ($n=100$, $N=10$, 50 repetitions). We consider a 2-community SBM with $P_{11} = a$, $P_{22}=0.1$, $P_{12} = 0.1$ (middle plot) and $P_{12} = \frac{an - (a + 0.1)N}{n}$ (right plot, the case of degree matching).}
	\label{fig:chi2_vs_SgnQ}
	\vspace{-.5cm}
\end{figure}

%The EST was studied in the planted subgraph model \cite{bhaskara2010detecting} in the context of the log-density conjecture.  
Theorem \ref{thm:EST} follows from standard results in probabilistic combinatorics \citep{alon2016probabilistic}. 
It is conjectured in \citet{bhaskara2010detecting} that EST attains the CLB in the Erd\"{o}s-Renyi setting considered by \cite{Ery1,Ery0}. This suggests that the CLB in Theorem~\ref{thm:compLB} is likely not tight when $N=o(\sqrt{n})$ and $(a-c)/\sqrt{c}\to\infty$. 
However, this is not because our inequalities in proving the CLB are loose. A possible reason is that the prediction from the low-degree polynomial conjecture does not provide a tight bound. It remains an open question whether other computational infeasibility frameworks provide a tight CLB in our problem. %We leave this to future study. 

\vspace{-2mm}

\subsection{The phase transition} 
\label{subsec:phase} 

We describe more precisely our results in terms of the phase transitions shown in Figure \ref{fig:Phase}. Consider the null and alternative hypotheses from Section \ref{subsec:identifiability}. For illustration purposes, we fix constants $\beta\in (0,1)$ and $\gamma\in\mathbb{R}$ and assume that $N= n^{1-\beta}$ and $(a-c)/\sqrt{c}=n^{-\gamma}$. In the two-dimensional space of $(\gamma, \beta)$, the region of $\beta>1/2$ and $\beta<1/2$ corresponds to that the size of the small community is $\gg \sqrt{n}$ and $o(\sqrt{n})$, respectively, and the regions of $\gamma>0$, $-1/2<\gamma<0$ and $\gamma<-1/2$ correspond to `weak node-wise signal', `moderate node-wise signal,' and the `strong node-wise signal', respectively. See Figure~\ref{fig:Phase}. By our results in Section~\ref{subsec:statLB}, the testing problem is statistically impossible if $\beta+2\gamma>1$ (orange region). By our results in Section~\ref{subsec:SgnQnull}, SgnQ has a full power if $\beta+\gamma<1/2$ (blue region). Our results in Section~\ref{subsec:compLB} state that the testing problem is computationally infeasible if both $\gamma>0$ and $\beta+\gamma>1/2$ (green and orange regions). Combining these results, when $\beta<1/2$, we have a complete understanding of the LB and CLB. 

\vspace{-.2cm}

\section{Numerical results}\label{sec:Simu}

\vspace{-.2cm}

{\bf Simulations}. First in Figure \ref{fig:chi2_vs_SgnQ} (left panel) we demonstrate the asymptotic normality of SgnQ under a null of the form $\Omega = \theta \theta'$, where $\theta_i$ are i.i.d. generated from $\mathrm{Pareto}(4, 0.375)$. 
%\begin{wrapfigure}[11]{r}{0.4\textwidth} %this figure will be at the right
%	\centering
%	\includegraphics[width=0.25\textwidth]{PDF/asymptotic_normality.png}
%	\caption{\small Histogram of SgnQ with $n = 500$ and $500$ repetitions. }\label{fig:asymptotic_normality}
%\end{wrapfigure} 
Though the degree heterogeneity is severe, SgnQ  properly standardized is approximately standard normal under the null. Next in Figure \ref{fig:chi2_vs_SgnQ} we compare the power of SgnQ in an asymmetric and symmetric SBM model. As our theory predicts, both tests are powerful when degrees are not calibrated in each model, but only SgnQ is powerful in the symmetric case. 
We also compare the power of SgnQ with the scan test to show evidence of a statistical-computational gap. We relegate these experiments to the supplement.
%	\begin{figure}[htb]
%		\centering
%		\includegraphics[width=.3\textwidth]{PDF/asymptotic_normality.png}
%		\caption{}\label{fig:asymptotic_normality}
%	\end{figure}

%%%%%%%%%%%

%%%%%%%%%%%%%	

%\vspace{30pt}

{\bf Real data}: 
Next we demonstrate the effectiveness of SgnQ in detecting small communities in coauthorship networks studied in \cite{JBES2022}. In Example 1, we consider the personalized network of Raymond Carroll, whose nodes consist of his coauthors for papers in a set of 36 statistics journals from the time period 1975 -- 2015. An edge is placed between two coauthors if they wrote a paper in this set of journals during the specified time period. The SgnQ p-value for Carroll's personalized network $G_{\mathsf{Carroll}}$ is $0.02$, which suggests the presence of more than one community. In \cite{JBES2022}, the authors identify a small cluster of coauthors from a collaboration with the National Cancer Institute. We applied the SCORE community detection module with $K = 2$ (e.g. \cite{KeJin22}) and obtained a larger community $ G_{\mathsf{Carroll}}^{\, 0}$ of size $218$ and a smaller community $G_{\mathsf{Carroll}}^{\, 1}$ of size $17$. Precisely, we removed Carroll from his network, applied SCORE on the remaining giant component, and defined $G_{\mathsf{Carroll}}^{\, 0}$ to be the complement of the smaller community. The SgnQ p-values in the table below suggest that both $G_{\mathsf{Carroll}}^{\, 0}$ and $G_{\mathsf{Carroll}}^{\, 1}$ are tightly clustered. Refer to the supplement for a visualization of Carroll's network and its smaller community labeled by author names. In Example 2, we consider three different coauthorship networks $G_{\mathsf{old}}$, $G_{\mathsf{recent}}$, and $G_{\mathsf{new}}$ corresponding to time periods (i) 1975-1997, (ii) 1995-2007, and (iii) 2005-2015 for the journals AoS, Bka, JASA, and JRSSB. Nodes are given by authors, and an edge is placed between two authors if they coauthored at least one paper in one of these journals during the corresponding time period. 
For each network, we perform a similar procedure as in the first example. First we compute the SgnQ p-value, which turns out to be $\approx 0$ (up to 16 digits of precision) for all networks. For each $i \in \{ \mathsf{old}, \mathsf{recent}, \mathsf{new} \}$, we apply SCORE with $K = 2$ to $G_i$ and compute the SgnQ p-value on both resulting communities, let us call them $G_{i}^0$ and $G_{i}^1$. We refer to the table below for the results.
For $G_{\mathsf{old}}$ and $G_{\mathsf{recent}}$, SCORE with $K = 2$ extracts a small community. The SgnQ p-value further supports the hypothesis that this small community is well-connected. In the last network, SCORE splits $G_{\mathsf{new}}$ into two similarly sized pieces whose p-values suggests they can be split into smaller subcommunities. 
%	The table below summarizes our results in Examples 1 and 2. 
\begin{center}
\scalebox{.8}{
	\renewcommand{\arraystretch}{1.2}
	\begin{tabular}{|c c c c c c c|} 
		\hline Example 
		& Network & Size & SgnQ p-value & Communities  & Sizes &SgnQ p-values \\ 
		\hline
		1 & $G_{\mathsf{Carroll}}$  &235 & 0.02 & $(G_{\mathsf{Carroll}}^{\, 0}, G_{\mathsf{Carroll}}^1)$ & (218, 17) & (0.134, 0.682) \\ 
		\hline 
		2 &$G_{\mathsf{old}}$ &2647 & 0 & $(G_{\mathsf{old}}^0, G_{\mathsf{old}}^1)$ & (2586, 61) & (0, 0.700) \\ 
%			\hline
		& $G_{\mathsf{recent}}$ &2554 & 0 & $(G_{\mathsf{recent}}^0,G_{\mathsf{recent}}^1)$ & (2540,14) & (0, 0.759)\\ 
%			\hline
	 & $G_{\mathsf{new}}$ &2920 & 0 & $(G_{\mathsf{new}}^0, G_{\mathsf{new}}^1)$ & (1685,1235) & (0, 0)  \\ 
		\hline \hline 
\end{tabular}}
\end{center}

%	\begin{center}
%		\scalebox{.8}{
%			\begin{tabular}{||c c c c c c||} 
%				\hline
%				Network & Size & SgnQ Pvalue & SCORE output & Sizes &SgnQ Pvalues \\ [0.5ex] 
%				\hline\hline
%				$G_1$ &2647 & 0 & $(G_1\rp{1}, G_1\rp{2})$ & (2586, 61) & (0, 0.700) \\ 
%				\hline
%				$G_2$ &2554 & 0 & $(G_2\rp{1},G_2\rp{2})$ & (2540,14) & (0, 0.759)\\ 
%				\hline
%				$G_3$ &2920 & 0 & $(G_3\rp{1}, G_3\rp{2})$ & (1685,1235) & (0, 0)  \\ 
%				\hline \hline 
%		\end{tabular}}
%	\end{center}

%	\begin{figure}[t]
%		\centering
%		\begin{subfigure}[b]{0.45\textwidth}
%			\centering
%			\includegraphics[trim={0.25cm 3.5cm 0 5cm},clip,width=.8\textwidth]{PDF/Carroll_network.pdf}
%		\end{subfigure} \hfill 
%		\begin{subfigure}[t]{0.5\textwidth}
%			\centering
%			\includegraphics[trim={0 0 0 0cm},clip,width=.8\textwidth, angle = 0]{PDF/Carroll_clique.png}
%		\end{subfigure}
%		\caption{\textbf{Left:} Carroll's personalized network, figure taken from \cite{JBES2022}. \textbf{Right:} A small community of $17$ authors extracted by SCORE and whose SgnQ p-value is $0.6818$. }
%		\label{fig:Carroll}
%	\end{figure}	
%%%%%%%%%%%%%%

%\section{Discussions} \label{sec:Discu} 
{\bf Discussions}: Global testing is a fundamental problem 
and often the starting point of a long line of research. For example, in the literature of 
Gaussian models, certain methods started as a global testing tool, but later 
grew into tools for variable selection,  classification,  and clustering and motivated many researches  (e.g., \cite{DJ2004, DJ2015}). The SgnQ test may also motivate tools for many other problems, such as estimating the locations of the clique and clustering more generally. For example, in \cite{JKLW2022}, the SgnQ test motivated a tool for estimating the number of communities (see also \cite{ma2021determining}). SgnQ is also extendable to clique detection in a tensor \citep{FengTensor, JinKeLiang21} and for network change point detection.  
The LB and CLB we obtain in this paper are also useful for studying other problems, such as 
clique estimation. If you cannot tell whether there is a clique in the network, then it is impossible 
to estimate the clique. Therefore, the LB and CLB are also valid for the clique estimation problem \citep{alon1998finding, ron2010finding}. 

The limiting distribution of SgnQ is $N(0,1)$. This is not easy to achieve 
if we use other testing ideas, such as the leading eigenvalues of the adjacency matrix:
the limiting distribution depends on many unknown parameters and it is hard to normalize \citep{liu2019community}. 
The p-value of the SgnQ test is easy to approximate and also useful in applications. 
For example, we can use it to measure the research diversity of a given author. 
Consider the ego sub-network of an author in a large co-authorship or citation network. 
A smaller p-value suggests that the ego network has more than $1$ communitiy and  has more diverse interests. The p-values can also be useful as a stopping criterion in hierarchical community detection modules.
%	A typical hierarchical community detection algorithm is recursive, and in each stage we need to decide whether we should 
%	further divide a sub-community. 
%	{\color{red}
%		Drawing inspiration from \cite{yin2017local, yin2018higher}, SgnQ may also serve as the basis for other clustering algorithms. \cite{yin2017local, yin2018higher} utilize small subgraphs referred to as `motifs' as the basis for a clustering algorithm called MAPPR. When the motif is a $4$-cycle, this algorithm optimizes a `higher-order' clustering coefficient based on the counts of $4$-cycles. This coefficient can be extended to \textit{signed} counts of $4$-cycles, i.e. SgnQ. It is an interesting direction to adapt MAPPR to signed cycle counts. 
%	}

\textbf{Acknowledgments.} We thank the anonymous referees for their helpful comments. We thank Louis Cammarata for assistance with the simulations in Section A.3. J. Jin was partially supported by NSF grant DMS-2015469. Z.T. Ke was supported in part by NSF CAREER Grant DMS-1943902. A.R. Zhang acknowledges the grant NSF  CAREER-2203741. 

%%HERE 

\appendix
\addcontentsline{toc}{section}{Appendix} % Add the appendix text to the document TOC
\part{Appendix} % Start the appendix part
\parttoc % Insert the appendix TOC

\section{Additional experiments}

\subsection{Visualization of Carroll's network} 

In Figure \ref{fig:Carroll}, we display a subgraph of high-degree nodes of Raymond Carroll's personalized coauthorship network (figure borrowed with permission from \citet{JBES2022}). On the right of Figure \ref{fig:Carroll} is shown the small community extracted by SCORE, and this cluster of size $17$ is labeled by author names.  

\begin{figure}[h]
	\centering
	\begin{subfigure}[b]{0.45\textwidth}
		\centering
		\includegraphics[trim={0.25cm 3.5cm 0 5cm},clip,width=\textwidth]{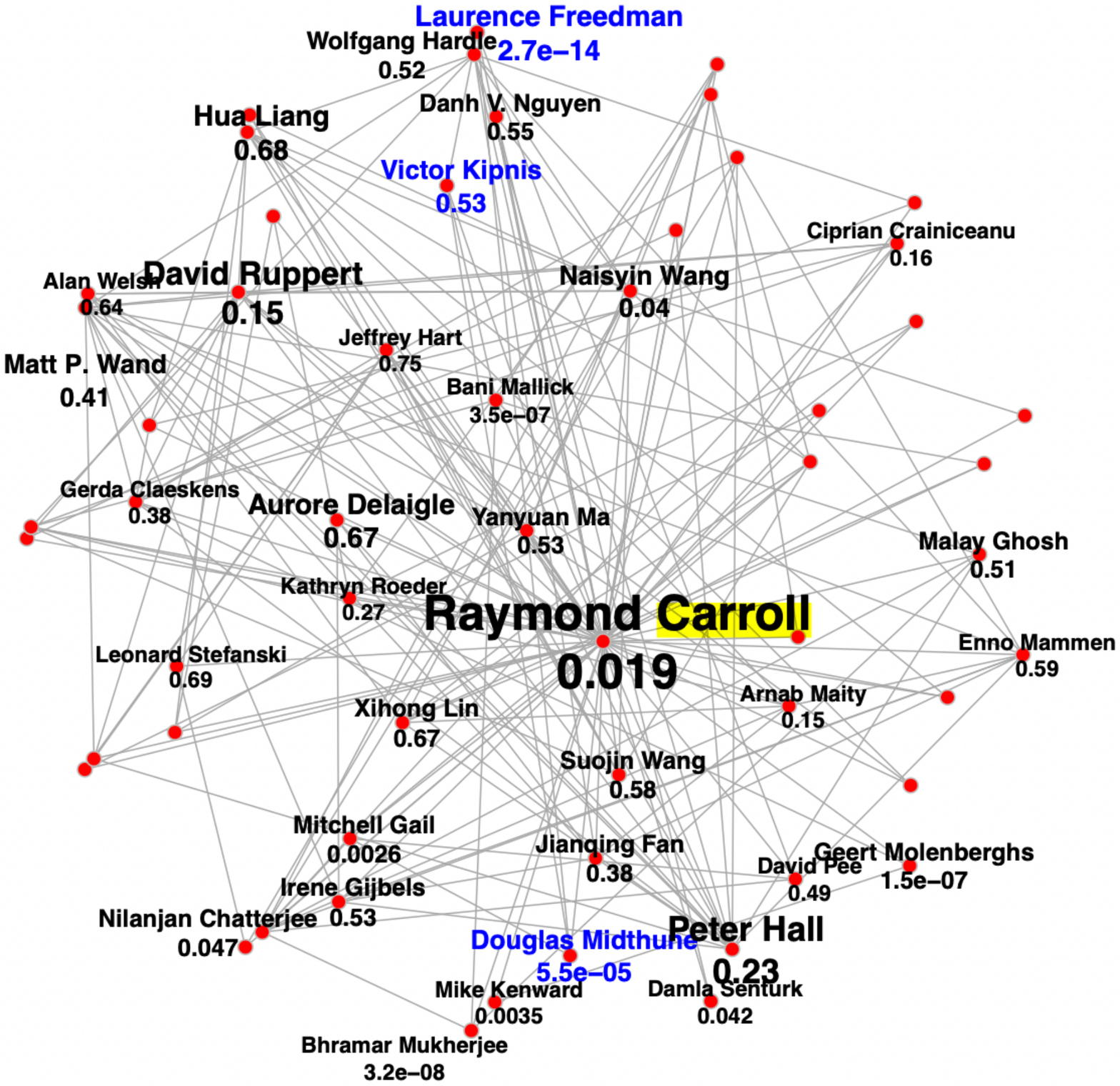}
	\end{subfigure} \hfill 
	\begin{subfigure}[t]{0.5\textwidth}
		\centering
		\includegraphics[trim={0 0 0 0cm},clip,width=\textwidth, angle = 0]{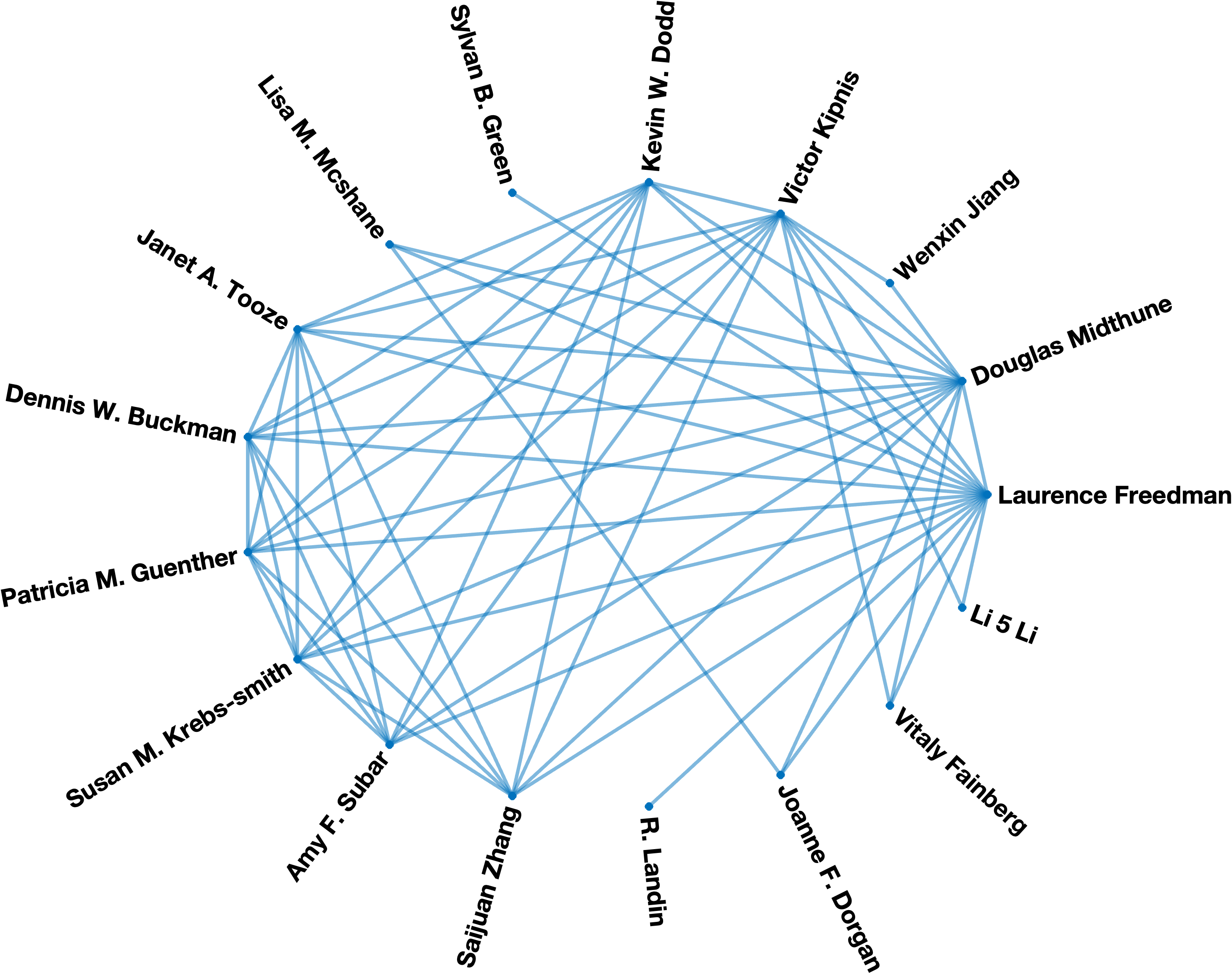}
	\end{subfigure}
	\caption{\textbf{Left:} Carroll's personalized network, figure taken from \cite{JBES2022}. \textbf{Right:} A small community of $17$ authors extracted by SCORE and whose SgnQ p-value is $0.6818$. }
	\label{fig:Carroll}
\end{figure}	

\subsection{SgnQ vs. Scan}

In this section we demonstrate evidence of a statistical-computational gap by means of numerical experiments. 

We consider a SBM null and alternative model (as in Example 2 with $\theta \equiv 1$) with
\[
P_0 = \begin{pmatrix}
	\alpha & \alpha \\
	\alpha & \alpha 
\end{pmatrix}, \qquad 
P_1 = \begin{pmatrix}
	a & b \\
	b & c 
\end{pmatrix}
\]
where $aN + b(n - N) =\alpha$. For this simple testing problem, we compare the power of SgnQ and the scan test. In our experiments, we set $\alpha = 0.2$ and allow the parameter $a$ to vary from $a = \alpha$ to $a = a_{\max} \equiv a n/N$. Once $a$ and $\alpha$ are fixed, the parameters $b$ and $c$ are determined by 
\begin{align*}
	c &= \frac{a N^2+\alpha n^2-2 \alpha n N}{(n-N)^2},
	\\ b &= \frac{nc - (a + c)N}{n - 2N}.
\end{align*}
In particular, $a_{\max}$ is the largest value of $a$ such that $b \geq 0$.

Since the scan test $\phi_{sc}$ we defined is extremely computationally expensive, we study the power of an `oracle' scan test $\ti \phi_{sc}$ which knows the location of the true planted subset $\mathcal{C}_1$. The power of the oracle scan test is computed as follows. Let $\kappa$ denote the desired level. 
\begin{enumerate}
	\item Using $M_{cal}$ repetitions under the null, we calculate  the (non-oracle) scan statistic $\phi_{sc}\rp{1}, \ldots, \phi_{sc}\rp{M_{cal}}$ for each repetition. We set the threshold $\hat \tau$ to be the empirical $1 - \kappa$ quantile of $\phi_{sc}\rp{1}, \ldots, \phi_{sc}\rp{M_{cal}}$. 
	\item Given a sample from the alternative model, we compute the power using $M_{pow}$ repetitions, where we reject if
	\[
	\ti \phi_{sc} \equiv \mf{1}_{\mc{C}_1} (A - \hat \eta \hat \eta') \mf{1}_{\mc{C}_1} > \hat \tau. 
	\]
\end{enumerate}

In our experiments, we set $M_{cal} = 75$ and $M_{pow} = 200$. 

Note that since $\ti \phi_{sc} \leq \phi_{sc}$, the procedure above gives an underestimate of the power of the scan test (provide the threshold is correctly calibrated), which is helpful since this can be used to show evidence of a statistical-computational gap.

In our plots we also indicate the statistical (information-theoretic) and computational thresholds in addition to the power. Inspired by the sharp characterization of the statistical threshold in  \cite[Equation (10)]{Ery1} for planted dense subgraph, in all plots we draw a black vertical dashed line at the first value of $a$ such that 
\[
(1/2) \sqrt{N} (a - c)/\sqrt{c (1 - c)} > 1. 
\]
We draw a blue vertical dashed line at the first value of $a$ such that
\[
N(a - c)/\sqrt{nc} > 1. 
\]

\begin{figure}[tb]
	%						\begin{subfigure}[b]{0.33\textwidth}
	%				\centering
	%				\includegraphics[width=\textwidth]{PDF/asymptotic_normality.png}
	%			\end{subfigure}
	\centering
	\includegraphics[width=0.45\textwidth]{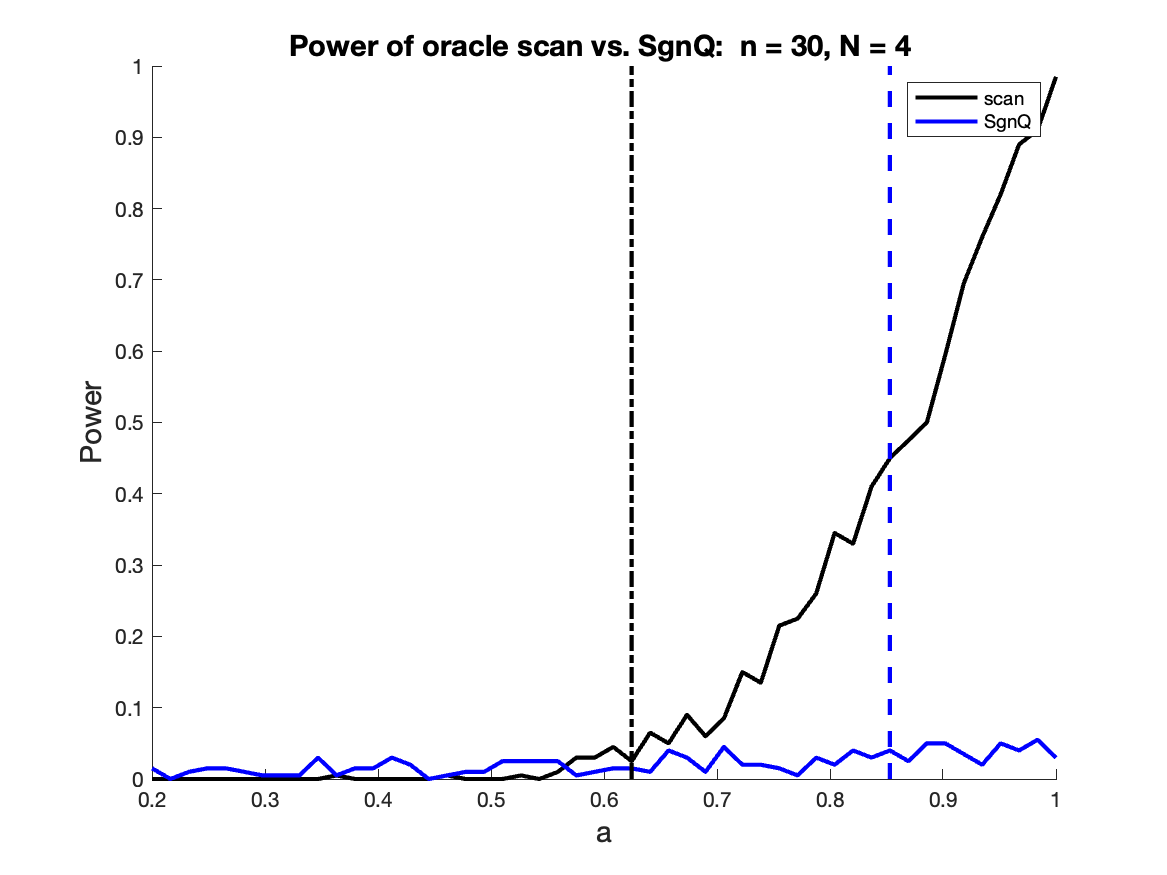}
	\includegraphics[width=.45\textwidth]{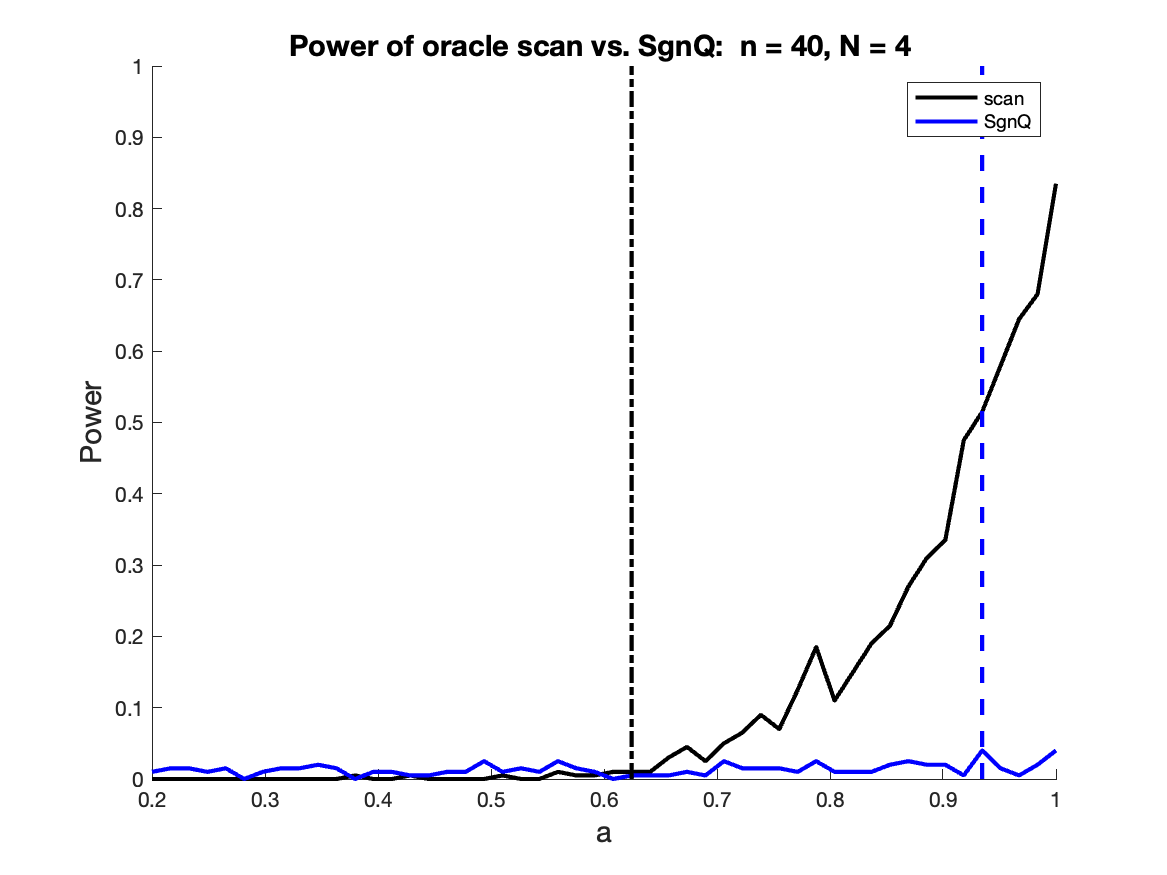}
	\includegraphics[width=.45\textwidth]{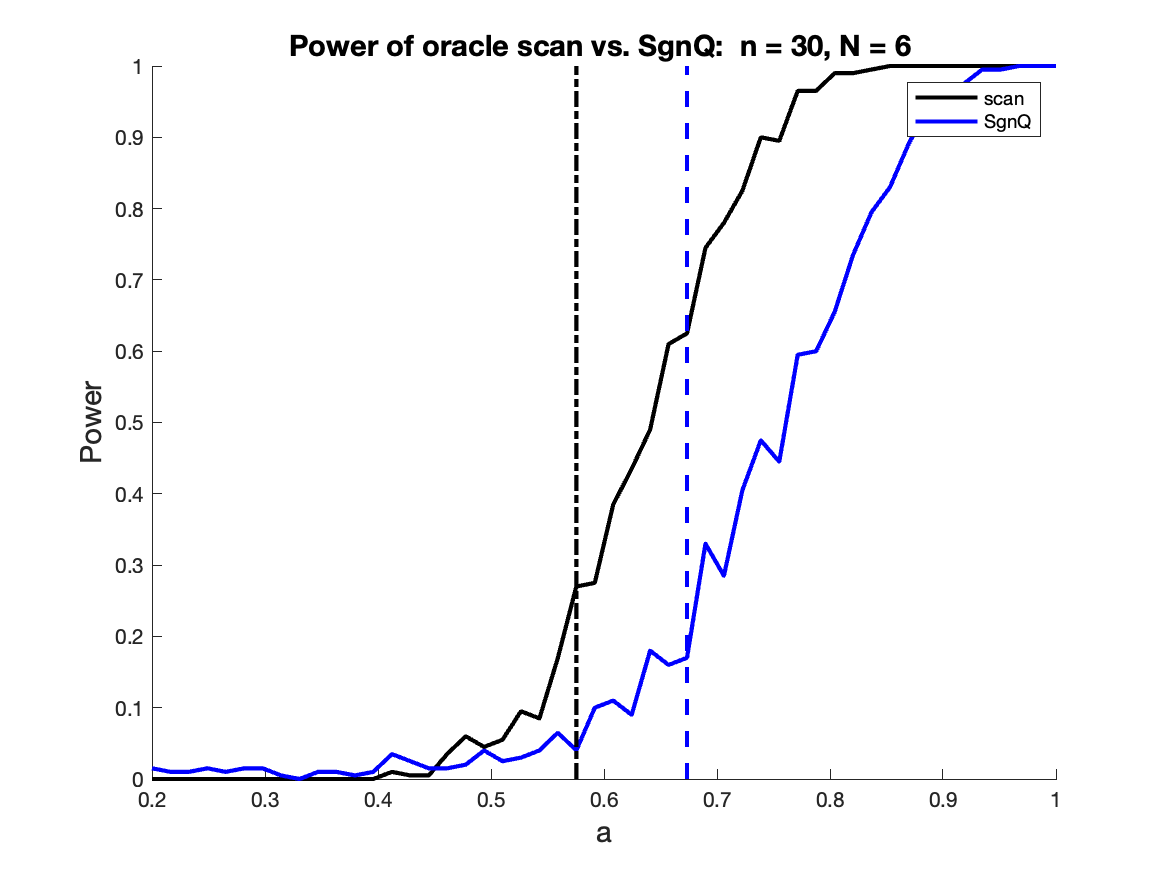}
	\includegraphics[width=.45\textwidth]{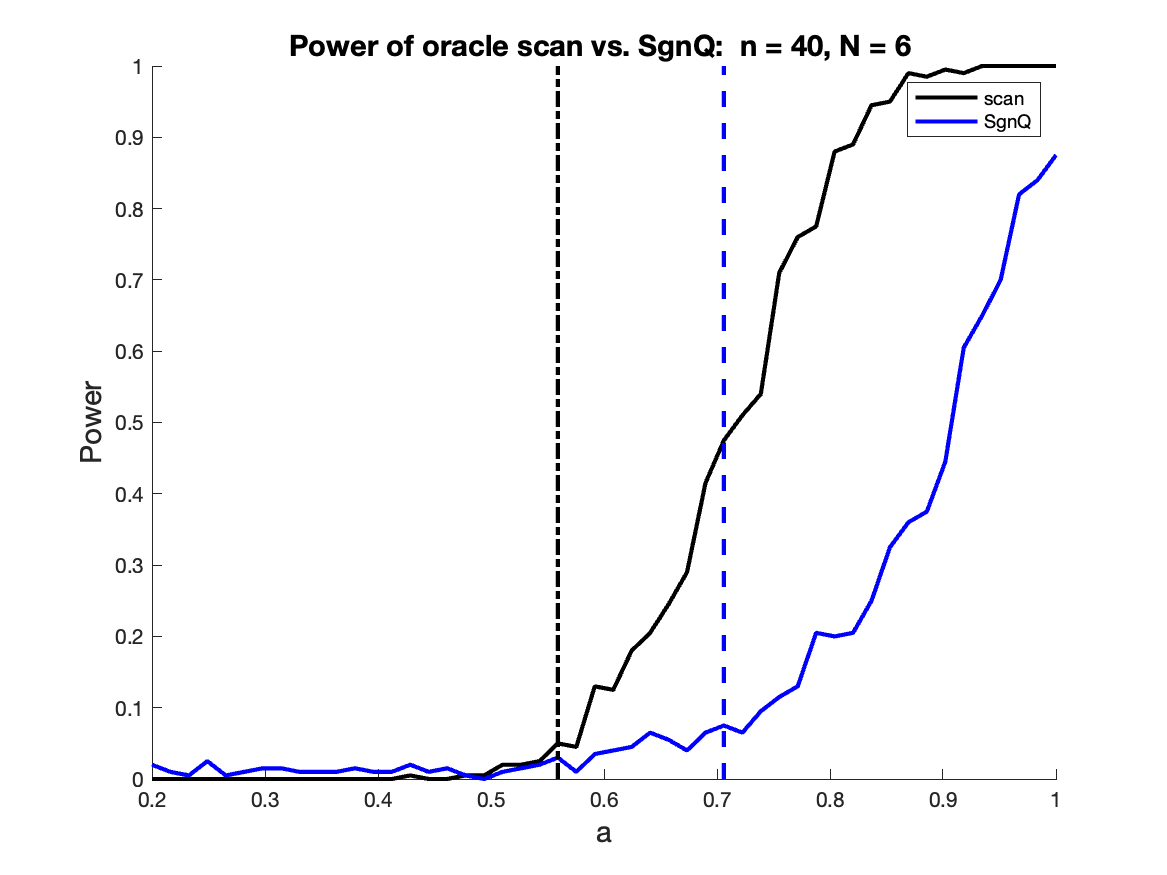}
	\includegraphics[width=.45\textwidth]{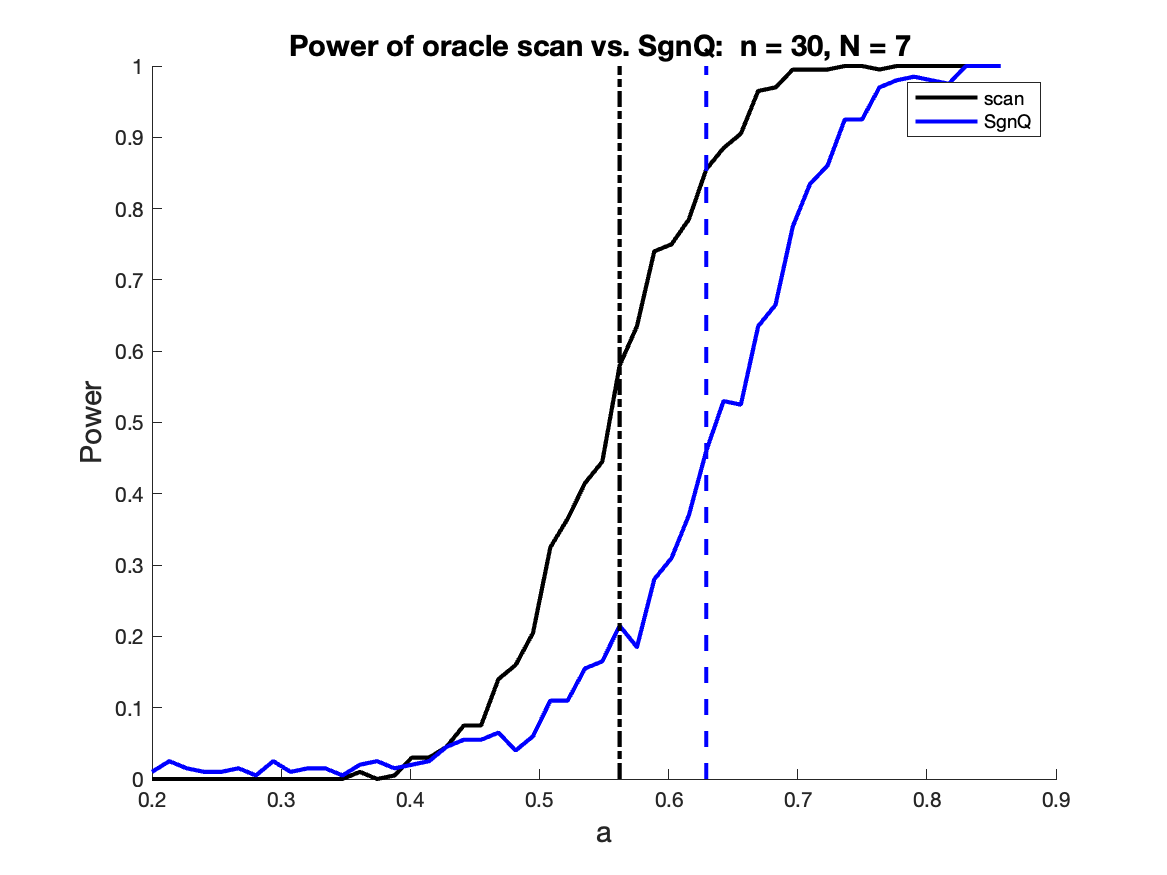}
	\includegraphics[width=.45\textwidth]{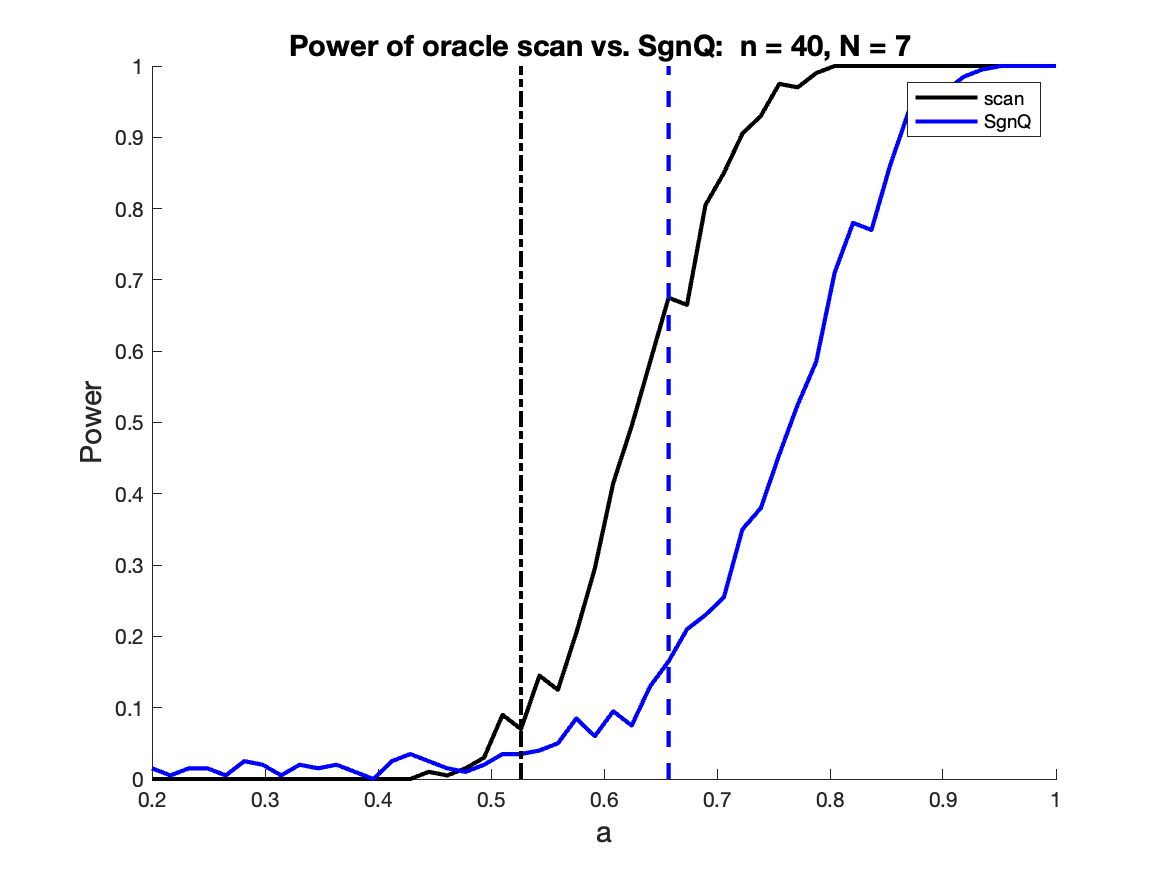}
	\caption{The power of SgnQ (blue curve) and oracle scan (black curve) for $n = 30,N \in \{4, 6, 7 \}$ (left) and $n = 40, N \in \{4, 6, 7\}$ (right). The black dashed line indicates the theoretical statistical threshold, and the blue dashed line indicates the theoretical computational threshold. 
	}
	\label{fig:scan_vs_SgnQ}
\end{figure}

\subsection{$\chi_2$ vs. SgnQ}
We also show additional experiments demonstrating the effect of degree-matching on the power of the $\chi^2$ test. We compute the power with respect to the following alternative models (as in Example 2 with $\theta \equiv 1$) with
\[
P\rp{1} = \begin{pmatrix}
	a & b \\
	b & c 
\end{pmatrix}, \qquad P\rp{2} = \begin{pmatrix}
	a & c \\
	c & c 
\end{pmatrix}
\]
where $b = \frac{cn - (a + c)N}{n  - 2N}$, $c$ is fixed, and $a$ ranges from $c$ to $a_{\max}' = c(n - N)/N$ for the experiments with $P\rp{1}$. Similar to before, $a_{\max}'$ is the largest value of $a$ such that $b \geq 0$. See Figure \ref{fig:more_scan_vs_SgnQ} for further details.

\begin{figure}[tb]
	%						\begin{subfigure}[b]{0.33\textwidth}
	%				\centering
	%				\includegraphics[width=\textwidth]{PDF/asymptotic_normality.png}
	%			\end{subfigure}
	\centering
	\includegraphics[width=.48\textwidth]{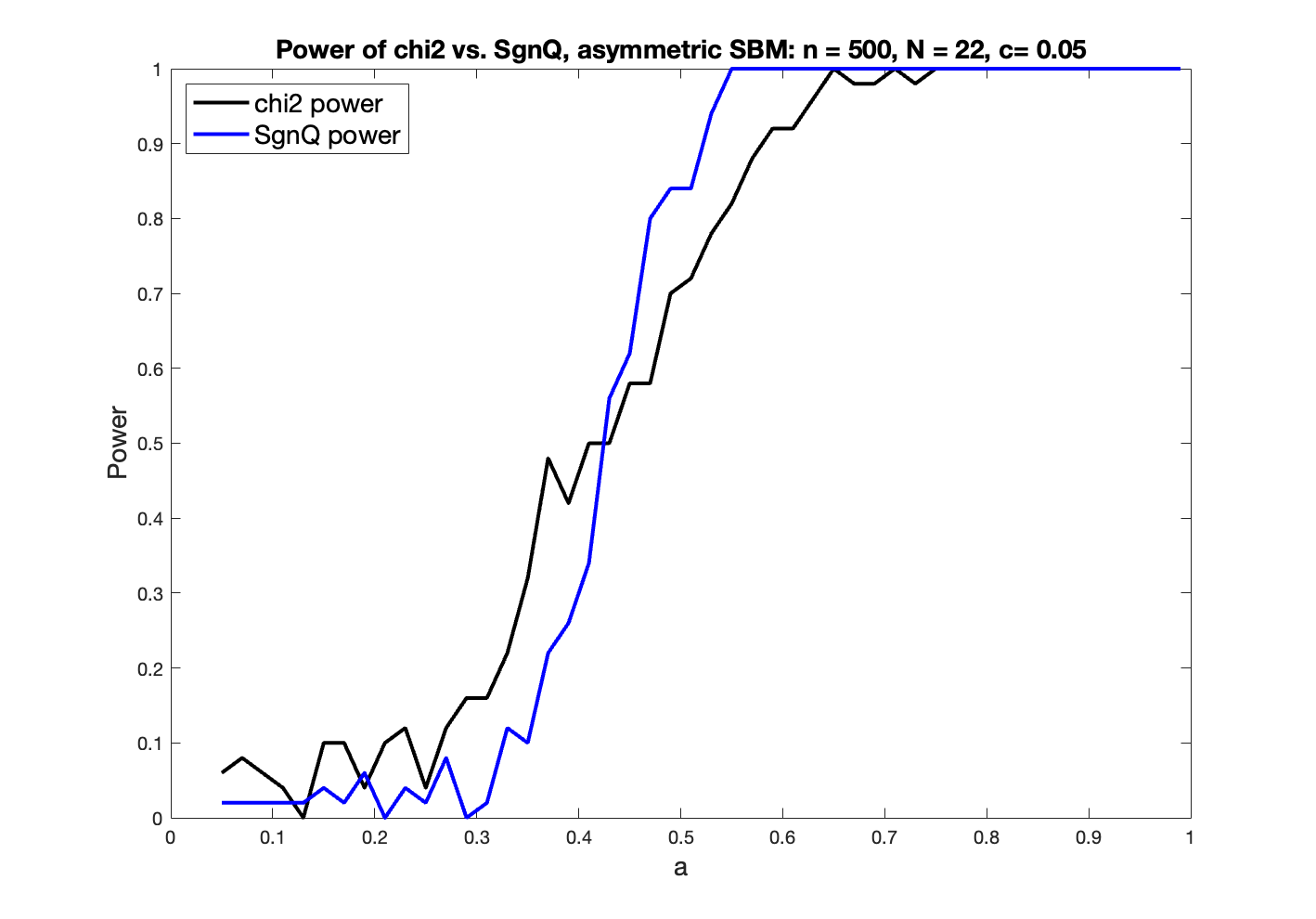}
	\includegraphics[width=.48\textwidth]{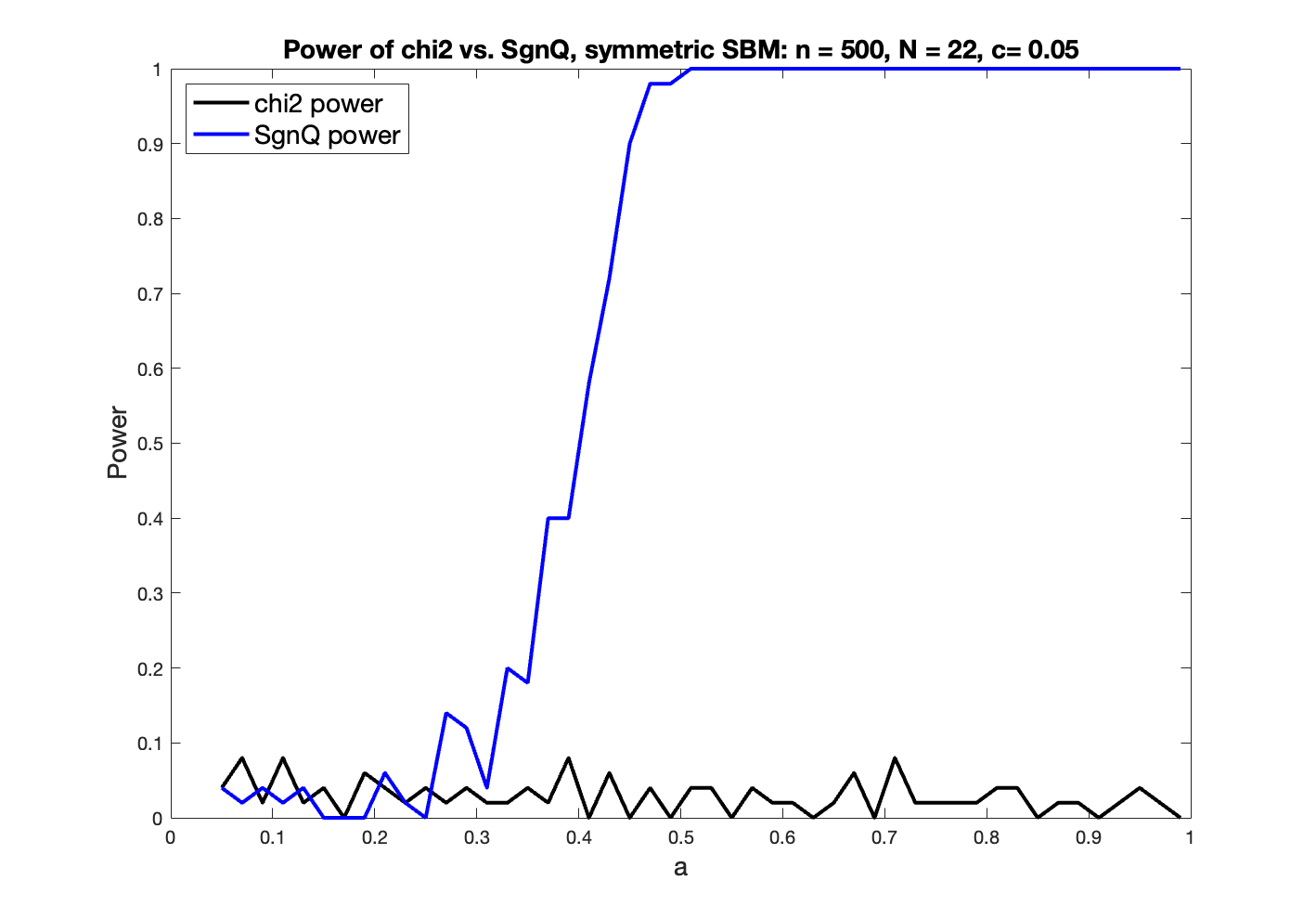}
	\includegraphics[width=.48\textwidth]{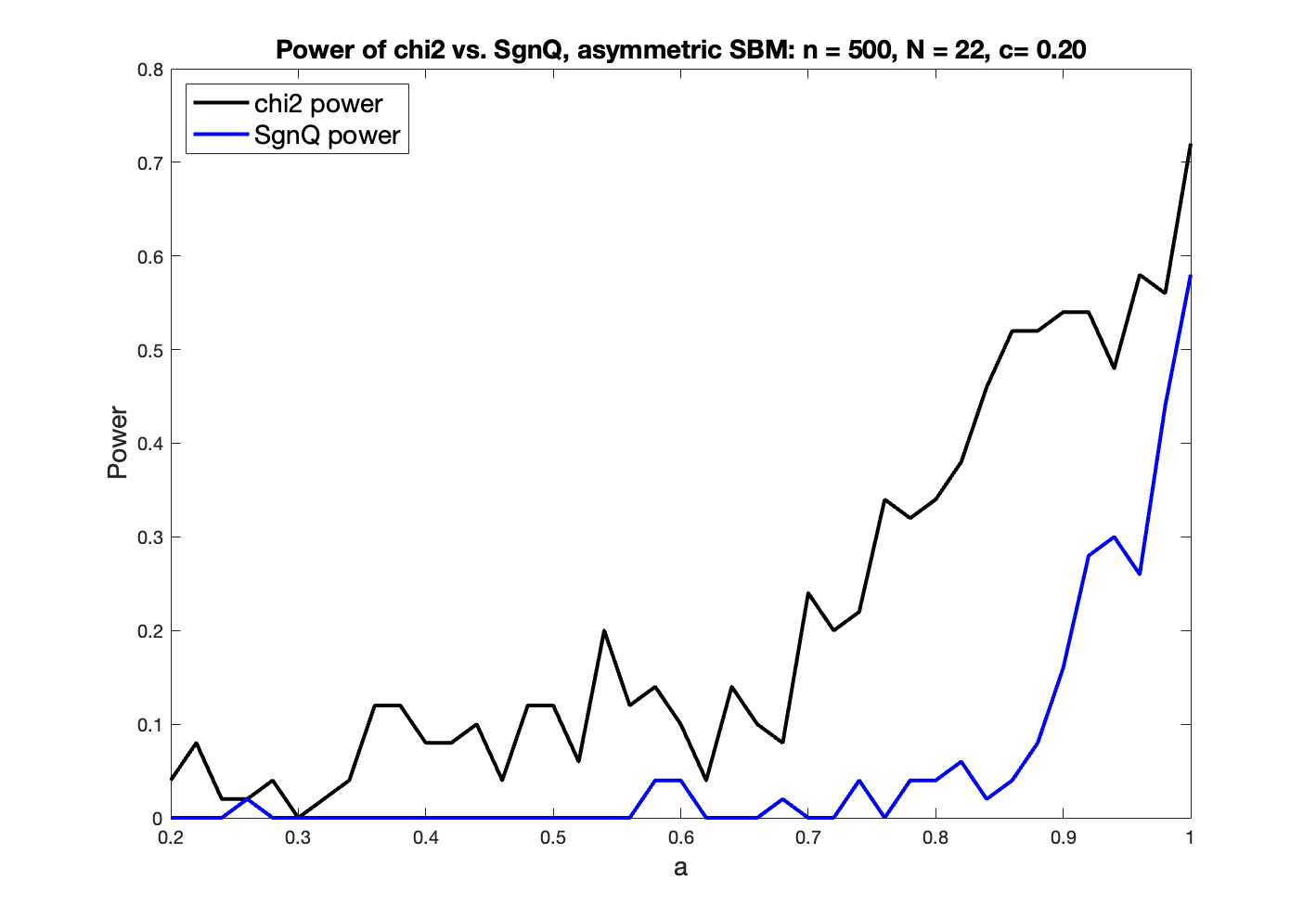}
	\includegraphics[width=.48\textwidth]{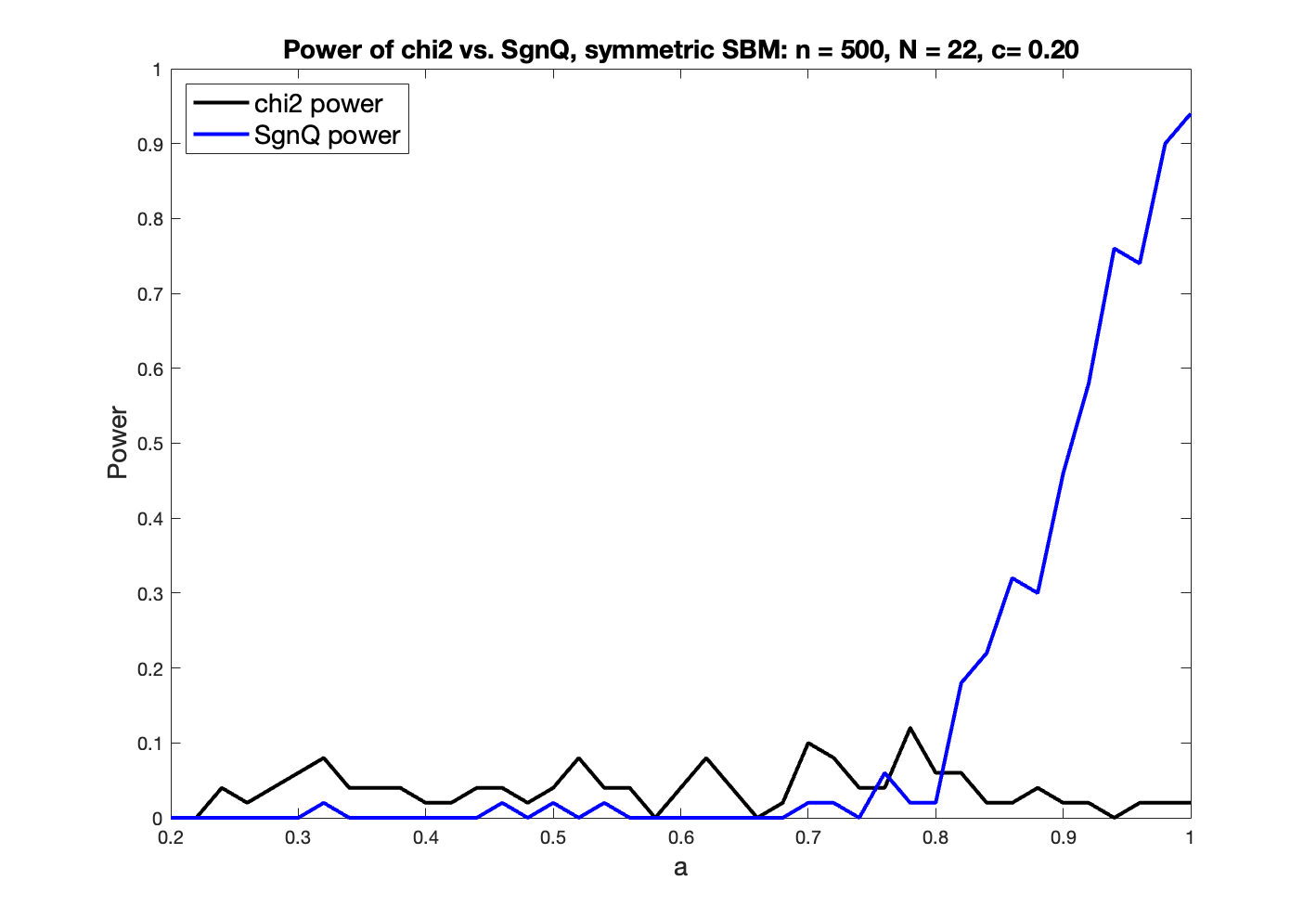}
	\caption{Power comparison of SgnQ and $\chi^2$ ($n=500$, $N=22$, 50 repetitions). We consider a 2-community SBM with $P_{11} = a$, $P_{22}=c$, $P_{12} = c$ (left) and $P_{12} = \frac{an - (a + c)N}{n}$ (right plot, the case of degree matching) where $c = 0.05$ (top row) and $c = 0.20$ (bottom row). 
	}
	\label{fig:more_scan_vs_SgnQ}
\end{figure}

\section{Proof of Lemma \ref{lem:identifiability} (Identifiability) }

%\textbf{Note:} We require a mild assumption that the diagonal elements of $\Omega$ are positive, and this will be added to the main text. 

%	To prove the lemma, we make use of two preliminary results. The first is from \cite[Lemma 3.1]{JinKeLuo21}, and the second is an immediate corollary of \cite[Theorem 4]{brualdi1974dad}. Both results are in line with Sinkhorn's work \cite{Sinkhorn} on matrix scaling.

To prove identifiability, we make use of the following result from  \cite[Lemma 3.1]{JinKeLuo21}, which is in line with Sinkhorn's work \cite{Sinkhorn} on matrix scaling.

\begin{lemma}[\cite{JinKeLuo21}] \label{lem:existence}  
	Given a matrix $A \in \mathbb{R}^{K,K}$ with strictly positive diagonal entries 
	and non-negative off-diagonal entries, 
	and a  strictly  positive vector $h \in \mathbb{R}^K$, 
	there exists a unique diagonal matrix $D = \diag(d_1, d_2, \ldots, d_K)$ such that 
	$D A D h    = 1_K$ and $d_k > 0$, $1 \leq k \leq K$. 
\end{lemma}  

%	\begin{lemma}[\cite{brualdi1974dad}]
%	Let $h$ denote a positive vector and suppose that $A \in \mb{R}^{m \times m}$ has positive diagonal elements. Then there exists a unique diagonal matrix $D$ with positive diagonal elements such that $DAD\mf{1}_m = h$.
%	\end{lemma}
%	
%	Specifically, Theorem 4 of \cite{brualdi1974dad} stat

We apply Lemma \ref{lem:existence}  with $h = (h_1, \ldots, h_K)'$ and $A = P$ to construct a diagonal matrix $D = \diag(d_1, \ldots, d_K)$ satisfying $DADh = 1_K$. Note that $P$ has positive diagonal entries since $\Omega$ does.

Define $P^* = DPD$ and  $D^* = \diag(d_1^*, \ldots, d_n^*) \in \mathbb{R}^n$ where
\[
d_i^* \equiv d_k \qquad \text{if } i \in \mc{C}_k
\]
Observe that
\[
\Pi D^{-1} = (D^*)^{-1} \Pi.
\]
%	and moreover $D^*$ is the unique matrix with this property. 
Define $\Theta^* = \Theta(D^*)^{-1}$, and let $\theta^* = \diag(\Theta^*)$.  Next, let $\overline{\Theta} = \frac{n}{\| \theta^*\|_1} \cdot \Theta^*$, let $\overline{\theta} = \diag(\overline{\Theta})$, and let $\overline{P} = \frac{\| \theta^*\|_1^2}{n^2} \cdot P^*$. Note that $\| \overline{\theta}\|_1 = n$ and $\overline{P}h \propto \mf{1}_K$. 

Using the previous definitions and observations, we have
\[
\Omega = \Theta \Pi D^{-1} D P D D^{-1} \Pi' \Theta = 
\Theta^* \Pi P^* \Pi' \Theta^*
= \overline{\Theta} \Pi \overline{ P} \Pi' \overline{ \Theta}
\]
which justifies existence.

%	\begin{align*}
%		\theta^*_i = 
%	\end{align*}
%Set $\Theta^* = \diag(\theta^*)$. 

To justify uniqueness, suppose that
\[
\Omega = 	\Theta\rp{1} \Pi P\rp{1} \Pi' \Theta\rp{1} = 
\Theta\rp{2} \Pi P\rp{2} \Pi' \Theta\rp{2},
\]
where $\theta\rp{i} = \diag(\Theta\rp{i})$ satisfy $\|  \theta\rp{i} \|_1 = n$ for $i = 1, 2$ and 
\[P\rp{1} h \propto \mf{1}_K , \qquad P\rp{2} h \propto \mf{1}_K. \]

%	Define a vector $\ti h \in \mathbb{R}^n$ by 
%	\[
%	\ti h_i = h_k \qquad \text{if } i \in \mathcal{C}_k.  
%	\]

Observe that
\[
\Pi P\rp{1} \Pi' \mf{1}_n = \alpha\rp{1} n \cdot \mf{1}_n, \qquad 
\Pi P\rp{2} \Pi' \mf{1}_n = \alpha\rp{2} n \cdot \mf{1}_n. 
\]
for positive constants $\alpha\rp{i}, i \in \{1,2\}$. Since $\Omega$ has nonnegative entries and positive diagonal elements, by Lemma \ref{lem:existence}, there exists a unique diagonal matrix $D$ such that 
\[
D \Omega D \mf{1}_n= \mf{1}_n.
\]
We see that taking $D =\frac{1}{\sqrt{\alpha\rp{i} n } } (\Theta\rp{i})^{-1}$ satisfies this equation for $i = 1, 2$, and therefore by uniqueness,
\[
\frac{1}{\sqrt{\alpha\rp{1} n }} (\Theta\rp{1})^{-1} = 
\frac{1}{\sqrt{\alpha\rp{2} n }} (\Theta\rp{2})^{-1}.  
\]
Since $\| \theta\rp{1} \|_1 = \| \theta \rp{2} \|_1 = n$, further we have $\alpha\rp{1} = \alpha\rp{2}$, and hence
\[
\Theta\rp{1} = \Theta \rp{2}. 
\]

It follows that
\[
\Pi P\rp{1} \Pi' = \Pi P\rp{2} \Pi',
\]
which, since we assume $h_i > 0$ for $i = 1, \ldots, K$, further implies that $P\rp{1} = P\rp{2}$.  \qed 

\section{Proof of Theorem \ref{thm:null-SgnQ}  (Limiting null of the SgnQ statistic)}

Consider a null DCBM with $\Omega = \theta^* (\theta^*)'$. Note that this is a different choice of parameterization than the one we study in the main paper. In \cite[Theorem 2.1]{JinKeLuo21} it is shown that the asymptotic distribution of $\psi_n$, the standardized version of SgnQ, is standard normal provided that
\begin{equation} \label{cond-theta}
	\|\theta^*\| \goto \infty, \;\;\;  \theta_{max}^* \goto 0,  \;\;\;  \mbox{and} \;\;\; (\|\theta^*\|^2/\|\theta^*\|_1) \sqrt{\log(\|\theta\|_1^*)} \goto 0. 
\end{equation} 

We verify that, in a DCBM with $\Omega = \alpha \theta \theta'$ and $\| \theta \|_1 = n$, these conditions are implied by the assumptions in \eqref{nullconditions}, restated below:
\begin{equation} 
	\label{eqn:our_assns_sup}
	n\alpha\to\infty, \qquad \mbox{and}\qquad \alpha\theta^2_{\max} \log(n^2\alpha)\to 0
\end{equation} 

In the parameterization of \cite{JinKeLuo21}, we have $\theta^* = \sqrt{\alpha} \theta$.  First, $\| \theta^* \|^2 \to \infty$ because by \eqref{eqn:our_assns_sup}, 
\[
\| \theta^* \|^2 \geq \frac{1}{n} \cdot \| \theta^* \|_1^2
= \alpha n \to \infty. 
\]

Next, $\theta_{\max}^* \to 0$ because by \eqref{eqn:our_assns_sup}, 
\[
\theta_{\max} = \sqrt{\alpha} \theta_{\max} \to 0. 
\]

To show the last part of \eqref{cond-theta}, note that
\begin{align*}
	(\|\theta^*\|^2/\|\theta^*\|_1) \sqrt{\log(\|\theta\|_1^*)}
	\leq \sqrt{\alpha} \theta_{\max} \sqrt{\log( \sqrt{\alpha} n )}
	= \frac{1}{\sqrt{2}} \sqrt{\alpha} \theta_{\max} \sqrt{\log( \alpha n^2 )}
	\to 0
\end{align*}
by \eqref{eqn:our_assns_sup}. Thus \eqref{cond-theta} holds, and $\psi_n$ is asymptotically standard normal under the null. \qed 

\section{ Proof of Lemma \ref{lemma:d} (Properties of $\ti \Omega$) }

%\textbf{Note:} There were typos in the statement of Lemma \ref{lemma:d}, which we fix below. It should read as follows.

\begin{lemma*} 
	The rank and trace of the matrix $\widetilde{\Omega}$ are $(K-1)$  and $\|\theta\|^2 \diag(\ti P)' g$, respectively.  When $K = 2$, $\tilde{\lambda}_1 = \mathrm{trace}(\widetilde{\Omega}) = \| \theta \|^2(ac - b^2)(d_0^2 g_1 + d_1^2 g_0)/ (a d_1^2 + 2b d_0 d_1 + c d_0^2)$.    
\end{lemma*} 

{\bf Proof of Lemma \ref{lemma:d}}.  
By basic algebra, 
\[
\widetilde{\Omega} = \Theta \Pi \widetilde{P} \Pi' \Theta, \qquad \mbox{where $\widetilde{P} = (P - (d' P d)^{-1} P d d' P)$}. 
\] 
It is seen $\widetilde{P} d = P d - (d' P d)^{-1} Pd d' P d = 0$,  so $\mathrm{rank}(\widetilde{P}) \leq K-1$.  
At the same time, since for any matrix $A$ and $B$ of the same size,  $\mathrm{rank}(A + B) \leq \mathrm{rank}(A) +  \mathrm{rank}(B)$, it follows $\widetilde{P} \geq (K-1)$, as $\mathrm{rank}(P) = K$ 
and $\mathrm{rank}(P dd'P) \leq 1$. This proves that $\mathrm{rank}(\widetilde{P}) =  K-1$. 

At the same time, since for any matrices $A$ and $B$, $\mathrm{trace}(AB) = \mathrm{trace}(BA)$, 
\[
\mathrm{trace}(\widetilde{\Omega}) =  \mathrm{trace}(\widetilde{P} \Pi' \Theta^2 \Pi)  = \|\theta\|^2  \mathrm{trace}(\widetilde{P} G) = \| \theta \|^2 \diag(\ti P)' g. 
\]
This proves the second item of the lemma. 

Last, when $K = 2$,  $\widetilde{\Omega}$ is rank $1$, and its eigenvalue is the same as its trace. First
\begin{align*}
	(\ti P)_{11} &= a - \frac{  (a d_1 + b d_0)^2   }{a d_1^2 + 2b d_0 d_1 + c d_0^2}
	=(ac - b^2) \frac{ d_0^2}{a d_1^2 + 2b d_0 d_1 + c d_0^2}
	\\  (\ti P)_{22} &= c - \frac{  (b d_1 + c d_0)^2   }{a d_1^2 + 2b d_0 d_1 + c d_0^2}
	=(ac - b^2) \frac{   d_1^2   }{a d_1^2 + 2b d_0 d_1 + c d_0^2}.
\end{align*}
Thus
\[
\ti \lambda_1 = \| \theta \|^2 \diag(\ti P)' g
=   \| \theta \|^2(ac - b^2) \cdot \frac{   d_0^2 g_1 + d_1^2 g_0    }{ a d_1^2 + 2b d_0 d_1 + c d_0^2 }
\]
This proves the last item and completes the proof of the lemma. 

\qed 

\section{ Proof of Theorem \ref{thm:alt-SgnQ} (Power of the SgnQ test) and Corollary \ref{cor:SgnQtest}  }
\label{sec:SgnQ}

\subsection{Setup and results }

\textit{Notation:} Given sequences of real numbers $A = A_{n}$ and $B = B_n$, we write $A \lesssim B$ to signify that $A = O(B)$, $A \asymp B$ to signify that $A \lesssim B$ and $B \lesssim A$, and $A \sim B$ to signify that $A/B = 1 + o(1)$. 

Throughout this section, we consider a DCBM with parameters $(\Theta, P)$ where $P \in \mb{R}^{2 \times 2}$ has unit diagonals, and we analyze the behavior of SgnQ under the alternative. At the end of this subsection we explain how Theorem \ref{thm:alt-SgnQ}  and Corollary \ref{cor:SgnQtest} follow from the results described next. Our results hinge on 
\[
\ti \lambda \equiv \ti \lambda_1 = \tr(\ti \Omega). 
\]

%	We also assume that $b = O(1)$ (equivalently $\| P \| = O(1)$). 
Given a subset $U \subset [n]$, let $\theta_U \in \R^{|U|}$ denote the restriction of $\theta$ to the coordinates of $U$. For notational convenience, we let $S = \{ i: \pi_i(1) = 1 \}$, which was previously written as $\mathcal{C}_1$ in the main paper.

In a DCBM where $P$ has unit diagonals, our main results hold under the following conditions.
\begin{align*}
	\Omega_{ij} &\les \theta_i \theta_j \num \label{eqn:assn1}
	\\ \| \theta \|_\infty &= O(1), \text{ and }  \num \label{eqn:assn2}
	\\ \| \theta \|_2^2 &\to \infty. 
	\num \label{eqn:assn3}
	\\ ( \| \theta \|_2^2/ \| \theta \|_1 ) \sqrt{ \log( \| \theta \|_1 )} &\to 0. 
	\num \label{eqn:assn4}
\end{align*}

First we justify that these assumptions are satisfied by an equivalent DCBM with the same $\Omega$ represented with the parameterization \eqref{identifiable} and satisfying \eqref{altconditions}. Thus all results proved in this section transfer immediately to the main paper.

\begin{lemma} \label{lem:compatibility}
	Consider a DCBM with parameters $(\Theta^*, P^*)$ satisfying \eqref{identifiable} and satisfying \eqref{altconditions}. Define $\Theta = \diag(\theta)$ where
	\[
	\theta_i = \begin{cases}
		\sqrt{a} \theta_i^* &\quad \text{if } i \in S \\
		\sqrt{c} \theta_i^* &\quad \text{if } i \in S^c, 
	\end{cases}
	\]
	and 
	\[
	P = \begin{pmatrix}
		1 & \frac{b}{\sqrt{ac}} \\
		\frac{b}{\sqrt{ac}} & 1 
	\end{pmatrix}.
	\]
	Then
	\[
	\Omega =	\Theta \Pi P \Pi \Theta = \Theta^* \Pi P^* \Pi' \Theta^*
	\]
	and  \eqref{eqn:assn1}--\eqref{eqn:assn4} are satisfied. 
\end{lemma}

\begin{proof}
	The statement regarding $\Omega$ follows by basic algebra. \eqref{eqn:assn1} follows if we can show that 
	\begin{align}
		\label{eqn:pf_assn1}
		\frac{b}{\sqrt{ac}} \les 1. 
	\end{align}
	Since
	\[
	b = \frac{cn - (a + c)N}{n - 2N}
	= c \cdot \frac{n - N}{n - 2N}
	- a \cdot \frac{N}{n - 2N},
	\]
	we have $a \geq c \gtrsim b$, so \eqref{eqn:pf_assn1} follows.
	
	Next, \eqref{eqn:assn2} follows directly from $a \theta_{\max,1}^2 \les 1$ since $c \theta_{\max, 0}^2 = o(1)$ by \eqref{altconditions}. 
	
	For \eqref{eqn:assn3}, 
	\[
	\| \theta \|_2^2
	\geq \frac{1}{n} \cdot \| \theta \|_1^2
	\geq c n \to \infty
	\]
	by \eqref{altconditions}. 
	
	For the last part, note that
	\[
	b = c \cdot \frac{n - N}{n - 2N}
	- a \cdot \frac{N}{n - 2N} \geq 0
	\Rightarrow a \eps \les c. 
	\]
	%	Thus,
	%	\begin{align*}
	%		\frac{ \| \theta \|_2^2 }{ \| \theta \|_1}
	%		&= \frac{ a \| \theta_S^* \|_2^2 + c \| \theta_{S^c}^* \|_2^2}{\sqrt{a} \| \theta_S^* \|_1 + \sqrt{c} \| \theta_{S^c}^* \|_1}
	%		\les \frac{   a (N/n) \| \theta^* \|_2^2 + c \| \theta_{S^c}^* \|_2^2    }{\sqrt{c} \| \theta^* \|_1}
	%		\\&\les \frac{ c \| \theta^* \|_2^2 }{ \sqrt{c} \| \theta^* \|_1 }
	%		\les \sqrt{c } \theta_{\max} = o\big( \frac{1}{\sqrt{ \log ( \| \theta \|_1 })} \big) 
	%		= o\big( \frac{1}{\sqrt{ \log c n^2 } } \big),
	%	\end{align*}
	%
	%=======================
	Thus,
	\begin{align*}
		\frac{ \| \theta \|_2^2 }{ \| \theta \|_1}
		&= \frac{ a \| \theta_S^* \|_2^2 + c \| \theta_{S^c}^* \|_2^2}{\sqrt{a} \| \theta_S^* \|_1 + \sqrt{c} \| \theta_{S^c}^* \|_1}
		\les \frac{   a (N/n) \| \theta^*_{S^c} \|_2^2 + c \| \theta_{S^c}^* \|_2^2    }{\sqrt{c} \| \theta^*_{S^c} \|_1}
		\\&\les \frac{ c \| \theta_{S^c}^* \|_2^2 }{ \sqrt{c} \| \theta_{S^c}^* \|_1 }
		\les \sqrt{c } \theta_{\max,0} = o\big( \frac{1}{\sqrt{ \log c n^2 } } \big) = o\big( \frac{1}{\sqrt{ \log ( \| \theta \|_1 })} \big) 
		,
	\end{align*}
	which implies \eqref{eqn:assn4}. Above we use that $a \geq c$ and $g_1 \asymp d_1 \asymp N/n$, by assumption. Precisely, in the first line, we used
	\begin{align*}
		a \| \theta^*_{S} \|_2^2 
		\asymp a \cdot (1 - N/n)^{-1} \frac{N}{n} \| \theta^*_{S^c} \|_2^2 
		\lesssim c \| \theta^*_{S^c} \|_2^2, 
	\end{align*}
	and  in the second line we used
	\begin{align*}
		\| \theta \|_1
		&\geq  \sqrt{c} \| \theta_{S^c}^* \|_1
		\asymp  \sqrt{c} (1 - N/n)^{-1} \| \theta^* \|_1 
		\asymp \sqrt{c} n. 
	\end{align*}
	%\begin{align*}
	%	\| \theta \|_1
	%	&= \sqrt{a} \| \theta_S^* \|_1 
	%	+ \sqrt{c} \| \theta_{S^c}^* \|_1
	%	\asymp \sqrt{a} \frac{N}{n} \| \theta_{S^c}^* \|_1 
	%	+ \sqrt{c} \| \theta_{S^c}^* \|_1
	%	\\&\lesssim \sqrt{ c n/ N} \cdot \frac{N}{n} \| \theta_{S^c}^* \|_1 
	%	+ \sqrt{c} \| \theta_{S^c}^* \|_1
	%\lesssim \sqrt{c} \| \theta^* \|_1 
	%	= \sqrt{c} n. 
	%\end{align*}
	
\end{proof}

With Lemma \ref{lem:compatibility} in hand, we restrict in the remainder of this section to the setting where $P$ has unit diagonals and \eqref{eqn:assn1}--\eqref{eqn:assn4} are satisfied.

%The first two assumptions are only needed to ensure that the signal under the alternative (\textit{i.e.}, the expectation of SgnQ) is sufficiently large.

Define $v_0 = \mf{1}' \Omega \mf{1}$, and let $\eta^* =( 1/\sqrt{v_0} ) \Omega \mf{1}$. Recall $\tilde \Omega = \Omega - \eta^* \eta^{*\T}$, and $\ti \lambda = \tr(\ti \Omega)$. Our main result concerning the alternative is the following. 

\begin{thm}[Limiting behavior of SgnQ test statistic]
	\label{thm:SgnQ_lim_behavior}
	Suppose that the previous assumptions hold and that $|\ti \lambda | / \sqrt{ \lambda_1 } \to \infty$. Then under the null hypothesis, as $n \to \infty$, $\E[Q ] \sim 2 \| \theta \|_2^4$, $\var(Q) \sim 8 \| \theta \|_2^8$, and $(Q - \E Q)/\sqrt{\var(Q)} \to N(0, 1)$ in law. Under the alternative hypothesis, as $n \to \infty$, $\E Q \sim \ti\lambda^4 $ and
	$\var(Q) \les  | \ti \lambda|^6 + | \ti \lambda|^2 \lambda_1^3 = o( \ti \lambda^8)$. 
\end{thm} 

Following \cite{JinKeLuo21}, we introduce some notation: 
\begin{align*}
	& \widetilde{\Omega}=\Omega-(\eta^*)(\eta^*)', \qquad \mbox{where}\quad \eta^* = \frac{1}{\sqrt{v_0}}\Omega {\bf 1}_n, \;\; v_0= {\bf 1}_n'\Omega {\bf 1}_n;\cr
	& \delta_{ij} = \eta_i(\eta_j-\teta_j)+\eta_j(\eta_i-\teta_i), \qquad \mbox{where}\quad \eta=\frac{1}{\sqrt{v}}(\mathbb{E}A){\bf 1}_n,\;\; \teta = \frac{1}{\sqrt{v}}A{\bf 1}_n,\;\; v= {\bf 1}_n'(\mathbb{E}A){\bf 1}_n;\cr
	&r_{ij} = (\eta_i^*\eta_j^*-\eta_i\eta_j) - (\eta_i-\teta_i)(\eta_j-\teta_j)+(1-\frac{v}{V})\teta_i\teta_j, \qquad \mbox{where}\;\; V = {\bf 1}_n'A{\bf 1}_n. 
\end{align*}

The \textit{ideal} and \textit{proxy} SgnQ  statistics, respectively, are defined as follows:
\begin{align}
	\label{eqn:idealSgnQ}	\widetilde{Q}_n &= \sum_{i,j,k,\ell (dist)} (\widetilde{\Omega}_{ij}  + W_{ij}) 
	(\widetilde{\Omega}_{jk}  + W_{jk}) (\widetilde{\Omega}_{k\ell}  + W_{k\ell}) (\widetilde{\Omega}_{\ell i} + W_{\ell i}) \\
	\label{eqn:proxySgnQ}	Q_n^*&= \sum_{i,j,k,\ell (dist)} (\widetilde{\Omega}_{ij}  + W_{ij} + \delta_{ij})(\widetilde{\Omega}_{jk}  + W_{jk} + \delta_{jk}) (\widetilde{\Omega}_{k\ell}  + W_{k\ell} + \delta_{k\ell}) (\widetilde{\Omega}_{\ell i} + W_{\ell i} + \delta_{\ell i}).
\end{align}
Moreover, we can express the original or \textit{real} SgnQ as
\begin{align*}
	Q_n = \sum_{i,j,k,\ell (dist)} \bigg[(\widetilde{\Omega}_{ij}  + W_{ij}+\delta_{ij}+r_{ij})  
	&(\widetilde{\Omega}_{jk}  + W_{jk}+\delta_{jk}+r_{jk})
	\nonumber \\ &(\widetilde{\Omega}_{k\ell}  + W_{k\ell }+\delta_{k\ell}+r_{k\ell}) (\widetilde{\Omega}_{\ell i} + W_{\ell i}+\delta_{\ell i}+r_{\ell i})\bigg]. 
\end{align*}

The next theorems handle the behavior of these statistics. Together the results imply Theorem \ref{thm:SgnQ_lim_behavior}. Again, the analysis of the null carries over directly from \cite{JinKeLuo21}, so we only need to study the alternative. The claims regarding the alternative follow from Lemmas \ref{lem:ideal_SgnQ}--\ref{lem:real_SgnQ_starQ} below.

\begin{thm}[Ideal SgnQ test statistic]
	\label{thm:ideal_SgnQ}
	Suppose that the previous assumptions hold and that $|\ti \lambda | / \sqrt{ \lambda_1 } \to \infty$. Then under the null hypothesis, as $n \to \infty$, $\E[ \ti Q ] =0$ and $\var(\ti Q) = 8 \| \theta \|_2^8 \cdot [1 + o(1)]$. Furthermore, under the alternative hypothesis, as $n \to \infty$, $\E[ \ti Q] \sim \ti \lambda^4$ and $\var(\ti Q) \les \lambda_1^4 + |\ti \lambda|^6 = o(\ti \lambda^8)$. 
\end{thm} 

\begin{thm}[Proxy SgnQ test statistic]
	\label{thm:proxy_SgnQ}
	Suppose that the previous assumptions hold and that $|\ti \lambda | / \sqrt{ \lambda_1 } \to \infty$. Then under the null hypothesis, as $n \to \infty$, $|\E[ \ti Q - Q^* ]| = o(\| \theta \|_2^4)$ and $\var(\ti Q - Q^*) = o(\| \theta \|_2^8)$. Furthermore, under the alternative hypothesis, as $n \to \infty$, 
	$|\E[ \ti Q - Q^*] | \les | \ti \lambda |^2 \lambda_1 = o( \ti \lambda^4) $ and 
	$\var(\ti Q - Q^*) \les | \ti \lambda |^2 \lambda_1^3 + |\ti \lambda|^6 = o(\ti \lambda^8) $. 
\end{thm} 

\begin{thm}[Real SgnQ test statistic]
	\label{thm:real_SgnQ}
	Suppose that the previous assumptions hold and that $|\ti \lambda | / \sqrt{ \lambda_1 } \to \infty$. Then under the null hypothesis, as $n \to \infty$, $|\E[ Q - \ti Q ]| = o(\| \theta \|_2^4)$ and $\var(Q -\ti Q) = o(\| \theta \|_2^8)$. Furthermore, under the alternative hypothesis, as $n \to \infty$, 
	$|\E[ Q - Q^*] | \les |\ti \lambda|^2 \lambda_1 = o( \ti \lambda^4)$ and 
	$\var(Q - Q^*) \les  |\ti \lambda|^2 \lambda_1^3 = o(\ti \lambda^8)$. 
\end{thm}

The previous work \cite{JinKeLuo21} establishes that under the assumptions above, if $\| \theta_S \|_1 / \| \theta \|_1 \asymp 1$, then SgnQ distinguishes the null and alternative provided that $|\lambda_2|/\sqrt{\lambda_1} \to \infty$. To compare with the results above, note that $\lambda_2 \asymp \ti \lambda$ if $ \| \theta_S \|_1 / \| \theta \|_1 \asymp 1$ (c.f. Lemma E.5 of \cite{JinKeLuo21}). Thus when $K = 2$, our main result extends the upper bound of \cite{JinKeLuo21} to the case when $\| \theta_S \|_1 / \| \theta \|_1 = o(1)$. We note that $| \ti \lambda | \gtrsim | \lambda_2 |$ in general (see Lemma \ref{lem:ti_lambda_lambda_12} and Corollary \ref{cor:ti_lambda_vs lam2}). 

The theorems above apply to the symmetric SBM. Recall that in this model, 
\begin{align*}
	\Omega_{ij} = \begin{cases}
		a &\quad \text{ if } i,j \in S  \\
		c   &\quad \text{ if } i,j \notin S  \\
		\ti b = \frac{n c - (a+c) N}{n - 2N}  &\quad \text{ otherwise. } 
	\end{cases}
\end{align*}
where $N = |S|$ and $a, b, c \in (0,1)$. To obtain this model from our DCBM, set 
\begin{align*}
	\num \label{eqn:SBM_P}
	P = \begin{pmatrix}
		1 & \ti  b/\sqrt{ac} \\
		\ti b/\sqrt{ac} & 1 \\
	\end{pmatrix},
\end{align*}
and 
\begin{align*}
	\num \label{eqn:SBM_theta}
	\theta = \sqrt{a} \mf{1}_S + \sqrt{c} \mf{1}_{S^c}. 
\end{align*}
The assumption \eqref{eqn:assn1} implies that $ \ti b \les \sqrt{ a c}$, which is automatically satisfied since we assume $a \geq c$. 
%Note that we may assume that $a \geq c$ without loss of generality by evaluating SgnQ on the original network and its complement.

In SBM, it holds that $\lambda_2 = \ti \lambda$ (see Lemma \ref{lem:ti_lambda_lambda_12}). Furthermore, explicit calculations in Section \ref{sec:SBM_calc} reveal that
\begin{align*}
	\num \label{eqn:SBM_eigen}
	\lambda_1 &\sim nc  , \text{ and }
	\\ \lambda_2 &= \ti \lambda \sim N(a-c). 
\end{align*}

In addition, with $P, a, \ti b, c$ as above, if we have 
\[
\theta_i
= \begin{cases}
	\rho_i \sqrt{a} &\qquad \text{if } i \in S 
	\\ \rho_i \sqrt{c} &\qquad \text{if } i \notin S
\end{cases}
\]
for $\rho > 0$ with $\rho_{\min} \gtrsim \rho_{\max}$ in the DCBM setting, a very similar calculation, which we omit, reveals that
\begin{align*}
	\num \label{eqn:DCBM_eigen}
	\lambda_1 &\asymp  nc  , \text{ and }
	\\ \ti \lambda &\asymp   N(a-c). 
\end{align*}

With the previous results of this subsection in hand (which are proved in the remaining subsections) we justify Theorem \ref{thm:alt-SgnQ} and Corollary \ref{cor:SgnQtest}. 

\begin{proof}[Proof of Theorem \ref{thm:alt-SgnQ}]
	The SgnQ test has level $\kappa$ by Theorem \ref{thm:null-SgnQ}, so it remains to study the type II error. Using Theorem \ref{thm:SgnQ_lim_behavior} and Lemma \ref{lem:compatibility}, the fact that the type II error tends to $0$ directly follows from Chebyshev's inequality and the fact that $ \| \hat \eta \|_2^2 - 1 \approx \| \theta \|_2^2$ with high probability. In particular, note that since $|\ti \lambda| \gg \sqrt{\lambda_1}$, the expectation of SgnQ under the alternative is much larger than its standard deviation, under the null or alternative. We omit the details as they are very similar to the proof of Theorem 2.6 in \cite[Supplement,pgs. 5--6]{JinKeLuo21}. 
\end{proof}

\begin{proof}[Proof of Corollary \ref{cor:SgnQtest}]
	This result follows immediately from \eqref{eqn:DCBM_eigen} and Theorem \ref{thm:alt-SgnQ}. 
\end{proof}

%	We also have the following characterization of the mean and variance of SgnQ under the alternative. 
%	
%%	We obtain from Theorem \ref{thm:SgnQ_lim_behavior} the following corollary regarding SBM.
%	
%	\begin{cor}[Limiting behavior of SgnQ for SBM]
%		\label{cor:SgnQ_lim_behavior_SBM}
%		Suppose that the previous assumptions hold in the DCBM model \eqref{eqn:SBM_P}, \eqref{eqn:SBM_theta} and that $| \lambda_2 | / \sqrt{ \lambda_1 }\sim  N|a - c|/\sqrt{n c} \to \infty$. Then under the null hypothesis, as $n \to \infty$, $\E[Q ] \sim 2 \| \theta \|_2^4$, $\var(Q) \sim 8 \| \theta \|_2^8$, and $(Q - \E Q)/\sqrt{\var(Q)} \to N(0, 1)$ in law. Under the alternative hypothesis, as $n \to \infty$, $\E Q \sim \lambda_2^4  \sim N^4(a -c)^4  $ and
%		$\var(Q) \les N^6(a - c)^6 + N^2(a - c)^2 \cdot n^3 c^3
%		= o(N^8 (a - c)^8)	= o( \lambda_2^8)$. 
%	\end{cor}

\subsection{Preliminary bounds}

Define $v_0 = \mf{1}^\T \Omega \mf{1}$, and let $\eta^* = 1/\sqrt{v_0} \cdot \Omega \mf{1}$. For the analysis of SgnQ, it is important is to understand $\tilde \Omega = \Omega - \eta^* \eta^{*\T}$. The next lemma establishes that $\ti \Omega$ is rank one and has a simple expression when $K = 2$. 

\begin{lemma}
	\label{lem:tilde_Omega}
	Let $f = ( \|\theta_{S^c}\|_1  , - \|\theta_S\|_1 )^\T$ It holds that
	\begin{align*}
		\tilde \Omega = \frac{(1 - b^2)}{ v_0 } \cdot \Theta \Pi f f^\T \Pi^\T \Theta. 
	\end{align*}
\end{lemma}

\begin{proof}
	Let $\rho_0 = \|\theta_S\|_1$ and $\rho_1 = \|\theta_{S^c}\|_1$. Note that
	\begin{align*}
		(\Omega \mf{1})_i = \theta_i \sum_j \theta_j \pi_i^\T P \pi_j 
		= \begin{cases}
			\theta_i (\rho_0 + b \rho_1) &\quad \text{ if } i \in S \\
			\theta_i (b \rho_0 +  \rho_1) &\quad \text{ if } i \notin S. 
		\end{cases}  
	\end{align*}
	Hence
	\begin{align*}
		v_0 = \mf{1}^\T \Omega \mf{1} =  \rho_0^2 + 2 b \rho_0 \rho_1 +  \rho_1^2. 
	\end{align*}
	If $i,j \in S$, then 
	\begin{align*}
		\ti \Omega_{ij} = \theta_i \theta_j \big( 1 - \frac{( \rho_0 + b \rho_1)^2}{v_0}\big)
		= \theta_i \theta_j \cdot \frac{ (1 - b^2) \rho_1^2}{v_0} 
		%			=: \theta_i \theta_j \alpha_{00}
	\end{align*}
	Similarly if $i \in S$ and $j \notin S$,
	\begin{align*}
		\ti \Omega_{ij} = \theta_i \theta_j \big( b - \frac{( \rho_0 + b \rho_1)(b \rho_0 +  \rho_1)}{v_0}\big)
		= - \theta_i \theta_j \cdot \frac{ (1 - b^2) \rho_0 \rho_1}{v_0} 
		%			=: \theta_i \theta_j \alpha_{01}
	\end{align*}
	and
	\begin{align*}
		\ti \Omega_{ij} = \theta_i \theta_j \big( 1 - \frac{(b \rho_0 +  \rho_1)^2}{v_0}\big)
		=  \theta_i \theta_j \cdot \frac{ (1 - b^2) \rho_0^2}{v_0} 
		%=: \theta_i \theta_j \alpha_{11}
	\end{align*}
	if $i, j \in S^c$. The claim follows.
\end{proof}

Let 
\begin{align*}
	w = \Theta \Pi f = \theta_S \| \theta_{S^c}\|_1 - \theta_{S^c}  \| \theta_{S}\|_1 =  
	\rho_1 \theta_S - \rho_0 \theta_{S^c}
	%		\| \theta_S\|_1^2 - \| \theta_{S^c}\|_1^2.
\end{align*}
Using the previous lemma, we have the rank one eigendecomposition
\begin{align*}
	\ti \Omega = \ti \lambda \ti \xi \ti \xi^\T,
	\num \label{eqn:tilde_omega_eigendecomp}
\end{align*}
where we define 
\begin{align*}
	\ti \xi &= \frac{\rho_1 \theta_S - \rho_0 \theta_{S^c}}{\| \rho_1\theta_S - \rho_0 \theta_{S^c} \|_2} = \frac{\rho_1 \theta_S - \rho_0 \theta_{S^c}}{ \sqrt{ \rho_1^2 \| \theta_S \|_2^2 + \rho_0^2 \| \theta_{S^c} \|_2^2 } }, \text{ and } 
	\num \label{eqn:tilde_xi_def}
	\\ \ti \lambda &= \frac{(1 - b^2)}{v_0} \cdot \big( \rho_1^2 \| \theta_S \|_2^2 + \rho_0^2 \| \theta_{S^c} \|_2^2 \big). 
	\num \label{eqn:tilde_lambda_def}
\end{align*}

Lemma E.5 of \cite{JinKeLuo21} implies that if $ \| \theta_S \|_1 / \| \theta \|_1 \asymp 1$, then $\lambda_2 \asymp \ti \lambda_1$ . If $ \| \theta_S \|_1 / \| \theta \|_1 = o(1)$, then this guarantee may not hold. Below, in the case $K = 2$, we express $\ti \lambda$ in terms of the eigenvalues and eigenvectors of $\Omega$. This allows us to compare $\lambda_2$ with $\ti \lambda$ more generally, as in Corollary \ref{cor:ti_lambda_vs lam2}.

\begin{lemma}
	\label{lem:ti_lambda_lambda_12}
	Let $\Omega$ have eigenvalues $\lambda_1, \lambda_2$ and eigenvectors $\xi_1, \xi_2$. Let $\ti \lambda$ denote the eigenvalue of $\ti \Omega$. Then
	\begin{align*}
		\num \label{eqn:ti_lambda_lambda_12}
		\ti \lambda = \frac{\lambda_1 \lambda_2 \big( \langle \xi_1, \mf{1} \rangle^2 + \langle \xi_2, \mf{1} \rangle^2 \big)}{ \lambda_1 \langle \xi_1, \mf{1} \rangle^2 + \lambda_2 \langle \xi_2, \mf{1} \rangle^2  } . 
	\end{align*}
\end{lemma}

\begin{proof}
	By explicit computation,
	\begin{align*}
		\ti \Omega &= \Omega - \eta^* \eta^{*\T} 
		\\ &= \lambda_1 \big( 1- \frac{\lambda_1 \langle \xi_1, \mf{1} \rangle^2}{v_0} \big) \xi_1 \xi_1^\T 
		+ \lambda_2\big(1 - \frac{\lambda_2 \langle \xi_2, \mf{1} \rangle^2}{v_0} \big) \xi_2 \xi_2^\T - \frac{ \lambda_1 \lambda_2 \langle \xi_1, \mf{1} \rangle  \langle \xi_2, \mf{1} \rangle   }{v_0} \big( \xi_1 \xi_2^\T \xi_2 + \xi_1^\T \big)
		\\&= \frac{ \lambda_1 \lambda_2 }{v_0} \, \big(  \langle \xi_2, \mf{1} \rangle 
		\xi_1 + \langle \xi_1, \mf{1} \rangle 
		\xi_2  \big) \cdot  \big(  \langle \xi_2, \mf{1} \rangle 
		\xi_1 + \langle \xi_1, \mf{1} \rangle 
		\xi_2  \big)^\T.
	\end{align*}
	From \eqref{eqn:tilde_xi_def} and \eqref{eqn:tilde_lambda_def}, it follows that
	\begin{align*}
		\ti \xi &= \frac{ \langle \xi_2, \mf{1} \rangle 
			\xi_1 + \langle \xi_1, \mf{1} \rangle 
			\xi_2}{ \sqrt{ \langle \xi_1, \mf{1} \rangle^2 + \langle \xi_2, \mf{1} \rangle^2 }}
		\\ \ti \lambda &= \frac{\lambda_1 \lambda_2}{v_0} \big( \langle \xi_1, \mf{1} \rangle^2 + \langle \xi_2, \mf{1} \rangle^2 \big). 
	\end{align*}
\end{proof}

\begin{cor}
	\label{cor:ti_lambda_vs lam2}
	It holds that
	\begin{align*}
		\num \label{eqn:ti_lambda_vs lam2_abs}
		|\lambda_2| \les | \ti \lambda | \les \lambda_1.
	\end{align*}
	If $\lambda_2 \geq 0$, then
	\begin{align*}
		\num \label{eqn:ti_lambda_vs lam2_pos}
		\lambda_2 \leq \ti \lambda \leq \lambda_1 
	\end{align*}
\end{cor} 

\begin{proof}
	Suppose that $\lambda_2 \geq 0$. Then 
	\begin{align*}
		\lambda_2 \big( \langle \xi_1, \mf{1} \rangle^2 + \langle \xi_2, \mf{1} \rangle^2 \big) 
		\leq  \lambda_1 \langle \xi_1, \mf{1} \rangle^2  +  \lambda_2 \langle \xi_2, \mf{1} \rangle^2 = v_0 
		\leq  \lambda_1 \big( \langle \xi_1, \mf{1} \rangle^2 + \langle \xi_2, \mf{1} \rangle^2 \big),
	\end{align*}
	implies \eqref{eqn:ti_lambda_vs lam2_pos}. 
	
	Suppose that $\lambda_2 < 0$. Note that 
	\begin{align*}
		\lambda_1 \big( \langle \xi_1, \mf{1} \rangle^2 + \langle \xi_2, \mf{1} \rangle^2 \big) \geq  \lambda_1 \langle \xi_1, \mf{1} \rangle^2 + \lambda_2 \langle \xi_2, \mf{1} \rangle^2 = v_0 \geq 0,
	\end{align*}
	which combined with \eqref{eqn:ti_lambda_lambda_12} implies that $| \ti \lambda | \geq |\lambda_2|$. 
	
	Next,
	\begin{align*}
		\lambda_2 &\leq \ti \xi^\T \Omega \ti \xi 
		= \ti \lambda + \langle \ti \xi, \eta^* \rangle^2,
	\end{align*}
	which implies that
	\begin{align*}
		|\ti \lambda | \leq | \lambda_2 | + \langle \ti \xi, \eta^* \rangle^2
		\leq \lambda_1 + \| \eta^* \|_2^2 
		\les \lambda_1 + \| \theta \|_1^2 \les \lambda_1,
	\end{align*}
	where the last inequality follows from Lemma \ref{lem:lambda1_bd}. 
\end{proof}
%Our goal in this section is to establish that the detection boundary is determined by $\ti \lam/ \sqrt{ \lambda_1 }$.

The next results are frequently used in our analyis of SgnQ.

\begin{lemma}
	\label{lem:v0_v1_bd}
	Let $v = \mf{1}^\T (\Omega - \diag(\Omega)) \mf{1}$ and $v_0 = \mf{1}^\T \Omega \mf{1}$.  Then
	\begin{align*}
		v_0 \sim v \sim \| \theta \|_1^2. 
		\num \label{eqn:v0_v1_bd}
	\end{align*}
\end{lemma}

\begin{proof}
	By \eqref{eqn:assn4}, $\| \theta \|_2^2 = o( \| \theta \|_1)$. By \eqref{eqn:assn3}, $\|\theta \|_1 \to \infty$. Hence
	\begin{align*}
		v = \mf{1}^{\T} (\Omega - \diag( \Omega) ) \mf{1}
		= \| \theta \|_1^2 - \| \theta \|_2^2
		\sim \| \theta \|_1^2 \sim v_0 = \mf{1}^{\T} \Omega \mf{1}. 
	\end{align*}
	
\end{proof}

The next result is a direct corollary of Lemmas \ref{lem:tilde_Omega} and \ref{lem:v0_v1_bd}. 

\begin{cor}
	Define $\beta \in \mb{R}^n$ by
	\begin{align}
		\label{eqn:beta_def}
		\beta = \sqrt{\frac{|1 - b^2|}{v_0}} \cdot  \big( \| \theta_{S^c} \|_1 \mf{1}_S
		+ \| \theta_{S} \|_1 \mf{1}_{S^c}  \big)
	\end{align}
	%		
	%		where $\ti \varepsilon \equiv \| \theta_S \|_1/\|\theta_{S^c}\|_1$. 
	Then
	\begin{align*}
		|\ti \Omega_{ij}| \les \beta_i \theta_i \beta_j \theta_j  
		\num \label{eqn:tiOm_bd}. 
	\end{align*}
\end{cor}

%	\begin{cor}
%		Define $\beta = \mf{1}_S + \ti \varepsilon \mf{1}_{S^c}$, where $\ti \varepsilon \equiv \| \theta_S \|_1/\|\theta_{S^c}\|_1$. Then
%		\begin{align*}
%			|\ti \Omega_{ij}| \les \beta_i \theta_i \beta_j \theta_j  
%			\num \label{eqn:tiOm_bd}. 
%		\end{align*}
%	\end{cor}

\begin{lemma}
	\label{lem:lambda1_bd}
	Let $\lambda_1$ denote the largest eigenvalue of $\Omega$. Then
	\begin{align*}
		\num \label{eqn:lambda1_bd}
		\lambda_1 \gtrsim \|  \theta \|_2^2.
	\end{align*}
\end{lemma}

\begin{proof}
	Using the universal inequality $a^2 + b^2 \geq \frac{1}{2}( a + b)^2$, we have
	\begin{align*}
		\lambda_1 &\geq 	\frac{\theta^{\T}  \Omega \theta}{\| \theta \|_2^2} 	
		\geq \frac{1}{\| \theta \|_2^2} \cdot  \sum_{i,j} \theta_i \theta_j \Omega_{ij}
		\geq \frac{1}{\| \theta \|_2^2} \cdot \big( \sum_{i,j \in S} \theta_i^2 \theta_j^2
		+ \sum_{i,j \notin S} \theta_i^2 \theta_j^2 \big)
		\\ &\geq \frac{ \| \theta_S \|_2^4 + \| \theta_{S^c} \|_2^4 }{ \| \theta \|_2^2 } 
		\gtrsim \| \theta \|_2^2. 
	\end{align*}
\end{proof}

\begin{lemma}
	\label{lem:eta_bd}
	Define $\eta = \frac{1}{\sqrt{v}} (\Omega - \diag(\Omega))\mf{1}$. Then
	\begin{align*}
		\num \label{eqn:etai_bd}
		\eta_i 	 \les		\eta_i^* \les \theta_i 
	\end{align*}
\end{lemma}

\begin{proof}
	The left-hand side is immediate, so we prove that $\eta_i^* \les \theta_i$. We have
	\begin{align*}
		(\Omega \mf{1})_i = \begin{cases}
			\theta_i (  \| \theta_S \|_1 + b \| \theta_{S^c} \|_1)
			\quad \text{if } i \in S \\
			\theta_i ( b \| \theta_S \|_1 +  \| \theta_{S^c} \|_1)
			\quad \text{if } i \notin S
		\end{cases} 
	\end{align*}
	Since $\Omega_{ii} = \theta_i^2$, 
	\begin{align*}
		\sqrt{v_0} \cdot \eta_i = \begin{cases}
			\theta_i (  \| \theta_S \|_1 + b \| \theta_{S^c} \|_1 )
			- \theta_i^2
			\quad \text{if } i \in S \\
			\theta_i ( b \| \theta_S \|_1 +  \| \theta_{S^c} \|_1 ) -\theta_i^2
			\quad \text{if } i \notin S. 
		\end{cases}
	\end{align*}
	Since $b = O(1)$, $\theta_i = O(1)$, and $v_0 \gtrsim \| \theta \|_1^2$ (c.f. Lemma \ref{lem:v0_v1_bd}),
	\begin{align*}
		\eta_i^* \les \frac{ \theta_i \| \theta \|_1 }{ \sqrt{ \| \theta \|_1^2 } } = \theta_i,
	\end{align*}
	as desired. 
\end{proof}

We use the bounds \eqref{eqn:v0_v1_bd} -- \eqref{eqn:etai_bd} throughout. We also use repeatedly that
\begin{align*}
	\| \theta \|_p^p &\les \| \theta \|_q^q, \text{ if } p \geq q,
	\num \label{eqn:theta_bd_ellp}
\end{align*}
which holds by \eqref{eqn:assn2}, and
\begin{align*}
	\| \beta \circ \theta \|_2^2 &= |\ti \lambda| \\
	|\beta_i | &\les 1  \\
	\| \beta \circ \theta^{\circ 2 } \|_1 &\leq \| \beta \circ \theta \|_2 \| \theta \|_2 \les \| \theta \|_2^2
	\num \label{eqn:beta_bds},  
\end{align*}
where the second line holds by Cauchy--Schwarz.

\subsection{Mean and variance of SgnQ}

The previous work \cite{JinKeLuo21} decomposes $\ti Q$ and $\ti Q - Q^*$ into a finite number of terms. For each term an exact expression for its mean and variance is derived in \cite{JinKeLuo21} that depends on $\theta$, $\eta$, $v$, and $\ti \Omega$. These expression are then bounded using the inequalities \eqref{eqn:assn2}, \eqref{eqn:assn3}, \eqref{eqn:v0_v1_bd}, \eqref{eqn:lambda1_bd}--\eqref{eqn:theta_bd_ellp}, as well as an inequality of the form
\begin{align*}
	|\ti \Omega_{ij}| \les \alpha \theta_i \theta_j.
\end{align*}

In our case, an inequality of this form is still valid, but it does not attain sharp results because it does not properly capture the signal $|\ti \lambda|$ from the smaller community. Instead, we use the inequality \eqref{eqn:tiOm_bd}, followed by the bounds in \eqref{eqn:beta_bds} to handle terms involving $\ti \Omega$.

Therefore, for terms of $\ti Q$ and $\ti Q - Q^*$ that do not depend on $\ti \Omega$, the bounds in \cite{JinKeLuo21} carry over immediately. In particular, their analysis of the null hypothesis carries over directly. Hence we can focus solely on the alternative hypothesis. 

Furthermore, any terms with zero mean in \cite{JinKeLuo21} also have zero mean in our setting : for every term that is mean zero, it is simply the sum of mean zero subterms, and each mean zero subterm is a product of independent, centered random variables (eg, $X_1$ below). 

\subsubsection{Ideal SgnQ}

The previous work \cite{JinKeLuo21} shows that $\ti Q = X_1 + 4 X_2 + 4X_3 + 2 X_4 + 4 X_5 + X_6$, where $X_1, \ldots, X_6$ are defined in their Section G.1. For convenience, we state explicitly the definitions of these terms.
\begin{align*}
	& X_1 = \sum_{i, j, k, \ell (dist)} W_{ij} W_{jk} W_{k \ell} W_{\ell i}, \qquad X_2 =  \sum_{i, j, k, \ell (dist)} \widetilde{\Omega}_{ij} W_{jk} W_{k \ell} W_{\ell i},\cr
	& X_3 = \sum_{i, j, k, \ell (dist)} \widetilde{\Omega}_{ij}  \widetilde{\Omega}_{jk} W_{k \ell} W_{\ell i}, \qquad\; X_4 = \sum_{i, j, k, \ell (dist)} \widetilde{\Omega}_{ij} W_{jk} \widetilde{\Omega}_{k \ell} W_{\ell i},\cr
	& X_5 = \sum_{i, j, k, \ell (dist)} \widetilde{\Omega}_{ij} \widetilde{\Omega}_{jk} \widetilde{\Omega}_{k \ell} W_{\ell i}, \qquad\;\, X_6 = \sum_{i, j, k, \ell (dist)} \widetilde{\Omega}_{ij} \widetilde{\Omega}_{jk} \widetilde{\Omega}_{k \ell} \widetilde{\Omega}_{\ell i}.  
\end{align*}

Since $X_1$ does not depend on $\ti \Omega$, the bounds for $X_1$ below are directly quoted from Lemma G.3 of \cite{JinKeLuo21}. Also note that $X_6$ is a non-stochastic term. 

\begin{lemma}
	\label{lem:ideal_SgnQ}
	Under the alternative hypothesis, we have
	\begin{alignat*}{1}
		\E[X_k] &= 0 \text{ for } 1 \leq k \leq 5, 
		\\ \var(X_1) &\les \| \theta \|_2^8 \les \lambda_1^4
		\\ \var(X_2) &\les \| \gam \|_2^4 \, \| \theta \|_2^4 \les | \ti \lambda |^2 \lambda_1^2 
		\\ \var(X_3) &\les \| \gam \|_2^8 \, \| \theta \|_2^2 
		\les | \ti \lambda |^4 \lambda_1
		\\ \var(X_4) &\les \| \gam \|_2^8 \leq |\ti \lambda |^4
		\\ \var(X_5) & \les \| \gam \|_2^{12}
		\les | \ti \lambda |^6,
		\text{ and } 
		\\ \E[ X_6]
		&= X_6 \sim  | \ti \lambda^4 | 
	\end{alignat*}
\end{lemma}

Since we assume $| \ti \lambda |/\sqrt{\lambda_1} \to \infty$ under the alternative hypothesis, it holds that
\begin{align*}
	\var( \ti Q) &\les \lambda_1^4 + | \ti \lambda|^6.
\end{align*}
Theorem \ref{thm:ideal_SgnQ} follows directly from this bound and that $\E X_6 = \E \ti Q \sim \ti \lambda^4$.

\subsubsection{Proxy SgnQ}

The previous work \cite{JinKeLuo21} shows that 
\begin{align*}
	\ti Q - Q^* = U_a + U_b + U_c,
\end{align*}
where
\begin{align*}
	U_a &= 4 Y_1 + 8Y_2 + 4Y_3 + 8 Y_4 + 4 Y_5 +4 Y_6
	\\ U_b &= 4Z_1 + 2 Z_2 + 8 Z_3 + 4 Z_4 + 4 Z_5 + 2 Z_6
	\\ U_c &= 4T_1 + 4 T_2 + F.
\end{align*}

These terms are defined in Section G.2 of  \cite{JinKeLuo21}, and for convenience, we define them explicitly below. The previous equations are obtained by expanding carefully $\tilde Q$ and $Q^*$ as defined in \eqref{eqn:idealSgnQ} and \eqref{eqn:proxySgnQ}. Thus, the terms on the right-hand-side above are referred as \textit{post-expansion} terms, and we can analyze each one individually. Now we proceed to their definitions. 

First $Y_1, \ldots, Y_6$ are defined as follows.

\begin{align*}
	& Y_1 = \sum_{i, j, k, \ell (dist)} \delta_{ij} W_{jk} W_{k \ell} W_{\ell i}, \qquad\;\;Y_2 =  \sum_{i, j, k, \ell (dist)} \delta_{ij} \widetilde{\Omega}_{jk}  W_{k \ell}  W_{\ell i},\cr
	& Y_3 = \sum_{i, j, k, \ell (dist)} \delta_{ij} W_{jk}   \widetilde{\Omega}_{k \ell}  W_{\ell i}, \qquad\;\;\; Y_4 = \sum_{i, j, k, \ell (dist)} \delta_{ij}  \widetilde{\Omega}_{jk} \widetilde{\Omega}_{k \ell}W_{\ell i},\cr
	& Y_5 = \sum_{i, j, k, \ell (dist)} \delta_{ij}  \widetilde{\Omega}_{jk} W_{k \ell}  \widetilde{\Omega}_{\ell i}, \qquad\quad Y_6 = \sum_{i, j, k, \ell (dist)} \delta_{ij}  \widetilde{\Omega}_{jk} \widetilde{\Omega}_{k \ell}  \widetilde{\Omega}_{\ell i}.  
\end{align*}

Next, $Z_1, \ldots, Z_6$ are defined as follows.
\begin{align*}
	& Z_1 = \sum_{i, j, k, \ell (dist)} \delta_{ij} \delta_{jk} W_{k \ell} W_{\ell i}, \qquad\;\;Z_2 =  \sum_{i, j, k, \ell (dist)} \delta_{ij} W_{jk}  \delta_{k \ell}  W_{\ell i},\cr
	& Z_3 = \sum_{i, j, k, \ell (dist)} \delta_{ij} \delta_{jk}   \widetilde{\Omega}_{k \ell}  W_{\ell i}, \qquad\;\;\; Z_4 = \sum_{i, j, k, \ell (dist)} \delta_{ij}  \widetilde{\Omega}_{jk} \delta_{k \ell}W_{\ell i},\cr
	& Z_5 = \sum_{i, j, k, \ell (dist)} \delta_{ij} \delta_{jk} \widetilde{\Omega}_{k \ell}  \widetilde{\Omega}_{\ell i}, \qquad\quad Z_6 = \sum_{i, j, k, \ell (dist)} \delta_{ij}  \widetilde{\Omega}_{jk} \delta_{k \ell}  \widetilde{\Omega}_{\ell i}.  
\end{align*}

Last, we have the definitions of $T_1, T_2$, and $F$. 

\begin{align*}
	&T_1= \sum_{i,j,k,\ell (dist)} \delta_{ij}\delta_{jk}\delta_{k\ell}W_{\ell i}, \qquad\;\; T_2= \sum_{i,j,k,\ell (dist)} \delta_{ij}\delta_{jk}\delta_{k\ell}\widetilde{\Omega}_{\ell i},\cr
	& F = \sum_{i,j,k,\ell (dist)} \delta_{ij}\delta_{jk}\delta_{k\ell}\delta_{\ell i}. 
\end{align*}

The following post-expansion terms below appear in Lemma G.5 of \cite{JinKeLuo21}. The term $Y_1$ does not depend on $\ti \Omega$, so we may directly quote the result. 

\begin{lemma}
	\label{lem:Ua}
	Under the alternative hypothesis, it holds that 
	\begin{alignat*}{2}
		%	&\text{Lemma G.4: }	\E U &= 0, \quad \var(U) \les \| \theta \|_2^2 \| \theta \|\theta_3^6 \|
		|\E Y_1| &= 0, \qquad &&\var(Y_1) \les \| \theta \|_2^2 \| \,  \|\theta\|_3^6 \les \lambda_1^4
		\\ |\E Y_2| &= 0, \quad &&\var(Y_2) \les 
		%	\frac{ \|\beta \circ \theta\|_2^2 \| \theta \|_2^8}{\| \theta \|_1} + 
		\| \beta \circ \theta \|_2^2 \, \| \theta \|_2^6 \les |\ti \lambda| \lambda_1^3 
		\\ |\E Y_3| &= 0, \quad &&\var(Y_3) \les \| \beta \circ \theta \|_2^4 \, \|\theta\|_2^4 \les |\ti \lambda|^2 \lambda_1^2
		\\ |\E Y_4| &\les   \| \beta \circ \theta \|_2^4 \| \theta \|_2^2 \les | \ti \lambda |^2 \lambda_1 , \quad &&\var(Y_4) \les \frac{\| \beta \circ \theta \|_2^6 \, \| \theta \|_2^6}{\| \theta \|_1} \les |\ti \lambda|^3 \lambda_1^2
		\\ |\E Y_5| &= 0 , \quad &&\var(Y_5) \les \frac{\| \beta \circ \theta \|_2^6 \, \| \theta \|_2^4 }{ \| \theta \|_1 } \les |\ti \lambda|^3 \lambda_1 
		%			\\ |\E Y_6| & = 0, \quad &&\var(Y_6) \les \frac{\| \beta \circ \theta \|_2^6 \, \| \theta \|_2^4   }{\| \theta \|_1} \les |\ti \lambda|^3 \lambda_1  .
		\\ |\E Y_6| & = 0, \quad &&\var(Y_6) \les \frac{ \| \gam \|_2^{12} \| \theta \|_2^2  }{\| \theta \|_1} \les | \ti \lambda |^6 .
	\end{alignat*}
\end{lemma}

As a result,
\begin{align*}
	\num \label{eqn:Ua_expectation}
	|\E U_a | \les | \ti \lambda |^2 \lambda_1= o( \ti \lambda^4).
\end{align*}
Also using Corollary \ref{cor:ti_lambda_vs lam2} and that $| \ti \lambda|/\sqrt{\lambda_1} \to \infty$, we have
\begin{align*}
	\num \label{eqn:Ua_variance}
	\var(U_a) \les \lambda_1^4 + |\ti \lambda|^3 \lambda_1^2 + |\ti \lambda |^6. 
\end{align*}

The terms below appear in Lemma G.7 of \cite{JinKeLuo21}. The bounds on $Z_1$ and $Z_2$ are quoted directly from \cite{JinKeLuo21}.

%	The terms below show up in Lemma G.7. $Z_1$ and $Z_2$ can be directly quoted from \cite{JinKeLuo21}.
%	\begin{align*}
%		\E Z_1 &\les \| \theta \|_2^4, \quad \var(Z_1) \les \|\theta\|_2^2 \| \theta \|_3^6 
%	\\	\E Z_2 &\les \| \theta \|_2^4 , \quad \var(Z_2) \les \frac{ \| \theta \|_2^6 \| \theta \|_3^3 }{ \| \theta \|_1} 
%	\\ \E Z_3 &\les 	
%	\end{align*}
\begin{lemma}
	\label{lem:Ub}
	Under the alternative hypothesis, it holds that 
	\begin{alignat*}{2}
		| \E Z_1 | &\les \| \theta \|_2^4 \les \lambda_1^2, \qquad &&\var(Z_1) \les \| \theta\|_2^2 \,\| \theta \|_3^6 \les \lambda_1^4
		\\ | \E Z_2 | &\les \| \theta \|_2^4 \les \lambda_1^2, \quad &&\var(Z_2) \les \frac{ \| \theta \|_2^6 \, \| \theta \|_3^3 }{ \| \theta \|_1 } \les \lambda_1^3
		\\ | \E Z_3| &= 0 , \quad &&\var(Z_3) \les \| \gam \|_2^4 \, \| \theta \|_2^6
		\les | \ti \lambda|^2 \lambda_1^3  
		\\ | \E Z_4 | &\les \| \gam \|_2^2 \, \| \theta \|_2^2 \les |\ti \lambda| \lambda_1, 
		\quad &&\var(Z_4) \les \frac{ \| \gam \|_2^4 \, \| \theta \|_2^6}{ \| \theta \|_1 } 		\les | \ti \lambda|^2 \lambda_1^2  
		\\ | \E Z_5 | &\les \| \gam \|_2^4 \, \| \theta \|_2^2 \les |\ti \lambda|^2 \lambda_1 ,
		\quad &&\var(Z_5) \les \frac{ \| \gam \|_2^8 \, \| \theta \|_2^6}{ \| \theta \|_1^2} \les | \ti \lambda |^4 \lambda_1 
		\\ | \E Z_6 | &\les \frac{ \| \gam \|_2^4 \, \| \theta \|_2^4}{ \| \theta \|_1^2} \les | \ti \lambda|^2  , 
		\quad &&\var(Z_6) \les \frac{ \| \gam \|_2^8 \, \| \theta \|_2^4}{ \| \theta \|_1^2 } \les | \ti \lambda|^4 . 
	\end{alignat*}
\end{lemma}

%	Note that $\lambda_2 < \ti \lambda$ and $\tr(\Omega) = \lambda_1 + \lambda_2 \geq 0$. It follows that $|\ti \lambda| \leq \lambda_1$. \pax{ Comparing $\ti \lambda$ and $\lambda_1$. Need to justify carefully?} 

Using Corollary \ref{cor:ti_lambda_vs lam2} and the fact that $|\ti \lambda|/\sqrt{\lambda_1} \to \infty$ under the alternative hypothesis, we have
\begin{align*}
	\num \label{eqn:Ub_expectation}
	| \E U_b | &\les |\ti \lambda|^2 \lambda_1,
\end{align*}
and 
\begin{align*}
	\num \label{eqn:Ub_variance}
	\var(U_b) &\les |\ti \lambda|^2 \lambda_1^3. 
\end{align*}

The terms below appear in Lemma G.9 of \cite{JinKeLuo21}. The bounds on $T_1$ and $F$ are quoted directly from \cite{JinKeLuo21} since they do not depend on $\tilde Omega$. 

\begin{lemma}
	\label{lem:Uc}
	Under the alternative hypothesis, it holds that 
	\begin{alignat*}{2}
		|\E T_1| &\leq \frac{ \| \theta \|_2^6 }{ \| \theta \|_1^2} \les \lambda_1,
		\quad &&\var(T_1) \les \frac{ \| \theta \|_2^6 \, \| \theta \|_3^3}{ \| \theta \|_1 } \les \lambda_1^3
		\\ | \E T_2 | &\leq \frac{ \| \gam \|_2^2 \, \| \theta \|_2^4 }{ \| \theta \|_1^2}
		\les | \ti \lambda| , 
		\quad &&\var(T_2) \les \frac{ \| \gam \|_2^4 \, \| \theta \|_2^8}{\| \theta \|_1^2} \les |\ti \lambda|^2 \lambda_1^2 
		\\ | \E F | &\les \| \theta \|_2^4 \les \lambda_1^2, 
		\quad &&\var(F) \les \frac{ \| \theta \|_2^{10} }{ \| \theta \|_1^2 } 
		\les \lambda_1^3 
	\end{alignat*}
\end{lemma}

%	\begin{align*}
%		|\E T_1| &\leq \frac{ \| \theta \|_2^6 }{ \| \theta \|_1^2},
%		\quad \var(T_1) \les \frac{ \| \theta \|_2^6 \, \| \theta \|_3^3}{ \| \theta \|_1 } 
%		\\ | \E T_2 | &\leq \frac{ \| \gam \|_2^2 \, \| \theta \|_2^4 }{ \| \theta \|_1^2}, 
%		\quad \var(T_2) \les \frac{ \| \gam \|_2^4 \, \| \theta \|_2^8}{\| \theta \|_1^2} 
%		\\ | \E F | &\les \| \theta \|_2^4, 
%		\quad \var(F) \les \frac{ \| \theta \|_2^{10} }{ \| \theta \|_1^2 } 
%	\end{align*}
Using Corollary \ref{cor:ti_lambda_vs lam2} and the fact that $|\ti \lambda|/\sqrt{\lambda_1} \to \infty$ under the alternative hypothesis, we have
\begin{align*}
	\num \label{eqn:Uc_expectation}
	|	\E U_c | &\les \lambda_1^2, 
\end{align*}
and
\begin{align*}
	\num \label{eqn:Uc_variance}
	\var(U_c) &\les |\ti \lambda|^2 \lambda_1^2. 
\end{align*}

Using Corollary \ref{cor:ti_lambda_vs lam2} and that $| \ti \lambda |/\sqrt{\lambda_1} \to \infty$, the inequalities \eqref{eqn:Ua_expectation}--\eqref{eqn:Uc_variance} imply Theorem \ref{thm:proxy_SgnQ}. 

%	If $\tilde \lambda^2 \gg \lambda_1$, all expectations above are $o(\tilde \lambda^4)$, and all variances are $o(\tilde \lambda^8)$. It follows that proxy SgnQ has full power assuming $\tilde \lambda^2 \gg \lambda_1$.  \pax{Maybe give a short argument about this.}

\subsubsection{Real SgnQ}

Our first lemma regarding real SgnQ plays the part of Lemma G.11 from \cite{JinKeLuo21}. 

\begin{lemma}
	\label{lem:real_SgnQ_tistarQ}
	Under the previous assumptions, as $n \to \infty$,
	\begin{itemize}
		\item Under the null hypothesis, $| \E[Q^* - \tilde Q^*]| = o(\| \theta \|_2^4)$
		and $\var(Q^* - \tilde Q^*) = o(\| \theta \|_2^8)$. 
		\item Under the alternative hypothesis, if $| \ti \lambda|/\sqrt{\lambda_1} \to \infty$, then $| \E[Q^* - \tilde Q^*]| \les | \ti \lambda |^2 \lambda_1$ 
		and $\var(Q^* - \tilde Q^*) \les | \ti \lambda |^2 \lambda_1^3 $. 
	\end{itemize}
\end{lemma}

The following lemma plays the part of Lemma G.12 from \cite{JinKeLuo21}. 

\begin{lemma}
	\label{lem:real_SgnQ_starQ}
	Under the previous assumptions, as $n \to \infty$,
	\begin{itemize}
		\item Under the null hypothesis, $| \E[Q - \tilde Q^*]| = o(\| \theta \|_2^4)$
		and $\var(Q - \ti Q^*) = o(\| \theta \|_2^8)$.
		\item Under the alternative hypothesis, if $| \ti \lambda|/\sqrt{\lambda_1} \to \infty$, then $| \E[Q - \tilde Q^*]| \les \lambda_1^2 + | \ti \lambda|^3$ 
		and $\var(Q - \ti Q^*) \les \lambda_1^4 $. 
	\end{itemize}
\end{lemma}

\subsection{Proofs of Lemmas \ref{lem:ideal_SgnQ}--\ref{lem:real_SgnQ_starQ}}

\subsubsection{Proof strategy}
\label{sec:proof_strategy}
First we describe our method of proof for Lemmas \ref{lem:ideal_SgnQ}--\ref{lem:Uc}. We borrow the following strategy from \cite{JinKeLuo21}. Let $T$ denote a term appearing in one of the Lemmas \ref{lem:ideal_SgnQ}--\ref{lem:Uc}, which takes the general form
\begin{align*}
	T = \sum_{i_1, \ldots, i_m \in \mathcal{R}} c_{i_1, \ldots, i_m} 
	G_{i_1, \ldots, i_m}
\end{align*}
where 
%	\begin{itemize}
%		\item 	$m = O(1)$,
%		\item $ \mathcal{R}$ is a subset of $[n]^m$,
%		\item $G_{i_1, \ldots, i_m}$ is a product of terms of the form $W_{i_s, i_{s'}}$ where $s, s' \in [m]$, and
%		\item $c_{i_1, \ldots, i_m}$ is a nonstochastic term, itself a product of terms of
%	\end{itemize}
\begin{itemize}
	\item 	$m = O(1)$,
	\item $ \mathcal{R}$ is a subset of $[n]^m$,
	\item $c_{i_1, \ldots, i_m} = \prod_{(s, s') \in A} \Gamma_{i_s, i_{s'}}\rp{s,s'} $ is a nonstochastic coefficient where $A \subset [m] \times [m]$ and 
	$\Gamma\rp{s, s'} \in \{ \ti \Omega, \eta^* \mf{1}^\T, \eta \mf{1}^\T , \mf{1} \mf{1}^\T \}$, and 
	\item $G_{i_1, \ldots, i_m} = \prod_{ (s, s') \in B } W_{i_s, i_{s'}}$ where $B \subset [m] \times [m]$. 
\end{itemize}
Since we are studying signed quadrilateral, one can simply take $m = 4$ above, though we wish to state the lemma in a general way. 

Define a \textit{canonical upper bound} $\overline{\Gamma_{i_s, i_{s'}}\rp{s, s'}} $ (up to constant factor) on $\Gamma_{i_s, i_{s'}}\rp{s,s'}$ as follows:
\begin{align*}
	\num \label{eqn:canon_ubd_Gam}
	\overline{\Gamma_{i_s, i_{s'}}\rp{s, s'}}
	= \begin{cases}
		\beta_{i_s} \theta_{i_s} \beta_{i_{s'}} \theta_{i_{s'}} \quad &\text{if } \Gamma\rp{s,s'} = \ti \Omega, 
		\\ 	\theta_{i_s}  \quad &\text{if } \Gamma\rp{s,s'}  \in \{ \eta^* \mf{1}^\T, \eta \mf{1}^\T \}
		\\ 1 \quad &\text{otherwise}. 
	\end{cases}
\end{align*}
Define 
\begin{align*}
	\num 	\label{eqn:cbar_def}
	\overline{c_{i_1, \ldots, i_m}}
	= \prod_{(s, s') \in A} \overline{ \Gamma_{i_s, i_{s'}}\rp{s,s'} }.
\end{align*}
By Corollary \ref{cor:ti_lambda_vs lam2} and Lemma \ref{lem:eta_bd},
\begin{align*}
	|c_{i_1, \ldots i_m} | \les  \overline{c_{i_1, \ldots, i_m}}. 
\end{align*}

In \cite{JinKeLuo21}, each term $T$ is decomposed into a sum of $L = O(1)$ terms:
\begin{align*}
	\num \label{eqn:T_decomp}
	T = \sum_{\ell = 1}^L T\rp{L}
	= \sum_{\ell = 1}^L \,  \sum_{i_1, \ldots, i_m \in \mathcal{R}\rp{\ell}} c_{i_1, \ldots, i_m} 
	G_{i_1, \ldots, i_m}.
\end{align*}

In our analysis below and that of \cite{JinKeLuo21}, an upper bound $\overline{ \E T }$ on $|\E T|$ is obtained by
\begin{align*}
	|\E T | &\leq \sum_{\ell = 1}^L | \E T\rp{\ell} |
	\leq \sum_{\ell = 1}^L \, \, \sum_{i_1, \ldots, i_m \in \mathcal{R}\rp{\ell}}
	| c_{i_1, \ldots, i_m} | \cdot | \E G_{i_1, \ldots, i_m} |
	\\&\leq \sum_{\ell = 1}^L \, \, \sum_{i_1, \ldots, i_m \in \mathcal{R}\rp{\ell}}
	\overline{c_{i_1, \ldots, i_m}}  \cdot | \E G_{i_1, \ldots, i_m} |
	\\&=: \overline{ \E T }.
	\num \label{eqn:expT_ubd}
\end{align*}

Also an upper bound $\overline{\var T}$ on $\var T$ is obtained by
\begin{align*}
	\var T
	&\leq L \sum_{\ell = 1}^L \var(T \rp{\ell}) 
	\\ &	\leq L \sum_{\ell = 1}^L \, \, 
	\sum_{ \substack{ i_1, \ldots, i_m \in \mathcal{R}\rp{\ell} \\
			i_1', \ldots, i_m' \in \mathcal{R}\rp{\ell} }	} 
	| c_{i_1, \ldots, i_m} c_{i_1', \ldots, i_m'} |
	\cdot \big | \cov\big( G_{i_1, \ldots, i_m}, G_{i_1', \ldots, i_m'}   \big) \big |
	\\ &	\leq L \sum_{\ell = 1}^L \, \, 
	\sum_{ \substack{ i_1, \ldots, i_m \in \mathcal{R}\rp{\ell} \\
			i_1', \ldots, i_m' \in \mathcal{R}\rp{\ell} }	} 
	\overline{c_{i_1, \ldots, i_m}} \cdot \overline{ c_{i_1', \ldots, i_m'}} 
	\cdot \big | \cov\big( G_{i_1, \ldots, i_m}, G_{i_1', \ldots, i_m'}   \big) \big |
	\\& =: \overline{\var T}. 
	\num \label{eqn:varT_ubd}
\end{align*}

In Lemmas \ref{lem:ideal_SgnQ}--\ref{lem:Uc}, all stated upper bounds are obtained in this manner and are therefore upper bounds on $\overline{\E T}$ and $\overline{\var T}$. 

Note that the definition of $\overline{\E T}$ and $\overline{\var T}$ depends on the specific decomposition \eqref{eqn:T_decomp} of $T$ given in \cite{JinKeLuo21}. Refer to the proofs below for details including the explicit decomposition. Again we remark that the difference between our setting and \cite{JinKeLuo21} is that the canonical upper bound on $|\ti \Omega_{ij}|$ used in \cite{JinKeLuo21} is of the form $\alpha \theta_i \theta_j$ rather than the inequality $\beta_i \theta_i \beta_j \theta_j$ which is required for our purposes.

The formalism above immediately yields the following useful fact that allows us to transfer bounds between terms that have similar structures. 

\begin{lemma}
	\label{lem:transfer}
	Suppose that 
	\begin{align*}
		T &= \sum_{i_1, \ldots, i_m \in \mathcal{R}} c_{i_1, \ldots, i_m} 
		G_{i_1, \ldots, i_m}, 
		\\  T^* &=  \sum_{i_1, \ldots, i_m \in \mathcal{R}} c^*_{i_1, \ldots, i_m} 
		G_{i_1, \ldots, i_m} , 
	\end{align*}
	where
	\begin{align*}
		|c_{i_1, \ldots, i_m}|	&\les  \overline{c^*_{i_1, \ldots, i_m}} 
	\end{align*}
	Then
	\begin{align*}
		| \E T| \les \overline{\E[T^*]}
	\end{align*}
	and
	\begin{align*}
		\var \, T \les \overline{\var \, T^*}.  
	\end{align*}
\end{lemma}

In the second part of our analysis, we show that Lemmas \ref{lem:real_SgnQ_tistarQ} and \ref{lem:real_SgnQ_starQ} follow from Lemmas \ref{lem:ideal_SgnQ}--\ref{lem:Uc} and repeated applications of Lemma \ref{lem:transfer}. 

\subsubsection{Proof of Lemma \ref{lem:ideal_SgnQ}}

The bounds for $X_1$ follow immediately from \cite{JinKeLuo21}.

In \cite[Supplement, pg.37]{JinKeLuo21} it is shown that
$\E X_2 = 0$, and
\begin{align*}
	\var(X_2) = 2 \sum_{i,j,k,\ell(dist.)} \ti \Omega_{ij}^2 \cdot \var(W_{jk} W_{k\ell} W_{\ell i}).
\end{align*}
Thus by \eqref{eqn:assn1} and  \eqref{eqn:assn2},
\begin{align*}
	\var(X_2) &\les 	\sum_{i,j,k,\ell(dist.)} \ti \Omega_{ij}^2 \cdot \var(W_{jk} W_{k\ell} W_{\ell i})
	\les \sum_{i,j,k,\ell} \beta_i^2 \theta_i^2 \beta_j^2 \theta_j^2 
	\cdot \Omega_{jk} \Omega_{k\ell} \Omega_{\ell i} 
	\\ &\les \sum_{i,j,k,\ell} \beta_i^2 \theta_i^2 \beta_j^2 \theta_j^2 
	\cdot \theta_{j} \theta_{k}^2 \theta_\ell^2 \theta_i = \| \beta \circ \theta \|_2^4 \, \| \theta \|_2^4
\end{align*}

In \cite[Supplement, pg. 38]{JinKeLuo21} it is shown that $\E X_3 = 0$ and
\begin{align*}
	\var(X_3) &\les \sum_{i, k, \ell (dist)} \big( \sum_{j \notin \{i, k,\ell\}} \ti \Omega_{ij} \ti \Omega_{jk} \big)^2 \cdot \var(W_{k\ell} W_{\ell i}).
\end{align*}
By \eqref{eqn:tiOm_bd} and \eqref{eqn:beta_bds},
\begin{align*}
	\big( \sum_{j \notin \{i, k,\ell\}} \ti \Omega_{ij} \ti \Omega_{jk} \big)^2
	&\leq \beta_i^2 \theta_i^2 \, \beta_k^2 \theta_k^2 \,
	\, \| \gam \|_2^4
\end{align*}
Thus by \eqref{eqn:assn1} and \eqref{eqn:assn2}, 
\begin{align*}
	\var(X_3) &\les \sum_{i,k,\ell} \beta_i^2 \theta_i^2 \, \beta_k^2 \theta_k^2 \,
	\, \| \gam \|_2^4 \cdot \Omega_{k\ell} \Omega_{\ell i} 
	\les \sum_{i,k,\ell} \beta_i^2 \theta_i^3 \, \beta_k^2 \theta_k^3 \,
	\, \| \gam \|_2^4 \cdot \theta_\ell^2 
	&\les \| \gam \|_2^8 \, \| \theta \|_2^2. 
\end{align*}

In \cite[Supplement, pg. 38]{JinKeLuo21} it is shown that $\E X_4 = 0$ and
\begin{align*}
	\var(X_4) \les \sum_{i,j,k,\ell(dist.)} \ti \Omega_{ij}^2 
	\ti \Omega_{k \ell}^2 \cdot \var(W_{jk} W_{\ell i}). 
\end{align*}
By \eqref{eqn:assn1} and \eqref{eqn:tiOm_bd},
\begin{align*}
	\var(X_4) &\les \sum_{i, j, k, \ell}
	\beta_i^2 \theta_i^2 
	\beta_j^2 \theta_j^2
	\beta_k^2 \theta_k^2
	\beta_\ell^2 \theta_\ell^2 \cdot \theta_j \theta_k \theta_\ell \theta_i 
	\les \| \gam \|_2^8. 
\end{align*}

In \cite[Supplement, pg. 39]{JinKeLuo21} it is shown that $\E X_5 = 0$ and 
\begin{align*}
	\var(X_5) &= 2 \sum_{i < \ell} \big( \sum_{ \substack{j, k \notin \{ i, \ell \} \\ j \neq k}} \ti \Omega_{ij} \ti \Omega_{jk} \ti \Omega_{k\ell} \big)^2 \cdot \var(W_{\ell i}).
\end{align*}
We have
\begin{align*}
	\big|	\sum_{ \substack{j, k \notin \{ i, \ell \} \\ j \neq k}} \ti \Omega_{ij} \ti \Omega_{jk} \ti \Omega_{k\ell} \big|
	\les \beta_i \theta_i \| \gam \|_2^4 \beta_\ell \theta_\ell. 
\end{align*}
Thus by \eqref{eqn:assn1} and \eqref{eqn:assn2}, 
\begin{align*}
	\var(X_5) 
	\les \sum_{i, \ell} \big( \beta_i \theta_i \| \gam \|_2^4 \beta_\ell \theta_\ell \big)^2 \cdot \theta_\ell \theta_i 
	\les \| \gam \|_2^{12}. 
\end{align*}

Note that $X_6$ is a nonstochastic term. Mimicking \cite[Supplement, pg. 39]{JinKeLuo21}, we have by \eqref{eqn:beta_bds}, 
\begin{align*}
	|X_6 - \ti \lambda^4| 
	&\les \sum_{i,j,k,\ell(not \, dist.)} \beta_i^2 \theta_i^2 
	\beta_j^2 \theta_j^2 \beta_k^2 \theta_k^2
	\beta_\ell^2 \theta_\ell^2
	\les \sum_{i, j, k}   \beta_i^2 \theta_i^2 
	\beta_j^2 \theta_j^2 \beta_k^4 \theta_k^4
	\les \| \gam \|_2^6 \les |\ti \lambda|^3. 
\end{align*}
This completes the proof. \qed

\subsubsection{Proof of Lemma \ref{lem:Ua}}

The bounds on $Y_1$ carry over directly from \cite[Lemma G.5]{JinKeLuo21}. 

In \cite[Supplement, pg. 43]{JinKeLuo21} it is shown that $\E Y_2 = 0$. To study $\var(Y_2)$, we write $Y = Y_{2a} + Y_{2b} + Y_{2c}$ where as in \cite[Supplement, pg. 43]{JinKeLuo21}, we define  
\begin{eqnarray}  \label{proof-Y2-decompose}
	Y_2 &=&  -\frac{1}{\sqrt{v}} \sum_{\substack{i, j, k, \ell (dist)\\s\neq j}}\eta_i \widetilde{\Omega}_{jk}  W_{js} W_{k \ell}  W_{\ell i}\cr
	&& - \frac{1}{\sqrt{v}} \sum_{i, k, \ell (dist)} \Bigl(\sum_{j\notin\{i,k,\ell\}}\eta_j \widetilde{\Omega}_{jk}\Bigr) W^2_{i\ell} W_{k \ell}\cr
	&& -   \frac{1}{\sqrt{v}} \sum_{\substack{i, k, \ell (dist)\\s\notin\{ i,\ell\}}} \Bigl(\sum_{j\notin\{i,k,\ell\}}\eta_j \widetilde{\Omega}_{jk}\Bigr)  W_{is}W_{k \ell}  W_{\ell i}\cr
	&\equiv& Y_{2a} + Y_{2b} + Y_{2c}. 
\end{eqnarray}
There it is shown that 
\begin{align*}
	\var(Y_{2a}) \les \frac{1}{v} \sum_{ijk\ell s} 
	\big | \eta_i \ti \Omega_{jk} +
	\eta_i \ti \Omega_{sk} +
	\eta_k \ti \Omega_{ji} +
	\eta_k \ti \Omega_{si} \big |^2
	\cdot \var(W_{js} W_{k\ell} W_{\ell i}).
\end{align*}		
We have by \eqref{eqn:etai_bd}
\begin{align*}
	\big | \eta_i \ti \Omega_{jk} +
	\eta_i \ti \Omega_{sk} +
	\eta_k \ti \Omega_{ji} +
	\eta_k \ti \Omega_{si} \big |
	\les \theta_i \beta_j \theta_j \beta_k \theta_k 
	+ \theta_i \beta_s \theta_s \beta_k \theta_k 
	+ \theta_k \beta_j \theta_j \beta_i \theta_i 
	+ \theta_k \beta_s \theta_s \beta_i \theta_i.
\end{align*}
Hence by \eqref{eqn:assn1}, \eqref{eqn:assn2}, and \eqref{eqn:v0_v1_bd}, 
\begin{align*}
	\var(Y_{2a}) 
	&\les 
	\frac{1}{v} \sum_{ijk\ell s} \big(  \theta_i \beta_j \theta_j \beta_k \theta_k 
	+ \theta_i \beta_s \theta_s \beta_k \theta_k 
	+ \theta_k \beta_j \theta_j \beta_i \theta_i 
	+ \theta_k \beta_s \theta_s \beta_i \theta_i  \big)^2 \cdot \theta_j \theta_s \theta_k \theta_\ell^2 \theta_i 
	\\ &\les \frac{ \| \gam \|_2^4 \| \theta \|_2^4 }{ \| \theta \|_1 }
\end{align*}

Next, in \cite[Supplement, pg. 43]{JinKeLuo21}, it is shown that
\begin{align*}
	\var(Y_{2b}) &\les
	\frac{1}{v} \sum_{ \substack{ ik\ell (dist) \\ i' k' \ell' (dist) }}
	|\alpha_{ik\ell} \alpha_{i'k'\ell'}| \cdot \E[  W_{i\ell}^2 W_{k\ell} , W_{i'\ell'}^2 W_{k' \ell'} ]
\end{align*}
where $\alpha_{ik\ell} = \sum_{j \notin \{ i, k, \ell \} } \eta_j \ti \Omega_{jk}$. By \eqref{eqn:beta_bds}, 
\begin{align*}
	| \alpha_{ik\ell} | \les  \| \gam \|_2 \| \theta \|_2 \, \theta_k.
\end{align*}
By \eqref{eqn:assn1}, \eqref{eqn:v0_v1_bd}, the inequalities above, and the casework in \cite[Supplement, pg.44]{JinKeLuo21} on $E[W_{i\ell}^2 W_{k\ell} , W_{i'\ell'}^2 W_{k' \ell'}]$, 
\begin{align*}
	\var(Y_{2b})
	&\les \frac{1}{v} \sum_{ \substack{ ik\ell (dist) \\ i' k' \ell' (dist) }} 
	\| \gam \|_2^2 \| \theta \|_2^2  \theta_k  \theta_{k'}
	\E[ W_{i\ell}^2 W_{k\ell} , W_{i'\ell'}^2 W_{k' \ell'} ] 
	\\&\les \frac{\| \gam \|_2^2 \| \theta \|_2^2 }{v} \big(  \sum_{ik\ell } \theta_i \theta_k^3 \theta_\ell^2 + \sum_{i k \ell i'} \theta_i \theta_k^3 \theta_\ell^3 \theta_{i'} 
	+ \sum_{ik\ell} \theta_i^2 \theta_k^2 \theta_\ell^2 \big)
	\\&\les \| \gam \|_2^2 \| \theta \|_2^6.
\end{align*}

Next, in \cite[Supplement, pg.44]{JinKeLuo21} it is shown that
\begin{align*}
	\var(Y_{2c}) \les \frac{1}{v} \sum_{\substack{ik\ell (dist) \\ s \notin \{ i, \ell \}} } \beta^2_{ik\ell} \var(W_{is} W_{k\ell} W_{\ell i} )
\end{align*}
where $\alpha_{ik\ell}$ is defined the same as with $Y_{2b}$. Thus
\begin{align*}
	\var(Y_{2c}) \les \frac{1}{v} \sum_{\substack{ik\ell (dist) \\ s \notin \{ i, \ell \}} } \|  \gam \|_2^2 \| \theta \|_2^2 \theta_k^2 \cdot \theta_k \theta_\ell^2 \theta_s
	\les  \frac{\| \gam \|_2^2 \| \theta \|_2^8}{\| \theta \|_1}.  
\end{align*}

Combining the results for $Y_{2a}, Y_{2b}, Y_{2c}$ gives the claim for $\var(Y_2)$. 

In \cite[Supplement, pg.45]{JinKeLuo21} it is shown that $\E Y_3 = 0$ and the decomposition 
\begin{align}\label{proof-Y3-decompose}
	Y_3 &= - \frac{2}{\sqrt{v}} \sum_{i, j, k, \ell (dist)}\eta_i \widetilde{\Omega}_{k\ell}  W^2_{jk}  W_{\ell i} - \frac{2}{\sqrt{v}} \sum_{\substack{i, j, k, \ell (dist)\\s\notin\{j,k\}}}\eta_i \widetilde{\Omega}_{k\ell}  W_{js} W_{jk}  W_{\ell i}\cr
	& \equiv Y_{3a} + Y_{3b},
\end{align}
is introduced. There it is shown that
\[
\mathrm{Var}(Y_{3a}) = \frac{4}{v} \sum_{\substack{i,j,k,\ell (dist)\\i',j',k',\ell' (dist)}} (\eta_i \widetilde{\Omega}_{k\ell} \eta_{i'} \widetilde{\Omega}_{k'\ell'})\cdot\mathbb{E}[W^2_{jk}  W_{\ell i}W^2_{j'k'}W_{\ell' i'}]. 
\]
Using \eqref{eqn:assn1}, \eqref{eqn:assn2} \eqref{eqn:beta_bds} and the casework in \cite[Supplement, pg.45]{JinKeLuo21}, 
\begin{align*}
	\var(Y_{3a})
	%	&\les \frac{1}{\| \theta \|_1^2} \bigg( \sum_{ijk\ell} [\beta_k^2 \beta_\ell^2 \theta_i^3 \theta_j \theta_k^3 \theta_\ell^3
	%	+ \beta_k \beta_j \beta_\ell \beta_i 
	%	\theta_i^3 \theta_j^2 \theta_k^2 \theta_\ell^3 
	%	+ \beta_k \beta_\ell \beta_i \beta_j \theta_i^3 \theta_j^2 \theta_k^3 \theta_\ell^2  ] 
	%	\\&\quad + \bigg) 
	&\les \frac{1}{\| \theta \|_1^2} \bigg( \sum_{ijk\ell} 
	[ \beta_k^2 \beta_\ell^2 + \beta_{i} \beta_j \beta_k \beta_\ell  ]\theta_i^2 \theta_j^2 \theta_k^2 \theta_\ell^2 
	+ \sum_{ijk\ell j' k'}  \beta_k \beta_\ell^2 \beta_{k'} \theta_i^3 \theta_j \theta_k^2 \theta_\ell^3 \theta_{j'} \theta_{k'}^2  \bigg) 
	\\&\les \frac{  \| \gam \|_2^4 \| \theta \|_2^4   }{\| \theta \|_1^2}
	+  \| \gam \|_2^4 \| \theta \|_2^4 \les \| \gam \|_2^4 \| \theta \|_2^4
\end{align*}

Similar to the study of $Y_{2a}$ we have
\begin{align*}
	\var(Y_{3b})
	&\les \frac{1}{v} \sum_{ijk\ell s}
	\big( \theta_i \beta_k \theta_k \beta_\ell \theta_\ell
	+ \theta_\ell \beta_k \theta_k \beta_i \theta_i
	+ \theta_i \beta_s \theta_s \beta_\ell \theta_\ell
	+ \theta_\ell \beta_s \theta_s \beta_i \theta_i
	\big)^2
	\cdot \var(W_{sj} W_{jk} W_{\ell i})
	\\&\les \frac{1}{v} \sum_{ijk\ell s}
	\big( \theta_i \beta_k \theta_k \beta_\ell \theta_\ell
	+ \theta_\ell \beta_k \theta_k \beta_i \theta_i
	+ \theta_i \beta_s \theta_s \beta_\ell \theta_\ell
	+ \theta_\ell \beta_s \theta_s \beta_i \theta_i
	\big)^2
	\cdot \theta_s \theta_j^2 \theta_k \theta_\ell \theta_i 
	\\&\les \frac{ \| \gam \|_2^4 \|\theta\|_2^4}{ \| \theta \|_1}.
\end{align*}

Combining the bounds on $\var(Y_{3a})$ and $\var(Y_{3b})$ yields the desired bound on $\var(Y_{3})$. 

Following \cite[Supplement, pg.46]{JinKeLuo21} we obtain the decomposition
\begin{align*}
	Y_4 
	&= -\frac{1}{\sqrt{v}}\sum_{\substack{i,j,\ell (dist)\\s\neq j}}\Bigl(\sum_{k\notin\{i,j,\ell\}} \eta_i\widetilde{\Omega}_{jk}\widetilde{\Omega}_{k\ell}\Bigr) W_{js} W_{\ell i} - \frac{1}{\sqrt{v}} \sum_{\substack{i,\ell (dist)\\s\neq i}} \Bigl(\sum_{j,k\notin\{ i,\ell\}}\eta_j \widetilde{\Omega}_{jk}\widetilde{\Omega}_{k\ell}\Bigr) W_{is}W_{\ell i}\cr
	&\equiv Y_{4a} + Y_{4b}. 
\end{align*}
First we study $Y_{4a}$, which is shown in \cite{JinKeLuo21} to have zero mean and satisfy the following:
\begin{align*}
	\var(Y_{4a}) &\les \frac{1}{v} \sum_{\substack{ij\ell (dist) \\ s \neq j} } 
	\alpha_{ij\ell}^2 \var(W_{js} W_{\ell i})
\end{align*}
where $\alpha_{ij\ell} = \sum_{k \notin \{i, j, \ell \}} \eta_i \ti \Omega_{jk} 
\ti \Omega_{k\ell}$. 
Simlar to previous arguments, we have
\begin{align*}
	\var(Y_{4a}) &	\les \frac{1}{\| \theta \|_1^2} 
	\sum_{ij\ell s} \theta_i^2 (\beta_j \theta_j)^2 (\beta_\ell \theta_\ell)^2 \| \gam \|_2^4 \cdot \theta_i \theta_j \theta_\ell \theta_s 
	\\&\les \frac{ \| \gam \|_2^4 \|\theta \|_2^2}{ \| \theta \|_1 }. 
\end{align*}

Next we study $Y_{4b}$ using the decomposition 
\[
Y_{4b} = -\frac{1}{\sqrt{v}} \sum_{i,\ell (dist)} \beta_{i\ell} W^2_{\ell i}- \frac{1}{\sqrt{v}} \sum_{\substack{i,\ell (dist)\\s\notin\{ i,\ell\}}} \beta_{i\ell} W_{is}W_{\ell i}\equiv \widetilde{Y}_{4b} + Y_{4b}^*. 
\]
from \cite[Supplement,pg.47]{JinKeLuo21}. There it is shown that only $\E \ti Y_{4b}$ is nonzero and
\begin{align*}
	| \E \ti Y_{4b} | \les \frac{1}{\| \theta \|_1} \sum_{i, \ell} |\alpha_{i\ell}| \theta_i \theta_\ell.
\end{align*}
where $\alpha_{i,\ell} = \sum_{j,k\notin\{ i,\ell\}}\eta_j \widetilde{\Omega}_{jk}\widetilde{\Omega}_{k\ell}$. In our case, we derive from \eqref{eqn:beta_bds},
\begin{align*}
	|\alpha_{i\ell}| \les \beta_\ell \theta_\ell \| \gam \|_2^3 \| \theta \|_2. 
\end{align*}
Using similar arguments from before,
\begin{align*}
	| \E \ti Y_{4b} | \les \frac{1}{\| \theta \|_1}  \sum_{i\ell} 
	\beta_\ell \theta_\ell \| \gam \|_2^3 \| \theta \|_2 \cdot \theta_i \theta_\ell
	\les \| \gam \|_2^4 \| \theta \|_2^2. 
\end{align*}

Now we study $\var(Y_{4b})$. Using the bound above on $|\alpha_{i\ell}|$ and direct calculations, 
\begin{align*}
	\mathrm{Var}(\widetilde{Y}_{4b}) &= \frac{2}{v} \sum_{i,\ell (dist)} \alpha^2_{i\ell}\cdot\mathrm{Var}( W^2_{i\ell}) 
	\les \frac{1}{\|\theta\|^2_1}\sum_{i,\ell} \beta_\ell^2 \theta_\ell^2 \| \gam \|_2^6 \| \theta \|_2^2 \cdot \theta_i\theta_\ell 
	\les \frac{\| \gam \|_2^8 \| \theta \|_2^2}{\|\theta\|_1},\cr
	\mathrm{Var}(Y^*_{4b}) &\leq \frac{1}{v} \sum_{\substack{i,\ell (dist)\\s\notin\{ i,\ell\}}} \alpha^2_{i\ell} \cdot \mathrm{Var}( W_{is}W_{\ell i}) \leq \frac{1}{\|\theta\|_1^2}\sum_{i,\ell,s} \beta_\ell^2 \theta_\ell^2 \| \gam \|_2^6 \| \theta \|_2^2 \cdot\theta_i^2\theta_\ell\theta_s \leq \frac{\| \gam \|_2^8 \| \theta \|_2^4 }{\|\theta\|_1}. 
\end{align*}
Combining the results above yields the required bounds on $\E Y_{4b}$ and $\var(Y_{4b})$. 

In \cite[Supplement, pg.48]{JinKeLuo21} it is shown that $\E Y_5 = 0$ and 
\begin{align*}
	\mathrm{Var}(Y_5)&\les \frac{1}{v}\sum_{\substack{j,k,\ell (dist)\\s\neq j}}\alpha^2_{jk\ell}\cdot\mathrm{Var}( W_{js}W_{k\ell})
\end{align*}
where
\begin{align*}
	\alpha_{jk\ell} \equiv \sum_{i\notin\{j,k,\ell\}} \eta_i\widetilde{\Omega}_{jk}\widetilde{\Omega}_{\ell i}.
\end{align*}
We have using \eqref{eqn:tiOm_bd}, \eqref{eqn:beta_bds} and the triangle inequality,
\begin{align*}
	|\alpha_{jk\ell}| \les \| \theta \|_2^2 (\beta_j \theta_j) (\beta_k \theta_k) (\beta_\ell \theta_\ell ).  
\end{align*}
Thus, by similar arguments to before,
\begin{align*}
	\var(Y_5) \les \frac{1}{\| \theta \|_1^2 } 
	\sum_{jk\ell} \big( \| \theta \|_2^4 (\beta_j \theta_j)^2 (\beta_k \theta_k)^2 (\beta_\ell \theta_\ell )^2 \big) \theta_j \theta_s \theta_k \theta_\ell 
	\les \frac{ \| \theta \|_2^4 \| \gam \|_2^6 }{\| \theta \|_1 }. 
\end{align*}

Next, in \cite[Supplement, pg.49]{JinKeLuo21} it is shown that $\E Y_6 = 0$ and 
\begin{align*}
	\mathrm{Var}(Y_6)& = \frac{8}{v}\sum_{j,s (dist)}\Bigl(\sum_{i,k,\ell (dist) \notin\{j\}} \eta_i\widetilde{\Omega}_{jk}\widetilde{\Omega}_{k\ell}\widetilde{\Omega}_{\ell i}\Bigr)^2\cdot \mathrm{Var}(W_{js}). 
\end{align*}
We have using \eqref{eqn:tiOm_bd}, \eqref{eqn:beta_bds} and the triangle inequality,
\begin{align*}
	\big| \sum_{i,k,\ell (dist) \notin\{j\}} \eta_i\widetilde{\Omega}_{jk}\widetilde{\Omega}_{k\ell}\widetilde{\Omega}_{\ell i} \big|
	\les \beta_j \theta_j \| \gam \|_2^5 \| \theta \|_2. 
\end{align*} 
Thus
\begin{align*}
	\var(Y_6) \les \frac{1}{\| \theta \|_1^2 } \sum_{j,s} 
	\big( \beta_j^2 \theta_j^2 \| \gam \|_2^{10} \| \theta \|_2^2  \big) \theta_j \theta_s 
	\les \frac{ \| \gam \|_2^{12} \| \theta \|_2^2  }{\| \theta \|_1}. 
\end{align*}

This completes the proof. \qed

\subsubsection{Proof of Lemma \ref{lem:Ub}}

The bounds on $Z_1$ and $Z_2$ carry over directly from \cite[Lemma G.7]{JinKeLuo21} since neither term depends on $\ti \Omega$. 

We consider $Z_3$. In \cite[Supplement, pg.61]{JinKeLuo21}, the decomposition 
\begin{align} \label{proof-Z3-decompose}
	Z_3 =& \sum_{\substack{i, j, k, \ell\\ (dist)}}\eta_i(\eta_j-\teta_j)\eta_j(\eta_k-\teta_k)   \widetilde{\Omega}_{k \ell}  W_{\ell i} +
	\sum_{\substack{i, j, k, \ell\\ (dist)}}\eta_i(\eta_j-\teta_j)^2\eta_k   \widetilde{\Omega}_{k \ell}  W_{\ell i} \cr
	&\hspace{-30pt}+\sum_{\substack{i, j, k, \ell\\ (dist)}}(\eta_i-\teta_i)\eta^2_j(\eta_k-\teta_k)   \widetilde{\Omega}_{k \ell}  W_{\ell i} +
	\sum_{\substack{i, j, k, \ell\\ (dist)}}(\eta_i-\teta_i)\eta_j (\eta_j-\teta_j)\eta_k   \widetilde{\Omega}_{k \ell}  W_{\ell i}\cr
	&\equiv Z_{3a}+Z_{3b} + Z_{3c}+Z_{3d}. 
\end{align}
is introduced. We study each term separately.

In \cite[Supplement, pg.61]{JinKeLuo21} it is shown that $\E Z_{3a} = 0$ and
the decomposition 
\begin{align*}
	Z_{3a} &=\frac{1}{v}\sum_{i, j, k, \ell (dist)} \alpha_{ijk\ell}  W^2_{jk}W_{\ell i} +   \frac{1}{v}\sum_{\substack{ i, j, k, \ell (dist)\\s\neq j,t\neq k, (s,t)\neq (k,j)}} \alpha_{ijk\ell} W_{js}W_{kt}W_{\ell i}\cr
	&\equiv \widetilde{Z}_{3a} + Z^*_{3a}. 
\end{align*}
is introduced, where $\alpha_{ijk\ell} \equiv \eta_i \eta_j \ti \Omega_{k\ell}$. Then
\begin{align*}
	\var(\ti Z_{3a})
	\les \sum_{ \substack{ ijk\ell (dist) \\ i'j'k'\ell' (dist)} }
	|\alpha_{ijk\ell}| |\alpha_{i'k'j'\ell'} | \cdot
	|\cov( W^2_{jk} W_{\ell i}, W_{j'k'}^2 W_{\ell'i'} )|. 
\end{align*}
Using the casework in \cite[Supplement, pg.62]{JinKeLuo21}, \eqref{eqn:assn1}, \eqref{eqn:assn2}, and \eqref{eqn:beta_bds}, we obtain
\begin{align*}
	\var(\ti Z_{3a}) &\les \frac{1}{v^2}
	\big( \sum_{ijk\ell} [\beta_k^2 \beta_\ell^2 + \beta_k \beta_\ell \beta_i \beta_j] \theta_i^3 \theta_j^3 \theta_k^3 \theta_\ell^3 
	+ \sum_{ijk\ell j'k'}  \beta_k \beta_\ell^2 \beta_{k'}   \theta_i^3 \theta_j^2 \theta_k^2 \theta_\ell^3 \theta_{j'}^2 \theta_{k'}^2 
	\big)
	\\ &\les \frac{1}{\| \theta \|_1^4} \big( \| \gam \|_2^4 \| \theta \|_2^2 +
	\| \gam \|_2^4 \| \theta \|_2^4 + \| \gam \|_2^4\| \theta \|_2^8  \big)
	\les \frac{   \| \gam \|_2^4\| \theta \|_2^8    }{ \| \theta \|_1^4} . 
\end{align*}
Similarly,
\begin{align*}
	\var(Z_{3a}^*) 
	&\les \frac{1}{v^2} \bigg( \sum_{ijk\ell s t}  \beta_k^2 \beta_\ell^2 \theta_i^3 \theta_j^3 \theta_k^3 \theta_\ell^3 \theta_s \theta_t
	+ \sum_{ijk\ell s t} [ \beta_k^2   \beta_\ell \beta_j + \beta_k \beta_\ell^2 \beta_j  ]\theta_i^2 \theta_j^3 \theta_k^3 \theta_\ell^3 \theta_s^2 \theta_t \bigg) 
	\\& \les \frac{1}{ \| \theta \|_1^4} \big( \|\gam \|_2^4 \| \theta \|_2^4 \| \theta \|_1^2 + \| \gam \|_2^4 \| \theta \|_2^6 \| \theta \|_1\big) 
	\les \frac{  \| \gam \|_2^4 \| \theta \|_2^4  }{\| \theta \|_1^2}. 
\end{align*}
It follows that
\begin{align*}
	\var(Z_{3a}) \les \| \gam \|_2^4. 
\end{align*}

Next, in \cite[Supplement, pg.63]{JinKeLuo21}, it is shown that $\E Z_{3b}] = 0$ and the decomposition 
\[
Z_{3b} = \frac{1}{v}\sum_{\substack{i,j,\ell (dist)\\s\neq j}} \beta_{ij\ell}   W^2_{js}W_{\ell i} + \frac{1}{v}\sum_{\substack{i,j,\ell (dist)\\s, t (dist) \notin \{j\}}} \beta_{ij\ell}   W_{js}W_{jt}W_{\ell i} \equiv \widetilde{Z}_{3b} + Z^*_{3b}. 
\]
is given. Using  \cite[Supplement, pg.63]{JinKeLuo21} we have
\begin{align*}
	\var( \ti Z_{3b} ) &\les 
	\sum_{ \substack{i, j, \ell, s, t \\ i', j', \ell', s', t' }} 
	|\alpha_{ij\ell}| 	|\alpha_{i'j'\ell'}| 
	| \cov( W_{js}^2 W_{\ell_i}, W_{j's}'^2 W_{\ell' i'} )|.
\end{align*}
where
\begin{align*}
	\alpha_{ij\ell} =  \sum_{k\notin\{i,j,\ell\}} \eta_i\eta_k\widetilde{\Omega}_{k\ell} .
\end{align*}
Using \eqref{eqn:beta_bds}, \eqref{eqn:v0_v1_bd}, and similar arguments to before,
\begin{align*}
	| 	\alpha_{ij\ell}  | \les 
	\theta_i (\beta_\ell \theta_\ell) \| \theta \|_2^2. 
\end{align*}
By the casework in \cite[Supplement, pg.63]{JinKeLuo21}, \eqref{eqn:assn1}, and \eqref{eqn:assn2}, 
\begin{align*}
	\var( \ti Z_{3b} ) &\les 
	\frac{1}{v^2} \bigg( \sum_{ij\ell s} \beta_\ell^2 \|\theta\|_2^4 \theta_i^3 \theta_j \theta_\ell^3 \theta_s 
	+ \sum_{ij\ell s j' s'} \beta_\ell^2 \|\theta\|_2^4 \theta_i^3 \theta_j \theta_\ell^3 \theta_s \theta_{j'} \theta_{s'} 
	+ \sum_{ij\ell s} \beta_\ell \beta_j \|\theta\|_2^4 \theta_i^2 \theta_j^2 \theta_\ell^2 \theta_s^2 
	\bigg) 
	\\ &\les \frac{1}{\| \theta \|_1^4 } 
	\big( \| \gam \|_2^2 \| \theta \|_2^6 \| \theta \|_1 
	+ \| \gam \|_2^2 \| \theta \|_2^6 \| \theta \|_1^3
	+ \| \gam \|_2^2 \| \theta \|_2^8 \big) \les \frac{  \| \gam \|_2^2 \| \theta \|_2^6  }{\| \theta \|_1}. 
\end{align*}
By a similar argument,
\begin{align*}
	\var(Z_{3b}^*) \les \frac{ \| \gam \|_2^2 \| \theta \|_2^8}{\| \theta \|_1}. 
\end{align*}
Hence by \eqref{eqn:assn2},
\begin{align*}
	\var(Z_{3b}) \les \frac{ \| \gam \|_2^2 \| \theta \|_2^8}{\| \theta \|_1}
	\les \| \gam \|_2^2 \| \theta \|_2^6. 
\end{align*}

For $Z_{3c}$, in \cite[Supplement, pg.64]{JinKeLuo21}, it is shown that $\E Z_{3c} = 0$ and the decomposition 
\[
Z_{3c} = \frac{1}{v}\sum_{\substack{i,k,\ell (dist)\\t\neq k}}\alpha_{ik\ell}W^2_{i\ell}W_{kt} + \frac{1}{v}\sum_{\substack{i,k,\ell (dist)\\s\notin\{ i,\ell\}, t\neq k}}\alpha_{ik\ell}W_{is}W_{kt}W_{\ell i}\equiv \widetilde{Z}_{3c} + Z^*_{3c}. 
\]
is given. We have
\begin{align*}
	|\alpha_{ik\ell}| = |\sum_{j\notin\{i,k,\ell\}} \eta_j^2\widetilde{\Omega}_{k\ell}|  \les (\beta_k \theta_k) (\beta_\ell \theta_\ell) \| \theta \|_2^2. 
\end{align*}
By the casework in \cite[Supplement, pg.65]{JinKeLuo21}
\begin{align*}
	\var(\tilde Z_{3c})
	&\les \sum_{ \substack{ik\ell (dist) \\ s \notin \{i, \ell \}, t \neq k}  }
	\sum_{ \substack{i'k'\ell' (dist) \\ s' \notin \{i', \ell' \}, t' \neq k'}  } |\alpha_{ik\ell} \alpha_{i'k'\ell'}|
	| \E W_{i\ell}^2 W_{kt} W_{i'\ell'}^2 W_{k' t'} | 
	\\&\les \frac{\| \theta \|_2^4}{\| \theta \|_1^4}
	\sum_{ik\ell t} 
	\bigg[ \beta_k^2 \beta_\ell^2 \theta_i \theta_k^3 \theta_\ell^3 \theta_t 
	+ \beta_k^2 \beta_\ell \beta_i \theta_i^2 \theta_k^3 \theta_\ell^2 \theta_t 
	+ \beta_k \beta_\ell^2 \beta_t \theta_i^1 \theta_k^2 \theta_\ell^3 \theta_t^2 
	\\&\quad  + \beta_k \beta_\ell \beta_t \beta_i \theta_i^2 \theta_k^2 \theta_\ell^2 \theta_t^2
	+ \beta_k^2 \beta_\ell \beta_i \theta_i^2 \theta_k^3 \theta_\ell^2 \theta_t^1
	+ \beta_k \beta_\ell^2  \beta_t \theta_i \theta_k^2 \theta_\ell^3 \theta_t^2
	+ \beta_k^2 \beta_\ell^2  \theta_i \theta_k^3 \theta_\ell^3 \theta_t \bigg]
	\\ &\quad +	\sum_{ik\ell t i' \ell'} \bigg[\beta_k^2 \beta_\ell \beta_{\ell'} 
	\theta_i \theta_k^3 \theta_\ell^2 \theta_t \theta_{i'} \theta_{\ell'}^2 
	+ \beta_k \beta_\ell \beta_{\ell'} \beta_t 
	\theta_i \theta_k^2 \theta_\ell^2 \theta_t^2 \theta_{i'} \theta_{\ell'}^2 
	\bigg]
	\\ 
\end{align*}
We have by \eqref{eqn:assn2} and \eqref{eqn:beta_bds} that
\begin{align*}
	\sum_{ik\ell t} 
	\bigg[ \beta_k^2 \beta_\ell^2 \theta_i \theta_k^3 \theta_\ell^3 \theta_t 
	+ &\beta_k^2 \beta_\ell \beta_i \theta_i^2 \theta_k^3 \theta_\ell^2 \theta_t 
	+ \beta_k \beta_\ell^2 \beta_t \theta_i^1 \theta_k^2 \theta_\ell^3 \theta_t^2 + \beta_k \beta_\ell \beta_t \beta_i \theta_i^2  \theta_k^2 \theta_\ell^2 \theta_t^2
	\\&\quad  
	+ \beta_k^2 \beta_\ell \beta_i \theta_i^2 \theta_k^3 \theta_\ell^2 \theta_t^1
	+ \beta_k \beta_\ell^2  \beta_t \theta_i \theta_k^2 \theta_\ell^3 \theta_t^2
	+ \beta_k^2 \beta_\ell^2  \theta_i \theta_k^3 \theta_\ell^3 \theta_t \bigg]
	\\& \les \| \gam \|_2^4 \| \theta \|_1^2 +  \| \gam \|_2^2 \| \theta \|_2^4 \| \theta \|_1
	+  \| \gam \|_2^4 \| \theta \|_2^4  +  \| \gam \|_2^4 \| \theta \|_1^2 
\end{align*}
and
\begin{align*}
	\sum_{ik\ell t i' \ell'} \bigg[\beta_k^2 \beta_\ell \beta_{\ell'} 
	\theta_i \theta_k^3 \theta_\ell^2 \theta_t \theta_{i'} \theta_{\ell'}^2 
	+ \beta_k \beta_\ell \beta_{\ell'} \beta_t 
	\theta_i \theta_k^2 \theta_\ell^2 \theta_t^2 \theta_{i'} \theta_{\ell'}^2 
	\bigg] \les \| \gam \|_2^2 \| \theta \|_2^4 \| \theta \|_1^3 +
	\| \gam \|_2^4 \| \theta \|_2^4 \| \theta \|_1^2
\end{align*}
Thus
\begin{align*}
	\var(\ti Z_{3c} )	&\les  \frac{\| \theta \|_2^4}{\| \theta \|_1^4} \big( 
	\| \gam \|_2^4 \| \theta \|_1^2 + \| \gam \|_2^2 \| \theta \|_2^4 \| \theta \|_1^3 
	+\| \gam \|_2^4 \| \theta \|_2^4 \| \theta \|_1^2 \big)  \les \frac{ \| \gam \|_2^4 \| \theta \|_2^8 }{ \| \theta \|_1 }
\end{align*}

To study $Z_{3c}^*$, in \cite[Supplement, pg.65]{JinKeLuo21} the decomposition
\[
Z_{3c}^*=  \frac{1}{v}\sum_{i,k,\ell (dist)}\alpha_{ik\ell}W_{ik}^2W_{\ell i}+ \frac{1}{v}\sum_{\substack{i,k,\ell (dist)\\s\notin\{ i,\ell\}, t\neq k, (s,t)\neq (k,i)}}\alpha_{ik\ell}W_{is}W_{kt}W_{\ell i}
\equiv Z_{3c,1}^* + Z_{3c,2}^*
\]
is used, where recall $	\alpha_{ik\ell} =  \sum_{j\notin\{i,k,\ell\}} \eta_j^2\widetilde{\Omega}_{k\ell}$.
Using a similar argument as before, we have
\begin{align*}
	\var( Z_{3c,1}^*)
	&\les \frac{\| \theta \|_2^4}{ \| \theta \|_1^4}
	\bigg( 
	\sum_{ik\ell} \beta_k^2 \beta_\ell^2 \theta_i^2 \theta_k^3 \theta_\ell^3 
	+ \sum_{ik\ell k'} [\beta_k \beta_{k'} \beta_\ell^2 +
	\beta_k \beta_{k'} \beta_i \beta_\ell ] \theta_i^3 \theta_k^2 \theta_\ell^3  \theta_{k'}^2 \bigg)
	\\ &\les  \frac{\| \theta \|_2^4}{ \| \theta \|_1^4} \big( \| \gam \|_2^4 \| \theta \|_2^2 + \| \theta \|_2^4 \| \gam \|_2^2 \| \theta \|_3^3 
	+ \| \gam \|_2^4 \| \theta \|_2^4 \big) 
	\les \frac{\| \gam \|_2^4 \| \theta \|_2^{10} }{\| \theta \|_1^4}.
\end{align*}
We omit the argument for $Z_{3c,2}^*$ as it is similar and simply state the bound:
\begin{align*}
	\var(Z_{3c,2}^*) \les \frac{ \| \gam \|_2^4 \| \theta \|_2^6}{ \| \theta \|_1^2 }. 
\end{align*}
Combining the results for $\ti Z_{3c}$ and $Z_{3c}^*$, we have
\begin{align*}
	\var(Z_{3c}) \les \frac{ \| \gam \|_2^4 \| \theta \|_2^8}{\| \theta \|_1} \les \| \gam \|_2^4 \| \theta \|_2^6.
\end{align*}

%Next we study $Z_{3d}$ using Lemma \ref{lem:transfer}. In \cite[Supplement,pg.66]{JinKeLuo21} it is shown that
%\[
%Z_{3d} =
%\sum_{\substack{i,j,k,\ell\\ (dist)}}(\eta_k\eta_j\widetilde{\Omega}_{j\ell})(\eta_i-\teta_i)(\eta_k-\teta_k) = 
%\sum_{\substack{i,j,k,\ell (dist) \\ s \neq i, t \neq k}}(\eta_k\eta_j\widetilde{\Omega}_{j\ell})W_{is} W_{kt}
%\]
%and 
%\[
%Z_{3c} = 
%=\sum_{\substack{i,j,k,\ell\\ (dist)}}(\eta_j^2\widetilde{\Omega}_{k\ell})(\eta_i-\teta_i)(\eta_k-\teta_k)
%= 
%\sum_{\substack{i,j,k,\ell (dist) \\ s \neq i, t \neq k}} (\eta_j^2\widetilde{\Omega}_{k\ell})W_{is} W_{kt}.
%\]
%
%We set 

Next we study $Z_{3d}$, which is defined as
\[
Z_{3d} =
\sum_{\substack{i,j,k,\ell\\ (dist)}}(\eta_k \eta_j \ti \Omega_{j\ell}) (\eta_i-\teta_i)(\eta_k-\teta_k)W_{\ell i} = 
\sum_{\substack{i,k,\ell (dist) \\ s \neq i, t \neq k}} \alpha_{ik\ell} W_{is} W_{kt} W_{\ell i}
\]
where $\alpha_{ik\ell} = \sum_{j \notin \{i, k, \ell\} } \eta_k \eta_j \ti \Omega_{j\ell}$. We see that $\E Z_{3d} = 0$. To study the variance, we use a similar decomposition to that of $Z_{3c}$. Write
\begin{align*}
	Z_{3d} = \frac{1}{v}\sum_{\substack{i,k,\ell (dist)\\t\neq k}}\alpha_{ik\ell}W^2_{i\ell}W_{kt} + \frac{1}{v}\sum_{\substack{i,k,\ell (dist)\\s\notin\{ i,\ell\}, t\neq k}}\alpha_{ik\ell}W_{is}W_{kt}W_{\ell i}\equiv \widetilde{Z}_{3d} + Z^*_{3d}. 
\end{align*}
Mimicking the arguments for $\widetilde{Z}_{3c}$ and $Z^*_{3c}$ we obtain
\begin{align*}
	\var(\ti Z_{3d}) &\les  \frac{ \| \gam \|_2^4 \| \theta \|_2^6}{\| \theta \|_1},
\end{align*}
and
\begin{align*}
	\var(Z^*_{3d}) \les \frac{ \| \gam \|_2^4 \| \theta \|_2^{10}}{\| \theta \|_1^4}.
\end{align*}
Hence
\begin{align*}
	\var(Z_{3d}) \les \frac{ \| \gam \|_2^4 \| \theta \|_2^6}{\| \theta \|_1}.
\end{align*}
Combining the results for $Z_{3a}, \ldots, Z_{3d}$, we have
\begin{align*}
	\E Z_3 = 0, \quad 
	\var(Z_{3})&\les \| \gam \|_2^4 \| \theta \|_2^6. 
\end{align*}

We proceed to study $Z_{4}$. In \cite[Supplement,pg.67]{JinKeLuo21}  the following decomposition is given: 
\begin{align} \label{proof-Z4-decompose}
	Z_4 &= 2\sum_{i, j, k, \ell(dist)}  \eta_i(\eta_j-\teta_j)  \widetilde{\Omega}_{jk} \eta_k (\eta_\ell - \teta_\ell )W_{\ell i}\cr
	&+ \sum_{i, j, k, \ell(dist)} \eta_i(\eta_j-\teta_j)  \widetilde{\Omega}_{jk}(\eta_k - \teta_k )\eta_\ell W_{\ell i}\cr
	& +  \sum_{i, j, k, \ell(dist)} (\eta_i-\teta_i)\eta_j  \widetilde{\Omega}_{jk} \eta_k (\eta_\ell - \teta_\ell )W_{\ell i}\cr
	&\equiv Z_{4a}+Z_{4b}+Z_{4c}. 
\end{align}
There it is shown that $\E Z_{4a} = 0$. To study $\var(Z_{4a})$, we note that $Z_{4a}$ and $Z_{3c}$ have similar structure. In particular we have the decomposition
\begin{align*}
	Z_{4a} = \frac{1}{v}\sum_{\substack{i,k,\ell (dist)\\t\neq k}}\alpha_{ik\ell}W^2_{i\ell}W_{kt} + \frac{1}{v}\sum_{\substack{i,k,\ell (dist)\\s\notin\{ i,\ell\}, t\neq k}}\alpha_{ik\ell}W_{is}W_{kt}W_{\ell i}\equiv \widetilde{Z}_{4a} + Z^*_{4a}. 
\end{align*}
where $\alpha_{ik\ell} = \sum_{j \notin \{i, k, \ell\}} \eta_j \eta_\ell \ti \Omega_{k\ell} $. 
Mimicking the argument for $\ti Z_{3c}$ we have
\begin{align*}
	\var(\ti Z_{4a}) &\les \frac{\| \gam \|_2^2 \| \theta \|_2^2}{\| \theta \|_1^4} \bigg(
	\sum_{ik\ell t} \big[ \beta_k^2 ( \theta_i \theta_k^2 \theta_\ell^2 \theta_t + 
	\theta_i^2 \theta_k^2 \theta_\ell^2 \theta_t) 
	+ \beta_k \beta_t \theta_i \theta_k^2 \theta_\ell^2 \theta_t^2 + \beta_k \beta_t \theta_i^2 \theta_k^2 \theta_\ell^2 \theta_t^2
	\\&\quad + \beta_k^2 \theta_i^2 \theta_k^2 \theta_\ell^2 \theta_t  
	+ \beta_k \beta_t \theta_i \theta_k^2 \theta_\ell^2 \theta_t^2 \big] +
	\sum_{ik\ell t i'\ell'} \big[ \beta_k^2 \theta_i \theta_k^2 \theta_\ell^2 \theta_t \theta_{i'} \theta_{\ell'}^2
	+ \beta_k \beta_t \theta_i \theta_k^2 \theta_\ell^2 \theta_t^2 \theta_{i'} \theta_{\ell'}^2  \big]
	\bigg)
	\\ &\les 
	\frac{\| \gam \|_2^2 \| \theta \|_2^2}{\| \theta \|_1^4} 
	\big(   \| \gam \|_2^2 \| \theta \|_2^2 \| \theta \|_1 + 
	\| \gam \|_2^2 \| \theta \|_2^4 \| \theta \|_1
	+ 	\| \gam \|_2^2 \| \theta \|_2^4 \| \theta \|_1^3
	+ 
	\\&\quad 	\| \gam \|_2^2 \| \theta \|_2^4 \| \theta \|_1^2
	\big) \les \frac{ \| \gam \|_2^4 \| \theta \|_2^6 }{\| \theta \|_1}. 
\end{align*}
For $\ti Z_{4a}^*$ we adapt the decomposition used for $\ti Z_{4c}^*$:
\[
Z_{4a}^*=  \frac{1}{v}\sum_{i,k,\ell (dist)}\alpha_{ik\ell}W_{ik}^2W_{\ell i}+ \frac{1}{v}\sum_{\substack{i,k,\ell (dist)\\s\notin\{ i,\ell\}, t\neq k, (s,t)\neq (k,i)}}\alpha_{ik\ell}W_{is}W_{kt}W_{\ell i}
=: Z_{4a,1}^* + Z_{4a,2}^*
\]
Mimicking the argument for $Z_{3c,1}^*$ and $Z_{3c,2}^*$, we have
\begin{align*}
	\var(Z_{4a,1}^*) &\les \frac{ \| \gam \|_2^2 \| \theta \|_2^2}{\| \theta \|_1^4} \big( \sum_{ik\ell} \beta_k^2 \theta_i^2 \theta_k^2
	\theta_\ell^2 + \sum_{ik\ell k'} \beta_k \beta_{k'}   \theta_i^2 \theta_k^2
	\theta_\ell^2 \theta_{k'}^2 \big) \les
	\frac{ \| \gam \|_2^4 \| \theta \|_2^8}{\| \theta \|_1^4},
\end{align*}
and
\begin{align*}
	\var(Z_{4a,2}^*) &\les \frac{ \| \gam \|_2^2 \| \theta \|_2^2}{\| \theta \|_1^4} \sum_{ik\ell s t} \big[
	\beta_k^2 \theta_i^2 \theta_k^2
	\theta_\ell^2 \theta_s \theta_t 
	+ \beta_k \beta_t \theta_i^2 \theta_k^2
	\theta_\ell^2 \theta_s \theta_t^2
	+ \beta_k \beta_s \theta_i^2 \theta_k^2
	\theta_\ell^2 \theta_s^2 \theta_t^2  
	\big]
	\\&\les \frac{ \| \theta \|_2^4 \| \theta \|_2^6}{\| \theta \|_1^2}. 
\end{align*}
It follows that
\begin{align*}
	\var(Z_{4a}) \les \frac{ \| \gam \|_2^4 \| \theta \|_2^6}{\| \theta \|_1}. 
\end{align*}

Next we study 
\begin{align*}
	Z_{4b} &= \sum_{\substack{i, j, k, \ell\\ (dist)}} \eta_i(\eta_j-\teta_j)  \widetilde{\Omega}_{jk}(\eta_k - \teta_k )\eta_\ell W_{\ell i}
	= \sum_{\substack{i, j, k, \ell\\ (dist)}}\alpha_{ijk\ell} (\eta_j-\teta_j)(\eta_k-\teta_k)W_{\ell i}
	\\ &= \frac{1}{v} \sum_{\substack{i, j, k, \ell (dist)\\ s \neq j, t \neq k }}
	\alpha_{i j k \ell} W_{js} W_{kt} W_{\ell i} 
\end{align*}
where $\alpha_{ijk\ell} = \eta_i\eta_\ell \widetilde{\Omega}_{jk}$. Mimicking the study of $Z_{3a}$, we have the decomposition
\begin{align*}
	Z_{4b} &=\frac{1}{v}\sum_{i, j, k, \ell (dist)} \alpha_{ijk\ell}  W^2_{jk}W_{\ell i} +   \frac{1}{v}\sum_{\substack{ i, j, k, \ell (dist)\\s\neq j,t\neq k, (s,t)\neq (k,j)}} \alpha_{ijk\ell} W_{js}W_{kt}W_{\ell i}\cr
	&\equiv \widetilde{Z}_{4b} + Z^*_{4b}. 
\end{align*}
Further we have, using \eqref{eqn:assn1}, \eqref{eqn:assn2}, \eqref{eqn:tiOm_bd}, and \eqref{eqn:beta_bds}, we have 
\begin{align*}
	\var(\ti Z_{4b}) &\les 
	\frac{1}{\| \theta \|_1^4} 
	\bigg( \sum_{ijk\ell} \big[ \, [\beta_j^2 \beta_k^2 + \beta_j \beta_k \beta_\ell \beta_i  ] \theta_i^3 \theta_j^3 \theta_k^3 \theta_\ell^3
	\big]
	+ \sum_{ijk\ell j' k'}
	\beta_j \beta_k \beta_{j'} \beta_{k'} 
	\theta_i^3 \theta_j^2 \theta_k^2 \theta_\ell^3 \theta_{j'}^2 \theta_{k'}^2 \bigg) 
	\\&\les \frac{1}{\| \theta \|_1^4}
	\big( \| \gam \|_2^4 \| \theta \|_2^4 + \| \gam \|_2^4 \| \theta \|_2^8 \big) \les \frac{ \| \gam \|_2^4 \| \theta \|_2^8 }{ \| \theta \|_1^4 }. 
\end{align*}
Similarly,
\begin{align*}
	\var(Z_{4b}^*) &\les \frac{1}{ \| \theta \|_1^4 }
	\bigg( \sum_{ijk\ell s t} \big[ \beta_j^2 \beta_k^2 
	\theta_i^2 \theta_j^2 \theta_k^2 \theta_\ell^2 \theta_s \theta_t 
	+  \beta_k^2 \beta_\ell \beta_j  
	\theta_i^2 \theta_j^3 \theta_k^3 \theta_\ell^3 \theta_s^2 \theta_t 
	+ \beta_j \beta_k^2 \beta_\ell 
	\theta_i^2 \theta_j^3 \theta_k^3 \theta_\ell^3 \theta_s^2 \theta_t \bigg) 
	\\&\les \frac{ \| \gam \|_2^4 \| \theta\|_2^4}{ \| \theta \|_1^2 }. 
\end{align*}
It follows that 
\begin{align*}
	\var(Z_{4b}) \les \frac{ \| \gam \|_2^4 \| \theta\|_2^4}{ \| \theta \|_1^2 }.
\end{align*}

We study $Z_{4c}$ using the decomposition 
\begin{align*}
	Z_{4c} &
	= \frac{1}{v}\sum_{i, \ell (dist)} \beta_{i\ell}  W^3_{\ell i}  + \frac{2}{v}\sum_{\substack{i, \ell (dist)\\s\notin\{i,\ell\}}} \beta_{i\ell}  W_{is}W^2_{\ell i} + \frac{1}{v}\sum_{\substack{i, \ell (dist)\\s\notin\{ i,\ell\}, t\notin\{\ell, i\}}} \beta_{i\ell}  W_{is}W_{\ell t}W_{\ell i}\cr
	&\equiv \widetilde{Z}_{4c} + Z^*_{4c} + Z^{\dag}_{4c}. 
\end{align*}
from \cite[Supplement, pg.68]{JinKeLuo21}. Only 
\[
\tilde Z_{4c} = \frac{1}{v}\sum_{i, \ell (dist)} \alpha_{i\ell}  W^3_{\ell i} 
\]
has nonzero mean, where $\alpha_{i\ell} = \sum_{j, k (dist)\notin\{i,\ell\}} \eta_j\eta_k\widetilde{\Omega}_{jk}$. By \eqref{eqn:tiOm_bd}
\begin{align*}
	|\alpha_{i\ell}| \les \| \gam \|_2^2 \| \theta \|_2^2.
\end{align*}
Hence
\begin{align*}
	| \E \ti Z_{4c} | \les \frac{1}{\| \theta \|_1^2} \sum_{i \ell } \| \gam \|_2^2 \| \theta \|_2^2 \theta_i \theta_\ell \les \| \gam \|_2^2 \| \theta \|_2^2. 
\end{align*}
Except for when $(i, \ell) = (\ell, i)$, the summands of $\ti Z_{4c}$ are uncorrelated. Thus
\begin{align*}
	\var(\ti Z_{4c}) \les \frac{1}{ \|\theta \|_1^4} \sum_{i \ell} 
	\| \gam \|_2^4 \| \theta \|_2^4 \theta_i \theta_\ell 
	\les \frac{ \| \gam \|_2^4 \| \theta \|_2^4}{\| \theta \|_1^2}.
\end{align*}
Applying the casework from \cite[Supplement, pg.68]{JinKeLuo21},
\begin{align*}
	\var(Z_{4c}^*) &\les \sum_{\substack{i,\ell (dist) \\ s \notin \{ i , \ell \} }} 
	\sum_{\substack{i',\ell' (dist) \\ s' \notin \{ i' , \ell' \} }} 
	|\alpha_{i\ell} \alpha_{i'\ell'} | \cov( W_{is} W_{\ell i}^2, W_{i' s'} W_{\ell' i'}^2 )| 
	\\&\les \frac{1}{ \|\theta \|_1^4} \big( \sum_{i \ell s}  \| \gam \|_2^4 \| \theta \|_2^4 \theta_i^2  \theta_\ell \theta_s
	+ \sum_{i\ell s \ell'} \| \gam \|_2^4 \| \theta \|_2^4 
	\| \theta \|_2^4 \theta_i^3  \theta_\ell \theta_s \theta_{\ell'} \big)
	\\&\les \frac{ \| \gam \|_2^4 \| \theta \|_2^4}{ \| \theta \|_1^4 } 
	\big( \| \theta \|_2^2 \| \theta \|_1^2 + \| \theta \|_2^2 \| \theta \|_1^3 \big)
	\les \frac{ \| \gam \|_2^4 \| \theta \|_2^6 }{ \| \theta \|_1 }. 
\end{align*}
Next, in \cite[Supplement, pg.69]{JinKeLuo21} it is shown that 
\begin{align*}
	\mathrm{Var}(Z^{\dag}_{4c}) &\les \frac{1}{v^2}\sum_{\substack{i, \ell (dist)\\s\notin\{ i,\ell\}, t\notin\{\ell, i\}}} \alpha^2_{i\ell}\cdot\mathrm{Var}( W_{is}W_{\ell t}W_{\ell i})
\end{align*}
Thus
\begin{align*}
	\mathrm{Var}(Z^{\dag}_{4c}) &\les
	\sum_{i \ell s } \| \gam \|_2^4 \| \theta \|_2^4 \theta_i^2  \theta_\ell^2 \theta_s \theta_t 
	\les \frac{ \| \gam \|_2^4 \| \theta \|_2^8 }{ \| \theta \|_1^2}. 
\end{align*}
Combining the results for $\ti Z_{4c}, Z_{4c}^*, Z_{4c}^\dag$, we have
\begin{align*}
	| \E Z_{4c} | \les \| \gam \|_2^2 \| \theta \|_2^2,
	\quad \var(Z_{4c}) \les \frac{ \| \gam \|_2^4 \| \theta \|_2^8}{\|\theta \|_1}. 
\end{align*}
Combining the results for $Z_{4a}, Z_{4b},$ and $Z_{4c}$, we have
\begin{align*}
	| \E Z_{4} | \les \| \gam \|_2^2 \| \theta \|_2^2,
	\quad \var(Z_4) \les   \frac{ \| \gam \|_2^4 \| \theta \|_2^6}{\| \theta \|_1}
\end{align*}

To study $Z_5$, we use the decomposition
\begin{align} \label{proof-Z5-decompose}
	Z_5 &= 2\sum_{i, j, k, \ell (dist)} \eta_i(\eta_j-\teta_j)\eta_j(\eta_k-\teta_k) \widetilde{\Omega}_{k \ell}  \widetilde{\Omega}_{\ell i} + \sum_{i, j, k, \ell (dist)} \eta_i(\eta_j-\teta_j)^2 \eta_k \widetilde{\Omega}_{k \ell}  \widetilde{\Omega}_{\ell i}\cr
	&\qquad  +\sum_{i, j, k, \ell (dist)} (\eta_i-\teta_i)\eta^2_j(\eta_k-\teta_k) \widetilde{\Omega}_{k \ell}  \widetilde{\Omega}_{\ell i} \cr
	&\equiv Z_{5a} + Z_{5b} + Z_{5c}. 
\end{align}
from \cite[Supplement, pg. 70]{JinKeLuo21}. We further decompose $Z_{5a}$ as in \cite[Supplement, pg.70]{JinKeLuo21}: 
\[
Z_{5a} =  \frac{2}{v}\sum_{j, k (dist)}\alpha_{jk}W_{jk}^2  + \frac{2}{v}\sum_{\substack{j, k (dist)\\s\neq j, t\neq k,\\(s,t)\neq (k,j)}}\alpha_{jk}W_{js}W_{kt} \equiv \widetilde{Z}_{5a} + Z^*_{5a}. 
\]
where $\alpha_{jk} = \sum_{i,\ell (dist)\notin\{j,k\}}\eta_i\eta_j\widetilde{\Omega}_{k \ell}  \widetilde{\Omega}_{\ell i}$. Note that by \eqref{eqn:tiOm_bd} and \eqref{eqn:beta_bds},
\begin{align*}
	| \alpha_{jk} | \les \sum_{i\ell} (\beta_k \theta_k) (\beta_\ell \theta_\ell)^2
	(\beta_i \theta_i) \les \theta_j (\beta_k \theta_k) \| \gam \|_2^3 \| \theta \|_2. 
\end{align*}
Only $\tilde Z_{5a}$ has nonzero mean. By \eqref{eqn:assn1} and \eqref{eqn:assn2}, 
\begin{align*}
	|\E  Z_{5a} | = 	|\E \ti Z_{5a} | \les \frac{1}{ \| \theta \|_1^2} 
	\sum_{jk} \theta_j (\beta_k \theta_k) \| \gam \|_2^3 \| \theta \|_2 \cdot \theta_j \theta_k \les \frac{ \| \gam \|_2^4 \| \theta \|_2^4 }{ \| \theta \|_1^2}. 
\end{align*}
Now we study the variance of $Z_{5a}$. In \cite[Supplement, pg.70]{JinKeLuo21} it is shown that
\begin{align*}
	\mathrm{Var}(\widetilde{Z}_{5a}) &\les \frac{1}{v^2}\sum_{j,k (dist)}\alpha_{jk}^2\,\mathrm{Var}(W^2_{jk}) \cr
	\mathrm{Var}(Z^*_{5a})&\les \frac{1}{v^2}\sum_{\substack{j, k (dist)\\s\neq j, t\neq k,\\(s,t)\neq (k,j)}} \alpha^2_{jk}\, \mathrm{Var}(W_{js}W_{kt}). 
\end{align*}
Thus by \eqref{eqn:assn2} and \eqref{eqn:beta_bds}, 
\begin{align*}
	\mathrm{Var}(\widetilde{Z}_{5a}) &\les \frac{\| \gam \|_2^6 \| \theta \|_2^4 }{ \| \theta \|_1^4} \big(
	\sum_{jk} \theta_j^3 \beta_k^2  \theta_k^3 
	\big) \les \frac{ \| \gam \|_2^8 \| \theta \|_2^6 }{ \| \theta \|_1^4} 
	\\ 
	\mathrm{Var}(Z^*_{5a})&\les 
	\frac{\| \gam \|_2^6 \| \theta \|_2^4 }{ \| \theta \|_1^4} \big(
	\sum_{jk} \theta_j^2  \beta_k^2 \theta_k^2 \cdot \theta_j \theta_s \theta_k \theta_t  
	\big) \les \frac{ \| \gam \|_2^8 \| \theta \|_2^6 }{ \| \theta \|_1^2 }. 
\end{align*}
We conclude that 
\begin{align*}
	\var(Z_{5a}) \les \frac{ \| \gam \|_2^8 \| \theta \|_2^6 }{ \| \theta \|_1^2}. 
\end{align*}

Next we study $Z_{5b}$ using the decomposition
\[
Z_{5b} = \frac{1}{v}\sum_{j, s (dist)} \alpha_j W^2_{js} + \frac{1}{v}\sum_{\substack{j\\ s, t (dist) \notin\{j\}}} \alpha_j W_{js}W_{jt}\equiv \widetilde{Z}_{5b} + Z^*_{5b}. 
\]
from \cite[Supplement, pg.71]{JinKeLuo21}, where $\alpha_j = \sum_{i,k,\ell (dist)\notin\{j\}}\eta_i\eta_k \widetilde{\Omega}_{k \ell}  \widetilde{\Omega}_{\ell i}$. Note that by \eqref{eqn:assn2} and \eqref{eqn:tiOm_bd},
\begin{align*}
	|\alpha_j| \les \sum_{ik\ell} \theta_i \theta_k (\beta_k \theta_k) (\beta_\ell \theta_\ell)^2 (\beta_i \theta_i) 
	\les \| \gam \|_2^4 \| \theta \|_2^2. 
\end{align*}
Only $\widetilde{Z}_{5b}$ above has nonzero mean, and we have 
\begin{align*}
	| \E Z_{5b} | =	| \E Z_{5b}| \les 
	\frac{ \| \gam \|_2^4  \| \theta \|_2^2}{ \| \theta \|_1^2} \sum_{j,s} \theta_j \theta_s  \les \| \gam \|_2^4 \| \theta \|_2^2. 
\end{align*}
Similarly for the variances,
\begin{align*}
	\var( \ti Z_{5b} ) &\les 	\frac{ \| \gam \|_2^8  \| \theta \|_2^4}{ \| \theta \|_1^4} \sum_{js} \theta_j \theta_s \les \frac{ \| \gam \|_2^8 \| \theta \|_2^4}{\| \theta \|_1^2} \\
	\var(  Z_{5b}^* ) &\les 	\frac{ \| \gam \|_2^8  \| \theta \|_2^4}{ \| \theta \|_1^4} \sum_{jst} \theta_j^2 \theta_s \theta_t
	\les \frac{ \| \gam \|_2^8 \| \theta \|_2^6}{\| \theta \|_1^2}, 
\end{align*}
and it follows that
\begin{align*}
	\var(Z_{5b}) \les \frac{ \| \gam \|_2^8 \| \theta \|_2^6}{\| \theta \|_1^2}. 
\end{align*}

Next we study 
\begin{align*}
	Z_{5c} &= \sum_{\substack{i, j, k, \ell\\ (dist)}} (\eta_j-\teta_j)\eta^2_i(\eta_k-\teta_k) \widetilde{\Omega}_{k \ell}  \widetilde{\Omega}_{\ell j} = \sum_{\substack{i, j, k, \ell\\ (dist)}}(\eta_i^2\widetilde{\Omega}_{k\ell}\widetilde{\Omega}_{\ell j})(\eta_j-\teta_j)(\eta_k-\teta_k)
	\\&= \frac{1}{v} \sum_{\substack{i, j, k, \ell (dist)\\ s \neq j, t \neq k }}(\eta_i^2\widetilde{\Omega}_{k\ell}\widetilde{\Omega}_{\ell j})W_{js} W_{kt}
	=  \frac{1}{v} \sum_{\substack{j, k (dist)\\ s \neq j, t \neq k }}\alpha_{jk} W_{js} W_{kt}
\end{align*}
where $\alpha_{jk} =\sum_{\substack{i, \ell (dist) \\ i, \ell \notin\{ j, k \}}}\eta_i^2\widetilde{\Omega}_{k\ell}\widetilde{\Omega}_{\ell j}$. Note that by \eqref{eqn:tiOm_bd} and  \eqref{eqn:v0_v1_bd} ,
\begin{align*}
	|\alpha_{jk}| \les \sum_{i\ell} \theta_i^2 (\beta_k \theta_k) (\beta_\ell \theta_\ell)^2 (\beta_j \theta_j) \les (\beta_j \theta_j) (\beta_k \theta_k) \| \theta \|_2^2 \| \gam \|_2^2 . 
	\num \label{eqn:Z5c_alpha_bd}
\end{align*}
We further decompose 
\begin{align*}
	Z_{5c}
	&= \frac{1}{v} \sum_{\substack{j, k \\ (dist) }}\alpha_{jk} W_{jk}^2 
	+ \frac{1}{v} \sum_{\substack{j, k (dist)\\ s,t \notin \{j, k \}  }}\alpha_{jk} W_{js} W_{kt}
	\equiv \ti Z_{5c} + Z_{5c}^*.
\end{align*}
Only the first term has nonzero mean. It follows that
\begin{align*}
	| \E Z_{5c} | = | \E \ti Z_{5c} |&\les \frac{\| \theta \|_2^2 \| \gam \|_2^2}{\| \theta \|_1^2} \sum_{j,k, s, t} (\beta_j \theta_j) (\beta_k \theta_k)   \cdot \theta_j  \theta_k 
	\les \frac{ \| \gam \|_2^4 \| \theta \|_2^4}{\| \theta \|_1^2}. 
\end{align*}
Note that $Z_{5c}$ and $Z_{5a}$ have the same form, but with a different setting of the coefficient $\alpha_{jk}$. Mimicking the variance bounds for $Z_{5a}$ we obtain the bound
\begin{align*}
	\var(Z_{5c}) \les \frac{ \| \gam \|_2^8 \| \theta \|_2^4}{\|\theta\|_1^2}. 
\end{align*}
Combining the previous bounds we obtain 
\begin{align*}
	| \E Z_5 | \les \| \gam \|_2^4 \| \theta \|_2^2, 
	\quad \var(Z_5) \les \frac{ \| \gam \|_2^8 \| \theta \|_2^6}{ \| \theta \|_1^2}. 
\end{align*}

Next we study $Z_6 = Z_{6a} + Z_{6b}$ as defined in \cite[Supplement, pg.72]{JinKeLuo21}, where 
\begin{align*}
	Z_{6a} &= \sum_{\substack{i, j, k, \ell\\ (dist)}}(\eta_i\eta_\ell \widetilde{\Omega}_{j\ell}\widetilde{\Omega}_{ki})(\eta_j-\teta_j)(\eta_k-\teta_k) = \frac{1}{v} \sum_{\substack{j, k (dist)\\ s \neq j, t \neq k }}\alpha_{jk}\rp{6a} W_{js} W_{kt} \\
	\\Z_{6b} &= 2\sum_{\substack{i, j, k, \ell\\ (dist)}} (\eta_i\eta_\ell \widetilde{\Omega}_{jk}\widetilde{\Omega}_{\ell i})(\eta_j-\teta_j)(\eta_k-\teta_k)
	= \frac{1}{v} \sum_{\substack{j, k (dist)\\ s \neq j, t \neq k }}\alpha_{jk}\rp{6b} W_{js} W_{kt}
\end{align*}
and
\begin{align*}
	\alpha_{jk}\rp{6a} &= \sum_{\substack{i, \ell (dist) \\ i, \ell \notin\{ j, k \}}} \eta_i \eta_\ell \ti \Omega_{jk} \ti \Omega_{\ell i} \\
	\alpha_{jk}\rp{6b} &= \sum_{\substack{i, \ell (dist) \\ i, \ell \notin\{ j, k \}}} \eta_i \eta_\ell \ti \Omega_{j\ell} \ti \Omega_{k i}.
\end{align*}
Thus $Z_{6a}$ and $Z_{6b}$ take the same form as $Z_{5c}$, but with a different setting of $\alpha_{jk}$. Note that by \eqref{eqn:beta_bds} and similar arguments from before,
\begin{align*}
	\max( |\alpha_{jk}\rp{6a}|, | \alpha_{jk}\rp{6b} | ) 
	\les (\beta_j \theta_j) (\beta_k \theta_k) \| \theta \|_2^2 \| \gam \|_2^2,
\end{align*}
which is the same as the upper bound on $|\alpha_{jk}|$ associated to $Z_{5c}$ given in \eqref{eqn:Z5c_alpha_bd}. It follows that
\begin{align*}
	| \E Z_{6} | \les \frac{ \| \gam \|_2^4 \| \theta \|_2^4}{\| \theta \|_1^2},
	\quad \var(Z_6) \les \frac{ \| \gam \|_2^8 \| \theta \|_2^4}{\|\theta\|_1^2}. 
\end{align*}
We have proved all claims in Lemma \ref{lem:Ub}. \qed 

\subsubsection{Proof of Lemma \ref{lem:Uc}}

The terms $T_1$ and $F$ do not depend on $\ti \Omega$, and thus the claimed bounds transfer directly from \cite[Lemma G.9]{JinKeLuo21}. Thus we focus on $T_2$. We use the decomposition $T_2 = 2(T_{2a} + T_{2b} + T_{2c} + T_{2d})$ from \cite[Supplement, pg.73]{JinKeLuo21} 
where
\begin{align*}
	T_{2a} &=  \sum_{i_1, i_2, i_3, i_4(dist)}\eta_{i_2}\eta_{i_3}\eta_{i_4}\big[(\eta_{i_1} - \tilde{\eta}_{i_1}) (\eta_{i_2} - \tilde{\eta}_{i_2})  (\eta_{i_3} - \tilde{\eta}_{i_3})    \big]\cdot  \widetilde{\Omega}_{i_4i_1},\cr
	T_{2b} &=  \sum_{i_1, i_2, i_3, i_4(dist)}\eta_{i_2}\eta_{i_3}^2\big[(\eta_{i_1} - \tilde{\eta}_{i_1}) (\eta_{i_2} - \tilde{\eta}_{i_2}) (\eta_{i_4} - \tilde{\eta}_{i_4})     \big]\cdot  \widetilde{\Omega}_{i_4i_1}, \cr
	T_{2c} &= \sum_{i_1, i_2, i_3, i_4(dist)}\eta_{i_1}\eta_{i_3}\eta_{i_4}\big[ (\eta_{i_2} - \tilde{\eta}_{i_2})^2  (\eta_{i_3} - \tilde{\eta}_{i_3})    \big]\cdot  \widetilde{\Omega}_{i_4i_1},\cr
	T_{2d} &= \sum_{i_1, i_2, i_3, i_4(dist)}\eta_{i_1}\eta_{i_3}^2\big[(\eta_{i_2} - \tilde{\eta}_{i_2})^2 (\eta_{i_4} - \tilde{\eta}_{i_4})     \big]\cdot  \widetilde{\Omega}_{i_4i_1}. 
\end{align*}
We study each term separately. 

For $T_{2a}$, in  \cite[Supplement, pg.89]{JinKeLuo21}, we have the decomposition $T_{2a} = X_{a1} + X_{a2} + X_{a3} + X_b$ where
\begin{align*}
	X_{a1} & = -\frac{1}{v^{3/2}}\sum_{i_1, i_2, i_3, i_4 (dist)}\sum_{j_3\neq i_3}\eta_{i_2}\eta_{i_3}\eta_{i_4} W^2_{i_1i_2}W_{i_3j_3}\widetilde{\Omega}_{i_1i_4}, \cr
	X_{a2} & = -\frac{1}{v^{3/2}}\sum_{i_1, i_2, i_3, i_4 (dist)}\sum_{j_2\neq i_2}\eta_{i_2}\eta_{i_3}\eta_{i_4} W^2_{i_1i_3}W_{i_2j_2}\widetilde{\Omega}_{i_1i_4}, \cr
	X_{a3} & = -\frac{1}{v^{3/2}}\sum_{i_1, i_2, i_3, i_4 (dist)}\sum_{j_1\neq i_1}\eta_{i_2}\eta_{i_3}\eta_{i_4} W^2_{i_2i_3}W_{i_1j_1}\widetilde{\Omega}_{i_1i_4}, \cr
	X_b &=-\frac{1}{v^{3/2}}\sum_{i_1, i_2, i_3, i_4 (dist)}
	\sum_{\substack{j_1, j_2, j_3 \\ j_k\neq i_\ell, k, \ell = 1, 2, 3}} 
	\eta_{i_2}\eta_{i_3}\eta_{i_4} W_{i_1j_1}W_{i_2j_2}W_{i_3j_3}\widetilde{\Omega}_{i_1i_4}. 
\end{align*}
There it is shown that $\E T_{2a} = 0$. Further it is argued that
\begin{align*}
	\var( X_{a1} ) &= \E X_{a1}^2
	\\&= \frac{1}{v^3} \sum_{\substack{i_1,i_2,i_3,i_4 (dist) \\ i_1',i_2',i_3',i_4'(dist)}}    \sum_{\substack{j_3, j_3'  \\  j_3 \neq i_3, j_3' \neq i_3'}} 
	\eta_{i_2}\eta_{i_3}\eta_{i_4}\eta_{i_2'}\eta_{i_3'}\eta_{i_4'}\mathbb{E}[ W^2_{i_1i_2} W_{i_3j_3} W^2_{i_1'i_2'} W_{i_3'j_3'}]\widetilde{\Omega}_{i_1i_4}\widetilde{\Omega}_{i_1'i_4'} \num \label{eqn:ABC_casework}
	\\&\equiv V_A + V_B + V_C, 
\end{align*}
where the terms $V_A, V_B, V_C$ correspond to the contributions from cases $A, B, C$, respectively, described in \cite[Supplement, pg.89]{JinKeLuo21}. Concretely, the nonzero terms of  \eqref{eqn:ABC_casework} fall into three cases:
\begin{itemize}
	\item[Case A.] $\{i_1, i_2\} = \{i_3',j_3'\}$ and
	$\{i_3, j_3 \} = \{i_1', i_2'\}$
	\item[Case B.] $\{i_3, j_3\} = \{i_3',j_3'\}$ 
	and $\{i_1, i_2\} = \{i_1',i_2'\}$ 
	\item[Case C.]  $\{i_3, j_3\} = \{i_3',j_3'\}$ 
	and $\{ i_1, i_2\} \neq \{i_1', i_2'\}$.
\end{itemize}
Here $V_A, V_B,$ and $V_C$ are defined to be the contributions from each case. 

%\begin{itemize} 
%	\item Case A. $W_{i_1 i_2} = W_{i_3'j_3'}$ and $W_{i_3j_3} = W_{i_1'i_2'}$. 
%	\item Case B. $ W_{i_3j_3} = W_{i_3'j_3'}$ and $W_{i_1i_2} = W_{i_1' i_2'}$. 
%	\item Case C. $W_{i_3 j_3} = W_{i_3'j_3'}$ and $W_{i_1i_2} \neq W_{i_1' i_2'}$. 
%\end{itemize} 

Applying \eqref{eqn:assn2}, \eqref{eqn:etai_bd}, and \eqref{eqn:tiOm_bd},
\begin{align*}
	|\eta_{i_2}\eta_{i_3}\eta_{i_4}\eta_{i_2'}\eta_{i_3'}\eta_{i_4'} \widetilde{\Omega}_{i_1i_4} \widetilde{\Omega}_{i_1'i_4'}| 
	&\les \theta_{i_2} \theta_{i_3} \theta_{i_4}
	\theta_{i_2'} \theta_{i_3'} \theta_{i_4'}
	(\beta_{i_1} \theta_{i_1}) 
	(\beta_{i_4} \theta_{i_4})
	(\beta_{i_1'} \theta_{i_1'})  
	(\beta_{i_4'} \theta_{i_4'}) 
	\\ &\les 
	\theta_{i_2} \theta_{i_3} \theta_{i_4}
	\theta_{i_2'} \theta_{i_3'} \theta_{i_4'}
	(\beta_{i_1} \theta_{i_1}) 
	(\beta_{i_4} \theta_{i_4})
	\theta_{i_1'}
	(\beta_{i_4'} \theta_{i_4'}) 
	\num \label{eqn:caseABC_bd}. 
\end{align*}
Note that using the last inequality reduces the required casework while still yielding a good enough bound. Mimicking the casework in Case A of \cite[Supplement, pg.90]{JinKeLuo21} and applying \eqref{eqn:beta_bds}, we have 
\begin{align*}
	V_A &\les \frac{1}{ \| \theta \|_1^6} 
	\sum_{ \substack{i_1, i_2, i_3 \\ i_4 ,i_4', j_3} } \, 
	\sum_{\substack{b_1, b_2 \\ (b_1 + b_2 = 1)}}
	\beta_{i_1} \beta_{i_4} \beta_{i_4'} \theta_{i_1}^{2 + b_1} 
	\theta_{i_2}^{2 + b_2} \theta_{i_3}^{3} \theta_{j_3}^{2} 
	\theta_{i_4}^{2} 		\theta_{i_4'}^{2} 
	\\&\les \frac{1}{ \| \theta \|_1^6}  \big( \| \gam\|_2^3 \| \theta\|_2^3 \| \theta \|_2^4 \| \theta \|_3^3 + \| \gam \|_2^3 \| \theta \|_2^3 \| \theta \|_2^2 \| \theta \|_3^6 \big) \les  \frac{ \| \gam \|_2^3 \| \theta \|_2^9}{ \| \theta \|_1^6}. 
\end{align*}
Similarly, applying \eqref{eqn:caseABC_bd} along with \eqref{eqn:etai_bd}, \eqref{eqn:tiOm_bd}, and \eqref{eqn:beta_bds} yields
\begin{align*}
	V_B &\les \frac{1}{ \| \theta \|_1^6} 
	\sum_{ \substack{i_1, i_2, i_3 \\ i_4 ,i_4', j_3} } \, 
	\sum_{\substack{c_1, c_2 \\ (c_1 + c_2 = 1)}}
	\beta_{i_1} \beta_{i_4} \beta_{i_4'} \theta_{i_1}^3 
	\theta_{i_2}^3 \theta_{i_3}^{2+c_1} \theta_{j_3}^{1+c_2}
	\theta_{i_2}^2 \theta_{i_4'}^2
	\les \frac{ \| \gam \|_2^3 \| \theta \|_2^7 }{ \| \theta \|_1^5 }. 
\end{align*}
and
\begin{align*}
	V_C &\les 
	\sum_{ \substack{i_1, i_2, i_3,i_4 \\ i_1' , i_2', i_4', j_3} } \, 
	\sum_{\substack{c_1, c_2 \\ (c_1 + c_2 = 1)}}
	\beta_{i_1} \beta_{i_4} \beta_{i_1'} \beta_{i_4'}   
	\theta_{i_1}^2 \theta_{i_2}^2 \theta_{i_3}^{2+c_1} 
	\theta_{j_3}^{1+c_2} \theta_{i_4}^2 \theta_{i_1'}^2
	\theta_{i_2'}^2 \theta_{i_4'}^2
	\les \frac{ \| \gam \|_2^4 \| \theta \|_2^{10}}{ \| \theta \|_1^5}. 
\end{align*}
Thus
\begin{align*}
	\var(X_{a1}) \les \| \gam \|_2^4. 
\end{align*}
The arguments for $X_{a2}$ and $X_{a3}$ are similar, and the corresponding $V_A, V_B, V_C$ satisfy the same inequalities above.  We simply state the bounds:
\begin{align*}
	\E X_{a_2} &= \E X_{a3} = 0,
	\quad \var(X_{a2}) \les\| \gam \|_2^4, \quad  \var(X_{a3}) \les \| \gam \|_2^4. 
\end{align*}

Next we consider $X_b$ as defined in \cite[Supplement, pg.89]{JinKeLuo21}. We have $\E X_b = 0$ and focus on the variance. In \cite[Supplement, pg.91]{JinKeLuo21} it is shown that 
\begin{align*} 
	\var(X_b) &= \mathbb{E}[X_{b}^2]  
	\\& = v^{-3} \sum_{\substack{i_1,i_2,i_3,i_4 (dist) \\ i_1',i_2',i_3',i_4'(dist)}}   
	\sum_{\substack{j_3, j_3'  \\  j_3 \neq i_3, j_3' \neq i_3'}} 
	\eta_{i_2}\eta_{i_3}\eta_{i_4}\eta_{i_2'}\eta_{i_3'}\eta_{i_4'} 
	\mathbb{E}[ W_{i_1j_1} W_{i_2j_2}W_{i_3j_3}  W_{i_1'j_1'} W_{i_2' j_2'}  W_{i_3'j_3'}] \widetilde{\Omega}_{i_1i_4}\widetilde{\Omega}_{i_1'i_4'}, 
\end{align*} 
Note that
\begin{align*}
	\mathbb{E}[ W_{i_1j_1} W_{i_2j_2}W_{i_3j_3}  W_{i_1'j_1'} W_{i_2' j_2'}  W_{i_3'j_3'}] \neq 0
\end{align*}
if and only if the two sets of random variables $\{ W_{i_1j_1}, W_{i_2j_2}, W_{i_3j_3}\}$ and $\{ W_{i_1'j_1'}, W_{i_2'j_2'}, W_{i_3'j_3'}\}$ are identical. Applying \eqref{eqn:etai_bd} and \eqref{eqn:tiOm_bd},
\begin{align*}
	|\eta_{i_2}\eta_{i_3}\eta_{i_4}\eta_{i_2'}\eta_{i_3'}\eta_{i_4'} 	\widetilde{\Omega}_{i_1i_4}\widetilde{\Omega}_{i_1'i_4'} |
	&\les \theta_{i_2}\theta_{i_3}\theta_{i_4}\theta_{i_2'}\theta_{i_3'}\theta_{i_4'} (\beta_{i_1} \theta_{i_1}) 
	(\beta_{i_4} \theta_{i_4})
	\theta_{i_1'}
	(\beta_{i_4'} \theta_{i_4'}) 
	\\&\les \beta_{i_1} \beta_{i_4} \beta_{i_4'} 
	\theta_{i_1}^{1 + a_1}   \theta_{j_1}^{a_2}   \theta_{i_2}^{1 + a_3} \theta_{j_2}^{a_4}  \theta_{i_3}^{1 + a_5} \theta_{j_3}^{a_6}  \theta_{i_4}^2 \theta_{i_4'}^2
\end{align*}
if $\E [ W_{i_1j_1} W_{i_2j_2}W_{i_3j_3}  W_{i_1'j_1'} W_{i_2' j_2'}  W_{i_3'j_3'}] \neq 0$, where $a_i \in \{0,1\}$ and $\sum_{i =1}^6 a_i = 3$. Thus by \eqref{eqn:assn1}, \eqref{eqn:assn2}, and \eqref{eqn:beta_bds},
\begin{align*}
	\var(X_b) &\les \max_a \frac{1}{ \| \theta \|_1^6}
	\sum_{\substack{i_1, i_2, i_3, i_4 \\ i_4', j_1, j_2, j_3} }
	\beta_{i_1} \beta_{i_4} \beta_{i_4'} 
	\theta_{i_1}^{2 + a_1}    \theta_{j_1}^{1+a_2}  \theta_{i_2}^{2 + a_3}  \theta_{j_2}^{1+a_4}  \theta_{i_3}^{2 + a_5} \theta_{j_3}^{1+ a_6}  \theta_{i_4}^2 \theta_{i_4'}^2
	\\&\les \frac{1}{ \| \theta \|_1^6}
	\sum_{\substack{i_1, i_2, i_3, i_4 \\ i_4', j_1, j_2, j_3} }
	\beta_{i_1} \beta_{i_4} \beta_{i_4'} 
	\theta_{i_1}^{2}    \theta_{j_1}^{1}  \theta_{i_2}^{2}  \theta_{j_2}^{1}  \theta_{i_3}^{2} \theta_{j_3}^{1}  \theta_{i_4}^2 \theta_{i_4'}^2
	\\&\les \frac{ \| \gam \|_2^3 \| \theta \|_2^3 \| \theta \|_2^4 \| \theta \|_1^3 }{ \| \theta \|_1^6} \les \frac{ \| \gam \|_2^3 \| \theta \|_2^7}{ \| \theta \|_1^3} \les \| \gam \|_2^3 \| \theta \|_2. 
\end{align*}
Combining the results for $X_{a1}, X_{a2}, X_{a3}$ and $X_{b}$, we conclude that
\begin{align*}
	\E T_{2a} = 0, \qquad \var(T_{2a}) \les \| \gam \|_2^4 \| \theta \|_2. 
\end{align*}

The argument for $T_{2b}$ is similar to the one for $T_{2a}$, so we simply state the results:
\begin{align*}
	\E T_{2b}  = 0
	, \qquad \var(T_{2b}) \les \| \gam \|_2^4 \| \theta \|_2. 
\end{align*}

Next we study $T_{2c}$, providing full details for completeness. Using the definition of $T_{2c}$ in \cite[Supplement, pg.92]{JinKeLuo21}, we have the following decomposition by careful casework.
\begin{align*}
	Y_{a} & =  -\frac{1}{v^{3/2}}\sum_{i_1, i_2, i_3, i_4 (dist)} \eta_{i_1}\eta_{i_3}\eta_{i_4} W^3_{i_2i_3}\widetilde{\Omega}_{i_1i_4},  \cr
	Y_{b1} & = -\frac{1}{v^{3/2}}
	\sum_{i_1, i_2, i_3, i_4 (dist)}
	\sum_{\substack{(i_2,j_2) \neq(j_3,  i_3) \\ j_2 \neq i_2, j_3 \neq i_3}}
	\eta_{i_1}\eta_{i_3}\eta_{i_4} W^2_{i_2j_2}W_{i_3j_3}\widetilde{\Omega}_{i_1i_4}, \cr
	Y_{b2} & = -\frac{1}{v^{3/2}}
	\sum_{i_1, i_2, i_3, i_4(dist)}
	\sum_{\ell_2 \notin \{ i_3, i_2 \}}
	\eta_{i_1}\eta_{i_3}\eta_{i_4} W^2_{i_2i_3}W_{i_2\ell_2}\widetilde{\Omega}_{i_1i_4}, \cr
	Y_{b3} & = -\frac{1}{v^{3/2}}
	\sum_{i_1, i_2, i_3, i_4 (dist)}
	\sum_{j_2 \notin \{ i_3, i_2 \}}
	\eta_{i_1}\eta_{i_3}\eta_{i_4} W^2_{i_2i_3}W_{i_2j_2}\widetilde{\Omega}_{i_1i_4}, \cr
	Y_c &=-\frac{1}{v^{3/2}} \sum_{i_1, i_2, i_3, i_4 (dist)}
	\sum_{\substack{j_2, \ell_2, j_3 \\ j_2\neq i_2, \ell_2\neq i_2, j_3\neq i_3\\  
			j_2\neq \ell_2, (i_2, j_2) \neq (j_3, i_3), (i_2, \ell_2) \neq ( j_3, i_3 ) }} 
	\eta_{i_1}\eta_{i_3}\eta_{i_4} W_{i_2j_2}W_{i_2\ell_2}W_{i_3j_3} \widetilde{\Omega}_{i_1i_4}. 
\end{align*}
Note that, by the change of variables $\ell_2 \to j_2$, it holds that $Y_{b2} = Y_{b3}$. 

The only term with nonzero mean is $Y_a$. We have by  \eqref{eqn:v0_v1_bd}, \eqref{eqn:tiOm_bd}, \eqref{eqn:etai_bd}, and \eqref{eqn:beta_bds} that
\begin{align*}
	| \E Y_a | &\les \frac{1}{ \| \theta \|_1^3} \sum_{i_1, i_2, i_3, i_4}
	\theta_{i_1} \theta_{i_3} \theta_{i_4} (\beta_{i_1} \theta_{i_1})
	(\beta_{i_4} \theta_{i_4}) \cdot |\E W_{i_2 i_3}^3 |
	\les \frac{1}{\| \theta \|_1^3} \sum_{i_1, i_2, i_3, i_4}
	\beta_{i_1} \beta_{i_4} \theta_{i_1}^2 \theta_{i_2} \theta_{i_3}^2 \theta_{i_4}^2 
	\\&\les \frac{ \| \gam \|_2^2 \| \theta \|_2^4 }{\| \theta \|_1^2}. 
\end{align*}
For the variance, by independence of $\{W_{ij}\}_{i>j }$, \eqref{eqn:assn2}, \eqref{eqn:tiOm_bd}, and \eqref{eqn:beta_bds}, we have
\begin{align*}
	\var(Y_a)
	&\les \frac{1}{\|\theta \|_1^6} 
	\sum_{i_2, i_3} \big( \sum_{i_1, i_4} \theta_{i_1} \theta_{i_3} \theta_{i_4} (\beta_{i_1} \theta_{i_1})
	(\beta_{i_4} \theta_{i_4}) \big)^2 \theta_{i_2} \theta_{i_3}
	\les \frac{1}{\|\theta \|_1^6} 
	\sum_{i_2, i_3} \| \gam \|_2^4 \| \theta \|_2^4 \theta_{i_2} \theta_{i_3}^2 
	\\&\les \frac{ \| \gam \|_2^4 \| \theta \|_2^6 }{ \| \theta \|_1^5}. 
\end{align*}

For $Y_{b1}, Y_{b2}, Y_{b3}$ we make note of the identity
\begin{align*}
	W_{ij}^2 &= (1 - 2 \Omega_{ij}) W_{ij} + \Omega_{ij}(1 -\Omega_{ij}) \equiv A_{ij} W_{ij} + B_{ij}. 
	\num \label{eqn:bernoulli_identity}
\end{align*}
Write
\begin{align*}
	Y_{b1} &= 
	-\frac{1}{v^{3/2}}
	\sum_{i_1, i_2, i_3, i_4 (dist)}
	\sum_{\substack{(i_2,j_2) \neq(j_3,  i_3) \\ j_2 \neq i_2, j_3 \neq i_3}}
	\eta_{i_1}\eta_{i_3}\eta_{i_4} A_{i_2j_2} W_{i_2j_2}W_{i_3j_3}\widetilde{\Omega}_{i_1i_4}
	\\&\quad -\frac{1}{v^{3/2}}
	\sum_{i_1, i_2, i_3, i_4 (dist)}
	\sum_{\substack{(i_2,j_2) \neq(j_3,  i_3) \\ j_2 \neq i_2, j_3 \neq i_3}}
	\eta_{i_1}\eta_{i_3}\eta_{i_4} B_{i_2j_2} W_{i_3j_3}\widetilde{\Omega}_{i_1i_4}
	\equiv Y_{b1,A} + Y_{b1,B}. 
\end{align*}
By similar arguments from before, and noting that $|A_{i_2, j_2}| \les 1$, 
\begin{align*}
	\var(Y_{b1,A}) &\les
	\frac{1}{\| \theta \|_1^6} 
	\sum_{\substack{(i_2,j_2) \neq(j_3,  i_3) \\ j_2 \neq i_2, j_3 \neq i_3}} 
	\bigg( \sum_{i_1, i_4} \eta_{i_1} \eta_{i_3} \eta_{i_4} (\beta_{i_1} \theta_{i_1}) (\beta_{i_4} \theta_{i_4}) \bigg)^2 | \E W_{i_2j_2}W_{i_3j_3} |
	\\&\les \frac{1}{\| \theta \|_1^6} 
	\sum_{i_2, j_2, i_3, j_3} 
	\bigg( \sum_{i_1, i_4} \eta_{i_1} \eta_{i_3} \eta_{i_4} (\beta_{i_1} \theta_{i_1}) (\beta_{i_4} \theta_{i_4}) \bigg)^2 \cdot \theta_{i_2} \theta_{j_2} \theta_{i_3} \theta_{j_3}
	\\& \les  \frac{1}{\| \theta \|_1^6} 
	\sum_{i_2, j_2, i_3, j_3 } \| \gam \|_2^4 \| \theta \|_2^4 
	\theta_{i_2} \theta_{j_2} \theta_{i_3}^3\theta_{j_3}  
	\les \frac{ \| \gam \|_2^4 \| \theta \|_2^6}{ \| \theta \|_1^3}. 
\end{align*}
Similarly, using $|B_{ij}| \les \Omega_{ij} \les \theta_i \theta_j$,  
\begin{align*}
	\var(Y_{b1,B}) &\les
	\frac{1}{\| \theta \|_1^6} 
	\sum_{i_3, j_3 (dist)} \bigg( \sum_{i_1, i_2, i_4, j_2} 
	\eta_{i_1} \eta_{i_3} \eta_{i_4} \theta_{i_2} \theta_{j_2} (\beta_{i_1} \theta_{i_1}) 
	(\beta_{i_4} \theta_{i_4}) \bigg)^2 
	\cdot | \E W_{i_3, j_3} | 
	\\&\les \frac{1}{\| \theta \|_1^6} 
	\sum_{i_3, j_3 } \| \gam \|_2^4 \| \theta\|_2^4 \| \theta \|_1^2 \theta_{i_3}^3 \theta_{j_3} 
	\les \frac{ \| \gam \|_2^4 \| \theta \|_2^6}{ \| \theta \|_1^3}. 
\end{align*}
It follows that
\begin{align*}
	\var(Y_{b1}) \les \frac{ \| \gam \|_2^4 \| \theta \|_2^6}{ \| \theta \|_1^3}
\end{align*}

To control $\var(Y_{b2})$, again we invoke the identity \eqref{eqn:bernoulli_identity} to write
\begin{align*}
	Y_{b2} &= 
	-\frac{1}{v^{3/2}}
	\sum_{i_1, i_2, i_3, i_4(dist)}
	\sum_{\ell_2 \notin \{ i_3, i_2 \}}
	\eta_{i_1}\eta_{i_3}\eta_{i_4} A_{i_2i_3} W_{i_2i_3}W_{i_2\ell_2}\widetilde{\Omega}_{i_1i_4}
	\\& \quad -\frac{1}{v^{3/2}}
	\sum_{i_1, i_2, i_3, i_4(dist)}
	\sum_{\ell_2 \notin \{ i_3, i_2 \}}
	\eta_{i_1}\eta_{i_3}\eta_{i_4} B_{i_2i_3}W_{i_2\ell_2}\widetilde{\Omega}_{i_1i_4}
	\equiv Y_{b2,A} + Y_{b2, B}.
\end{align*}
Using similar arguments from before, we have
\begin{align*}
	\var(Y_{b2, A})
	&\les \frac{1}{\| \theta \|_1^6}
	\sum_{i_2 i_3 \ell_2} 
	\bigg( \sum_{i_1 i_4} \theta_{i_1} \theta_{i_3} \theta_{i_4}
	(\beta_{i_1} \theta_{i_1} ) 	(\beta_{i_4} \theta_{i_4}) \bigg)^2 
	\theta_{i_2}^2 \theta_{i_3} \theta_{\ell_2} 
	\\&	\les \frac{1}{\| \theta \|_1^6}
	\sum_{i_2 i_3 \ell_2} \| \gam \|_2^4 \| \theta \|_2^4 		\theta_{i_2}^2 \theta_{i_3}^3 \theta_{\ell_2} 
	\les \frac{ \| \gam \|_2^4 \| \theta \|_2^8}{ \| \theta \|_1^5}. 
\end{align*}
Furthermore,
\begin{align*}
	\var(Y_{b2, B})
	&\les \frac{1}{\| \theta \|_1^6}
	\sum_{i_2, \ell_2} \bigg( \sum_{i_1, i_3, i_4} 
	\theta_{i_1} \theta_{i_3} \theta_{i_4}
	(\beta_{i_1} \theta_{i_1} ) 	(\beta_{i_4} \theta_{i_4}) \theta_{i_2} \theta_{i_3} 
	\bigg)^2
	\theta_{i_2} \theta_{\ell_2}
	\\& \les \frac{1}{\| \theta \|_1^6}
	\sum_{i_2, \ell_2} \| \gam \|_2^4 \| \theta \|_2^8 \theta_{i_2}^3 \theta_{\ell_2} \les \frac{ \| \gam \|_2^4 \| \theta \|_2^{10}}{\| \theta \|_1^5}. 
\end{align*}
Since $Y_{b2} = Y_{b3}$, we have
\begin{align*}
	\var(Y_{b2}) = \var( Y_{b3}) \les \frac{ \| \gam \|_2^4 \| \theta \|_2^{10}}{\| \theta \|_1^5}. 
\end{align*}

Next we study the variance of $Y_{2c}$. For notational brevity, let
\begin{align*}
	\mc{R}_{i_1, i_2, i_3} =\bigg \{ (j_2, \ell_2, j_3)\bigg| j_2\neq i_2, \ell_2\neq i_2, j_3\neq i_3 j_2\neq \ell_2, (i_2, j_2) \neq (j_3, i_3), (i_2, \ell_2) \neq ( j_3, i_3 ) 
	\bigg\}.
\end{align*} 
We have
\begin{align*}
	&\var(Y_{c})
	\\&= \frac{1}{v^3} \sum_{\substack{i_1, i_2, i_3, i_4 (dist)
			\\i_1', i_2', i_3', i_4' (dist) }}
	\, \, \sum_{\substack{(j_2, \ell_2, j_3) \in \mc{R}_{i_1, i_2, i_3} \\ 
			(j_2', \ell_2', j_3') \in \mc{R}_{i_1', i_2', i_3'}}	} 	
	\eta_{i_1}\eta_{i_3}\eta_{i_4}  \widetilde{\Omega}_{i_1i_4}
	\eta_{i_1'}\eta_{i_3'}\eta_{i_4'}  \widetilde{\Omega}_{i_1'i_4'}
	\E\big[ W_{i_2j_2}W_{i_2\ell_2}W_{i_3j_3}
	W_{i_2'j_2'}W_{i_2'\ell_2'}W_{i_3'j_3'}
	\big] 
	\num \label{eqn:VarY2c_mainbd}
\end{align*}
Note that $W_{i_2j_2}W_{i_2\ell_2}W_{i_3j_3}$ and $W_{i_2'j_2'}W_{i_2'\ell_2'}W_{i_3'j_3'}$ above are uncorrelated unless
\[\bigg\{ \{i_2, j_2\}, \{i_2, \ell_2\}, \{i_3, j_3\} \bigg \}
= \bigg\{ \{i_2', j_2'\}, \{i_2', \ell_2'\}, \{i_3', j_3'\} \bigg\}.\] 
In particular, $i_3' \in \{ i_2, j_2, \ell_2, i_3, j_3 \}$ when the above holds. Hence for some choice of $a_i \in \{0, 1\}$ with $\sum_{i = 1}^5 a_i = 1$, 
\begin{align*}
	\var(Y_{c}) &\les \frac{1}{v^3} 
	\sum_{ \substack{i_1, i_2, i_3, i_4 \\ i_1', i_4',  j_2, \ell_2, j_3} } 
	\theta_{i_2}^{a_1} 
	\theta_{j_2}^{a_2} 
	\theta_{\ell_2}^{a_3} 
	\theta_{i_3}^{a_4} 
	\theta_{j_3}^{a_5} 	
	\cdot
	\theta_{i_1} \theta_{i_3}
	\theta_{i_4} (\beta_{i_1} \theta_{i_1})
	(\beta_{i_4} \theta_{i_4})
	\theta_{i_1'} \theta_{i'_4} (\beta_{i_1'} \theta_{i_1'})
	(\beta_{i_4'} \theta_{i_4'})
	\cdot 	\theta_{i_2}^2 \theta_{j_2} \theta_{\ell_2}
	\theta_{i_3} \theta_{j_3}
	\\& \les 
	\frac{1}{v^3} 
	\sum_{ \substack{i_1, i_2, i_3, i_4 \\ i_1', i_4',  j_2, \ell_2, j_3} }
	\beta_{i_1} \beta_{i_1'} \beta_{i_4} \beta_{i_4'} 
	\theta_{i_1}^2 \theta_{i_2}^{2 + a_1}
	\theta_{i_3}^{ 2 + a_4} 
	\theta_{i_4}^2
	\theta_{i_1'}^2
	\theta_{i_4'}^2 \theta_{j_2}^{ 1 + a_2}
	\theta_{\ell_2}^{1 + a_3}
	\theta_{j_3}^{1 + a_5}
	\\&\les 	\frac{1}{v^3} 
	\sum_{ \substack{i_1, i_2, i_3, i_4 \\ i_1', i_4',  j_2, \ell_2, j_3} }
	\beta_{i_1} \beta_{i_1'} \beta_{i_4} \beta_{i_4'} 
	\theta_{i_1}^2 \theta_{i_2}^{2 }
	\theta_{i_3}^{ 2 } 
	\theta_{i_4}^2
	\theta_{i_1'}^2
	\theta_{i_4'}^2 \theta_{j_2}^{ 1 }
	\theta_{\ell_2}^{1 }
	\theta_{j_3}^{1 }
	\les \frac{  \| \gam \|_2^4 \| \theta \|_2^8  }{ \|\theta \|_1^3}, 
\end{align*}
where in the last line we apply \eqref{eqn:assn2} followed by \eqref{eqn:beta_bds}. Combining our results above we have
\begin{align*}
	|\E T_{2c} | \les \frac{ \| \gam \|_2^2 \| \theta \|_2^4 }{\| \theta \|_1^2},
	\qquad 
	\var(T_{2c}) \les \frac{ \| \gam \|_2^4 \| \theta \|_2^6}{ \| \theta \|_1^2 }. 
\end{align*}

The argument for $T_{2d}$ is omitted since it is similar to the one for $T_{2c}$ (note that the two terms have similar structure). The results are stated below.
\begin{align*}
	| \E T_{2d} | &\les \frac{ \| \gam \|_2^2 \| \theta \|_2^4}{ \| \theta \|_1^2},
	\qquad \var(T_{2d}) \les \frac{ \| \gam \|_2^4 \| \theta \|_2^8}{ \| \theta \|_1^3}. 
\end{align*}

Combining the results for $T_{2a}, \ldots, T_{2d}$ yields 
\begin{align*}
	|\E T_2 | \les  \frac{ \| \gam \|_2^2 \| \theta \|_2^4}{ \| \theta \|_1^2} ,
	\qquad \var(T_2) \les \frac{ \| \gam \|_2^4 \| \theta \|_2^8}{ \| \theta \|_1^2}, 
\end{align*}
as desired. \qed 

\subsubsection{Proof of Lemma \ref{lem:real_SgnQ_tistarQ}}	

As before, we only need to analyze the alternative hypothesis. In \cite[Supplement,pg.103]{JinKeLuo21} it is shown that $\ti Q^* - Q^*$ is a sum of $O(1)$ terms of the form 
\begin{align*}
	Y = 	\Bigl(\frac{v}{V}\Bigr)^{N_{\tilde{r}}}\sum_{i,j,k,\ell (dist)}a_{ij}b_{jk}c_{k\ell}d_{\ell i},
	\num \label{compareXY-0}
\end{align*}
where $a,b,c,d\in \{\widetilde{\Omega}, W, \delta, -(\teta-\eta)(\teta-\eta)^\T\}$, and $N_{\ti r}$ denotes the number of $a, b, c, d$ that are equal to $-(\teta-\eta)(\teta-\eta)^\T$.

Similarly, let $N_W$ denote the number of $a,b,c,d$ that are equal to $W$, and $N_{\tilde \Omega}$ and $N_\delta$ are similarly defined. Write
\beq \label{compareXY-1}
Y =  \Bigl(\frac{v}{V}\Bigr)^m X, \quad\quad \mbox{where}\quad X =  \sum_{i,j,k,\ell (dist)}a_{ij}b_{jk}c_{k\ell}d_{\ell i}. 
\eeq
Note that for this proof, we do not need the explicit decomposition: we only will use the fact that $\ti Q^* - Q^*$ is a sum of $O(1)$ terms. At times, we refer to these terms of the form $Y$ composing $\ti Q^* - Q^*$ as \textit{post-expansion sums}. 

In \cite{JinKeLuo21} it is shown that $4 \geq N_{\tilde r} \geq 1$ for every post-expansion sum (note that the upper bound of $4$ is trivial). It turns out that this is the \textit{only} constraint on the post-expansion sums; so we need to analyze every single possible combination of nonnegative integers $(N_{\tilde \Omega}, N_W, N_\delta, N_{\tilde r})$ where their sum is $4$ and $N_{\tilde r} \geq 1$ and then arrange $a,b,c,d \in \{ \tilde \Omega, W, \delta, - (\tilde \eta - \eta)(\tilde \eta - \eta)^{\mathsf{T}}\}$ in all possible ways according to \eqref{compareXY-0}. This leads to a total of $34$ possibilities, all of which are shown in Table \ref{tb:remainder} reproduced from \cite{JinKeLuo21}. 

\begin{table}[tb!]  
	\centering
	\caption{\textit{Note: This table and caption reproduced from Table G.4 of \cite{JinKeLuo21}.}The $34$ types of the $175$ post-expansion sums for $(\widetilde{Q}^*_n-Q_n^*)$.}    \label{tb:remainder}
	\scalebox{.95}{
		\begin{tabular}{lccclcc}
			Notation  & $\#$ & $N_{\tilde{r}}$ &  ($N_{\delta}, N_{\widetilde{\Omega}}, N_W)$ &     Examples & $N^*_W$\\
			\hline 
			$R_1$  & 4  &1&(0, 0, 3)   & $\sum_{i, j, k,\ell (dist)} \tilde{r}_{ij}W_{jk} W_{k\ell} W_{\ell i}$ & 5\\  
			$R_2$  & 8  &1&(0, 1, 2)   & $\sum_{i, j, k,\ell (dist)} \tilde{r}_{ij} \widetilde{\Omega}_{jk} W_{k\ell} W_{\ell i}$ & 4\\ 
			$R_3$  & 4  & &     & $\sum_{i, j, k,\ell (dist)} \tilde{r}_{ij} W_{jk}\widetilde{\Omega}_{k\ell} W_{\ell i}$ &  4\\ 
			$R_4$ & 8 &1&(0, 2, 1)   & $\sum_{i, j, k,\ell (dist)} \tilde{r}_{ij} \widetilde{\Omega}_{jk} \widetilde{\Omega}_{k\ell} W_{\ell i}$ & 3\\  
			$R_5$ & 4 &&   & $\sum_{i, j, k,\ell (dist)} \tilde{r}_{ij} \widetilde{\Omega}_{jk}W_{k\ell} \widetilde{\Omega}_{\ell i}$ & 3\\   
			$R_6$ & 4 & 1& (0, 3, 0)   & $\sum_{i, j, k,\ell (dist)} \tilde{r}_{ij}\widetilde{\Omega}_{jk} \widetilde{\Omega}_{k\ell} \widetilde{\Omega}_{\ell i}$  & 2\\   
			$R_7$   & 8  & 1 & (1, 0, 2)   & $\sum_{i, j, k,\ell (dist)} \tilde{r}_{ij} \delta_{jk} W_{k\ell} W_{\ell i}$ & 5\\  
			$R_8$   & 4  &  &   & $\sum_{i, j, k,\ell (dist)} \tilde{r}_{ij} W_{jk}\delta_{k\ell} W_{\ell i}$ & 5\\ 
			$R_9$  & 8  &1 &(1, 1, 1)   & $\sum_{i, j, k,\ell (dist)} \tilde{r}_{ij}\delta_{jk}\widetilde{\Omega}_{k\ell}  W_{\ell i}$ & 4   \\ 
			$R_{10}$  & 8  & &   & $\sum_{i, j, k,\ell (dist)} \tilde{r}_{ij}\widetilde{\Omega}_{jk}  W_{k \ell}\delta_{\ell i}$ &   
			4  \\ 
			$R_{11}$ & 8  &&   & $\sum_{i, j, k,\ell (dist)} \tilde{r}_{ij}W_{jk}\delta_{k\ell }\widetilde{\Omega}_{\ell i}  $ &   
			4 \\ 
			$R_{12}$ & 8 &1& (1, 2, 0)   & $\sum_{i, j, k,\ell (dist)}\tilde{r}_{ij} \delta_{jk}  \widetilde{\Omega}_{k\ell} \widetilde{\Omega}_{\ell i}$ & 3\\ 
			$R_{13}$ & 4 & & & $\sum_{i, j, k,\ell (dist)}\tilde{r}_{ij} \widetilde{\Omega}_{jk} \delta_{k\ell}  \widetilde{\Omega}_{\ell i}$ & 3\\
			$R_{14}$ & 8 & 1 &(2, 0, 1)   & $\sum_{i, j, k,\ell (dist)} \tilde{r}_{ij}\delta_{jk} \delta_{k\ell} W_{\ell i}$    &  5\\  
			$R_{15}$ & 4 &  &   & $\sum_{i, j, k,\ell (dist)} \tilde{r}_{ij}\delta_{jk} W_{k\ell }\delta_{\ell i} $    &  5\\ 
			$R_{16}$ & 8 & 1 & (2, 1, 0) & $\sum_{i, j, k,\ell (dist)} \tilde{r}_{ij}\delta_{jk} \delta_{k\ell} \widetilde{\Omega}_{\ell i}$    &  4\\  
			$R_{17}$  & 4 &  &  & $\sum_{i, j, k,\ell (dist)} \tilde{r}_{ij}\delta_{jk}\widetilde{\Omega}_{k\ell} \delta_{\ell i}$    & 4 \\  
			$R_{18}$ & 4 &  1 & (3,  0,  0)  & $\sum_{i, j, k,\ell (dist)} \widetilde{r}_{ij}\delta_{jk} \delta_{k\ell} \delta_{\ell i}$ & 5\\ 
			\hline 
			$R_{19}$   & 4  & 2 & (0, 0, 2)   & $\sum_{i, j, k,\ell (dist)} \tilde{r}_{ij} \tilde{r}_{jk} W_{k\ell} W_{\ell i}$ &6 \\ 
			$R_{20}$      & 2  &  &    & $\sum_{i, j, k,\ell (dist)} \tilde{r}_{ij} W_{jk}\tilde{r}_{k\ell} W_{\ell i}$ & 6\\  
			$R_{21}$   & 4 & 2 & (0, 2, 0)   & $\sum_{i, j, k,\ell (dist)} \tilde{r}_{ij} \tilde{r}_{jk} \widetilde{\Omega}_{k\ell} \widetilde{\Omega}_{\ell i}$ & 4\\ 
			$R_{22}$     & 2  &  &    & $\sum_{i, j, k,\ell (dist)} \tilde{r}_{ij} \widetilde{\Omega}_{jk}\tilde{r}_{k\ell} \widetilde{\Omega}_{\ell i}$ & 4\\
			$R_{23}$     & 4  & 2 & (2, 0, 0)   & $\sum_{i, j, k,\ell (dist)} \tilde{r}_{ij} \tilde{r}_{jk} \delta_{k\ell} \delta_{\ell i}$ & 6\\ 
			$R_{24}$      & 2  &  &    & $\sum_{i, j, k,\ell (dist)} \tilde{r}_{ij} \delta_{jk}\tilde{r}_{k\ell} \delta_{\ell i}$ & 6\\    
			$R_{25}$  & 8  & 2 & (0, 1, 1)   & $\sum_{i, j, k,\ell (dist)} \tilde{r}_{ij} \tilde{r}_{jk} \widetilde{\Omega}_{k\ell} W_{\ell i}$ & 5\\ 
			$R_{26}$      & 4  &  &    & $\sum_{i, j, k,\ell (dist)} \tilde{r}_{ij} \widetilde{\Omega}_{jk}\tilde{r}_{k\ell} W_{\ell i}$ & 5\\  
			$R_{27}$   & 8  & 2 & (1, 1, 0)   & $\sum_{i, j, k,\ell (dist)} \tilde{r}_{ij} \tilde{r}_{jk} \delta_{k\ell} \widetilde{\Omega}_{\ell i}$ & 5\\ 
			$R_{28}$     & 4  &  &    & $\sum_{i, j, k,\ell (dist)} \tilde{r}_{ij} \delta_{jk}\tilde{r}_{k\ell} \widetilde{\Omega}_{\ell i}$ & 5\\
			$R_{29}$      & 8  & 2 & (1, 0, 1)   & $\sum_{i, j, k,\ell (dist)} \tilde{r}_{ij} \tilde{r}_{jk} \delta_{k\ell} W_{\ell i}$ & 6\\ 
			$R_{30}$      & 4  &  &    & $\sum_{i, j, k,\ell (dist)} \tilde{r}_{ij} \delta_{jk}\tilde{r}_{k\ell} W_{\ell i}$ & 6\\  
			\hline 
			$R_{31}$   & 4  & 3 & (0, 0, 1)   & $\sum_{i, j, k,\ell (dist)} \tilde{r}_{ij} \tilde{r}_{jk} \tilde{r}_{k\ell} W_{\ell i}$ & 7\\ 
			$R_{32}$   & 4  & 3 & (0, 1, 0)   & $\sum_{i, j, k,\ell (dist)} \tilde{r}_{ij} \tilde{r}_{jk} \tilde{r}_{k\ell} \widetilde{\Omega}_{\ell i}$ & 6\\  
			$R_{33}$  & 4  & 3 & (1, 0, 0)   & $\sum_{i, j, k,\ell (dist)} \tilde{r}_{ij} \tilde{r}_{jk} \tilde{r}_{k\ell} \delta_{\ell i}$ & 7\\   
			\hline
			$R_{34}$     & $1$  & 4 & (0, 0, 0)   & $\sum_{i, j, k,\ell (dist)} \tilde{r}_{ij} \tilde{r}_{jk} \tilde{r}_{k\ell} \tilde{r}_{\ell i}$ & 8\\   
			\hline
		\end{tabular} 
	} 
\end{table} 

In \cite[Supplement,pg.103]{JinKeLuo21} it is shown that 
\begin{align*}
	| \E[Y - X]| &\leq o( \| \theta \|_2^{-2} ) \sqrt{ \E[X^2]} + o(1), \text{ and }
	\\ \var(Y) &\leq 2 \var(X) + o( \| \theta \|_2^{-4} ) \E[X^2] + o(1). 
	\num \label{eqn:XvsY_moments}
\end{align*}
The proof of \eqref{eqn:XvsY_moments} in \cite{JinKeLuo21} only requires the heterogeneity assumptions \eqref{eqn:assn2}--\eqref{eqn:assn4} and the following two conditions. First, we must have the tail inequality
\beq \label{lem-event-tail}
\mathbb{P}(|V-v|>t) \leq \begin{cases}
	2\exp\bigl( -\frac{C_1}{\|\theta\|_1^2} t^2\bigr), & \mbox{when } x_n\|\theta\|_1 \leq t \leq \|\theta\|_1^2,\\
	2\exp\bigl( -C_2t \bigr), &\mbox{when }t > \|\theta\|_1^2. 
\end{cases} 
\eeq 
Second, it must hold that $|Y-X|$ is dominated by a polynomial in $V$. See \cite[Lemma G.10 and G.11]{JinKeLuo21} for further details.  Both conditions are satisfied in our setting, so indeed \eqref{eqn:XvsY_moments} applies.

Let $N_W$ and $N_\delta$ denote the number of $a, b, c, d$ that are equal to $W$ and $\delta$, respectively. As in \cite{JinKeLuo21}, we define 
\begin{align*}
	N_W^* = N_W + N_\delta + 2 N_{\ti r}
	\num \label{eqn:NW*}
\end{align*}
and divide our analysis into parts based on this parameter.

\paragraph{Analysis of terms with $N^*_W\leq 4$} For convenience, we reproduce Table G.5 from \cite{JinKeLuo21} in Table \ref{tb:Order-4}. The left column of Table \ref{tb:Order-4} lists all of the terms with $N^*_W\leq 4$ , where note that factors of $(\frac{v}{V})^{N_{\ti r}}$ are removed. In the right column terms are listed that have similar structure to those on the left. Precisely, a term in the left column has the form 
\begin{align*}
	X &= \sum_{i_1, \ldots, i_m \in \mathcal{R}} c_{i_1, \ldots, i_m} 
	G_{i_1, \ldots, i_m}, 	
\end{align*}
and its adjacent term on the right column has the form
\[
X^* =  \sum_{i_1, \ldots, i_m \in \mathcal{R}} c^*_{i_1, \ldots, i_m} 
G_{i_1, \ldots, i_m}, 
\]
analogous to $T$ and $T^*$ from Lemma \ref{lem:transfer}. By inspection, we see that for each term in the left column, the canonical upper bounds $\overline{c_{i_1, \ldots, i_m}}$ and $ \overline{c^*_{i_1, \ldots, i_m}}$  on the coefficients $c_{i_1, \ldots, i_m}$ and $c^*_{i_1, \ldots, i_m}$ satisfy 
\[\overline{c_{i_1, \ldots, i_m}} \les \overline{c^*_{i_1, \ldots, i_m}}.\] 
Recall that these canonical upper bounds were defined in Section \ref{sec:proof_strategy}. Thus the conclusion of Lemma \ref{lem:transfer} applies, and we have for each term $X$ in the left column of Table \ref{tb:Order-4},
\begin{align*}
	|\E X| \les \overline{ \E X^* }, 
	\qquad \var(X) \les \overline{\var(X^*)}. 
\end{align*}

\begin{table}[tb!]  
	\centering
	\caption{ \textit{For clarity, this table and caption are borrowed from Table G.5 of \cite{JinKeLuo21}.}The $14$ types of post-expansion sums with $N^*_W\leq 4$. The right column displays the post-expansion sums defined before which have similar forms as the post-expansion sums in the left column. For some terms in the right column, we permute $(i,j,k,\ell)$ in the original definition for ease of comparison with the left column. (In all expressions, the subscript ``$i,j,k,\ell (dist)$" is omitted.)}    \label{tb:Order-4}
	\scalebox{.9}{
		\begin{tabular}{lc | l c}
			&     Expression  &  & Expression\\
			\hline   
			$R_2$     & $\sum (\teta_i-\eta_i)(\teta_j-\eta_j) \widetilde{\Omega}_{jk} W_{k\ell} W_{\ell i}$ & $Z_{1b}$ &  $\sum (\teta_i-\eta_i)\eta_j (\teta_j-\eta_j)\eta_k W_{k\ell} W_{\ell i}$\\ 
			$R_3$  & $\sum (\teta_i-\eta_i)(\teta_j-\eta_j) W_{jk}\widetilde{\Omega}_{k\ell} W_{\ell i}$ &  $Z_{2a}$ & $\sum \eta_\ell (\teta_j-\eta_j)W_{jk}\eta_k(\teta_i-\eta_i)W_{i\ell}$\\ 
			$R_4$ &  $\sum (\teta_i-\eta_i)(\teta_j-\eta_j) \widetilde{\Omega}_{jk} \widetilde{\Omega}_{k\ell} W_{\ell i}$ & $Z_{3d}$ & $\sum (\teta_i-\eta_i)\eta_j(\teta_j-\eta_j)\eta_k\widetilde{\Omega}_{k\ell}W_{\ell  i}$\\  
			$R_5$ &  $\sum(\teta_i-\eta_i)(\teta_j-\eta_j) \widetilde{\Omega}_{jk}W_{k\ell} \widetilde{\Omega}_{\ell i}$ &$Z_{4b}$ & $\sum \widetilde{\Omega}_{ij}(\teta_j-\eta_j)\eta_kW_{k\ell}\eta_\ell(\teta_i-\eta_i)$ \\   
			$R_6$ &  $\sum (\teta_i-\eta_i)(\teta_j-\eta_j)\widetilde{\Omega}_{jk} \widetilde{\Omega}_{k\ell} \widetilde{\Omega}_{\ell i}$  & $Z_{5a}$ & $\sum \eta_i(\teta_j-\eta_j)\widetilde{\Omega}_{jk}\widetilde{\Omega}_{k\ell}\eta_\ell (\teta_i-\eta_i)$ \\   
			$R_9$  & $\sum(\teta_i-\eta_i)(\teta_j-\eta_j)^2\eta_k \widetilde{\Omega}_{k\ell}  W_{\ell i}$ & $T_{1d}$ &   
			$\sum \eta_{\ell}(\teta_j-\eta_j)^2\eta^2_k(\teta_i-\eta_i)W_{i\ell}$\\ 
			& $\sum(\teta_i-\eta_i)(\teta_j-\eta_j)\eta_j (\teta_k-\eta_k)\widetilde{\Omega}_{k\ell}  W_{\ell i}$ & $T_{1a}$ &   
			$\sum \eta_{\ell}(\teta_j-\eta_j)\eta_j(\teta_k-\eta_k)\eta_k(\teta_i-\eta_i)W_{i\ell}$\\ 
			$R_{10}$  & $\sum (\teta_i-\eta_i)^2(\teta_j-\eta_j) \widetilde{\Omega}_{jk}  W_{k \ell}\eta_\ell$ &   
			$T_{1c}$ & $\sum (\teta_j-\eta_j)\eta_kW_{k\ell}\eta_\ell (\teta_i-\eta_i)^2\eta_j$ \\ 
			& $\sum (\teta_i-\eta_i)(\teta_j-\eta_j) \widetilde{\Omega}_{jk}  W_{k \ell}(\teta_\ell -\eta_\ell)\eta_i$ & $T_{1a}$ & $\sum (\teta_j-\eta_j)\eta_kW_{k\ell}(\teta_\ell-\eta_\ell)\eta_i (\teta_i-\eta_i)\eta_j$  
			\\ 
			$R_{11}$  & $\sum (\teta_i-\eta_i)(\teta_j-\eta_j) W_{jk}\eta_k (\teta_\ell-\eta_\ell)\widetilde{\Omega}_{\ell i}$ & $T_{1a}$  & $\sum (\teta_i-\eta_i)\eta_k W_{kj} (\teta_j-\eta_j)\eta_\ell (\teta_\ell-\eta_\ell)\eta_i$   
			\\ 
			& $\sum (\teta_i-\eta_i)(\teta_j-\eta_j)W_{jk}(\teta_k-\eta_k)\eta_\ell \widetilde{\Omega}_{\ell i}  $ & $T_{1b}$ & $\sum \eta_i (\teta_k-\eta_k) W_{kj} (\teta_j-\eta_j)\eta^2_\ell (\teta_i-\eta_i)$   
			\\ 
			$R_{12}$ &  $\sum (\teta_i-\eta_i)(\teta_j-\eta_j)^2\eta_k  \widetilde{\Omega}_{k\ell} \widetilde{\Omega}_{\ell i}$ & $T_{2c}$ & $\sum \eta_i(\teta_j-\eta_j)^2\eta_k\widetilde{\Omega}_{k\ell}\eta_\ell(\teta_i-\eta_i)$\\ 
			&  $\sum (\teta_i-\eta_i)(\teta_j-\eta_j)\eta_j(\teta_k-\eta_k) \widetilde{\Omega}_{k\ell} \widetilde{\Omega}_{\ell i}$ & $T_{2a}$ & $\sum \eta_i(\teta_j-\eta_j)\eta_j(\teta_k-\eta_k)\widetilde{\Omega}_{k\ell}\eta_\ell(\teta_i-\eta_i)$\\ 
			$R_{13}$ & $\sum (\teta_i-\eta_i)(\teta_j-\eta_j) \widetilde{\Omega}_{jk} (\teta_k-\eta_k)\eta_\ell  \widetilde{\Omega}_{\ell i}$ & $T_{2b}$ & $\sum \eta_i(\teta_j-\eta_j)\widetilde{\Omega}_{jk}(\teta_k-\eta_k)\eta_{\ell}^2(\teta_i-\eta_i)$\\
			$R_{16}$ &  $\sum (\teta_i-\eta_i)(\teta_j-\eta_j)^2\eta_k(\teta_k-\eta_k)\eta_{\ell} \widetilde{\Omega}_{\ell i}$    &  $F_b$ &  $\sum \eta_i (\teta_j-\eta_j)^2\eta_k (\teta_k-\eta_k)\eta_\ell^2(\teta_i-\eta_i)$\\    
			&  $\sum (\teta_i-\eta_i)(\teta_j-\eta_j)^2\eta^2_k(\teta_\ell-\eta_\ell)\widetilde{\Omega}_{\ell i}$ & $F_b$ &  $\sum \eta_i (\teta_j-\eta_j)^2 \eta_k^2 (\teta_\ell-\eta_\ell)\eta_\ell(\teta_i-\eta_i)$\\ 
			&  $\sum (\teta_i-\eta_i)(\teta_j-\eta_j)\eta_j(\teta_k-\eta_k)^2\eta_\ell \widetilde{\Omega}_{\ell i}$  &  $F_b$ &  $\sum \eta_i (\teta_j-\eta_j)\eta_j(\teta_k-\eta_k)^2\eta_\ell^2(\teta_i-\eta_i)$\\
			&  $\sum (\teta_i-\eta_i)(\teta_j-\eta_j)\eta_j(\teta_k-\eta_k)\eta_k(\teta_\ell -\eta_\ell) \widetilde{\Omega}_{\ell i}$  &  $F_a$ &  $\sum \eta_i (\teta_j-\eta_j)\eta_j(\teta_k-\eta_k)\eta_k(\teta_\ell-\eta_\ell)\eta_\ell (\teta_i-\eta_i)$\\ 
			$R_{17}$  & $\sum (\teta_i-\eta_i)(\teta_j-\eta_j)\eta_j(\teta_k-\eta_k)\widetilde{\Omega}_{k\ell}(\teta_\ell-\eta_\ell)\eta_i$  & $F_a$ &  $\sum \eta_i (\teta_j-\eta_j)\eta_j(\teta_k-\eta_k)\eta_k(\teta_\ell-\eta_\ell)\eta_\ell (\teta_i-\eta_i)$\\ 
			& $\sum (\teta_i-\eta_i)(\teta_j-\eta_j)^2\eta_k \widetilde{\Omega}_{k\ell}(\teta_\ell-\eta_\ell)\eta_i$    & $F_b$ &  $\sum \eta_i (\teta_j-\eta_j)^2 \eta_k^2 (\teta_\ell-\eta_\ell)\eta_\ell(\teta_i-\eta_i)$\\ 
			& $\sum (\teta_i-\eta_i)^2(\teta_j-\eta_j)^2 \eta_k\widetilde{\Omega}_{k\ell} \eta_\ell$    & $F_c$ &  $\sum\eta_\ell(\teta_i-\eta_i)^2\eta_k^2(\teta_j-\eta_j)^2\eta_\ell$\\
			$R_{21}$   & $\sum (\teta_i-\eta_i)(\teta_j-\eta_j)^2(\teta_k-\eta_k) \widetilde{\Omega}_{k\ell} \widetilde{\Omega}_{\ell i}$ & $F_b$ &  $\sum \eta_i (\teta_j-\eta_j)^2\eta_k (\teta_k-\eta_k)\eta_\ell^2(\teta_i-\eta_i)$\\
			$R_{22}$     & $\sum (\teta_i-\eta_i)(\teta_j-\eta_j) \widetilde{\Omega}_{jk}(\teta_k-\eta_k)(\teta_\ell-\eta_\ell) \widetilde{\Omega}_{\ell i}$ &$F_a$ &  $\sum \eta_i (\teta_j-\eta_j)\eta_j(\teta_k-\eta_k)\eta_k(\teta_\ell-\eta_\ell)\eta_\ell (\teta_i-\eta_i)$\\ 
			\hline
		\end{tabular} 
	} 
\end{table} 

As discussed in Section \ref{sec:proof_strategy}, the upper bounds on the means and variances in Lemmas \ref{lem:ideal_SgnQ}--\ref{lem:Uc} are in fact upper bounds on $\overline{ \E X^* }$ and $\overline{\var(X^*)}$. By \eqref{eqn:XvsY_moments} and Lemmas \ref{lem:ideal_SgnQ}--\ref{lem:Uc}, for every post-expansion sum $Y$ with $N_W^* \leq 4$ we have
\begin{align*}
	|\E Y| &\leq | \E X | + o( \| \theta \|_2^{-2} ) \sqrt{ \E[X^2]} 
	= | \E X | + o( \| \theta \|_2^{-2} ) \sqrt{ \E[X]^2 + \var(X)} 
	\\&\les  \ti \lambda^2 \lambda_1  +  
	o( \| \theta \|_2^{-2} ) \cdot \sqrt{  \ti \lambda^4 \lambda_1^2 + \lambda_1^4 + \ti \lambda^6 + \ti \lambda^2 \lambda_1^3 }
	\\&\les \ti \lambda^2 \lambda_1 + \lambda_1^2 + \ti \lambda^3 + |\ti \lambda| \lambda_1^{3/2}
	= o(\ti \lambda^4)
\end{align*}
by the assumption that $| \ti \lambda |/\sqrt{\lambda_1} \to \infty$. Similarly,
\begin{align*}
	\var(Y) &\les  \var(X) + o( \| \theta \|_2^{-4} ) \E[X^2] 
	= \var(X) + o( \| \theta \|_2^{-4} )( \E[X]^2 + \var(X) )
	\\&\les  \lambda_1^4 + \ti \lambda^6 + \ti \lambda^2 \lambda_1^3 
	+ o( \| \theta \|_2^{-4} ) \cdot
	\big(  \ti \lambda^4 \lambda_1^2 + \lambda_1^4 + \ti \lambda^6 + \ti \lambda^2 \lambda_1^3 \big) \les o(\ti \lambda^8). 
\end{align*}

\paragraph{Analysis of terms with $N^*_W > 4$} 

Recall that
\begin{align*}
	\eta=\frac{1}{\sqrt{v}}(\mathbb{E}A){\bf 1}_n,\;\; \teta = \frac{1}{\sqrt{v}}A{\bf 1}_n,\;\; v= {\bf 1}_n'(\mathbb{E}A){\bf 1}_n\cr. 
\end{align*}
Define
\begin{align}
	\label{eqn:G_def}
	G_i = \teta_i - \eta_i. 
\end{align}

Among the post-expansion sums in Table \eqref{tb:remainder} satisfying $N^*_W =5$, only $R_7, R_8,$ and $R_{25}$--$R_{28}$ depend on $\ti \Omega$. Each of these terms falls into one of the types
\begin{align*}
	J'_5 &= \sum_{i, j, k,\ell (dist)} \widetilde{\Omega}_{jk} (G_iG_j G_kG_\ell W_{\ell i}),\\
	J'_6 &= \sum_{i, j, k,\ell (dist)} \widetilde{\Omega}_{k\ell}(G_iG^2_j G_k W_{\ell i}) \\
	J_9 &= \sum_{i, j, k,\ell (dist)} \eta_k \widetilde{\Omega}_{\ell i} (G_iG_j^2 G_kG_\ell)
	\\ J_{10} &= \sum_{i, j, k,\ell (dist)} \eta_\ell \widetilde{\Omega}_{\ell i}(G_iG_j^2 G^2_k).
\end{align*}
See \cite[Supplement, Section G.4.10.2]{JinKeLuo21} for more details.

To handle $J'_5$ and $J'_6$, we compare them to 
\begin{align*}
	J_5 &= \sum_{i, j, k,\ell (dist)} \eta_j\eta_k (G_iG_j G_kG_\ell W_{\ell i}) \\
	J_6 &= \sum_{i, j, k,\ell (dist)} \eta_k\eta_\ell (G_iG^2_j G_k W_{\ell i}),
\end{align*}
both of which are considered in \cite[Supplement, Section G.4.10.2]{JinKeLuo21}. Note that neither $J_5$ nor $J_5$ depends on $\ti \Omega$. Setting $T = J'_5$ and $T^* = J_5$ in Lemma \ref{lem:transfer} and noting that 
$| \ti \Omega_{jk} | \les \theta_j \theta_k$ by \eqref{eqn:beta_bds}, we see that the hypotheses of Lemma \ref{lem:transfer} are satisfied. In \cite[Supplement, Section G.4.10.2]{JinKeLuo21}, it is shown that 
\begin{align*}
	\E[ J_5^2 ] \leq \overline{ \E[ J_5 ] }^2 + 
	\overline{ \var(J_5) } = o( \| \theta \|_2^8 ).
\end{align*}
Applying Lemma \ref{lem:transfer}, we conclude that 
\[\E[ J_5^{'2}] = o( \| \theta \|_2^8 ).\]

Similarly, it is shown in \cite[Supplement, Section G.4.10.2]{JinKeLuo21} that 
\begin{align*}
	\E[ J_6^2 ] \leq \overline{ \E[ J_6 ] }^2 + 
	\overline{ \var(J_6) } = o( \| \theta \|_2^8 ).
\end{align*}
Setting $T = J'_6$ and $T^* = J_6$, the hypotheses of Lemma \ref{lem:transfer} are satisfied because $| \ti \Omega_{k\ell}| \les \theta_k \theta_\ell$. We conclude that 
\[\E[ J_6^{'2}] = o( \| \theta \|_2^8 ).\]

The terms $J_9$ and $J_{10}$ can be analyzed explicitly using the strategy described in Section \ref{sec:proof_strategy}. We omit the full details and instead give a simplified proof in the case where $\| \theta \|_2 \gg [\log(n)]^{5/2}$. The event
\beq  \label{eventE}
E = \cap_{i=1}^n E_i, \qquad \mbox{where}\quad E_i= \big\{ \sqrt{v} |G_i|\leq C_0\sqrt{\theta_i\|\theta\|_1\log(n)} \bigr\}. 
\eeq
is introduced in \cite[Supplement,pg.110]{JinKeLuo21}.
By applying Bernstein's inequality and the union bound, it is shown that $E$ holds with probability at least $1 - n^{-C_0/2.01}$. Applying the crude bound $|G_i| \leq n$ and triangle inequality, we see that $|J_9| \les n^9$ with high probability, and thus for $C_0$ sufficiently large,
\begin{align*}
	\E[ |J_9|^2 \cdot \mf{1}_{E^c} ]= o(1). 
\end{align*}
Under the event $E$, we have by \eqref{eqn:tiOm_bd}, 
\begin{align*}
	|J_9|&\leq \sum_{i,j,k,\ell}|\eta_k\widetilde{\Omega}_{\ell i}||G_iG_j^2G_kG_\ell|\cr
	&\les \sum_{i,j,k,\ell}(\theta_i \theta_k\theta_\ell)\frac{\sqrt{\theta_i\theta_j^2\theta_k\theta_\ell\|\theta\|_1^5 [\log(n)]^5}}{\sqrt{v^5}} \cr
	&\les \frac{[\log(n)]^{5/2}}{\sqrt{\|\theta\|_1^5}} \Bigl(\sum_i \theta^{3/2}_i\Bigr)\Bigl(\sum_j \theta_j\Bigr)\Bigl(\sum_k \theta^{3/2}_k\Bigr)\Bigl(\sum_\ell \theta^{3/2}_\ell\Bigr) \cr
	&\les \frac{ [\log(n)]^{5/2}}{\sqrt{\|\theta\|_1^3}} \Bigl(\sum_i \theta^{3/2}_i\Bigr)^3\cr
	&\les \frac{ [\log(n)]^{5/2}}{\sqrt{\|\theta\|_1^3}} \Bigl(\sum_i\theta_i^2\Bigr)^{3/2}\Bigl(\sum_i\theta_i\Bigr)^{3/2}\cr
	&\les [\log(n)]^{5/2}\|\theta\|^3.
\end{align*}
It follows that 
\begin{align*}
	\E[ J_9^2 ] =\var(J_9) + \E[ J_9]^2 = o(\| \theta \|_2^8 ). 
\end{align*}

We give a similar, simplified argument for $J_{10}$ assuming that $\| \theta \|_2 \gg [\log(n)]^{5/2}$. Under the event $E$, we have 
\begin{align*}
	|J_{10}| &\leq  \sum_{i, j, k,\ell} |\eta_\ell \widetilde{\Omega}_{\ell i}| |G_iG_j^2 G^2_k|\cr
	&\les \sum_{i,j,k,\ell}(\theta_i\theta_\ell^2)\frac{\sqrt{\theta_i\theta_j^2\theta^2_k\|\theta\|_1^5[\log(n)]^5}}{\sqrt{v^5}}\cr
	&\les \frac{[\log(n)]^{5/2}}{\sqrt{\|\theta\|_1^5}}\Bigl(\sum_i \theta_i^{3/2}\Bigr)\Bigl(\sum_j \theta_j\Bigr)\Bigl(\sum_k\theta_k\Bigr)\Bigl(\sum_\ell \theta_\ell^2\Bigr)\cr
	&\les \frac{[\log(n)]^{5/2}}{\sqrt{\|\theta\|_1^5}}\bigl(\|\theta\|\sqrt{\|\theta\|_1}\bigr)\|\theta\|^2_1\|\theta\|^2\cr
	&\les [\log(n)]^{5/2}\|\theta\|^3;
\end{align*} 
Hence
\begin{align*}
	\E[ J_{10}^2 ] =\var(J_{10}) + \E[ J_{10}]^2 = o(\| \theta \|_2^8 ). 
\end{align*}

Next we consider the terms with $N_W^*=6$. The only term that depends on $\ti \Omega$ is $R_{32}$, which has the form
\[
K_5' = \sum_{i, j, k,\ell (dist)} \widetilde{\Omega}_{ik} G_iG_j^2G_kG_\ell^2.
\]
The variance of $K_5'$ can be analyzed explicitly using the strategy described in Section \ref{sec:proof_strategy}. To save space, we give a simplified argument when $\| \theta \|_2 \gg [\log(n)]^{3/2}$. Again let $E$ denote the event \eqref{eventE}. Under this event we have
\begin{align*}
	|K_5'| &\les \sum_{i,j,k,\ell} (\theta_i\theta_k)\frac{\sqrt{\theta_i\theta_j^2\theta_k\theta_\ell^2}\|\theta\|_1^3[\log(n)]^3}{v^3}\cr
	&\les \frac{[\log(n)]^3}{\|\theta\|_1^3}\Bigl(\sum_{i}\theta_i^{3/2}\Bigr)\Bigl(\sum_j\theta_j\Bigr)\Bigl(\sum_k \theta_k^{3/2}\Bigr)\Bigl(\sum_\ell\theta_\ell\Bigr) \cr
	&\les \frac{[\log(n)]^3}{\|\theta\|_1^3}\bigl(\|\theta\|\sqrt{\|\theta\|_1}\bigr)^2\|\theta\|_1^2\cr
	&\les [\log(n)]^3\|\theta\|^2,
\end{align*}
Above we apply \eqref{eqn:tiOm_bd} and \eqref{eqn:beta_bds} as well as Cauchy--Schwarz. It follows that
\begin{align*}
	\E[ K_5^{'2} ] =\var(K_5') + \E[ K_5']^2 = o(\| \theta \|_2^8).  
\end{align*}

Finally, all terms with $N^*_W \geq 7$ have no dependence on $\ti \Omega$, and thus the bounds carry over immediately (see \cite[Supplement, Section G.4.10.4]{JinKeLuo21} for details). This completes the proof of the lemma. \qed 

\subsubsection{Proof of Lemma \ref{lem:real_SgnQ_starQ}}

Define 
\[
\epsilon_{ij}^{(1)} = \eta_i^*\eta_j^*-\eta_i\eta_j, \quad \epsilon_{ij}^{(2)} =  (1-\frac{v}{V})\eta_i\eta_j, \quad \epsilon^{(3)}_{ij}= -(1-\frac{v}{V})\delta_{ij}. 
\]
Note that $\epsilon_{ij}\rp{1}$ is a nonstochastic term. As shown in \cite[Supplement, pg. 119]{JinKeLuo21}, we have
\begin{align*}
	| \epsilon_{ij}\rp{1} | \les \frac{ \| \theta \|_\infty }{\| \theta \|_1} \cdot \theta_i \theta_j,
\end{align*}
which implies that
\begin{align*}
	| \epsilon_{ij}\rp{1} | \les \frac{ 1 }{\| \theta \|_2^2} \cdot \theta_i \theta_j
	\num \label{eqn:epsilon1_bd}
\end{align*}
by \eqref{eqn:assn2}. 

As discussed in \cite[Supplement, Section G.3]{JinKeLuo21}, $Q - Q^*$ is a finite sum of terms of the form
\beq \label{lastlemma-1}
\sum_{i,j,k,\ell (dist)} a_{ij}b_{jk}c_{k\ell}d_{\ell i}, \qquad \mbox{where}\quad a,b,c,d\in \{\widetilde{\Omega}, W,\delta, \tilde{r}, \epsilon\rp{1}, \epsilon\rp{2}, \epsilon\rp{3} \}.
\eeq

Let $Y$ denote an arbitrary term of the form above, and given $X \in \{\widetilde{\Omega}, W,\delta, \tilde{r}, \epsilon\rp{1}, \epsilon\rp{2}, \epsilon\rp{3} \}$, let $N_X$ denote the total number of $a, b, c, d$ that are equal to $X$. It holds that
\[
Y= \big( \frac{v}{V} \big)^{N_{\ti r}} (-1)^{N_\epsilon^{(3)}}\Bigl(1 - \frac{v}{V}\Bigr)^{N^{(2)}_{\epsilon}+N^{(3)}_{\epsilon}} X, \qquad X\equiv \sum_{i,j,k,\ell (dist)} a_{ij}b_{jk}c_{k\ell}d_{\ell i}. 
\]
where
\beq \label{lastlemma-6}
\begin{cases}
	a, b, c, d \in \{\widetilde{\Omega}, W, \delta, (V/v)\tilde{r}, \epsilon^{(1)}, \eta \eta^\T \}, \\
	\mbox{number of $\eta_i\eta_j$ in the product is $N^{(2)}_{\epsilon}$},\\
	\mbox{number of $\delta_{ij}$ in the product is $N_{\delta}+N_{\epsilon}^{(3)}$},\\
	\mbox{number of any other term in the product is same as before}. 
\end{cases}
\eeq
Let $x_n$ denote a sequence of real numbers such that  $\sqrt{ \log(\| \theta \|_1 )} \ll x_n \ll \| \theta \|_1$. Mimicking the argument in \cite[Supplement,pg.121]{JinKeLuo21}, it holds that 
\begin{align*}
	\E[Y^2] \les \Bigl(\frac{x^2_n}{\|\theta\|_1^2}\Bigr)^{N^{(2)}_{\epsilon}+N^{(3)}_{\epsilon}}\cdot \mathbb{E}[X^2]  + o(1),
\end{align*}
By  \eqref{eqn:assn4}, there exists a sequence $\log( \| \theta \|_1 ) \ll x_n \ll \| \theta \|_1 /\| \theta \|_2^2$. Hence,
\begin{align*}
	\E[Y^2] \les \Bigl(\frac{1}{\|\theta\|_2^4}\Bigr)^{N^{(2)}_{\epsilon}+N^{(3)}_{\epsilon}}\cdot \mathbb{E}[X^2]  + o(1),
	\num \label{eqn:EY2_bd}
\end{align*}
Thus we focus on controlling $\E[ X^2]$. 

Consider a new random variable $X^*$ defined to be
\begin{align*}
	X^* \equiv \sum_{i,j,k,\ell (dist)} a_{ij}^* b_{jk}^* c_{k\ell}^* d_{\ell i}^*
\end{align*}
where 
\begin{align*}
	a^* &= 
	\begin{cases}
		\frac{ 1 }{\| \theta \|_2^2} \cdot \theta \theta^\T  &\quad \text{ if } a = \epsilon\rp{1} 
		\\
		\theta \theta^\T &\quad \text{ if } a \in \{ \ti \Omega, \eta \eta^\T \}
		\\
		a &\quad \text{ otherwise }
	\end{cases}  
	\\ b^* &= 
	\begin{cases}
		\frac{ 1 }{\| \theta \|_2^2} \cdot \theta \theta^\T  &\quad \text{ if } b = \epsilon\rp{1}
		\\
		\theta \theta^\T &\quad \text{ if } b \in \{ \ti \Omega, \eta \eta^\T \}
		\\
		b &\quad \text{ otherwise }
	\end{cases}  
	\\ c^* &= 
	\begin{cases}
		\frac{ 1 }{\| \theta \|_2^2} \cdot \theta \theta^\T  &\quad \text{ if } c = \epsilon\rp{1} 
		\\
		\theta \theta^\T &\quad \text{ if } c \in \{ \ti \Omega, \eta \eta^\T \}
		\\
		c &\quad \text{ otherwise }
	\end{cases}  
	\\ d^* &= 
	\begin{cases}
		\frac{ 1 }{\| \theta \|_2^2} \cdot \theta \theta^\T  &\quad \text{ if } d = \epsilon\rp{1}
		\\
		\theta \theta^\T &\quad \text{ if } \in \{ \ti \Omega, \eta \eta^\T \}
		\\
		d &\quad \text{ otherwise }. 
	\end{cases}  
\end{align*}

Also define 
\[
\ti X = \sum_{ijk\ell (dist)} \ti a_{ij} \ti b_{jk} \ti c_{k\ell} \ti d_{\ell i}
\]
where 
\begin{align*}
	\ti a &= 
	\begin{cases}
		\theta \theta^\T  &\quad \text{ if } a \in \{\epsilon\rp{1} , \ti \Omega, \eta \eta^\T \} 
		\\
		a &\quad \text{ otherwise }
	\end{cases}  
	\\ \ti b &= 
	\begin{cases}
		\theta \theta^\T  &\quad \text{ if } b  \in \{\epsilon\rp{1} , \ti \Omega, \eta \eta^\T \} 
		\\
		b &\quad \text{ otherwise }
	\end{cases}  
	\\ \ti c &= 
	\begin{cases}
		\theta \theta^\T  &\quad \text{ if } c \in \{\epsilon\rp{1} , \ti \Omega, \eta \eta^\T \}  
		\\
		c &\quad \text{ otherwise }
	\end{cases}  
	\\ \ti d &= 
	\begin{cases}
		\theta \theta^\T  &\quad \text{ if } d \in  \{\epsilon\rp{1} , \ti \Omega, \eta \eta^\T \}
		\\
		d &\quad \text{ otherwise }. 
	\end{cases}  
\end{align*}
Note that $X^* = \big( \frac{1}{\| \theta \|_2^2 } \big)^{N_{\epsilon}\rp{1}} \ti X$ and $\ti a, \ti b, \ti c, \ti d \in \{\theta \theta^\T,  W, \delta, (V/v)\ti r \}$. Later we show that
\begin{align*}
	\E[ X^2] \les \E[X^{*2}]
	\num \label{eqn:XX*_comparison}
\end{align*}

First we bound $\E[\ti X^2]$ in the case when $N_W + N_\delta + N_{\ti r} = 0$. Note that for all such terms in $Q - Q^*$, we have $N_\epsilon\rp{1} + N_\epsilon\rp{2} + N_\epsilon\rp{3} + N_{\ti \Omega} = 4$ and $N_{\ti \Omega} < 4$. In particular, $\ti X$ and $X^*$ are nonstochastic.  If $N_{\ti \Omega} = 3$, then by \eqref{eqn:etai_bd} and \eqref{eqn:beta_bds}, 
\begin{align*}
	|\ti X| &= \big| 
	\sum_{ijk\ell (dist) } \ti \Omega_{ij} 
	\ti \Omega_{jk} 	\ti \Omega_{k\ell} \theta_i \theta_\ell \big| \les 
	\frac{1}{\| \theta \|_2^2} 
	\sum_{ijk\ell} \beta_i \theta_i^2 \beta_j^2 \theta_j^2 \beta_k^2 \theta_k^2 \beta_\ell \theta_\ell^2
	\les \| \gam \|_2^6 \| \theta \|_2^2
\end{align*}
If $N_{\ti \Omega} = 2$, there are two cases. First,
\begin{align*}
	|\ti X| &=\big| \sum_{ijk\ell (dist)}
	\ti \Omega_{ij} 	\ti \Omega_{jk} 	\theta_k \theta_\ell \theta_\ell \theta_i  
	\big|
	\les \sum_{ijk\ell }
	\beta_i \theta_i \beta_j^2 \theta_j^2 \beta_k \theta_k^2 \theta_\ell^2 \theta_i 
	\les \| \gam \|_2^4 \| \theta \|_2^4,
\end{align*}
and second
\begin{align*}
	|\ti X| &=\big| 
	\sum_{ijk\ell (dist)} \ti \Omega_{ij} \theta_j \theta_k \ti \Omega_{k\ell} \theta_\ell \theta_i \big|
	\les 
	\sum_{ijk\ell } \beta_i \theta_i^2 \beta_j \theta_j^2  \beta_k \theta_k^2 \beta_\ell \theta_\ell^2 
	\les \| \gam \|_2^4 \| \theta \|_2^4
\end{align*}
Finally if $N_{\ti \Omega} = 1$, 
\begin{align*}
	|\ti X| &= \big| \sum_{ ijk\ell (dist)} 
	\ti \Omega_{ij} \theta_j \theta_k^2 \theta_\ell^2 \theta_i  \big| 
	\les \sum_{ ijk\ell } 
	\ti \beta_i \theta_i^2 \beta_j \theta_j^2  \theta_k^2 \theta_\ell^2 
	\les \| \gam \|_2^2 \| \theta \|_2^6. 
\end{align*}
Note that when $N_W + N_\delta + N_{\ti r} = 0$
\[
|X| \les |X^*|
\]
by \eqref{eqn:etai_bd}, \eqref{eqn:tiOm_bd}, and \eqref{eqn:epsilon1_bd}. By the bounds above, we conclude that
\begin{align*}
	|Y| \les \big( \frac{1}{\| \theta \|_2^2} \big)^{N_\eps\rp{1} + N_\eps\rp{2} + N_\eps\rp{3}} |\ti X| \les \max_{1 \leq k \leq 3}
	\| \gam \|_2^{2k} \| \theta \|_2^{2(4 -  k)}
	\les | \ti \lambda |^3. 
	\num \label{eqn:realSgnQ_nonstoch}
\end{align*}

Next we bound $\E[\ti X^2]$ in the case when $N_W + N_\delta + N_{\ti r} > 0$. By Lemma \ref{lem:tilde_Omega} and the definition of $f \in \mathbb{R}^2$ there,  we have $\ti \Omega_{ij} = \alpha_i \alpha_j \theta_i \theta_j $ where $\alpha  = \Pi f$. Observe that in Lemmas \ref{lem:ideal_SgnQ}--\ref{lem:real_SgnQ_tistarQ},  we bound the mean and variance of all terms of the form
\[
Z\equiv \sum_{i,j,k,\ell (dist)} a_{ij} b_{jk} c_{k\ell} d_{\ell i}, \qquad \mbox{where}\quad a, b, c, d\in \{\widetilde{\Omega}, W, \delta, (V/v) \tilde{r}\}.
\]
As a result, the proofs of Lemmas \ref{lem:ideal_SgnQ}--\ref{lem:real_SgnQ_tistarQ} produce a function $F$ such that 
\begin{align*}
	\E[ Z^2] \leq F(\theta, \beta; N_{\ti \Omega}, N_W,  N_\delta, N_{\ti r}),
\end{align*}
where recall that $|\alpha_i| \leq \beta_i$.
%The resulting inequalities depend on $\theta$ and $\beta \circ \theta$. 

Note that in what follows, we use $'$ to denote a new variable rather than the transpose. As a direct corollary to the proofs of Lemmas \ref{lem:ideal_SgnQ}--\ref{lem:real_SgnQ_tistarQ}, if we define a new matrix $\ti \Omega' = \alpha'_i \alpha'_j \theta_i \theta_j$ where $\alpha'$ is a vector with a coordinate-wise bound of the form $|\alpha'_i| \leq \beta'_i$, then 
\[
Z' \equiv \sum_{i,j,k,\ell (dist)} a_{ij} b_{jk} c_{k\ell} d_{\ell i}, \qquad \mbox{where}\quad a, b, c, d\in \{\widetilde{\Omega}', W, \delta, (V/v)\tilde{r}\}
\]
satisfies 
\begin{align*}
	\E[Z^{'2}] \leq F(\theta, \beta'; N_{\ti \Omega'}', N_W',  N_\delta', N_{\ti r}'),
\end{align*}
where, for example, $N'_\delta$ counts the number of appearances of $\delta$ in $Z'$. This can be verified by tracing each calculation in Lemmas \ref{lem:ideal_SgnQ}--\ref{lem:real_SgnQ_tistarQ} line by line, replacing all occurences of $\tilde \Omega$ with $\tilde \Omega'$, and replacing every usage of the bound $|\alpha_i| \leq \beta_i$ with $|\alpha_i'| \leq \beta'_i$ instead. In other words, our proofs make no use of the specific value of $\alpha = \Pi f$. 

In particular, if $\alpha = \mf{1}$, then $\ti \Omega' = \theta \theta^\T$. In this case we may set $\beta = \mf{1}$. Observe that $\ti X$ has the form of $Z'$  with this choice of $\ti \Omega'$. Hence,
\begin{align*}
	\E[\ti X^2] \leq F(\theta, \mf{1}; \ti N_{\ti \Omega'}, \ti N_W,  \ti N_\delta, \ti N_{\ti r}). 
	\num \label{eqn:X2_F_bd}
\end{align*}
By careful inspection of the bounds in Lemmas \ref{lem:ideal_SgnQ}--\ref{lem:real_SgnQ_tistarQ}, we see that
\begin{align*}
	F(\theta, \mf{1}; N_{\ti \Omega'}, N_W,  N_\delta, N_{\ti r}) \les \| \theta \|_2^{12}. 
	\num \label{eqn:F_bd}
\end{align*}

In  \cite[Supplement, Section G.3]{JinKeLuo21}  it is shown that all terms in the decomposition of $Q - Q^*$ satisfy $N_\epsilon\rp{1} + N_\epsilon\rp{2} + N_\epsilon\rp{3} > 0$. Using this fact along with \eqref{eqn:EY2_bd}, \eqref{eqn:XX*_comparison}, \eqref{eqn:X2_F_bd} and \eqref{eqn:F_bd}, 
\begin{align*}
	\E[Y^2] &\les \Bigl(\frac{1}{\|\theta\|_2^4}\Bigr)^{N^{(2)}_{\epsilon}+N^{(3)}_{\epsilon}}\cdot \big( \frac{1}{\| \theta \|_2^2 } \big)^{2N_{\epsilon}\rp{1}} \cdot  \mathbb{E}[\ti X^2]  + o(1)
	\les \| \theta \|_2^8. 
	\num \label{eqn:realSgnQ_stoch}
\end{align*}

Observe that \eqref{eqn:realSgnQ_nonstoch} and \eqref{eqn:realSgnQ_stoch} recover the bounds in Lemma \ref{lem:real_SgnQ_starQ} under the alternative hypothesis, and the bounds under the null hypothesis transfer directly from \cite[Lemma G.12]{JinKeLuo21}. Thus it only remains to justify \eqref{eqn:XX*_comparison} when $N_W + N_\delta + N_{\ti r} > 0$. Let us write 
\begin{align*}
	X &= \sum_{i_1, \ldots, i_m} c_{i_1, \ldots, i_m}
	G_{i_1, \ldots, i_m}
	\\ X^* &=  \sum_{i_1, \ldots, i_m} c^*_{i_1, \ldots, i_m}
	G_{i_1, \ldots, i_m}
\end{align*}
in the form described in Section \ref{sec:proof_strategy}, where now
\begin{itemize}
	\item $c_{i_1, \ldots, i_m} = \prod_{(s, s') \in A} \Gamma_{i_s, i_{s'}}\rp{s,s'} $ is a nonstochastic term where $A \subset [m] \times [m]$ and 
	\[\Gamma\rp{s, s'} \in \{ \ti \Omega, \eta^* \mf{1}^\T, \eta \mf{1}^\T , \mf{1} \mf{1}^\T,  \epsilon\rp{1} , \eta \eta^\T\}\]
	\item $c^*_{i_1, \ldots, i_m} = \prod_{(s, s') \in A} \Gamma_{i_s, i_{s'}}\rp{s,s'} $ is a nonstochastic term where $A \subset [m] \times [m]$ and 
	\[\Gamma\rp{s, s'} \in \{ \eta^* \mf{1}^\T, \eta \mf{1}^\T , \mf{1} \mf{1}^\T, \theta \theta^\T/ \| \theta \|_2^2,  \theta \theta^\T  \}\]
	\item $G_{i_1, \ldots, i_m} = \prod_{ (s, s') \in B } W_{i_s, i_{s'}}$ where $B \subset [m] \times [m]$.
\end{itemize}
If $\Gamma\rp{s, s'} \in \{ \theta \theta^\T ,   \theta \theta^\T/ \| \theta \|_2^2 \}$, we simply let $\overline{\Gamma\rp{s, s'}} = \Gamma\rp{s, s'}$ and define 
\[
\overline{c^*_{i_1, \ldots, i_m}} 
= \prod_{(s, s') \in A} \overline{ \Gamma_{i_s, i_{s'}}\rp{s,s'} }
\]
as in Section \ref{sec:proof_strategy}. We also define the canonical upper bound $\overline{\E X^*}$ on $| \E X^*|$ and the canonical upper bound $\overline{\var(X^*)}$ on $\var(X^*)$ similarly to Section \ref{sec:proof_strategy}. By the discussion above and \eqref{eqn:X2_F_bd},
\begin{align*}
	\overline{\E[X^*]} 
	\equiv \big( \frac{1}{\| \theta \|_2^2 } \big)^{N_{\epsilon}\rp{1}} \sqrt{ F(\theta, \mf{1}; \ti N_{\ti \Omega'}, \ti N_W,  \ti N_\delta, \ti N_{\ti r}) },
\end{align*}
and
\begin{align*}
	\overline{\var(X^*)} 
	\equiv \big( \frac{1}{\| \theta \|_2^2 } \big)^{2N_{\epsilon}\rp{1}} F(\theta, \mf{1}; \ti N_{\ti \Omega'}, \ti N_W,  \ti N_\delta, \ti N_{\ti r}). 
\end{align*}
Next observe that 
\begin{align*}
	|c_{i_1, \ldots, i_m} |
	\les 	|c^*_{i_1, \ldots, i_m} |
	\les |\overline{c^*_{i_1, \ldots, i_m}} |. 
\end{align*}
By a mild extension of Lemma \ref{lem:transfer} it follows that
\begin{align*}
	| \E X | &\les \overline{  \E X^*  }
	\\ \var( X) &\les \overline{ \var(X^*)},
\end{align*}
which verifies \eqref{eqn:XX*_comparison} and completes the proof.  \qed 	

\subsection{Calculations in the SBM setting}
\label{sec:SBM_calc}
We compute the order of $\lambda_1$ and $\ti \lambda_1 = \lambda_2$ in the SBM setting (which are the two nonzero eigenvalues of $\Omega$).  By basic algebra, 
$\lambda_1,\lambda_2$ are also the two nonzero eigenvalues of the following matrix
\begin{equation*}
	\left[
	\begin{array}{cc}
		N & 0 \\
		0 & n-N  
	\end{array} 
	\right]^{1/2}\times
	\left[
	\begin{array}{cc}
		a & b \\
		b & c  
	\end{array} 
	\right]
	\times
	\left[
	\begin{array}{cc}
		N & 0 \\
		0 & n-N  
	\end{array} 
	\right]^{1/2}= 		 	 \left[
	\begin{array}{cc}
		aN & \sqrt{N(n-N)}b \\
		\sqrt{N(n-N)}b & (n-N)c  
	\end{array} 
	\right],  
\end{equation*}
where $b$  is given by (\ref{model2}).
By direct calculations and pluging the definitions of $b$,
\begin{align*}
	\lambda_1=&\frac{aN+(n-N)c+\sqrt{(aN-(n-N)c)^2+4N(n-N)b^2}}{2}\\
	=&\frac{aN+(n-N)c+|(n-N)c-aN|\frac{n}{n-2N}}{2}. 
\end{align*}
Recall that 
\[
b = \frac{nc - N(a+c)}{n - 2 N}. 
\] 
It is required that $b \geq 0$. Therefore, 
\begin{equation}\label{condition1}
	n c - (a+c) N \geq 0,\qquad \text{and so}\qquad (n-N)c \geq aN.  
\end{equation}
%	This can be interpreted that $c$ can not be too small; otherwise, we can not have degree matching.
By direct calculations, it follows that
\[
\lambda_1=\frac{(n-N)^2c-aN^2}{n-2N}=\frac{(n-N)c((n-N)-\frac{aN}{(n-N)c}N)}{n-2N}\sim \frac{(n-N)c(n-N)}{n-2N}\sim nc
\]
where in the last two $\asymp$, we have used  $(n-N)c \geq aN$ and  $N=o(n)$.
Similarly,
\[
\lambda_2=\frac{aN+(n-N)c-\sqrt{(aN-(n-N)c)^2+4N(n-N)b^2}}{2}=\frac{(a-c)N(n-N)}{n-2N}\sim N(a-c). 
\]
%	In the case where $N\asymp n$,  the SNR of SgnQ test is shown to be 
%	$|\lambda_2| / \sqrt{\lambda_1}$.  We conjecture that this continues to be the case when 
%	$N = o(n)$. This gives the previous result on SgnQ.  

\section{Proof of Theorem \ref{thm:chi2} (Powerlessness of $\chi^2$ test)} 

%\textbf{Note:} In Theorem \ref{thm:chi2}, we require the following additional assumptions, which will later be added to the main text,
%\begin{enumerate}
%	\item $\alpha^2 n \to \infty$
%	\item $\sum_{ij} (\Omega_{ij} - \alpha)^2 = o( \alpha n^{3/2} )$. 
%\end{enumerate}

We compare the SgnQ test with the $\chi^2$ test. Recall we assume 
$\theta_i = \mf{1}_n$. The $\chi^2$ test statistic is defined to be
\[
X_n = \frac{1}{\hat \alpha (1 - \hat \alpha)(n-1) } \sum_{i = 1}^n \big( (A \mf{1}_n)_i - \hat \alpha n \big)^2, \qquad \text{ where } \hat \alpha = \frac{1}{n(n-1)} \sum_{i \neq j} A_{ij}. 
\]
We also define an idealized $\chi^2$ test statistic by
\[
\ti X_n = \frac{1}{ \alpha (1 -  \alpha)(n-1) } \sum_{i = 1}^n \big( (A \mf{1}_n)_i -  \alpha n \big)^2, \qquad \text{ where } \alpha = \frac{1}{n(n-1)} \sum_{i \neq j} \Omega_{ij} . 
\]
The $\chi^2$ test is defined to be
\[
\chi^2_n = \mf{1}\bigg[ \frac{ |X_n - n|}{\sqrt{2n}} > z_{\gamma/2} \bigg],
\]
where $z_\gamma$ is such that $\p[ |N(0,1)| \geq z_\gamma ] = \gamma$. Similarly, the idealized $\chi^2$ test is defined by
\[
\ti \chi^2_n = \mf{1}\bigg[ \frac{ |\ti X_n - n|}{\sqrt{2n}} > z_{\gamma/2} \bigg],
\]
In certain degree-homogeneous settings, the $\chi^2$ test is known to have full power \cite{Ery1, cammarata2022power}. 

We prove the following, which directly implies Theorem \ref{thm:chi2}. 
\begin{thm} \label{thm:chi2_sup} 
	%		Consider a class of null DCBMs with $K = 1$ of the form $\Omega = \alpha \mf{1} \mf{1}'$ with $\alpha \in (0,1)$, and a class of alternatives with the property that $\Omega \mf{1} = \alpha$		
	Suppose that \eqref{altconditions} holds and that $|\ti \lambda|/\sqrt{ \lambda_1} \to \infty$, and recall that under these conditions, the power of the SgnQ test goes to $1$. Next suppose that the following regularity conditions hold under the null and alternative: 
	\begin{itemize}
		\item[(i)] $\theta = \mf{1}_n$
		\item[(ii)] $\alpha \to 0$ 
		\item[(iii)] $\alpha^2 n \to \infty$
		\item[(iv)] $\sum_{ij} (\Omega_{ij} - \alpha)^2 = o( \alpha n^{3/2} )$. 
	\end{itemize}
	Then the power of both the $\chi^2$-test and idealized $\chi^2$-test goes to $\gamma$ (which is the prescribed level of the test).  
\end{thm} 

Note that the previous theorem implies Theorem \ref{thm:chi2}. By  Theorem \ref{thm:alt-SgnQ}, SgnQ has full power even without the extra regularity conditions (i)--(iv).  On the other hand, for any fixed alternative DCBM satisfying  (i)--(iv), Theorem \ref{thm:chi2_sup} implies that $\chi^2$ has power $\kappa$.  

\begin{proof}[Proof of Theorem \ref{thm:chi2_sup}]
	Theorem \ref{thm:alt-SgnQ} confirms that SgnQ  has full power provided that \eqref{altconditions} holds and that $|\ti \lambda|/\sqrt{ \lambda_1} \to \infty$. It remains to justify the powerlessness of the $\chi^2$ test. 
	
	Consider an SBM in the alternative such that $\Omega \mf{1} = (\alpha n) \mf{1}$ and $|\ti \lambda|/\sqrt{ \lambda_1} \asymp N(a - c)/\sqrt{nc} \to \infty$. To do this we select an integer $N > 0$ to be the size of the smaller community and set $b = \frac{ cn - (a + c)N }{n - 2N}$. The remaining regularity conditions are satisfied if $c \to 0$ and $cn \ll N(a - c)^2 \ll c n^{3/2}$. We show that both $X_n$ and $\ti X_n$ are asymptotically normal under the specified alternative, which is enough to imply Theorem \ref{thm:chi2_sup}.
	
	In \cite{cammarata2022power} it is shown that
	\beq \label{chi1-equivalent} 
	\hat T_n \equiv [(n-1) \hat{\alpha} (1 - \hat \alpha)](X_n-n) = \sum_{i,j,k \text{ (dist.)}}(A_{ik}-\hat{\alpha})(A_{jk}-\hat{\alpha}). 
	\eeq 
	
	We introduce an idealized version $T_n$ of $\hat T_n$, which is
	\begin{align*}
		T_n = \sum_{i,j,k \text{ (dist.)}}(A_{ik}-{\alpha})(A_{jk}- {\alpha}),
	\end{align*}
	
	Following \cite{cammarata2022power}, we have 
	\begin{equation} \label{chi2-proof-2}
		\frac{X_n-n}{\sqrt{2n}}=\left(\frac{n-2}{n-1}\right)^{1/2}U_nV_nZ_n. 
	\end{equation}
	where 
	\begin{align*}
		U_n = \frac{\alpha_n(1-\alpha_n)}{\hat{\alpha}_n(1-\hat{\alpha}_n)} & \text{,} & V_n = \frac{\hat{T}_n}{T_n} & \text{,} & Z_n=\frac{\frac{T_n}{(n-1)\alpha_n(1-\alpha_n)}}{\sqrt{\frac{2n(n-2)}{(n-1)}}}.
	\end{align*}
	
	Since the terms of $\hat \alpha$ are bounded,  the law of large numbers implies that $U_n \stackrel{\p}{\goto} 1$. Furthermore, since $\alpha n \to \infty$ by assumption that $\alpha^2 n \to \infty$,
	a straightforward application of the Berry-Esseen theorem implies that 
	\begin{equation*}
		\sqrt{\frac{n(n-1)}{2}}\frac{\hat{\alpha}_n-\alpha_n}{\sqrt{\alpha_n(1-\alpha_n)}} \Rightarrow \mathcal{N}( \mu ,1). 
	\end{equation*}
	
	With the previous fact, mimicking the argument in \cite[pg.32]{cammarata2022power}, it also follows that 
	\[
	V_n \stackrel{\p}{\goto} 1,
	\]
	\textit{provided we can show that} $Z_n \Rightarrow N(0,1)$. We omit the details since the argument is very similar.
	
	Thus it suffices to study $Z_n$. We first analyze $T_n$, which we decompose as
	\begin{align*}
		T_n &= 
		\sum_{i,j,k \text{ (dist.)}}(A_{ik}-\Omega_{ik})(A_{jk}-\Omega_{jk})
		+ 2 \sum_{ijk (dist)} (\Omega_{ik}- \alpha)(A_{jk}-\Omega_{jk})
		\\&\quad +  \sum_{ijk (dist)} (\Omega_{ik} - \alpha) (\Omega_{jk} - \alpha)
		\equiv T_{n1} + T_{n2} + T_{n3}.
	\end{align*}
	
	Observe that $T_{n3}$ is non-stochastic. The second and third term are negligible compared to $T_{n1}$. Define $\overline{\Omega} = \Omega - \alpha \mf{1} \mf{1}'$. By direct calculations,
	\begin{align*}
		\E T_{n2} =0, \qquad 
	\end{align*}
	and
	\begin{align*}
		\var(T_{n2}) &= 8\sum_{j < k (dist)} \big( \sum_{i \notin \{ j, k \}} \overline{\Omega}_{ik} \big)^2 \Omega_{jk} (1 - \Omega_{jk}) 
		= 8\sum_{j < k (dist)} \big( \overline{\Omega}_{jk} + \overline{\Omega}_{kk} \big)^2 \Omega_{jk} (1 - \Omega_{jk})
		\les \alpha n^2. 
	\end{align*}
	
	Next,
	\begin{align*}
		|T_{n3}|  &= \big|  \sum_{ijk} \overline{\Omega}_{ik} \overline{\Omega}_{jk} 
		- \sum_{ijk (not \, dist.)} \overline{\Omega}_{ik} \overline{\Omega}_{jk} \big| 
		= \big|  \sum_{ijk (not \, dist.)} \overline{\Omega}_{ik} \overline{\Omega}_{jk} \big| 
		\\&\les \big| \sum_{ij} \overline{ \Omega}_{ii}  \overline{ \Omega}_{ji} \big| 
		+\big|  \sum_{ik} \oO_{ik}^2\big|  + \big|  \sum_i \overline{\Omega}_{ii}^2 \big| 
		= 0 + o( \alpha n^{3/2} ) + n = o( \alpha n^{3/2} ),
	\end{align*}
	where we apply the third regularity condition.  
	
	Now we focus on $T_{n1}$. By direct calculations
	\begin{align*}
		\E T_{n1} = 0,
	\end{align*}
	and
	\begin{align*}
		\var \, T_{n1} &= 
		2	\sum_{i, j, k (dist) } \Omega_{ik}(1 -\Omega_{ik}) \Omega_{jk} (1 - \Omega_{jk})
		\\&= 2\sum_{i, j, k  } \Omega_{ik}(1 -\Omega_{ik}) \Omega_{jk} (1 - \Omega_{jk})-
		2\sum_{i, j, k (not \, dist.) } \Omega_{ik}(1 -\Omega_{ik}) \Omega_{jk} (1 - \Omega_{jk})
		\\&= 2 \mf{1}' \Omega^2 \mf{1} 
		- 2	\sum_{i, j, k (not \, dist.) } \Omega_{ik}(1 -\Omega_{ik}) \Omega_{jk} (1 - \Omega_{jk})
		%	\\&\sim 2 n (n-1) (n-2) \alpha^2 
		%	- 	2\sum_{i, j, k (not \, dist.) } \Omega_{ik}(1 -\Omega_{ik}) \Omega_{jk} (1 - \Omega_{jk})
	\end{align*}
	Note that
	\begin{align*}
		2 \mf{1}' \Omega^2 \mf{1} \sim 2n(n-1)(n-2) \alpha^2 
	\end{align*}
	since $\alpha \to 0$. Moreover, with some simple casework we can show
	\begin{align*}
		\sum_{i, j, k (not \, dist.) } \Omega_{ik}(1 -\Omega_{ik}) \Omega_{jk} (1 - \Omega_{jk})
		\les \alpha n^2 = o( \alpha^2 n^3 ),
	\end{align*}
	where we use that $\alpha n \goto\infty$ (because $\alpha^2 n \goto\infty$). Hence
	\begin{align*}
		\var \, T_{n1} \sim 2 n (n-1) (n-2) \alpha^2 (1 - \alpha)^2 \sim 
		2 n (n-1) (n-2) \alpha^2 (1 - \alpha)^2. 
	\end{align*}
	
	To study $T_{n1}$ we apply the martingale central limit theorem using a similar argument to \cite{cammarata2022power}). Define $W_{ij} = A_{ij} - \Omega_{ij}$ and 
	\begin{align*} 
		T_{n,m}&=\sum_{(i,j,k)\in I_m}W_{ik}W_{jk}, \qquad \text{ and } \qquad T_{n,0}=0,\cr
		Z_{n,m}&=\sqrt{\frac{n-1}{2n(n-2)}}\frac{T_{n,m}}{(n-1)\alpha_n(1-\alpha_n)}, \qquad \text{ and } \qquad Z_{n,0}=0.
	\end{align*}
	where 
	\begin{equation*}
		I_m=\{(i,j,k)\in [ m ]^3 \text{ s.t. $i,j,k$ are distinct}\},
	\end{equation*}
	and $m \leq n$. Define a filtration $\{ \mc{F}_{n,m} \}$ where  $\mathcal{F}_{n,m}=\sigma\{W_{ij}, (i,j)\in[m]^2\}$ for all $m\in[n]$, and let  $\mathcal{F}_{n,0}$ be the trivial $\sigma$-field. It is straightforward to verify that $T_{n,m}$  and $Z_{n,m}$ are martingales with respect to this filtration. We further define a martingale difference sequence
	\[
	X_{n,m} = Z_{n,m} - Z_{n,m-1}
	\]
	for all $m \in [n]$.
	
	If we can show that the following conditions hold
	\begin{align}
		\text{(a) }& \sum_{m=1}^n\mathbb{E}[X_{n,m}^2|\mathcal{F}_{n,m-1}]\xrightarrow{\mathbb{P}}1 \label{chi2-proof-5},\\
		\text{(b) }& \forall\epsilon>0, \sum_{m=1}^n\mathbb{E}[X_{n,m}^2\mf{1}\{|X_{n,m}>\epsilon|\}|\mathcal{F}_{n,m-1}]\xrightarrow{\mathbb{P}}0\label{chi2-proof-6},
	\end{align}
	then the Martingale Central Limit Theorem implies that $Z_n\Rightarrow\mathcal{N}(0,1)$.
	
	Our argument follows closely \cite{cammarata2022power}. First consider \eqref{chi2-proof-5}. It suffices to show that 
	\begin{equation}     \label{a0}
		\E\left[\sum_{m=1}^n\mathbb{E}[X_{n,m}^2|\mathcal{F}_{n,m-1}]\right] \xrightarrow{n \to \infty} 1,
	\end{equation}
	and
	\begin{equation}     \label{b0}
		\var\left(\sum_{m=1}^n\mathbb{E}[X_{n,m}^2|\mathcal{F}_{n,m-1}]\right) \xrightarrow{n\to\infty}0. 
	\end{equation}
	For notational brevity, define
	\begin{equation*}
		C_n:=(n-1)\alpha_n(1-\alpha_n)\sqrt{\frac{2n(n-2)}{n-1}}.
	\end{equation*}
	
	Mimicking the argument in \cite[pgs.33-34]{cammarata2022power} shows the following. Note that all sums below are indexed up to $m- 1$.
	\begin{align*} 
		\E[C_n^2X_{n,m}^2|&\mathcal{F}_{n,m-1}] = 4\sum_{k\neq j;\; i\neq l}W_{jk}W_{il}\E\left[W_{mk}W_{mi}\right]
		+ 4\sum_{k\neq j;\; i\neq l}W_{jk}\E\left[W_{im}W_{km}W_{lm}\right]
		\\&+ \sum_{i\neq j;\; k\neq l}\E\left[W_{im}W_{jm}W_{km}W_{lm}\right].
		\num \label{chi2-proof-7}
	\end{align*}
	Continuing, we have
	\begin{align*}
		\E[C_n^2X_{n,m}^2|\mathcal{F}_{n,m-1}] &= 
		4\sum_{i}\sum_{j\neq i,l\neq i}W_{ij}W_{il} 
		\Omega_{mi}(1-\Omega_{mi})
		+ 2\sum_{i, j (dist)} \Omega_{im}(1 - \Omega_{im}) \Omega_{jm} ( 1 - \Omega_{jm} )
		\\&= 4\sum_{ij\ell (dist)} W_{ij}W_{il} 
		\Omega_{mi}(1-\Omega_{mi}) 
		+ 4 \sum_{i, j (dist)} W_{ij}^2 \Omega_{mi} (1 - \Omega_{mi})
		\\&\qquad + 2\sum_{i, j (dist)} \Omega_{im}(1 - \Omega_{im}) \Omega_{jm} ( 1 - \Omega_{jm} ). 
		\num \label{eqn:condl_2m}
	\end{align*}
	Computing expectations,
	\begin{align*}
		\E[	\E[C_n^2 X_{n,m}^2&|\mathcal{F}_{n,m-1}] ]
		\\  &= 4 \sum_{i,j (dist)} \Omega_{ij}(1 - \Omega_{ij}) \Omega_{mi} (1 - \Omega_{mi})
		+ 2\sum_{i, j (dist)} \Omega_{im}(1 - \Omega_{im}) \Omega_{jm} ( 1 - \Omega_{jm} )
	\end{align*}
	Summing over $m$ and a simple combinatorial argument yields
	\begin{align*}
		C_n^2 \E\big[ \sum_{m = 1}^n	\E[X_{n,m}^2|\mathcal{F}_{n,m-1}] \big ]
		&= 2 \sum_{i, j, k (dist)} \Omega_{ik}(1 -\Omega_{ik}) \Omega_{jk} (1 - \Omega_{jk}) \sim C_n^2. 
	\end{align*}
	
	Using the identity
	\[
	W_{ij}^2 = (1 - 2 \Omega_{ij}) W_{ij} + \Omega_{ij}(1 - \Omega_{ij} ),
	\]
	we 
	have
	\begin{align*}
		\E[C_n^2X_{n,m}^2|\mathcal{F}_{n,m-1}] &= 
		4\sum_{ij\ell (dist)} W_{ij}W_{il} 
		\Omega_{mi}(1-\Omega_{mi}) 
		+ 4 \sum_{i, j (dist)} W_{ij}^2 \Omega_{mi} (1 - \Omega_{mi})
		\\&= 	24 \sum_{i < j < \ell} W_{ij}W_{il} 
		\Omega_{mi}(1-\Omega_{mi}) 
		+ 8 \sum_{i<j} W_{ij} (1 - 2 \Omega_{ij}) \Omega_{mi} (1 - \Omega_{mi})
		\\ &\quad + 4 \sum_{i < j} \Omega_{ij}(1 - \Omega_{ij}) \Omega_{mi} (1 - \Omega_{mi}). 
	\end{align*}
	Thus
	\begin{align*}
		\sum_{m=1}^n\mathbb{E}[C_n^2 X_{n,m}^2|\mathcal{F}_{n,m-1}]
		&= 24 \sum_{i < j < \ell} \big( \sum_{m > \max(i,j,\ell) } \Omega_{mi}(1 - \Omega_{mi}) \, \,\big ) W_{ij} W_{i\ell} 
		\\ &\quad + 8 \sum_{i < j} \big( \sum_{m > \max(i,j,\ell) } \Omega_{mi}(1 - \Omega_{mi}) \, \,\big) (1 - 2 \Omega_{ij}) W_{ij}. 
	\end{align*}
	All terms above are uncorrelated. Hence,
	\begin{align*}
		\var\left(\sum_{m=1}^n\mathbb{E}[C_n^2 X_{n,m}^2|\mathcal{F}_{n,m-1}]\right)
		&=  24^2 \sum_{i < j < \ell}  \big( \sum_{m > \max(i,j,\ell) } \Omega_{mi}(1 - \Omega_{mi}) \, \,\big )^2 \Omega_{ij}(1  - \Omega_{ij}) \Omega_{i\ell} (1 - \Omega_{i\ell}) 
		\\&\quad + 64 \sum_{i<j} \big( \sum_{m > \max(i,j,\ell) } \Omega_{mi}(1 - \Omega_{mi}) \, \,\big )^2 (1 - 2 \Omega_{ij})^2 \Omega_{ij}(1 - \Omega_{ij})
		\\&\les n^2 \cdot C_n^2, 
	\end{align*} 
	whence,
	\begin{align*}
		\var\left(\sum_{m=1}^n\mathbb{E}[ X_{n,m}^2|\mathcal{F}_{n,m-1}]\right)
		&\les \frac{n^2}{C_n^2} \asymp \frac{n^2}{\alpha^2 n^3} \to 0
	\end{align*}
	since $\alpha^2  n \to \infty$. Thus we have shown \eqref{a0} and \eqref{b0}, which together prove \eqref{chi2-proof-5}.
	
	Next we prove \eqref{chi2-proof-6}, again following the argument in \cite{cammarata2022power}. In \cite[pg.36]{cammarata2022power} it is shown that it suffices to prove 
	\begin{equation} \label{chi2-proof-10}
		\sum_{m=1}^n\mathbb{E}[X_{n,m}^4]  \xrightarrow{n\to\infty} 0.
	\end{equation}
	Further in \cite[pg.37]{cammarata2022power}, it is shown that 
	\begin{align*}
		\E[C_n^4X_{n,m}^4]=& 16\biggl[\sum_{i<j}\E[W_{jm}^4]\E[(W_{ij}+W_{im})^4]\cr
		& +3\sum_{\substack{i<j,u<v\\i\neq u, j\neq v}}\E[W_{jm}^2]\,\E[(W_{ij}+W_{im})^2]\,\E[W_{vm}^2]\,\E[(W_{uv}+W_{um})^2]\\
		&+3\sum_{\substack{i<j,v\\j\neq v}}\E[W_{jm}^2]\,\E[W_{vm}^2]\,\E[(W_{ij}+W_{im})^2(W_{iv}+W_{im})^2]\\
		&+3\sum_{\substack{i,u<j\\i\neq u}}\E[(W_{ij}+W_{im})^2]\,\E[(W_{uj}+W_{um})^2]\, \E[W_{jm}^4]\biggr].
	\end{align*}
	Going through term by term, we have for $n$ sufficiently large
	\begin{align*}
		\sum_{i<j}\E[W_{jm}^4]\E[(W_{ij}+W_{im})^4]
		&\les \sum_{i,j} \Omega_{jm} \big( \Omega_{ij}+ \Omega_{im} \big)
		\les \alpha^2 n^2
		%\les \mf{1}' \Omega^2 \mf{1} \les \alpha^2 n^3.
	\end{align*}
	Next
	\begin{align*}
		\sum_{\substack{i<j,u<v\\i\neq u, j\neq v}}\E[W_{jm}^2]\,&\E[(W_{ij}+W_{im})^2]\,\E[W_{vm}^2]\,\E[(W_{ij}+W_{im})^2]
		\les \sum_{ijuv} \Omega_{jm} ( \Omega_{ij}+ \Omega_{jm} )
		\Omega_{vm} (\Omega_{uv}+ \Omega_{um}) 
		\\&=  \sum_{ijuv} \Omega_{jm} \Omega_{ij}
		\Omega_{vm} \Omega_{uv} 
		+  \sum_{ijuv} \Omega_{jm}  \Omega_{ij} 
		\Omega_{vm} \Omega_{um}
		+  \sum_{ijuv} \Omega_{jm}^2
		\Omega_{vm}  \Omega_{uv}
		\\& \quad +  \sum_{ijuv} \Omega_{jm}^2
		\Omega_{vm} \Omega_{um}
		\\&\les \alpha^4 n^4 + \alpha^3 n^3
	\end{align*}
	With a similar argument, we also have, for $n$ sufficiently large, 
	\begin{align*}
		\sum_{\substack{i<j,v\\j\neq v}}\E[W_{jm}^2]\,\E[W_{vm}^2]\,\E[(W_{ij}+W_{im})^2(W_{iv}+W_{im})^2] &\les \alpha^2 n^2 + \alpha^3 n^3
		\\ \sum_{\substack{i,u<j\\i\neq u}}\E[(W_{ij}+W_{im})^2]\,\E[(W_{uj}+W_{um})^2]\, \E[W_{jm}^4]\biggr] &\les \alpha^3 n^3 + \alpha^2 n^2. 
	\end{align*} 
	Thus
	\begin{align*}
		\sum_{m=1}^n\mathbb{E}[X_{n,m}^4] \les \frac{ \alpha^4 n^5 }{C_n^4}
		\sim  \frac{ \alpha^4 n^5 }{ \alpha^4 n^6}  \goto 0, 
	\end{align*}
	which verifies \eqref{chi2-proof-10}. Since \eqref{chi2-proof-10} implies \eqref{chi2-proof-6}, this completes the proof.
\end{proof}

\section{Proof of Theorem~\ref{thm:statLB1} (Statistical lower bound)}

Let $f_0(A)$ be the density under the null hypothesis. Let $\mu(\Pi)$ be the density of $\Pi$, and let $f_1(A|\Pi)$ be the conditional density of $A$ given $\Pi$. The $L_1$ distance between two hypotheses is 
\[
\ell^*\equiv \frac{1}{2} \mathbb{E}_{A\sim f_0}\bigl| \mathbb{E}_{\Pi\sim \mu}L(A, \Pi) -1 \bigr|, \qquad L(A, \Pi) = f_1(A|\Pi)/f_0(A). 
\]
Define 
\beq \label{LBproof1-eq1}
{\cal M}=\bigl\{\Pi: \mbox{$\Pi$ is an eligible membership matrix and $\sum_{i}\pi_i(1)\leq 2n\epsilon$} \bigr\}. 
\eeq
Write $L^{{\cal M}}(A,\Pi)=L(A,\Pi)\cdot 1\{\Pi\in {\cal M}\}$ and define $L^{{\cal M}^c}(A,\Pi)$ similarly.  
By direct calculations, we have
\begin{align} \label{LBproof1-eq2}
	\ell^*&= \frac{1}{2} \mathbb{E}_{A\sim f_0}\bigl| \mathbb{E}_{\Pi\sim \mu}L^{\cal M}(A, \Pi) -1 + \mathbb{E}_{\Pi\sim \mu}L^{{\cal M}^c}(A, \Pi) \bigr|\cr
	&\leq \frac{1}{2} \mathbb{E}_{A\sim f_0}\bigl| \mathbb{E}_{\Pi\sim \mu}L^{\cal M}(A, \Pi) -1\bigr| +\frac{1}{2}\mathbb{E}_{A\sim f_0} \mathbb{E}_{\Pi\sim \mu}L^{{\cal M}^c}(A, \Pi)\cr
	&\equiv \frac{1}{2}\ell_0+\frac{1}{2}\ell_1. 
\end{align} 
Note that $\mathbb{E}_{A\sim f_0} \mathbb{E}_{\Pi\sim \mu}L^{{\cal M}^c}(A, \Pi)=\int_{\Pi\in{\cal M}^c} f_1(A|\Pi)\mu(\Pi)d\Pi dA=\int_{\Pi\in{\cal M}^c}\mu(\Pi)d\Pi=\mu({\cal M}^c)$. We bound the probability of $\mu\in {\cal M}^c$. Note that $\pi_i(1)$ are independent Bernoulli variables with mean $\epsilon$, where $\epsilon\asymp n^{-1}N$. It follows by Bernstein inequality that if $t = 100\sqrt{N \log N}$, the we have conservatively, 
\begin{align*}
	\mathbb{P}\Bigl(\Bigl|\sum_{i}\pi_i(1)-N\Bigr| > t\Bigr)\leq 2\exp\biggl(-\frac{t^2/2}{n\eps + t/3}\biggr) \leq 2\exp\biggl(-\frac{100^2 N (\log N) /2}{200N}\biggr) \les N^{-c} = o(1)
	\num \label{eqn:clique_size_concentration}
\end{align*}
for some $c > 0$. 
%=======
%	\[
%	\mathbb{P}\Bigl(\Bigl|\sum_{i}\pi_i(1)-n\epsilon\Bigr|\Bigr)\leq 2\exp\biggl(-\frac{(n\epsilon)^2/2}{n\epsilon + n\epsilon/3}\biggr)\leq 2\exp(-3n\epsilon/8)\leq 2\exp(-CN)=o(1). 
%	\]
It follows that 
\beq  \label{LBproof1-eq3}
\ell_1=\mu({\cal M}^c) = o(1). 
\eeq
By Cauchy-Schwarz inequality, 
\begin{align*}
	\ell_0^2 & \leq \mathbb{E}_{A\sim f_0}\bigl| \mathbb{E}_{\Pi\sim \mu}L^{\cal M}(A, \Pi) -1\bigr|^2\cr
	&= \mathbb{E}_{A\sim f_0}\bigl( \mathbb{E}_{\Pi\sim \mu}L^{\cal M}(A, \Pi))^2  -2\mathbb{E}_{A\sim f_0}\mathbb{E}_{\Pi\sim \mu}L^{{\cal M}}(A, \Pi) + 1\cr
	&= \mathbb{E}_{A\sim f_0}\bigl( \mathbb{E}_{\Pi\sim \mu}L^{\cal M}(A, \Pi))^2  -2\bigl[ 1 - \mathbb{E}_{A\sim f_0}\mathbb{E}_{\Pi\sim \mu}L^{{\cal M}^c}(A, \Pi)\bigr] + 1\cr
	&\leq \mathbb{E}_{A\sim f_0}\bigl( \mathbb{E}_{\Pi\sim \mu}L^{\cal M}(A, \Pi))^2 -1 +o(1), 
\end{align*}
where the third line is from $\mathbb{E}_{A\sim f_0}\mathbb{E}_{\Pi\sim \mu}L(A, \Pi)=1$ and the last line is from \eqref{LBproof1-eq3}. We plug it into \eqref{LBproof1-eq2} to get
\beq \label{LBproof1-eq}
\ell^*\leq \sqrt{\ell_2-1} + o(1), \qquad\mbox{where}\quad \ell_2\equiv \mathbb{E}_{A\sim f_0}\bigl( \mathbb{E}_{\Pi\sim \mu}L^{{\cal M}}(A, \Pi))^2. 
\eeq 
It suffices to prove that $\ell_2\leq 1+o(1)$. %We remark that $\ell^2$ is a truncated $\chi^2$-distance between two hypotheses. 

Below, we study $\ell_2$. Let $\widetilde{\Pi}$ be an independent copy of $\Pi$. Define
\[
S(A, \Pi, \widetilde{\Pi}) = L(A, \Pi)\cdot L(\widetilde{\Pi}, A). 
\]
It is easy to see that
\beq  \label{LBproof1-eq4}
\ell_2 =   \mathbb{E}_{A\sim f_0, \Pi,\tilde{\Pi}\sim \mu}\bigl[S(A,\Pi,\widetilde{\Pi})\cdot 1\{\Pi\in {\cal M}, \widetilde{\Pi}\in {\cal M}\}\bigr].  
\eeq
Denote by $p_{ij}$ and $q_{ij}(\Pi)$ the values of $\Omega_{ij}$ under the null and the alternative, respectively. Write $\delta_{ij}(\Pi)=(q_{ij}(\Pi)-p_{ij})/p_{ij}$. 
By definition, 
\[
S(A, \Pi, \widetilde{\Pi}) = \prod_{i<j} \left[\frac{q_{ij}(\Pi)q_{ij}(\widetilde{\Pi})}{p_{ij}^2}\right]^{A_{ij}}\left[\frac{(1-q_{ij}(\Pi))(1-q_{ij}(\widetilde{\Pi}))}{(1-p_{ij})^2}\right]^{1-A_{ij}}. 
\]
Write for short $q_{ij}(\Pi)=q_{ij}$, $q_{ij}(\widetilde{\Pi})=\tilde{q}_{ij}$, $\delta_{ij}(\Pi)=\delta_{ij}$ and $\delta_{ij}(\widetilde{\Pi})=\tilde{\delta}_{ij}$. By straightforward calculations, we have the following claims:
\beq \label{LBproof1-eq7}
\mathbb{E}_{A\sim f_0}[S(A,\Pi,\widetilde{\Pi})] = \prod_{i<j}\Bigl(1+\frac{p_{ij}\delta_{ij}\tilde{\delta}_{ij}}{1-p_{ij}}\Bigr), 
\eeq
and
\begin{align}  \label{LBproof1-eq8}
	\ln S(A, \Pi, \widetilde{\Pi}) & = \sum_{i<j}A_{ij}\ln\biggl[\frac{(1+\delta_{ij})(1+\tilde{\delta}_{ij})}{(1-\frac{p_{ij}}{1-p_{ij}}\delta_{ij})(1-\frac{p_{ij}}{1-p_{ij}}\tilde{\delta}_{ij})}\biggr] \cr
	&\qquad + \ln \biggl[ \Bigl(1-\frac{p_{ij}}{1-p_{ij}}\delta_{ij}\Bigr)\Bigl(1-\frac{p_{ij}}{1-p_{ij}}\tilde{\delta}_{ij}\Bigr) \biggr]. 
\end{align}
The expression \eqref{LBproof1-eq8} may be useful for the case of $Nc\to 0$. In the current case of $Nc\to\infty$, we use \eqref{LBproof1-eq7}. It follows from \eqref{LBproof1-eq4} that
\begin{align} \label{LBproof1-eq9}
	\ell_2 &= \mathbb{E}_{\Pi,\tilde{\Pi}\sim \mu}\biggl[ \prod_{i<j}\Bigl(1+\frac{p_{ij}\delta_{ij}\tilde{\delta}_{ij}}{1-p_{ij}}\Bigr)  \cdot 1\{\Pi\in {\cal M}, \widetilde{\Pi}\in {\cal M}\}\biggr]\cr
	&=\mathbb{E}_{\Pi,\tilde{\Pi}\sim \mu}\biggl[ \exp\biggl( \sum_{i<j}\ln\Bigl(1+\frac{p_{ij}\delta_{ij}\tilde{\delta}_{ij}}{1-p_{ij}}\Bigr)\biggr)  \cdot 1\{\Pi\in {\cal M}, \widetilde{\Pi}\in {\cal M}\}\biggr]\cr
	&\leq \mathbb{E}_{\Pi,\tilde{\Pi}\sim \mu}\biggl[ \exp(X)  \cdot 1\{\Pi\in {\cal M}, \widetilde{\Pi}\in {\cal M}\}\biggr], \quad\mbox{with}\;\; X \equiv \sum_{i<j}\frac{p_{ij}\delta_{ij}\tilde{\delta}_{ij}}{1-p_{ij}}. 
\end{align}
where the last line is from the universal inequality of $\ln(1+t)\leq t$. 

We further work out the explicit expressions of $p_{ij}$, $\delta_{ij}$ and $\tilde{\delta}_{ij}$. Let $h=(\epsilon, 1-\epsilon)'$, and recall that $\alpha_0=a\epsilon+b(1-\epsilon)$. The condition of $b$ in \eqref{model2} guarantees that 
\[
Ph=\alpha_0 {\bf 1}_2, \qquad \alpha_0=a\epsilon+b(1-\epsilon). 
\]
By direct calculations, 
\beq \label{LBproof1-add}
\alpha_0 = \frac{c(1-\epsilon)^2-a\epsilon^2}{1-2\epsilon}. 
\eeq
It follows that
\beq \label{LBproof1-eq5}
P = \alpha_0 {\bf 1}_2{\bf 1}_2' + M, \qquad \mbox{where}\quad M = \frac{a-c}{1-2\epsilon}\xi\xi', \quad \xi= (1-\epsilon,  -\epsilon)'. 
\eeq 
Write $z_i=\pi_i-h$. Since $Ph=\alpha_0 {\bf 1}_2$ and $z_i'{\bf 1}_2=0$,  we have
\begin{align*}
	\Omega_{ij} &= \theta_j\theta_j (h+z_i)'P(h+z_i) \cr
	&= \theta_i\theta_j (h'Ph + z_i'Pz_j)\cr
	&= \theta_i\theta_j (\alpha_0 + z_i'Pz_j) \cr
	&= \theta_i\theta_j (\alpha_0 + z_i'Mz_j)\cr
	&= \theta_i\theta_j \Bigl[ \alpha_0 + \frac{a-c}{1-2\epsilon}(\xi'z_i)(\xi'z_j)\Bigr]. 
\end{align*}
Let $t_i$ be the indicator that node $i$ belongs to the first community and write $u_i = t_i - \frac{N}{n}$. Then, $\pi_i=(t_i, 1-t_i)$ and $z_i = u_i (1, -1)'$. It follows that $
\xi'z_i =  u_i$. Therefore, 
\beq \label{LBproof1-eq6}
\Omega_{ij}= \theta_i\theta_j \Bigl[ \alpha_0 + \frac{a-c}{1-2\epsilon}u_iu_j\Bigr], \qquad \mbox{where}\quad u_i \overset{iid}{\sim} \mathrm{Bernoulli}(\epsilon)-\epsilon. 
\eeq
Consequently, 
\[
p_{ij}=\alpha_0\theta_i\theta_j, \qquad  \delta_{ij}(\Pi)=  \frac{a-c}{(1-2\epsilon)\alpha_0}u_iu_j. 
\]
We plug it into \eqref{LBproof1-eq9} to obtain  
\beq \label{LBproof1-eq10}
X = \sum_{i<j}\frac{\theta_i\theta_j}{1-\alpha_0\theta_i\theta_j} \frac{(a-c)^2}{(1-2\epsilon)^2\alpha_0}u_iu_j\tilde{u}_i\tilde{u}_j.   
\eeq

Below, we use \eqref{LBproof1-eq10} to bound $\ell^2$. Since $\alpha_0\theta_{\max}^2=O(c\theta_{\max}^2)=o(1)$, by Taylor expansion of $(1-\alpha_0\theta_i\theta_j)^{-1}$, we have  
\[
X = \frac{(a-c)^2}{(1-2\epsilon)^2\alpha_0} \sum_{i<j}\sum_{s=1}^\infty \alpha_0^{s-1} \theta_i^s\theta_j^su_iu_j\tilde{u}_i\tilde{u}_j. 
\]
Let $b_i = \theta_i \theta_{\max}^{-1} < 1$. We re-write $X$ as
\[
X = \gamma  \sum_{s=1}^\infty w_s X_s, 
\]
where
\begin{align} \gamma = \frac{\theta_{\max}^2 (a-c)^2}{(1 - \alpha_0 \theta_{\max}^2)(1-2\epsilon)^2\alpha_0}, \;\; w_s =(1-\alpha_0 \theta_{\max}^2) \alpha_0^{s-1} \theta_{\max}^{2s-2}, \;\; \text{ and } 
	X_s = \sum_{i<j}b_i^s b_j^s u_i u_j \tilde{u}_i \tilde{u}_j. 
	\label{eqn:X_decomp_data}
\end{align}
Let $\check{\mathbb{E}}$ be the conditional expectation by conditioning on the event of $\{\Pi\in {\cal M}, \widetilde{\Pi}\in {\cal M}\}$. It follows from \eqref{LBproof1-eq9} that 
\begin{align} \label{LBproof1-eq11}
	\ell_2 &= \mathbb{P}(\Pi\in {\cal M}, \widetilde{\Pi}\in {\cal M})\cdot  \check{\mathbb{E}}[\exp(X)]\cr
	&=  \mathbb{P}(\Pi\in {\cal M}, \widetilde{\Pi}\in {\cal M}) \cdot \check{\mathbb{E}}\Bigl[\exp\Bigl(\gamma \sum_{s=1}^\infty w_s X_s\Bigr)\Bigr] \cr
	&\leq \mathbb{P}(\Pi\in {\cal M}, \widetilde{\Pi}\in {\cal M})\cdot \sum_{s=1}^\infty w_s \check{\mathbb{E}}[\exp(\gamma X_s)]\cr
	&= \sum_{s=1}^\infty w_s \; \mathbb{E}\bigl[\exp(\gamma X_s)\cdot 1\{  \Pi\in {\cal M}, \widetilde{\Pi}\in {\cal M}\}\bigr].
\end{align}
The third line follows using Jensen's inequality and that $\sum_{s \geq 1} w_s = 1$.

It suffices to bound the term in \eqref{LBproof1-eq11} for each $s\geq 1$. Note that 
\beq \label{LBproof1-eq12}
X_s \leq  Y_s^2, \qquad Y_s = \sum_i b_i^s u_i\tilde{u}_i. 
\eeq 
We recall that $u_i=t_i-\epsilon$, where $t_i=\pi_i(1)\in \{0,1\}$. The event $\{\Pi\in {\cal M}, \widetilde{\Pi}\in {\cal M}\}$ translates to $\max\{\sum_i t_i, \; \sum_i\tilde{t}_i\}\leq 2n\epsilon$. Note that
\[
u_i\tilde{u}_i = \begin{cases} (1-\epsilon)^2, & \mbox{when }t_i+\tilde{t}_i=2,\cr
	-\epsilon(1-\epsilon), & \mbox{when }t_i+\tilde{t}_i=1,\cr
	\epsilon^2, &\mbox{where }t_i+\tilde{t}_i=0. 
\end{cases}
\]
It follows that $|u_i\tilde{u}_i|\leq (t_i+\tilde{t}_i)/2 + \epsilon^2$. Note that $\epsilon=O(N/n)$. Therefore, on this event, 
\[
|Y_s|\leq  \sum_i [(t_i+\tilde{t}_i)/2+ \epsilon^2]\leq 2n\epsilon+n\epsilon^2 \leq 3N.  
\]
We immediately have
\beq \label{LBproof1-eq13}
\mathbb{E}\bigl[\exp(\gamma X_s)\cdot 1\{  \Pi\in {\cal M}, \widetilde{\Pi}\in {\cal M}\}\bigr]\leq \mathbb{E}\Biggl[\exp(\gamma Y_s^2) \cdot 1\{ |Y_s|\leq 3 N\}\Biggr]. 
\eeq
The following lemma is useful. 
%%%%%%%%%%%%%%%%%%%%
\begin{lemma} \label{lem:truncated-Bernstein}
	Let $Z$ be a random variable satisfying that  
	\[
	\mathbb{P}(|Z|>t)\leq 2\exp\Bigl(-\frac{t^2/2}{\sigma^2 + bt}\Bigr), \qquad\mbox{for all $t>0$}. 
	\]
	Then, for any $\gamma >0$ and $B>0$ such that $\gamma(\sigma^2+bB)<1/2$, we have
	\[
	\mathbb{E}\bigl[\exp(\gamma Z^2)1\{ |Z|\leq B\}\bigr] \leq 1 + \frac{4\gamma (\sigma^2+bB)}{1-2\gamma (\sigma^2+bB)}. 
	\]
\end{lemma}

Note that $Y_s=\sum_i b_i^s u_i\tilde{u}_i$ is a sum of independent, mean-zero variables, where $|b_i^s u_i \tilde{u}_i |\leq 2$ and $\sum_i\mathrm{Var}(b_i^s u_i\tilde{u}_i)\leq \sum_i b_i^{2s}2\epsilon^2\leq 2n\epsilon^2$. It follows from Bernstein's inequality that
\[
\mathbb{P}(|Y_s|>t)\leq \exp\biggl( -\frac{t^2/2}{ 2n\epsilon^2 + 2t}\biggr), \qquad \mbox{for all }t>0. 
\]
To apply Lemma \ref{lem:truncated-Bernstein}, we set 
\begin{align*}
	b &= 2, \qquad \sigma^2=2n\epsilon^2\leq 2n^{-1} N^2, \qquad Z = Y_s, \qquad B = 3N,
\end{align*}
and $\gamma$ as in \eqref{eqn:X_decomp_data}. The choice of $B$ is in light of \eqref{LBproof1-eq13}. 	Furthermore, by \eqref{LBproof1-add}, we have $\alpha_0\asymp  c$. Also we have $\theta_{\max}^2 \alpha_0 \to 0$. Hence,
$$\gamma =  \frac{\theta_{\max}^2 (a-c)^2}{(1 - \alpha_0 \theta_{\max}^2)(1-2\epsilon)^2\alpha_0} \leq C \cdot \big( \frac{ \theta_{\max}^2 (a - c)^2}{c} \big).$$
%	\[
%	\gamma = \frac{(a-c)^2}{(1-2\epsilon)^2(1-\alpha_0)\alpha_0} \leq C\frac{(a-c)^2}{c}, \qquad\mbox{and}\qquad B=\theta_{\max}^sN. 
%	\]  
Thus by the hypothesis $\frac{ \theta_{\max}^2 N(a - c)^2}{c} \to 0$, it holds that $\gamma( \sigma^2 + b B) < 1/2$ for $n$ sufficiently large. Applying Lemma \ref{lem:truncated-Bernstein}, we obtain
%	We apply Lemma~\ref{lem:truncated-Bernstein} to and plug the results into \eqref{LBproof1-eq13} 
%	(note: $\theta_{\max}\geq n^{-1}\|\theta\|^2 = 1$)
\begin{align*}
	\mathbb{E}\bigl[\exp(\gamma X_s)\cdot 1\{  \Pi\in {\cal M}, \widetilde{\Pi}\in {\cal M}\}\bigr] &\leq 1+ C(\gamma(\sigma^2+bB))\cr
	&\leq 1 + C \cdot \big( \frac{ \theta_{\max}^2 N (a - c)^2}{c} \big) 
	%		&\leq 1+ C\frac{(a-c)^2}{c}\theta_{\max}^{2s}N. 
\end{align*}%	{\color{red}{Question: If $\theta_{\max} > 1$ and $s$ is very large, could it be the case that 
%			\begin{align*}
%				\gamma( \sigma^2 + b B) \asymp \frac{(a - c)^2}{c} \big( \theta_{\max}^{2s - 2} n^{-2}N + \theta_{\max}^{2s} N \big) \gg 1 ?
%			\end{align*}
%			If so we have some trouble applying Lemma~\ref{lem:truncated-Bernstein}. 
%	}}
%	
%	{\color{red}{A fix: Define $w_s = (1 - \alpha_0 \theta_{\max}^2 ) \alpha_0^{s-1} \theta_{\max}^{2s - 2}$, $b_i = \theta_i \theta_{\max}^{-1}$ and
%		\begin{align*}
%			X_s = \sum_{i < j} b_i b_j u_i u_j \ti u_i \ti u_j
%		\end{align*}
%	Note that $\sum w_s  = 1$, so Jensen can be applied. We have 
%		\begin{align*}
%			X = \frac{ \theta_{\max}^2 (a -c)^2}{(1 - \alpha_0 \theta_{\max}^2)\alpha_0} 
%			\sum_{s = 1}^\infty w_s X_s.
%		\end{align*}
%	Let $Y_s^2 = ( \sum_i b_i u_i \ti u_i )^2$. We have $\abs{X_s} \leq Y_s^2$. Since $|b_i| \leq 1$, we can apply Berstein to $Y_s$ with $b \asymp 1, \sigma^2 \asymp n^{-1} N^2$, and we have $|Y_s| = O(N)$. Now Lemma \ref{lem:truncated-Bernstein} can be applied and the rest of the argument should go through. 
%	}}
We further plug it into \eqref{LBproof1-eq11} to get
%	\begin{align*}
%		\ell_2 &\leq \sum_{s=1}^\infty w_s\Bigl[1+ C\frac{(a-c)^2}{c}\theta_{\max}^{2s}N \Bigr]\cr
%		&\leq 1 + C\sum_{s=1}^\infty \frac{(a-c)^2}{c}w_s \theta_{\max}^{2s}N\cr
%		&\leq 1+ C\sum_{s=1}^\infty \frac{(a-c)^2}{c}(1-\alpha_0)\alpha_0^{s-1}\theta_{\max}^{2s}N\cr
%		&\leq 1 + O\biggl( \frac{(a-c)^2}{c}\theta_{\max}^2N   \biggr) \sum_{s=1}^\infty (\alpha_0\theta^2_{\max})^{s-1}\cr
%		&\leq 1+O\biggl( \frac{(a-c)^2}{c}\theta_{\max}^2N   \biggr), 
%	\end{align*}
\begin{align*}
	\ell_2 &\leq \sum_{s=1}^\infty w_s\Bigl[1+ C \cdot \big( \frac{ \theta_{\max}^2 N (a - c)^2}{c} \big) \Bigr] \leq 1 + \big( \frac{ \theta_{\max}^2 N (a - c)^2}{c} \big),
	%	&\leq 1 + C\sum_{s=1}^\infty \frac{(a-c)^2}{c}w_s \theta_{\max}^{2s}N\cr
	%	&\leq 1+ C\sum_{s=1}^\infty \frac{(a-c)^2}{c}(1-\alpha_0)\alpha_0^{s-1}\theta_{\max}^{2s}N\cr
	%	&\leq 1 + O\biggl( \frac{(a-c)^2}{c}\theta_{\max}^2N   \biggr) \sum_{s=1}^\infty (\alpha_0\theta^2_{\max})^{s-1}\cr
	%	&\leq 1+O\biggl( \frac{(a-c)^2}{c}\theta_{\max}^2N   \biggr), 
\end{align*}
where we use that $\sum w_s  = 1$.

%	where the last line is because $\alpha_0\theta_{\max}^2=O(c\theta_{\max}^2)=o(1)$. 

It follows immediately that
\[
\ell_2\leq 1+o(1), \qquad \mbox{if}\quad \theta_{\max}\frac{\sqrt{N}(a-c)}{\sqrt{c}}\to 0. 
\]
This proves the claim. \qed

\subsection{Proof of Lemma~\ref{lem:truncated-Bernstein}}

Let $X$ denote a nonnegative random variable, and define $\overline{F}(x) = \p_X[ X \geq x]$. For any positive number $\beta>0$, we have
\begin{align*}
	\mathbb{E}[\exp(\gamma X)1\{X<\beta\}] & =\int_0^\beta e^{\gamma x} \, \mathrm{d} \p_X(x) \cr
	&= -e^{\gamma x} \bar{F}(x)\bigg|_0^\beta + \int_0^{\beta} \gamma e^{\gamma x} \bar{F}(x)dx  \cr
	&= 1 - e^{\gamma \beta}\bar{F}(\beta) + \int_0^{\beta} \gamma e^{\gamma x}\bar{F}(x)dx\cr
	&\leq 1 + \int_0^{\beta}\gamma e^{\gamma x} \bar{F}(x)dx. 
\end{align*}
We apply it to $X=Z^2$ and $\beta=B^2$ to get 
\begin{align*}
	\mathbb{E}\bigl[\exp(\gamma Z^2)1\{ |Z|\leq B\}\bigr] &\leq 1 + \int_0^{B^2}\gamma \exp(\gamma x) \mathbb{P}(|Z|>\sqrt{x})dx \cr
	&\leq 1+ 2\gamma \int_0^{B^2}\exp(\gamma x) \exp\biggl\{-\frac{x}{2(\sigma^2 + b\sqrt{x})}\biggr\}dx\cr
	&\leq 1+ 2\gamma \int_0^{B^2} \exp(\gamma x) \exp\biggl\{-\frac{x}{2(\sigma^2 + bB)}\biggr\}dx\cr
	&\leq 1 + 2 \gamma \int_{0}^{\infty}\exp\biggl\{-  \frac{1-2\gamma (\sigma^2 + bB)}{2(\sigma^2 + bB)} x\biggr\} dx\cr
	&\leq 1 + \frac{4\gamma (\sigma^2 + bB)}{1-2\gamma (\sigma^2 + bB)}. 
\end{align*}
This proves the claim. 
\qed

\section{Proof of Theorem \ref{thm:scan} (Tightness of the statistical lower bound) }

Let $\rho \in \mb{R}^n$. We consider the global testing problem in the DCBM model where  
\begin{enumerate}[label=\emph{\Alph*}), ref=\emph{\Alph*}]
	\item $P = \begin{pmatrix}
		1 & b \\
		b & 1 
	\end{pmatrix}$
	\item $b = \ti b/\sqrt{ ac }$, 
	\item 	$\theta_i = \rho_i \sqrt{a}$ for $i \in S$, 
	\item $\theta_i = \rho_i \sqrt{c}$ for $i \notin S$, and
	\item $ a N_0 + \ti b(n - N_0) = \ti b N_0 + c (n - N_0)$,  \label{assn:degree_matching}
\end{enumerate}
%	Note that by applying Sinkhorn's theorem, any DCBM with $K = 2$ can be represented in this form  Indeed, we may take $\sqrt{a}$ and $\sqrt{c}$ to be the unique scale factors such that $D = \diag( (\sqrt{a}, \sqrt{c}) )$ satisfies $DPDh = \mf{1}$. 

Recall that $h = (N_0/n, 1 - N_0/n)^\T$, and $N_0$ is the size of the smaller community in the alternative. Observe that the null model $K = 1$ is parameterized by setting $a = c = \ti b = 1$.

%	An equivalent parameterization of the model above is to have heterogeneity parameter $\ti \theta = \rho$ and  structure matrix
%	\[
%	\ti P = \begin{pmatrix}
%		a & \ti b \\ 
%		\ti b & c
%	\end{pmatrix},
%	\]
%	where $\ti P h \propto \mf{1}$. We use the previous parameterization for consistency.

Recall that $\eps = N/n$. We define
\begin{align*}
	\alpha_0 \equiv \frac{a N_0 + \ti b(n - N_0)}{n}. 
\end{align*}
Note that by Assumption \eqref{assn:degree_matching},
\begin{align*}
	\tilde 	b &= \frac{n c - (a+c) N_0}{n - 2N_0} \num \label{model2}
	\\	a \epsilon &= O(c) \num \label{eqn:aeps_vs_c}, \text{ and }
	\\ c &\sim \ti b \sim \alpha_0 \num \label{eqn:alpha0_vs_c_b} .
\end{align*}

Our assumptions in this section are the following: 
\begin{enumerate}[label=\emph{\alph*}), ref=\emph{\alph*}]
	%	\item $a \geq c$ \label{assn:ageqc}
	\item There exists an absolute constant $C_\rho > 0$ such that	$	\rho_{\max} \leq C_\rho \, \rho_{\min} $
	\label{assn:mild_deg_het}
	\item $\frac{\rho_{\max}^2 \alpha_0 n}{\sqrt{\log n}}\to \infty $ \label{assn:diverging_degrees}
	\item An integer $N$ is known such that $N_0 = N[1 + o(1)]$. 
	\label{assn:N0_estimate}
	%		\footnote{Assumption \eqref{assn:N_known} can be relaxed using a multiscale version of our scan statistic (see e.g. \cite{AriVer14}), although we omit the details here. } \label{assn:N_known}
\end{enumerate}

Note that since we tolerate a small error in the clique size by Assumption \eqref{assn:N0_estimate}, our setting indeed matches that of the statistical lower bound, by \eqref{eqn:clique_size_concentration}. 

Define the signed scan statistic
\begin{align*}
	\phi_{sc} = \max_{D \subset [n]: |D| = N} 
	\mf{1}'_D \big( A - \hat \eta \hat \eta^\T  \big) \mf{1}_D. 	
	\num \label{eqn:scan_stat_def}
\end{align*}
For notational brevity, define $n\rp{2} = \binom{n}{2}$. Let
\[
\hat \gamma = \frac{1}{n\rp{2}} \sum_{i,j} A_{ij}.  
\]
The estimator $\hat \gamma$ provides a constant factor approximation of the edge density of the least-favorable null model. See Lemma \ref{lem:gamma_estimate} for further details.

Next let 
\begin{align}
	h(u) = (1 + u) \log( 1+ u) - u, 
\end{align}
and note that this function is strictly increasing on $\mb{R}_{\geq 0}$. Define a random threshold $\hat \tau$ to be  
\begin{align*}
	\hat \tau =  C^*  \hat \gamma N^2 h^{-1}\bigg( \frac{C^* N \log(\frac{ne}{N})}{\hat \gamma N^2}  \bigg)
	\num \label{eqn:tauhat_def}
\end{align*}
Let $C^*>0$ denote a sufficiently large constant, to be determined, that depends only on $C_\rho$ from Assumption \eqref{assn:mild_deg_het}. 
Finally define the scan test to be 
\begin{align*}
	\vp_{sc} = \mf{1}\big[ |\phi_{sc}| >   \hat \tau  \big]
\end{align*}
Note that, if we assume $a \geq c$, as in the main text, then $b < 1$. In this case, we can simply take
\begin{align*}
	\vp_{sc} = \mf{1}\big[ \phi_{sc} >   \hat \tau  \big],
\end{align*}
and the same guarantees hold. On the other hand, if $b > 1$, then the scan test skews negative, as our proof shows. 

\begin{thm}
	\label{thm:scan_test}
	If 
	\begin{align*}
		h \bigg( \frac{ \| \theta_S \|_1^2 |1 - b^2|}{  \rho_{\max}^2 \alpha_0 N_0^2} \bigg)
		\gg \frac{ \log\frac{ne}{N_0} }{\rho_{\max}^2 \alpha_0 N_0} 
		\num \label{eqn:scan_SNR},
	\end{align*}
	then the type 1 and 2 error of $\vp_{sc}$ tend to $0$ as $n \to \infty$.
\end{thm}

We interpret the previous result in the following concrete settings.
\begin{cor}
	\label{cor:scan_test}
	If \[
	\frac{\rho_{\max}^2 \alpha_0 N_0}{ \log \frac{ne}{N_0} } \to 0,
	\]
	then $\vp_{sc}$ has type 1 and 2 errors tending to $0$ as $n \to \infty$, provided that 
	\[
	\frac{\rho_{\max}^2 N_0 (a - c)}{ \log \frac{ne}{N_0} } \gg 1. 
	\]
	If 
	\[
	\frac{\rho_{\max}^2 \alpha_0 N_0}{ \log \frac{ne}{N_0} } \to \infty ,
	\]
	then $\vp_{sc}$ has type 1 and 2 errors tending to $0$ as $n \to \infty$, provided that 
	\[
	\frac{ \rho_{\max}^2 N_0 (a - c) }{ \sqrt{\rho_{\max}^2 N_0 \alpha_0 \log \frac{ne}{N_0} } } \gg 1. 
	\]
\end{cor}

\begin{proof}
	Note that
	\[
	\| \theta_S \|_1^2 |1 - b^2|
	= \rho_{\max}^2 N_0^2 (a - \ti b^2/\sqrt{c} )
	\sim  \rho_{\max}^2 N_0^2 (a - c ). 
	\]
	In the first case, 
	\[
	h \bigg( \frac{ \| \theta_S \|_1^2 |1 - b^2|}{  \rho_{\max}^2 \alpha_0 N_0^2} \bigg)
	\gg h \bigg( \frac{ \log \frac{ne}{N_0} }{  \rho_{\max}^2 \alpha_0 N_0} \bigg)
	\gtrsim \frac{ \log \frac{ne}{N_0} }{  \rho_{\max}^2 \alpha_0 N_0}. 
	\]
	We use the fact that $h(u) \gtrsim u$ for $u \geq 1$. 
	
	In the second case, 
	\begin{align*}
		h \bigg( \frac{ \| \theta_S \|_1^2 |1 - b^2|}{  \rho_{\max}^2 \alpha_0 N_0^2} \bigg) 
		&\gg h \bigg( \frac{N_0 \cdot \sqrt{ \rho_{\max}^2 N_0 \alpha_0 \log \frac{ne}{N_0} }}{  \rho_{\max}^2 \alpha_0 N_0^2} \bigg) 
		= h \bigg( \sqrt{ \frac{ \log \frac{ne}{N_0} }{  \rho_{\max}^2 \alpha_0 N_0 }  } \bigg) 
		\gtrsim \frac{ \log\frac{ne}{N_0} }{\rho_{\max}^2 \alpha_0 N_0}. 
	\end{align*}
	
\end{proof}

The upper bounds in the second part of Corollary \ref{cor:scan_test} is the best possible up to logarithmic factors. For example, suppose that $\theta_{\max} \les \theta_{\min}$ in Theorem \ref{thm:statLB1}. Then the upper bound for the second case of Corollary \ref{cor:scan_test} matches the lower bound of Theorem \ref{thm:statLB1} up to logarithmic factors.

To prove Theorem \ref{thm:scan}, first we establish concentration of $\hat \gamma$. 

\begin{lemma}
	\label{lem:gamma_estimate}
	Recall 
	\begin{align*}
		\hat \gamma = \frac{1}{n\rp{2}} \sum_{i,j (dist)} A_{ij}.
	\end{align*}
	There exists an absolute constant $C > 0$ such that for all $\delta > 0$, it holds that
	\begin{align*}
		| \hat \gamma - \E \hat \gamma | \leq \frac{ C \sqrt{ \rho_{\max}^2 \alpha_0 \log(1/\delta) } }{n}
	\end{align*}
	with probability at least $1 - \delta$. 
\end{lemma}

\begin{proof}
	As a preliminary, we claim that
	\begin{align*}
		(\Omega \mf{1})_i \asymp \rho_{\max}^2 \alpha_0 n. 
		\num \label{eqn:Om1}
	\end{align*}
	To see this, note that if $i \in S$, then by \eqref{assn:degree_matching}
	\begin{align*}
		(\Omega \mf{1})_i 
		&= \sum_{j} \Omega_{ij} 
		= \theta_i ( \| \theta_S \|_1 + b \| \theta_{S^c} \|_1 )
		\\ &\asymp \rho_{\max} \sqrt{a} \cdot \big( \sqrt{a} N \rho_{\max} + 
		\frac{\ti b }{\sqrt{ac}} \cdot \sqrt{c} \rho_{\max} \big)
		= \rho_{\max}^2 \alpha_0 n. 
	\end{align*}
	The claim for $i \notin S$ follows by a similar argument applying  \eqref{assn:degree_matching}. 
	It follows that 
	\begin{align*}
		v_0 = \mf{1}^\T \Omega \mf{1} \asymp \rho_{\max}^2 \alpha_0 n^2
	\end{align*}
	
	The expectation is
	\begin{align*}
		\E \hat \gamma = \frac{1}{n\rp{2}} \sum_{i,j (dist)} \Omega_{ij}, 
	\end{align*}
	and the variance is 
	\begin{align*}
		\var(\hat \gamma) =  \frac{1}{(n\rp{2})^2} \sum_{i,j (dist)} \Omega_{ij}(1 - \Omega_{ij}). 
	\end{align*}
	
	By Bernstein's inequality,
	\begin{align*}
		\p\big[ n\rp{2}  \big|\hat \gamma - \E \hat \gamma \big| > t  \big] 
		\leq 2 \exp\bigg( -\frac{ct^2}{  \sum_{i,j (dist)} \Omega_{ij}  + t  }   \bigg). 
		\num \label{eqn:bernstein_again}
	\end{align*}
	
	By Assumptions \eqref{assn:mild_deg_het} and \eqref{assn:diverging_degrees}, 
	\begin{align*}
		\sum_{i,j (dist)} \Omega_{ij} \asymp \rho_{\max}^2 \alpha_0 n^2 \gg n. 
	\end{align*}
	
	Setting 
	\[
	t = \tau \equiv C \sqrt{  \rho_{\max}^2 \alpha_0 n^2 \log(1/\delta)  }
	\]
	for a large enough absolute constant $C> 0$, \eqref{eqn:bernstein_again}  implies that 
	\[
	| \hat \gamma - \E \hat \gamma | \leq \frac{\tau}{n^2} 
	\asymp \frac{ \sqrt{  \rho_{\max}^2 \alpha_0 \log(1/\delta} ) }{n} 
	\]
	with probability at least $1 -\delta$. 
\end{proof}

Next we control the error arising from the plug-in effect of approximating $\eta^*$ by $\hat \eta $. 
\begin{lemma}
	Given $D \subset [n]$, define
	\[
	L_D \equiv \mf{1}^\T_D (\eta^* \eta^{* \T} - \hat \eta \hat \eta^\T ) \mf{1}_D. 
	\]
	Then under the null and alternative hypothesis, 
	\begin{align*}
		\max_{|D| = N}  | L_D |  \les \sqrt{ N_0^3 \rho_{\max}^2 \alpha_0 \log( \frac{ne}{N_0} )  }
	\end{align*}
	with probability at least $1 - \binom{n}{N}^{-1} -2 v_0^{-c_1}$, for an absolute constant $c_1 > 0$.
	\label{lem:scan_plugin}
\end{lemma}

\begin{proof}
	
	In this proof, $c > 0$ is an absolute constant that may vary from line to line.
	
	Given $D \subset [n]$, let 
	\begin{align*}
		L_D \equiv \mf{1}^\T_D (\eta^* \eta^{* \T} - \hat \eta \hat \eta^\T ) \mf{1}_D
		= \mf{1}^\T_D \eta^* ( \eta^{*} - \hat \eta)^\T \mf{1}_D 
		+ \mf{1}^\T_D ( \eta^{*} - \hat \eta)  \hat \eta^{\T} \mf{1}_D 
		\num \label{eqn:LD_decomp}
	\end{align*}
	Our first goal is to control
	\begin{align*}
		\big| \mf{1}_D^\T	(\hat \eta - \eta^*)  	\big| .  
	\end{align*}
	
	Define $\overline{\Omega} = \Omega - \diag(\Omega)$. Note that
	\begin{align*}
		\hat \eta - \eta^* &= 
		\frac{A \mf{1}}{\sqrt{V}} - \frac{ \Omega \mf{1}}{\sqrt{v_0}} 
		= \big( \frac{A \mf{1}}{\sqrt{V}} - \frac{A \mf{1}}{\sqrt{v_0}} \big) 
		+  \big(\frac{A \mf{1}}{\sqrt{v_0}} -  \frac{ \overline{ \Omega} \mf{1}}{\sqrt{v_0}}   \big)
		+ \big( \frac{ \overline{ \Omega} \mf{1}}{\sqrt{v_0}}  -  
		\frac{  \Omega \mf{1}}{\sqrt{v_0}}
		\big) 
		\num \label{eqn:etahat_etastar}
	\end{align*}
	We study each term of \eqref{eqn:etahat_etastar}. First note that 
	\begin{align*}
		(\overline{\Omega} \mf{1})_i = (\Omega \mf{1})_i - \Omega_{ii}
		= \rho_{\max}^2  \alpha_0 n + O(1),
	\end{align*}
	and thus
	\begin{align*}
		v_0 = \sum_i (\Omega \mf{1})_i &\sim \sum_i (\overline{\Omega} \mf{1})_i  = v, \text{ and}
		\\	|v_0 - v| &\les 1
		\num \label{eqn:v0_vs_v}
	\end{align*}
	
	Next note that 
	\begin{align*}
		\var\big( \mf{1}_D^\T \big( A \mf{1} - \overline{\Omega} \mf{1} \big) \big)
		&\les  \sum_{\substack{i \in [n], j \in D \\ i \neq j}} \Omega_{ij} 
		\les |D| \rho_{\max}^2 \alpha_0 n. 
	\end{align*}
	By Bernstein's inequality,
	\begin{align*}
		\p\big[  \big|  \mf{1}_D^\T \big( A \mf{1} - \overline{\Omega} \mf{1} \big)  \big|  \geq t \big] \leq 2 \exp\bigg(  -\frac{ct^2}{|D| \rho_{\max}^2 \alpha_0 n + t}  \bigg)
		\num \label{eqn:bernstein_plugin_statUB}
	\end{align*}
	for all $t > 0$. 
	Setting
	\[
	t = \tau \equiv \sqrt{4/c} \cdot  \sqrt{ |D| \rho_{\max}^2 \alpha_0 n \log(1/\delta) }, 
	\]
	we have
	\begin{align*}
		\frac{1}{\sqrt{v_0}} \big|  \mf{1}_D^\T \big( A \mf{1} - \overline{\Omega} \mf{1} \big)  \big|  \les  \frac{  \sqrt{ |D| \rho_{\max}^2 \alpha_0 n \log(1/\delta) }   }{\sqrt{\rho_{\max}^2 \alpha_0 n^2}}
		= \sqrt{ (|D|/n) \cdot \log(1/\delta) }
		\num \label{eqn:LD1}
	\end{align*}
	with probability at least $1 - \delta$. 
	
	%====================================
	%
	%Next 
	%	\begin{align*}
	%		\var\big( (A \mf{1})_i - (\overline{\Omega} \mf{1})_i \big)
	%		&\asymp \sum_{j: j \neq i} \Omega_{ij} 
	%		\asymp \rho_{\max}^2 \alpha_0 n 
	%	\end{align*}
	%for all $1 \leq i \leq n$. By Bernstein's inequality,
	%	\begin{align*}
	%		\p\big[ \big| (A \mf{1})_i - (\overline{\Omega} \mf{1})_i \big| \geq t \big]
	%		\les \exp\big( - \frac{ct^2}{ \rho_{\max}^2 \alpha_0 n  + t} \big)
	%	\end{align*}
	%for all $t > 0$ and $1 \leq i \leq n$, where $c>0$ is an absolute constant. Setting 
	%\[
	%t = \tau \equiv \sqrt{4/c} \cdot \sqrt{ \rho_{\max}^2 \alpha_0 n  \log n  }, 
	%\] 

	%=========================
	%
	%and applying the union bound and assumption \eqref{assn:diverging_degrees}, we have that $\| A \mf{1} - \overline{\Omega} \mf{1} \|_\infty \leq \tau$ with probability at least $1 - n^{-1}$.
	%
	%
	%
	%Thus with probability at least $1 -n^{-1}$, 
	%\begin{align*}
	%	\mf{1}_D^\T A \mf{1}  
	%	\les 	\mf{1}_D^\T \Omega  \mf{1} + |D| \tau 
	%	\asymp |D|  \rho_{\max}^2  \alpha_0 n + |D|  \sqrt{ \rho_{\max}^2 \alpha_0 n } \cdot \sqrt{ \rho_{\max}^2 \alpha_0 n  \log n  }.
	%\end{align*}
	
	%\pax{Not sure if this tail bound is good enough.}
	
	Next, it is shown in \cite[Supplement, pg.100]{JinKeLuo21} that for $\sqrt{ \log \| \theta \|_1} \ll x_n \ll \| \theta \|_1$, 
	\begin{align*}
		\p\big[ |V - v| > x_n \| \theta \|_1 \big] 
		= \p\bigg[ |\sqrt{V} - \sqrt{ v} | > \frac{ x_n \| \theta \|_1}{\sqrt{V} + \sqrt{v}} \bigg]
		\leq 2 \exp(-c x_n^2 ).
	\end{align*}
	Hence
	%Since $f(y) = \sqrt{y}$ is $1$-Lipschitz for $y \geq 1$, 
	\begin{align*}
		\p\bigg[ |\sqrt{V} - \sqrt{ v} | > \frac{x_n \| \theta \|_1}{\sqrt{v}} \bigg] \leq 2 \exp(-c x_n^2 ),
	\end{align*}
	Note that by \eqref{eqn:aeps_vs_c} and \eqref{eqn:alpha0_vs_c_b},
	\begin{align*}
		\frac{ \|\theta \|_1 }{ \sqrt{v}} \asymp \frac{ N_0 \rho_{\max} \sqrt{a} + (n - N_0) \rho_{\max} \sqrt{c} }{\rho_{\max} \sqrt{\alpha_0} n}
		\asymp 1. 
	\end{align*}
	By \eqref{eqn:v0_vs_v}, we have
	\begin{align*}
		\p\bigg[ |\sqrt{V} - \sqrt{ v_0} | >  \frac{x_n \| \theta \|_1}{\sqrt{v}} \bigg] \leq 2 \exp(-c x_n^2 ).
		\num \label{eqn:V_vs_v0}
	\end{align*}
	Hence with probability at least $1 - 2 \exp(-c x_n^2 )$,
	\[
	V \gtrsim v_0. 
	\]
	It follows that 
	\begin{align*}
		\p\bigg[ \big| \frac{1}{\sqrt{V}} - \frac{1}{\sqrt{v_0}} \big| \geq \frac{x_n \| \theta \|_1}{ v_0 \sqrt{v} }  \bigg] 
		&= 	\p\bigg[ \frac{  |\sqrt{V} - \sqrt{ v_0} |   }{\sqrt{V \cdot v_0} } 
		\geq \frac{x_n \| \theta \|_1}{ v_0  \sqrt{v}}  \bigg]  
		\leq 2 \exp(-c x_n^2 ).
	\end{align*}
	Hence with probability at least $1 - \delta - 2 \exp(-c x_n^2) $,
	\begin{align*}
		\bigg|\mf{1}_D^\T ( \frac{A \mf{1}}{\sqrt{V}} - \frac{A \mf{1}}{\sqrt{v_0}}  )\bigg|
		&\leq   \frac{x_n \cdot \big(|D|  \rho_{\max}^2  \alpha_0 n + \sqrt{ |D| \rho_{\max}^2 \alpha_0 n \log(1/\delta) } \big) }{v_0}
		\\& \asymp \frac{x_n \cdot \big(|D|  \rho_{\max}^2  \alpha_0 n + \sqrt{ |D| \rho_{\max}^2 \alpha_0 n \log(1/\delta) } \big) }{ \rho_{\max}^2 \alpha_0 n^2}. 
		\num \label{eqn:LD2}
	\end{align*}
	
	%Next, 
	%\begin{align*}
	%	\mf{1}_D^\T \big( A \mf{1} - \overline{\Omega} \mf{1} \big)
	%	&= \sum_{\substack{i \in [n], j \in D \\ i \neq j}} (A_{ij} - \Omega_{ij}). 
	%\end{align*}
	%Setting
	%\[
	%t = \tau \equiv \sqrt{4/c} \cdot  \sqrt{ |D| \rho_{\max}^2 \alpha_0 n \log(1/\delta) }, 
	%\]
	%by Bernstein's inequality we have
	%\[
	%\frac{1}{\sqrt{v_0}} \big|  \mf{1}_D^\T \big( A \mf{1} - \overline{\Omega} \mf{1} \big)  \big|  \les  \frac{  \sqrt{ |D| \rho_{\max}^2 \alpha_0 n \log(1/\delta) }   }{\sqrt{\rho_{\max}^2 \alpha_0 n^2}}
	%\]
	%with probability at least $1 - \delta$. 
	
	For the last term of \eqref{eqn:etahat_etastar}, 
	\begin{align*}
		\mf{1}_D^\T \big( \frac{ \overline{ \Omega} \mf{1}}{\sqrt{v_0}}  -  
		\frac{  \Omega \mf{1}}{\sqrt{v_0}}
		\big) 
		&= \frac{ \sum_{i \in D} \Omega_{ii} }{\sqrt{v_0}} 
		\asymp \frac{    \rho_{\max}^2 a |D \cap S| +   \rho_{\max}^2 c |D \cap S^c|    }{   \sqrt{\rho_{\max}^2 \alpha_0 n^2}     }
		\\ &\les \rho_{\max} a \epsilon / \sqrt{\alpha_0} \les \rho_{\max} \sqrt{ c} \les 1. 
		\num \label{eqn:LD3}
	\end{align*}
	
	Next we control $\mf{1}_D^\T \hat \eta$. By \eqref{eqn:bernstein_plugin_statUB} and \eqref{eqn:V_vs_v0}, 
	\begin{align*}
		| \mf{1}_D^\T \hat \eta | =	\frac{ | \mf{1}_D^\T A \mf{1}  |}{\sqrt{V}} 
		\les \frac{ |D|  \rho_{\max}^2  \alpha_0 n + \sqrt{ |D| \rho_{\max}^2 \alpha_0 n \log(1/\delta) } }{ \sqrt{v_0} - c x_n }
		\num \label{eqn:1_D_etahat}
	\end{align*}
	with probability at least $1 - \delta - 2\exp(-c x_n^2)$. It also holds that
	\begin{align*}
		| \mf{1}_D^\T  \eta^* | =	\frac{| \mf{1}_D^\T \Omega \mf{1}|}{\sqrt{v_0}} 
		= \frac{ |D| \rho_{\max}^2 \alpha_0 n }{ \rho_{\max} \sqrt{\alpha_0} n }
		= |D| \rho_{\max} \sqrt{\alpha_0}.  
		\num \label{eqn:1_D_etastar}
	\end{align*}
	
	Next we set $x_n= \sqrt{ \log \| \theta \|_1 } \asymp \sqrt{ \log v_0}$. Then from \eqref{eqn:LD2} and \eqref{eqn:1_D_etahat}, 
	\begin{align*}
		\bigg|\mf{1}_D^\T ( \frac{A \mf{1}}{\sqrt{V}} - \frac{A \mf{1}}{\sqrt{v_0}}  )\bigg|
		& \asymp \frac{ \sqrt{ \log v_0} \cdot \big(|D|  \rho_{\max}^2  \alpha_0 n + \sqrt{ |D| \rho_{\max}^2 \alpha_0 n \log(1/\delta) } \big) }{ \rho_{\max}^2 \alpha_0 n^2}
		\\ &\asymp \sqrt{\log v_0} \cdot \big( (|D|/n) + \frac{\sqrt{(|D|/n) \log(1/\delta)}}{  \rho_{\max} \sqrt{\alpha_0} n } \big) 	,
		\num \label{eqn:LD2_next}
	\end{align*}
	and
	\begin{align*}
		| \mf{1}_D^\T \hat \eta |  &\les \frac{ |D|  \rho_{\max}^2  \alpha_0 n + \sqrt{ |D| \rho_{\max}^2 \alpha_0 n \log(1/\delta) } }{ \sqrt{v_0} }
		\\ &\asymp \frac{ |D|  \rho_{\max}^2  \alpha_0 n + \sqrt{ |D| \rho_{\max}^2 \alpha_0 n \log(1/\delta) } }{ \rho_{\max} \sqrt{\alpha_0} n }
		\\ & \asymp |D| \rho_{\max} \sqrt{\alpha_0}  
		+ \sqrt{ (|D|/n) \cdot \log(1/\delta) }
		\num \label{eqn:1_D_etahat_next}
	\end{align*}
	with probability at least $1 - \delta -2 v_0^{-c_1}$. 
	
	By \eqref{eqn:LD1},\eqref{eqn:LD3}, \eqref{eqn:1_D_etastar},  \eqref{eqn:LD2_next}, and   \eqref{eqn:1_D_etahat_next}
	\begin{align*}
		|L_D | &\leq \big| \mf{1}^\T_D \eta^* ( \eta^{*} - \hat \eta)^\T \mf{1}_D  \big|
		+ \big| \mf{1}^\T_D ( \eta^{*} - \hat \eta)  \hat \eta^{\T} \mf{1}_D \big|
		\\& \les \big( |D| \rho_{\max} \sqrt{\alpha_0} 
		+ \sqrt{ (|D|/n) \cdot \log(1/\delta) } \big) \cdot \big(
		\sqrt{\log v_0} (|D|/n) + \sqrt{ (|D|/n) \log(1/\delta) } + 1
		\big).
		%\\& \les |D| \rho_{\max} \sqrt{\alpha_0} \cdot \bigg( \sqrt{\log v_0} \cdot \big( (|D|/n) + \frac{\sqrt{(|D|/n) \log(1/\delta)}}{  \rho_{\max} \sqrt{\alpha_0} n } \big)  + \sqrt{ (|D|/n) \log(1/\delta) } + 1\bigg)
	\end{align*}
	%Similarly,
	%\begin{align*}
	%&\big| \mf{1}^\T_D ( \eta^{*} - \hat \eta)  \hat \eta^{\T} \mf{1}_D \big|
	%\\ &\les  \big( |D| \rho_{\max} \sqrt{\alpha_0} n 
	%+ \sqrt{ (|D|/n) \cdot \log(1/\delta) } \big) 
	%\cdot \big(  \big)
	%\end{align*}
	with probability at least $1 - \delta -2 v_0^{-c_1}$.

	It follows that, setting $\delta = 1/\binom{n}{N}^2$ above and applying the union bound,
	\begin{align*}
		\max_{|D| = N}  | L_D |
		&\les \big( N \rho_{\max} \sqrt{\alpha_0}  
		+ \sqrt{ N \epsilon  \cdot \log(\frac{ne}{N}) } \, \big) \cdot \big(
		\epsilon\sqrt{\log v_0}  + \sqrt{ N \epsilon \cdot \log(\frac{ne}{N}) } + 1
		\big)
		%\\& \les \sqrt{ N^3 \rho_{\max}^2 \alpha_0 \log( \frac{ne}{N} )  }
		%+ N \epsilon \log( \frac{ne}{N} )
		%\les \sqrt{ N^3 \rho_{\max}^2 \alpha_0 \log( \frac{ne}{N} )  }
	\end{align*}
	with probability at least $1 - \binom{n}{N}^{-1}  - 2 v_0^{-c_1} \to 1$.  Note that
	\begin{align*}
		\frac{n \log \frac{ne}{N} }{ \log v_0 } 
		&\asymp 	\frac{n \log \frac{ne}{N} }{ \log( \rho_{\max}^2 \alpha_0 n^2 )} \gtrsim 1 \Rightarrow 
		\\ \frac{N^2}{n} \log \frac{ne}{N} &\gtrsim \frac{N^2}{n^2} \log( \rho_{\max}^2 \alpha_0 n^2 ) \Rightarrow
		\\   \sqrt{ N \epsilon \cdot \log(\frac{ne}{N}) } &\gtrsim  \epsilon\sqrt{\log v_0}. 
	\end{align*}
	Further, since $(N/n) \log \frac{ne}{N} \ll 1$ and $\rho_{\max}^2 \alpha_0 n \to \infty $ by Assumption  \eqref{assn:diverging_degrees}, 
	\begin{align*}
		N \log \frac{ne}{N} &\les \rho_{\max}^2 \alpha_0 n^2 \Rightarrow 
		\\ \frac{N}{n} \sqrt{ \log \frac{ne}{N} } &\les \sqrt{N \rho_{\max}^2 \alpha_0} \Rightarrow 
		\\ N \epsilon  \log \frac{ne}{N} &\les \sqrt{N^3 \rho_{\max}^2 \alpha_0 \log \frac{ne}{N} }. 
	\end{align*}
	Hence 
	\begin{align*}
		\max_{|D| = N}  | L_D |  \les \sqrt{ N^3 \rho_{\max}^2 \alpha_0 \log( \frac{ne}{N} )  }
		+ N \epsilon \log( \frac{ne}{N} )
		\les \sqrt{ N^3 \rho_{\max}^2 \alpha_0 \log( \frac{ne}{N} )  }
	\end{align*}
	with probability at least $1 - \binom{n}{N}^{-1}-2 v_0^{-c_1}$. Recalling that $N = N_0[1 + o(1)]$ yields the statement of the lemma. 

\end{proof}

Next we study an ideal version of $\phi_{sc}$.

\begin{lemma}
	\label{lem:scan_ideal}
	Define the ideal scan statistic
	\[
	\ti \phi_{sc} = \max_{|D| = N} \mf{1}_D^\T(  A - \eta^* \eta^{*\T} ) \mf{1}_D,
	\] 	
	and corresponding test
	\[
	\ti	\vp_{sc} = \mf{1}\bigg[ \ti \phi_{sc}  >  \ti \tau   \bigg],
	\]
	where 
	\[
	\ti \tau \equiv \ti C \hat \gamma N^2 h^{-1}\bigg( \frac{\ti C N \log(\frac{ne}{N})}{\hat \gamma N^2}  \bigg),  
	\]
	and $\tilde C > 0$ is a sufficiently large absolute constant that depends only on $C_\rho$ from Assumption \eqref{assn:mild_deg_het}. Then under the null hypothesis,
	\begin{align*}
		\p\big[  | \ti \phi_{sc} | > \ti \tau   \big] \leq 
		n^{-c_0} + \exp\big( - N \log\frac{ne}{N}  \big)
	\end{align*}
	and under the alternative hypothesis,
	\begin{align*}
		\p \big[  | \ti \phi_{sc} | \leq  \ti \tau   \big]
		\leq n^{-c_0} + \big( \frac{N}{ne} \big)^{10}
	\end{align*}
	for $n$ sufficiently large, where $c_0$ is an absolute constant.
\end{lemma}

\begin{proof}
	
	In this proof, $c>0$ is an absolute constant that may vary form line to line. 	
	
	%	Recall the scan test is
	%	\begin{align*}
	%		\vp = \mf{1}\bigg[ |\phi_{sc}|  >   \ti C  \hat \gamma N^2 h^{-1}\bigg( \frac{\ti C N \log(\frac{ne}{N})}{\hat \gamma N^2}  \bigg)  \bigg]
	%	\end{align*}	
	%	where $\ti C$ is an absolute constant depending on $C_\rho$ to be specified.
	Define the ideal scan statistic
	\[
	\ti \phi_{sc} = \max_{|D| = N} \mf{1}_D^\T(  A - \eta^* \eta^{*\T} ) \mf{1}_D. 
	\] 	
	Also define
	\[
	Z_D \equiv \sum_{i,j \in D (dist)} (A_{ij} - \Omega_{ij} )
	\]

	First consider the type 1 error. Under the null hypothesis, we have $\eta^* = \theta = \rho$ and $\alpha_0 = 1$. 
	Observe that 
	\[
	\sigma^2_D \equiv \var(Z_D) = \var\big( \sum_{i,j \in D (dist)} (A_{ij} - \theta_i \theta_j ) \big)
	\les \| \theta_D \|_1^2 \asymp \rho_{\max}^2 N^2
	\sim \rho_{\max}^2 N_0^2
	\]
	By the Bennett inequality, \cite[Theorem 2.9.2]{Ver18},
	\begin{align*}
		\p\big[ \sum_{i,j \in D} (A_{ij} - \theta_i \theta_j ) > t \big]
		\leq \exp\bigg( - \sigma_D^2 \, h\bigg( \frac{t}{\sigma_D^2} \bigg) \bigg),
		\num \label{eqn:scan_bennett}
	\end{align*}
	where $h(u) = (1 + u) \log(1 + u) - u$.

	%	First suppose that $N \rho_{\max}^2 \to \infty$. 
	
	%By Bernstein's inequality,
	%	\begin{align*}
	%		\p\bigg[ | \sum_{i,j \in D} (A_{ij} - \theta_i \theta_j ) | > t\bigg]
	%		\les \exp\big( 
	%		- \frac{ct^2}{ \rho_{\max}^2 N + t}
	%		\big)
	%	\end{align*}
	%	for all $t > 0$. 
	
	Next, by Lemma \ref{lem:gamma_estimate},
	\[
	|\hat \gamma - \E \hat \gamma| \les \frac{ \sqrt{ \log n} }{n}
	\]
	with probability $n^{-c_0}$. Also recall that
	\[
	\E \, \hat \gamma = \frac{1}{n\rp{2}}	\sum_{i,j (dist)} \Omega_{ij} \asymp \rho_{\max}^2 \alpha_0 = \rho_{\max}^2  \gg \frac{ \sqrt{ \log n }}{n}
	\]
	by Assumptions \eqref{assn:mild_deg_het} and  \eqref{assn:diverging_degrees}. It follows that there exist absolute constants $c_0, c_\gamma, C_\gamma  > 0$ such that 
	\begin{align*}
		c_\gamma \rho_{\max}^2 < \hat \gamma  <C_\gamma  \rho_{\max}^2
		\num \label{eqn:event_E}
	\end{align*}
	with probability $n^{-c_0}$. Let $\mc{E}$ denote this event. Under $\mc{E}$, we have that for $\ti C$ sufficiently large,
	\[
	\ti C \hat \gamma N^2 h^{-1}\bigg( \frac{\ti C N \log(\frac{ne}{N})}{\hat \gamma N^2}  \bigg) 
	\geq \sigma_D^2 h^{-1} \bigg ( \frac{2N \log\frac{ne}{N}}{\sigma^2_D} \bigg)
	\]
	It follows from this, the union bound, and the Bennett inequality,
	\begin{align*}
		\p\bigg[ | \ti  \phi_{sc} | >   \ti C  \hat \gamma N^2 h^{-1}\bigg( \frac{\ti C N \log(\frac{ne}{N})}{\hat \gamma N^2}  \bigg) \bigg] 
		&\leq \p[\mc{E}^c ] + \p\bigg[ | \ti  \phi_{sc} | >   \ti C  \hat \gamma N^2 h^{-1}\bigg( \frac{\ti C N \log(\frac{ne}{N})}{\hat \gamma N^2}  \bigg),\, \, \mc{E}  \bigg] 
		\\&\leq n^{-c_0} + \sum_{|D| = N} \p\bigg[  |Z_D|   >  \ti C  \hat \gamma N^2 h^{-1}\bigg( \frac{\ti C N \log(\frac{ne}{N})}{\hat \gamma N^2}  \bigg)  \bigg]
		\\&\leq n^{-c_0} + \sum_{|D| = N} \p\bigg[ | Z_D |  >  \sigma_D^2 h^{-1} \bigg ( \frac{2N \log\frac{ne}{N}}{\sigma^2_D} \bigg)  \bigg]
		%&\leq n^{-c_0} + binom{n}{N} \cdot \exp\big(  - 2N \log\frac{ne}{N}  \big)
		\\&\leq n^{-c_0} + \big( \frac{ne}{N} \big)^N \exp\big( - 2N \log\frac{ne}{N}  \big).
	\end{align*}
	This shows that the type 1 error for the ideal scan statistic is $o(1)$.
	
	Next consider the type 2 error. We have by Lemma \eqref{lem:tilde_Omega},
	\begin{align*}
		\mf{1}_S^\T(  A - \eta^* \eta^{*\T} ) \mf{1}_S
		&= \sum_{i,j \in S (dist)} (A_{ij} - \Omega_{ij} )
		+ \mf{1}_S^\T \ti \Omega \mf{1}_S		= Z_S + \| \theta_S \|_1^2 (1 - b^2) \cdot \frac{ \| \theta_{S^c} \|_1^2}{ v_0 }. 
	\end{align*}
	%	First
	%	\begin{align*}
	%		\| \theta_S \|_1^2 (1 - b^2) \cdot \frac{ \| \theta_{S^c} \|_1^2}{ v_0 }
	%		\asymp \rho_{\max}^2 (a - \frac{\ti b^2}{c} ) N^2 \cdot \frac{  (n - N)^2 \rho_{\max}^2 c   }{   \rho_{\max}^2 \alpha_0 n^2  }
	%		\asymp \rho_{\max}^2 (a - \frac{\ti b^2}{c}) N^2. 
	%	\end{align*}
	Note that by \eqref{eqn:v0_vs_v}
	\begin{align*}
		\| \theta_S \|_1^2 (1 - b^2) \cdot \frac{ \| \theta_{S^c} \|^2_1}{ v_0 } \sim 
		\| \theta_S \|_1^2 (1 - b^2).
	\end{align*}
	Next,
	\begin{align*}
		\var(Z_S) = \sum_{i,j \in S (dist)} \Omega_{ij}(1-\Omega_{ij})
		\les \| \theta_S \|_1^2 \asymp \rho_{\max}^2 N a
		\sim  \rho_{\max}^2 N_0 a
	\end{align*}
	By Bernstein's inequality,
	\begin{align*}
		|Z_S| \les \sqrt{ \| \theta_S \|_1^2 \log(1/\delta) } \vee \log(1/\delta)
		\leq \| \theta_S \|_1 \log(1/\delta) 
	\end{align*} 
	with probability at least $1 - \delta$. Setting $\delta = (\frac{N}{ne})^{10}$, we have
	\[
	|Z_S| \les \| \theta_S \|_1 \log \big( \frac{ne}{N} \big) 
	\]
	with probability at least $1 - (\frac{N}{ne})^{10}$. 
	
	Next we show that 
	\begin{align*}
		\| \theta_S \|_1 |1 - b^2| \gtrsim  \log\frac{ne}{N}
		\num \label{eqn:scan_alt_var_bd}
	\end{align*}
	using \eqref{eqn:scan_SNR}, which we rewrite as
	\begin{align*}
		\| \theta_S \|_1^2 |1 - b^2|
		\gg \gamma N_0^2  h^{-1} \bigg(\frac{\log \frac{ne}{N_0}}{\gamma N_0}\bigg)
		\sim \gamma N^2  h^{-1} \bigg(\frac{\log \frac{ne}{N}}{\gamma N}\bigg)
		\num \label{eqn:scan_SNR_rewrite}
	\end{align*}
	where $\gamma = \rho_{\max}^2 \alpha_0$. Recall that $\alpha_0 = 1$ under the null, and $\alpha_0 \sim c$ under the alternative. Let
	\[
	u = \frac{\log \frac{ne}{N}}{\gamma N}. 
	\]
	Consider two cases: (i) $u \leq 0.01$,  
	and (ii) $u \geq 0.01$. For $u' \leq h^{-1}(0.01)$, we have $h(u') \asymp (u')^2$, and therefore $h^{-1}(u) \asymp u^2$ for $u \leq 0.01$. In this case \eqref{eqn:scan_SNR_rewrite} implies 
	\begin{align*}
		\| \theta_S \|_1^2 |1 - b^2|
		\gg \gamma N^2 \sqrt{ \frac{ \log \frac{ne}{N} }{ \gamma N}  }
		= \sqrt{ \gamma N^3 \log \frac{ne}{N} }.  
	\end{align*}
	In addition, 
	\[
	\| \theta_S \|_1 = N \sqrt{a} \rho_{\max},
	\]
	so that
	\begin{align*}
		\| \theta_S \|_1 (1 - b^2)
		\gg \sqrt{ \frac{   \gamma  N  \log \frac{ne}{N} }{a \rho_{\max}^2 } }
		\gtrsim \log \frac{ne}{N}
	\end{align*}
	since $u \leq 0.01$ and $a \rho_{\max}^2 \les 1$. Thus in case (i), \eqref{eqn:scan_alt_var_bd} is satisfied for $n$ sufficiently large. 
	
	Now consider case (ii) where $u \geq 0.01$. Note that $h(u) \leq (u+1)\log(u+1)$, and thus
	\[
	\frac{1}{2}(u + 1)  \leq u \leq h^{-1}( (u+1)\log(u+1) ). 
	\]
	Let $\vp \equiv	(u+1)\log(u+1) \geq u$ and observe that
	\[
	u + 1 = \frac{ \vp}{\log (u + 1) } \geq \frac{ \vp }{ \log \vp}. 
	\]
	Hence
	\[
	h^{-1}( (u+1)\log(u+1) ) \geq \frac{1}{2} \cdot  \frac{ (u+1)\log(u+1) }{\log\big[ (u+1)\log(u+1)\big]}. 
	\] 
	Applying \eqref{eqn:scan_SNR_rewrite}, 
	\begin{align*}
		\| \theta_S \|_1^2 |1 - b^2|
		&\gg  \gamma N^2 \cdot  \frac{ (\frac{\log \frac{ne}{N}}{\gamma N}+1)\log(\frac{\log \frac{ne}{N}}{\gamma N}+1) }{\log\big[ (\frac{\log \frac{ne}{N}}{\gamma N}+1)\log(\frac{\log \frac{ne}{N}}{\gamma N}+1)\big]}
		\gtrsim N \log \frac{ne}{N}. 
	\end{align*}	
	Hence
	\begin{align*}
		\| \theta_S \|_1 |1 - b^2|
		\gg \frac{  \log \frac{ne}{N} }{ \sqrt{a } \rho_{\max} } \gtrsim \log \frac{ne}{N}. 
	\end{align*}
	Thus in case (ii), \eqref{eqn:scan_alt_var_bd} is also satisfied. 
	
	Next we have, 
	\begin{align*}
		\p\bigg[ 
		| \ti \phi_{sc} | \leq 
		\ti C  \hat \gamma N^2 &h^{-1}\bigg( \frac{\ti C N \log(\frac{ne}{N})}{\hat \gamma N^2}  \bigg)
		\bigg]
		\\&\leq n^{-c_0} + 
		\p\bigg[ 
		| \ti \phi_{sc} | \leq 
		\ti C  \hat \gamma N^2 h^{-1}\bigg( \frac{\ti C N \log(\frac{ne}{N})}{\hat \gamma N^2}  \bigg), \, \mc{E} 
		\bigg]
		\\&\leq 
		n^{-c_0} + 
		\p\bigg[ \, \,
		\bigg| \| \theta_S \|_1^2 (1 - b^2) + Z_S \bigg| 
		\leq C \gamma N^2 h^{-1} \bigg(   \frac{C N \log(\frac{ne}{N})}{\gamma N^2}  \bigg) \,
		\bigg]
		\\& \leq n^{-c_0} 
		+ \p\bigg[
		|Z_S| \geq \big| \| \theta_S \|_1^2 (1 - b^2)  \big| - C \gamma N^2 h^{-1} \bigg(   \frac{C N \log(\frac{ne}{N})}{\gamma N^2}  \bigg)
		\bigg],
	\end{align*}
	where $C> 0$ is a sufficiently large absolute constant. In the second line and third lines we use the event $\mc{E}$ from \eqref{eqn:event_E}, and in the last line we use the triangle inequality. By \eqref{eqn:scan_SNR}, we have conservatively that 
	\[
	\big| \, \| \theta_S \|_1^2 (1 - b^2)  \big| - C \gamma N^2 h^{-1} \bigg(   \frac{C N \log(\frac{ne}{N})}{\gamma N^2}  \bigg) \geq \frac{1}{2} \big| \| \theta_S \|_1^2 (1 - b^2)  \big|
	\gg \| \theta_S \|_1 \log \frac{ne}{N}
	\]
	for $n$ sufficiently large. Thus for $n$ sufficiently large,
	\begin{align*}
		\p\bigg[ 
		| \ti \phi_{sc} | \leq 
		\ti C  \hat \gamma N^2 h^{-1}\bigg( \frac{\ti C N \log(\frac{ne}{N})}{\hat \gamma N^2}  \bigg)
		\bigg]
		&\leq n^{-c_0} + \p\bigg[
		|Z_S| \geq\frac{1}{2} \big| \, \| \theta_S \|_1^2 (1 - b^2) \big| \,
		\bigg]
		\\&\leq n^{-c_0} + \big( \frac{N}{ne} \big)^{10}. 
	\end{align*}
	Therefore the type 2 error for the ideal scan statistic is also $o(1)$. 
\end{proof}

\begin{lemma}
	Let $\phi_{sc}$ denote the scan statistic defined in \eqref{eqn:scan_stat_def}, and let $\hat \tau$ denote the random threshold defined in \eqref{eqn:tauhat_def}. Then under the null hypothesis,
	\[
	\p\big[ | \phi_{sc} | > \hat \tau \big] 
	\leq \binom{n}{N}^{-1} + v_0^{-c_1} +  
	n^{-c_0} + \exp\big( - N \log\frac{ne}{N}  \big),
	\]
	and under the alternative hypothesis,for $n$ sufficiently large we have
	\[
	\p\big[ | \phi_{sc} | < \hat \tau \big] 
	\leq \binom{n}{N}^{-1} + v_0^{-c_1} + n^{-c_0} + \big( \frac{N}{ne} \big)^{10}. 
	\]
	\label{lem:real_scan_test}
\end{lemma}

\begin{proof}
	%	Under null, we use that $h(u) \leq u^2$ and deal with plug-in effect. 
	
	We show that the plug-in effect is negligible compared to the threshold and signal-strength.
	
	By Lemma \ref{lem:scan_plugin}, 
	\[
	\max_{|D| = N}  | L_D |  \les \sqrt{ N_0^3 \gamma \log( \frac{ne}{N_0} )  }
	\]
	with high probability. Since $h(u) \leq u^2$ for $u \geq 0$, it follows that 
	\begin{align*}
		h\bigg(  \frac{ \sqrt{N_0^3 \gamma \log( \frac{ne}{N_0} )} }{\gamma N_0^2} \bigg) 
		&\leq \frac{ N_0^3 \gamma \log( \frac{ne}{N_0} )}{\gamma^2 N_0^4}	
		= \frac{ \log \frac{ne}{N_0} }{ \gamma N_0 } \Rightarrow
		\\ \sqrt{N_0^3 \gamma \log( \frac{ne}{N_0} )} 
		&\leq \gamma N_0^2 h^{-1} \bigg(   \frac{ \log \frac{ne}{N_0} }{ \gamma N_0 }    \bigg) \Rightarrow
		\\ \sqrt{N^3 \gamma \log( \frac{ne}{N} )} 
		&\leq [1+ o(1)]\gamma N^2 h^{-1} \bigg(   \frac{ \log \frac{ne}{N} }{ \gamma N }    \bigg).
	\end{align*}

	Under the null, we have by Lemma \ref{lem:scan_ideal} that
	\begin{align*}
		\p\big[
		| \phi_{sc} | \geq \hat \tau 
		\big] 
		&\leq 		\p\big[ \,
		| \ti \phi_{sc} | \geq \hat \tau - \max_{|D| = N} |L_D | \,
		\big] 
		\\& 	\leq \binom{n}{N}^{-1} + v_0^{-c_1} + 
		\p\bigg[
		| \ti \phi_{sc} | \geq  C^*  \hat \gamma N^2 h^{-1}\bigg( \frac{C^* N \log(\frac{ne}{N})}{\hat \gamma N^2}  \bigg) - \gamma N^2 h^{-1} \bigg(   \frac{ \log \frac{ne}{N} }{ \gamma N }    \bigg)
		\bigg] 
		\\&\leq \binom{n}{N}^{-1} + v_0^{-c_1} +  
		n^{-c_0} + \exp\big( - N \log\frac{ne}{N}  \big)
	\end{align*}
	for $C^*> 0$ a sufficiently large absolute constant. It suffices to take $C^* \geq 2\ti C$. 
	
	Under the alternative hypothesis, we have by Lemma  \ref{lem:scan_ideal} that 
	\begin{align*}
		\p\big[
		| \phi_{sc} | \leq \hat \tau 
		\big] 
		&\leq \p\big[ \,
		| \ti \phi_{sc} | \leq \hat \tau + \max_{|D| = N} |L_D | \,
		\big] 
		\\&\leq \binom{n}{N}^{-1} + v_0^{-c_1}
		+ \p\bigg[ \, | \ti \phi_{sc} | \leq 
		C^*  \hat \gamma N^2 h^{-1}\bigg( \frac{C^* N \log(\frac{ne}{N})}{\hat \gamma N^2}  \bigg) + \gamma N^2 h^{-1} \bigg(   \frac{ \log \frac{ne}{N} }{ \gamma N }    \bigg)
		\bigg]
		\\&\leq \binom{n}{N}^{-1} + v_0^{-c_1}
		+ \p\bigg[ \, | \ti \phi_{sc} | \leq 
		2C^*  \hat \gamma N^2 h^{-1}\bigg( \frac{C^* N \log(\frac{ne}{N})}{\hat \gamma N^2}  \bigg) 
		\bigg]
		\\&\leq \binom{n}{N}^{-1} + v_0^{-c_1} + n^{-c_0} + \big( \frac{N}{ne} \big)^{10}
	\end{align*}
	for $n$ sufficiently large.

\end{proof}

Observe that Theorem \ref{thm:scan} follows directly from Lemma \ref{lem:real_scan_test}. 

%%%%%%%%%%%%%%%%%%%%%%%%%%%%
%%%%%%%%%%%%%%%%%%%%%%%%%%%
%%%%%%%%%%%%%%%%%%%%%%%%%%%%%%

\section{Proof of Theorem \ref{thm:compLB} (Computational lower bound) }

In this section, we provide the proof of Theorem \ref{thm:compLB}. For convenience, we denote $b = \frac{nc-(a+c)N}{n-2N}, d = \frac{c(n-N)^2-a N^2}{n(n-2N)}$. Under $H_0$, all upper triangular entries $A$ are i.i.d. Bernoulli distributed with probability $d$. Then an orthonormal basis of the adjacency matrix of graph $D$ is
\begin{equation*}
	f_{\Gamma}(A) = \prod_{i<j: (i, j)\in \Gamma} \frac{A_{ij} - d}{\sqrt{d(1-d)}}.
\end{equation*}
Here, $\Gamma \subseteq \{(i, j): 1\leq i <j\leq n\}$ takes all subsets of all upper triagonal entries of $A$. Denote $|\Gamma|$ as the cardinality of $\Gamma$ and $B(D) = \left\{\Gamma \subseteq \{\text{unordered pairs } (i, j): i\neq j, i,j\in [n]\}, \Gamma\neq \emptyset, |\Gamma|\leq D\right\}$ as all subsets of off-diagonal entries of $A$ of cardinality at most $D$. By Proposition \ref{pr:low-degree} and the property of the orthonormal basis function of $A$,
\begin{equation*}
	\begin{split}
		& \sup_{\substack{f \text{ is polynomial; } \text{degree}(f) \leq D\\\mathbb{E}_{H_0}f(A) = 0; \textrm{Var}_{H_0}(A)=1}}\mathbb{E}_{H_1}f(A) = \|LR^{\leq D} - 1\| \\
		= & \left\{\sum_{\Gamma \in B(D)}  \left(\mathbb{E}_{H_0} f_{\Gamma}(A) (LR^{\leq D}(A) - 1)\right)^2\right\}^{1/2} \overset{(*)}{=}  \left\{\sum_{\Gamma \in B(D)} \left(\mathbb{E}_{H_0}f_{\Gamma}(A) LR(A) \right)^2\right\}^{1/2}\\
		= & \left\{\sum_{\Gamma \in B(D)} \mathbb{E}_{H_1} \left(f_{\Gamma}(A)\right)^2\right\}^{1/2}  {=} \left\{\sum_{\Gamma\in B(D)} \left(\mathbb{E}_{H_1} \prod_{(i, j)\in \Gamma} \frac{A_{ij} - d}{\sqrt{d(1-d)}} \right)^2\right\}^{1/2}.
	\end{split}
\end{equation*}
Here, $(*)$ is due to $\mathbb{E}_{H_0} f_{\Gamma}LR^{\leq D} = \mathbb{E}_{H_0} f_{\Gamma}LR$ by the property of projection and $\mathbb{E}_{H_0}f_{\Gamma}(A)=0$ for any $\Gamma\in B(D)$. Therefore, to establish the desired computational lower bound, we only need to prove
\begin{equation*}
	\sum_{\Gamma\in B(D)} \left(\mathbb{E}_{H_1} \prod_{(i, j)\in \Gamma} \frac{A_{ij} - d}{\sqrt{d(1-d)}} \right)^2 = o(1)
\end{equation*}
under the described asymptotic regime. For convenience, we denote
$$p_1 = \frac{a-d}{\sqrt{d(1-d)}}, \quad p_2 = \frac{b-d}{\sqrt{d(1-d)}}, \quad p_3 = \frac{c-d}{\sqrt{d(1-d)}}.$$
We can calculate that
\begin{equation*}
	a - d = \frac{(n-N)^2(a - c)}{n(n-2N)}, \quad b - d = - \frac{(n-N)N(a - c)}{n(n-2N)}, \quad c - d = \frac{N^2(a-c)}{n(2-2N)}.
\end{equation*}
and
\begin{equation}\label{eq:relation-gamma-delta}
	c-d = - \frac{N}{n-N}\left(b-d\right) = \left(\frac{N}{n-N}\right)^2\left(a-d\right).
\end{equation}
Since $b = \frac{c(n-N)-aN}{n-2N} \geq 0$ and $N\leq n/3$, we know $a \leq c(n-N)/N$ and 
$$c\geq d = \frac{c(n-N)^2-aN^2}{n(n-2N)} \geq \frac{c(n-N)^2 - N(n-N)c}{n(n-2N)} \geq (n-N)/n \cdot c\geq 2/3\cdot c.$$ Under the asymptotic regime of this theorem, we have $d = \frac{c(n-N)^2 - aN^2}{n(n-2N)}$ and
\begin{equation}\label{eq:p_1-asymptotic}
	p_1 = \frac{(n-N)^2(a-c)}{n(n-2N) \sqrt{d(1-d)}} \asymp \frac{a-c}{\sqrt{c}}, 
\end{equation}
i.e., there exists constant $\delta>1$ such that $\delta^{-1} c \leq p_1 \leq \delta c$. By \eqref{eq:relation-gamma-delta}, we have $p_3 = -N/(n-N) p_2 = N^2/(n-N)^2 p_1$. For any fixed $\Gamma \subseteq\{(i, j): 1\leq i<j\leq n\}$, 
\begin{equation*}
	\begin{split}
		& \mathbb{E}_{H_1}\prod_{(i, j)\in \Gamma} \frac{A_{ij} - d}{\sqrt{d(1-d)}} =  \mathbb{E}_{\Pi} \left\{ \mathbb{E}\left\{\prod_{(i,j)\in \Gamma} \frac{A_{ij}-d}{\sqrt{d(1-d)}}\Bigg|  \begin{array}{ll}\text{$A$ has two communities assigned by } \Pi
		\end{array}\right\}\right\}\\
		= & \mathbb{E}_{\Pi} p_1^{|\Gamma\cap K\otimes K|}\cdot p_2^{|\Gamma\cap K\otimes K^c|}\cdot  p_3^{|\Gamma\cap K^c \otimes K^c|} = \mathbb{E}_{\Pi}\prod_{(i, j)\in \Gamma} \left\{p_1\cdot \left(-N/(n-N)\right)^{\pi_i +\pi_j -2}\right\}\\
		= & p_1^{|\Gamma|}\cdot \left(\frac{-N}{n-N}\right)^{\sum_{(i,j)\in \Gamma} (\pi_i + \pi_j -2)} = p_1^{|\Gamma|}\cdot \left(\frac{-N}{n-N}\right)^{\sum_{(i,j)\in \Gamma} (\pi_i + \pi_j -2)}\\
		= & p_1^{|\Gamma|} \cdot \prod_{i=1}^n \left(\frac{-N}{n-N}\right)^{(\pi_i - 1)\cdot |\{j': (i, j')\in \Gamma\}|} \overset{(a)}{=} p_1^{|\Gamma|}\cdot \prod_{i=1}^n \left\{\left(\frac{N}{n}\right) +  \frac{n-N}{n}\left(\frac{-N}{n-N}\right)^{|\{j': (i, j')\in \Gamma\}|}\right\}.
	\end{split}
\end{equation*}
Here, (a) is because $\mathbb{P}(\pi_i =1) = N/n$; $\mathbb{P}(\pi_i =2) = (n-N)/n$. 
Thus, the following fact holds: if there exists a node $i$ that appears exactly one time in $\Gamma$, i.e., $|\{j': (i, j') \in \Gamma\}| = 1$, $\mathbb{E}_{H_1}\prod_{(i, j)\in \Gamma} \frac{A_{ij} - d}{\sqrt{d(1-d)}}=0$. On the other hand, for all $\Gamma$ that each node appear zero times or at least two times, we have
\begin{equation*}
	\begin{split}
		& \mathbb{E}_{H_1}\prod_{(i, j)\in \Gamma} \frac{A_{ij} - d}{\sqrt{d(1-d)}} \leq p_1^{|\Gamma|}\cdot \left\{\frac{N}{n} + \frac{n-N}{n}\left(\frac{-N}{n-N}\right)^{2}\right\}^{|\{i: i \text{ appears at least 2 times in }\Gamma\}|}\\ 
		\leq & p_1^{|\Gamma|} \cdot \left(\frac{2N}{n}\right)^{|\{i: i \text{ appears at least 2 times in }\Gamma\}|}.
	\end{split}
\end{equation*}
Finally, we denote 
$$B_0(D) = \left\{\Gamma\in B(D): \text{ each node in } [n] \text{ appears zero time or at least 2 times}\right\},$$
$$m(\Gamma) = |\{i: i \text{ appears in some pair of } \Gamma\}|.$$
For any $\Gamma\in B_0(D)$, we must have $m(\Gamma) \leq |\Gamma| \leq m(\Gamma)(m(\Gamma)-1)/2$. Then,
\begin{equation*}
	\begin{split}
		& \sum_{\Gamma\in B(D)} \left(\mathbb{E}_{H_1} \prod_{(i, j)\in \Gamma} \frac{A_{ij} - d}{\sqrt{d(1-d)}} \right)^2 = \sum_{\Gamma\in B_0(D)} \left(\mathbb{E}_{H_1} \prod_{(i, j)\in \Gamma} \frac{A_{ij} - d}{\sqrt{d(1-d)}} \right)^2\\
		= & \sum_{\Gamma\in B_0(D)} p_1^{2|\Gamma|} \cdot \left(\frac{2N}{n}\right)^{2|\{i: i \text{ appears at least 2 times in }\Gamma\}|} \leq \sum_{\Gamma\in B_0(D)} p_1^{2|\Gamma|} \cdot \left(\frac{2N}{n}\right)^{2m(\Gamma)}\\
		= & \sum_{m=2}^{D} \sum_{g=m}^{D\wedge m(m-1)/2} ~ \sum_{\substack{\Gamma \in B_0(D)\\ m(\Gamma)=m\\|\Gamma|=g}} p_1^{2g}\left(\frac{2N}{n}\right)^{2m}
		\overset{(a)}{\leq} \sum_{m=2}^D \sum_{g=m}^{D\wedge \frac{m(m-1)}{2}} \binom{n}{m} m^g p_1^g \left(\frac{2N}{n}\right)^m\\
		\leq & \sum_{m=2}^D \sum_{g=m}^{D\wedge \frac{m(m-1)}{2}} \frac{m^g p_1^{2g} (2N)^{2m}}{m!\cdot  n^m} \leq \sum_{m=2}^D \frac{D\max\left\{(mp_1^2)^{m}, (mp_1^2)^{D\wedge m(m-1)/2}\right\}\cdot (2N)^{2m}}{n^m}\\
		= & D\sum_{m=2}^D \left(\frac{\max\{mp_1^2, (mp_1^2)^M\}\cdot (2N)^2}{n}\right)^m \overset{(b)}{=} o(1)
	\end{split}
\end{equation*}
Here, $M = \max_{m\geq 1} \frac{D\wedge m(m-1)/2}{m} \leq \sqrt{D/2-1}$; (a) is because the number of $\Gamma\in B_0(D)$ with $m(\Gamma)=m$ and $|\Gamma|=g$ is at most $\binom{n}{m}\cdot m^g$; (b) is due to the asymptotic assumption and \eqref{eq:p_1-asymptotic}, which leads to
$$ \frac{N}{\sqrt{n}}\left(p_1 \vee p_1^M\right) \leq n^{-\varepsilon}.$$
We have thus finished the proof of this theorem. \quad $\square$

\begin{prop}[Proposition 1.15 of \cite{kunisky2019notes}]\label{pr:low-degree}
	Given data $A$, consider the simple hypothesis testing problem: $H_0$ versus $H_1$. Let the likelihood ratio function be $LR(A) = \frac{p_{H_1}(A)}{p_{H_0}(A)}$. Define $\|f\| = \sqrt{\mathbb{E}_{H_0}f^2(A)}$ and $f^{\leq D}$ as the projection of any function $f$ to the subspace of polynomials of degree at most $D$, i.e., $f^{\leq D} = \textrm{argmin}_{\substack{g\text{ is polynomial}\\ \text{degree}(g)\leq D}} \|f - g\|$. Then for any positive integer $D$, we have
	\begin{equation*}
		\|LR^{\leq D}(A)-1\| = \max_{\substack{f: \textrm{degree}(f)\leq D\\ \mathbb{E}_{H_0}f^2(A)=1\\ \mathbb{E}_{H_0}f(A)=0}} \mathbb{E}_{H_1} f(A);
	\end{equation*}
	\begin{equation*}
		\frac{LR^{\leq D}(A)-1}{\|LR^{\leq D}(A)-1\|} =  \textrm{argmax}_{\substack{f: \textrm{degree}(f)\leq D\\ \mathbb{E}_{H_0}f^2(A)=1\\ \mathbb{E}_{H_0}f(A)=0}} \mathbb{E}_{H_1} f(A).
	\end{equation*}
\end{prop}

\section{Proof of Theorem \ref{thm:EST} (Power of EST)}

%\textbf{Note:} We require a slight modification of the threshold of EST to prove it has full power. This modification will be incorporated to the main paper later.

The EST statistic is defined to be
\[
\phi_{EST}^{(v)}  \equiv \sup_{ |S| \leq v } \sum_{i,j \in S} A_{ij}, 
\]
and the EST is defined to be
\[
\vp_{EST} = \mf{1}\big[ \phi_{EST}^{(r)} \geq e \big],
\]
where $v, e$ are relatively prime and satisfy 
\[
\frac{\omega}{1 - \beta} < \frac{v}{e} <  \delta. 
\]

Such $v$ and $e$ exist because 
\[
\frac{\omega}{1 - \beta} < \delta, 
\]
by assumption. Furthermore, we have
\[
v < e
\]
since $\omega, \delta \in (0, 1)$. 

To prove the statement, we require some preliminaries. Let $G(n, p)$ denote an Erd\H{o}s-R\'{e}nyi graph with parameter $p$. A graph $H$ with $v$ vertices and $e$ edges is said to be \textit{balanced} if for all (not necessarily induced) subgraphs $H' \subset H$ with $v'$ vertices and $e'$ edges, it holds that 
\[
e/v > e'/v'. 
\]

Next, the power of EST hinges on two well-known facts from probabilistic combinatorics. The first concerns the appearance of an arbitrary graph $H$ in $G(n, p)$. 
\begin{thm}[Adapted from Theorem 4.4.2. of \cite{alon2016probabilistic}]
	\label{thm:H_witness}
	Let $H$ denote a graph with $v$ vertices and $e$ edges. Then if $p \ll n^{-v/e}$, the random graph $G(n, p)$ does not have $H$ as a subgraph, with high probability as $n \to \infty$.
	
	On the other hand, if $H$ is balanced and $p \gg n^{-v/e}$, the random graph $G(n, p)$ contains  $H$ as a subgraph, with high probability as $n \to \infty$. 
\end{thm} 

\begin{thm}[\cite{rucinski1986strongly,catlin1988graphs}]
	\label{thm:balance_existence} 
	There exists a balanced graph with $v$ vertices and $e$ edges if and only if $1 \leq v - 1 \leq e \leq \binom{v}{2}$. 
\end{thm}

%The previous theorem of \cite{rucinski1986strongly,catlin1988graphs} in fact holds for a notion known as `strongly balanced,' although we do not discuss this here. 

Now we continue the proof. Recall that $v$ and $e$ are integers chosen such that $\frac{\omega}{1 - \beta} < v/e < \delta$. 

%Since we have $v < e$, there exists a balanced graph $H$ with those parameters.

\textit{Type 1 error:} 
Observe that
\[
b = \frac{cn - (a + c)N}{n - 2N}
= c \cdot \frac{n - N}{n - 2N}
- a \cdot \frac{N}{n - 2N},
\]
and thus
\begin{align*}
	\alpha &= a \eps + b(1 - \eps) 
	= a \eps + (1 - \eps) \big(c \cdot \frac{n - N}{n - 2N}
	- a \cdot \frac{N}{n - 2N}\big)
	\\&= a \bigg( \frac{N}{n} - (1 - \eps) \frac{N}{n - 2N} \bigg)
	+ (1 - \eps) \cdot \frac{n - N}{n - 2N} \cdot c 
	= -a \cdot \frac{N^2}{n(n - 2N)}
	+ (1 - \eps) \cdot \frac{n - N}{n - 2N} \cdot c 
	\sim c.,
\end{align*}
where above we use that $a \eps \leq c$. 

Thus under the alternative, $A$ is distributed as Erd\H{o}s-R\'{e}nyi with parameter 
\[
\alpha  \sim c = n^{-\delta} 
\ll n^{-v/e},
\]
by our choice of $v$ and $e$. By the first part of Theorem \ref{thm:H_witness}, no subset of size $v$ of $A$ contains more than $e$ edges, with high probability as $n \to \infty$. 

To be more precise, there are a finite number of graphs $H_1, \ldots, H_{L}$ with $v$ vertices and at least $e$ edges, where $L$ is a constant depending only on $v$. For each graph $H_i$, Theorem \ref{thm:H_witness} contains $H_i$ as a subgraph with probability tending $0$ as $n \to \infty$. The type 1 error of EST thus vanishes by the  union bound. 

\textit{Type 2 error:} Let $H$ denote a balanced graph on $v$ vertices and $e$ edges, whose existence is guaranteed by Theorem \ref{thm:balance_existence}. Consider the induced subgraph on $\mathcal{C}_1$, the smaller community, which is an Erd\H{o}s-R\'{e}nyi random graph on $N$ vertices with parameter $a = n^{-\omega}$.  By our choice of $v$ and $e$, we have
\[
a = n^{-\omega} = N^{-\frac{\omega}{1 - \beta}}
\gg N^{-v/e}.
\]
By Theorem \ref{thm:H_witness}, $\mathcal{C}_1$ contains a copy of $H$ with high probability. Since $H$ has $e$ edges, we conclude that $\phi_{EST}\rp{v} \geq e$, and thus the null is rejected with high probability as $n \to \infty$. 

%HERE 

\bibliography{bibliography}
\bibliographystyle{iclr2023_conference}

\end{document}